\documentclass{article}

\usepackage[utf8]{inputenc}
\usepackage[stable]{footmisc}
\usepackage{amsmath}  
\usepackage{graphicx} 
\usepackage{amssymb,amsthm}
\usepackage{bm,bbm}
\usepackage{booktabs}
\usepackage{dsfont}
\usepackage{graphicx}
\usepackage{subcaption}
\usepackage{blindtext}
\usepackage{url}
\usepackage{algorithm,algorithmic}
\usepackage{setspace}
\usepackage{enumitem}
\usepackage{thmtools}
\usepackage{thm-restate}
\usepackage{authblk}
\usepackage{comment}
\usepackage{geometry}
\usepackage{array}

\usepackage[normalem]{ulem}

\newcommand{\dummy}{\mathord{\color{black!33}\bullet}}

\providecommand{\bbN}{\mathbb{N}}

\providecommand{\bbR}{\mathbb{R}}
\newcommand{\R}{\mathbb{R}}

\providecommand{\bbE}{\mathbb{E}}
\newcommand{\E}{\mathbb{E}}

\newcommand{\minL}{\underbar{L}}

\providecommand{\CE}{G}
\providecommand{\CN}{\mathcal{N}}
\providecommand{\CD}{\mathcal{D}}

\providecommand{\CC}{\mathcal{C}}

\providecommand{\CP}{\mathcal{P}}

\providecommand{\CV}{\mathcal{V}}

\DeclareMathOperator*{\Law}{Law}

\DeclareMathOperator*{\diag}{diag}

\providecommand{\Id}{\mathrm{Id}}
\providecommand{\argmin}{\operatorname*{\arg\min}}

\newcommand{\supp}[1]{\operatorname{supp}(#1)} 
\providecommand{\dist}{\operatorname{dist}}

\newcommand{\thetaG}{\theta_{\mathrm{good}}^*}
\newcommand{\thetaGtilde}{\tilde\theta_{\mathrm{good}}}
\newcommand{\mAlphaBeta}[1]{m^{G,L}_{\alpha, \beta}(#1)}
\newcommand{\mAlphaBetaRED}[1]{m(#1)}
\newcommand{\mAlphaBetanoarg}{m^{G,L}_{\alpha, \beta}}

\newcommand{\mAlpha}[1]{m^G_{\alpha}(#1)}
\newcommand{\mAlphanoarg}{m^{G}_{\alpha}}

\newcommand{\Ibeta}[1]{I^L_{\beta}[#1]}
\newcommand{\qbeta}[1]{q^L_{\beta}[#1]}
\newcommand{\qbetanoarg}{q^L_{\beta}}
\newcommand{\qbetahalf}[1]{q^L_{\beta/2}[#1]}
\newcommand{\qa}[1]{q^L_{a}[#1]}

\newcommand{\Qbeta}[1]{Q^L_{\beta}[#1]}
\newcommand{\m}[1]{m(#1)}
\newcommand{\NN}{\mathbb{N}}

\newcommand{\Lloc}{L_{\text{loc}}^2}

\newcommand{\nmini}{n_{\text{batch}}}
\newcommand{\alphaInit}{\alpha_{\text{init}}}
\newcommand{\sigmaEpoch}{\sigma_{\text{epoch}}}
\newcommand{\LM}{\chi}
\newcommand{\betaEpoch}{\beta_{\text{epoch}}}
\newcommand{\DF}{\kappa}
\newcommand{\betaMin}{\beta_{\text{min}}}

\providecommand*{\abs}[1]{\left|{#1}\right|} 
\providecommand*{\absbig}[1]{\big|{#1}\big|} 
\providecommand*{\N}[1]{\left\|{#1}\right\|} 
\providecommand*{\Nnormal}[1]{\|{#1}\|} 
\providecommand*{\Nbig}[1]{\big\|{#1}\big\|} 
\providecommand*{\NBig}[1]{\Big\|{#1}\Big\|} 

\newtheorem{theorem}{Theorem}[section]
\newtheorem{proposition}[theorem]{Proposition}
\newtheorem{lemma}[theorem]{Lemma}

\newtheorem{assumption}[theorem]{Assumption}

\newtheorem{definition}[theorem]{Definition}
\theoremstyle{remark}

\newtheorem{remark}[theorem]{Remark}

\numberwithin{equation}{section}

\newcommand{\overbar}[1]{\makebox[0pt]{$\phantom{#1}\mkern 1.5mu\overline{\mkern-1.5mu\phantom{#1}\mkern-1.5mu}\mkern 1.5mu$}#1}
\renewcommand{\underbar}[1]{\makebox[0pt]{$\phantom{#1}\mkern 1.5mu\underline{\mkern-1.5mu\phantom{#1}\mkern-1.5mu}\mkern 1.5mu$}#1}

\newcommand{\divergence}{\textrm{div}}

\newcommand{\minobj}{\underbar G}

\newcommand{\omegaa}[0]{\omega_{\alpha}^G}

\newcommand{\indivmeasure}[0]{\varrho} 

\newif\ifrevisedOne

\revisedOnetrue
\newif\ifrevisedTwo

\revisedTwotrue

\usepackage{cite} 
\usepackage{hyperref}
\hypersetup{
	final,
    colorlinks=true,
    linkcolor=blue,
    filecolor=magenta,
    urlcolor=cyan,
}

\usepackage{blindtext}
\usepackage{url}
\usepackage{xcolor}

\definecolor{darkgreen}{rgb}{0,0.4,0}

\title{{\usefont{OT1}{bch}{b}{n}
	\huge CB\texorpdfstring{\textsuperscript{2}}{2}O: Consensus-Based Bi-Level Optimization}}

\date{}

\author[1]{Nicol\'as Garc\'ia Trillos\thanks{Email: \texttt{garciatrillo@wisc.edu}}}
\author[1]{Sixu Li\thanks{Email: \texttt{sli739@wisc.edu}}}
\author[2]{Konstantin Riedl\thanks{Email: \texttt{Konstantin.Riedl@maths.ox.ac.uk}}}
\author[3]{Yuhua Zhu\thanks{Email: \texttt{yuhuazhu@ucla.edu}}}

\affil[1]{University of Wisconsin-Madison, Department of Statistics}
\affil[2]{University of Oxford, Mathematical Institute}
\affil[3]{University of California, Los Angeles, Department of Statistics and Data Science}

\begin{document}

\maketitle

\begin{abstract}
    \noindent
    Bi-level optimization problems, where one wishes to find the global minimizer of an upper-level objective function over the globally optimal solution set of a lower-level objective, arise in a variety of scenarios throughout science and engineering, machine learning, and artificial intelligence. In this paper, we propose and investigate, analytically and experimentally, consensus-based bi-level optimization (CB\textsuperscript{2}O), a multi-particle metaheuristic derivative-free optimization method designed to solve bi-level optimization problems when both objectives may be nonconvex. Our method leverages within the computation of the consensus point a carefully designed particle selection principle implemented through a suitable choice of a quantile on the level of the lower-level objective, together with a Laplace principle-type approximation w.r.t.\@ the upper-level objective function, to ensure that the bi-level optimization problem is solved in an intrinsic manner. We give an existence proof of solutions to a corresponding mean-field dynamics, for which we first establish the stability of our consensus point w.r.t.\@ a combination of Wasserstein and $L^2$ perturbations, and consecutively resort to PDE considerations extending the classical Picard iteration to construct a solution. For such solution, we provide a global convergence analysis in mean-field law showing that the solution of the associated nonlinear nonlocal Fokker-Planck equation converges exponentially fast to the unique solution of the bi-level optimization problem provided suitable choices of the hyperparameters. The practicability and efficiency of our CB\textsuperscript{2}O algorithm is demonstrated through extensive numerical experiments in the settings of constrained global optimization, sparse representation learning, and robust (clustered) federated learning.
\end{abstract}



\tableofcontents

\section{Introduction}
\label{sec:intro}
A variety of problems arising in science and engineering, machine learning and artificial intelligence as well as economics and operations research are naturally posed as nonconvex bi-level optimization problems of the form
\begin{equation}
    \label{eq:bilevel_opt}
    \thetaG := \argmin_{\theta^*\in \Theta} G(\theta^*)
    \quad \text{s.t.\@}\quad
    \theta^* \in \Theta := \argmin_{\theta \in \bbR^d} L(\theta),
\end{equation}
where the aim is to minimize an upper-level objective function~$G$ over the solution set~$\Theta$ of a lower-level problem represented by the objective $L$~\cite{dempe2010optimality, beck2014first, gong2021Biojective,jiang2023conditional}.
Due to the widespread relevance of this type of problem, it has attracted great attention across several disciplines. In machine learning and artificial intelligence, for example,
this type of problem appears in hyperparameter optimization~\cite{franceschi2018bilevel},
continual learning \cite{chaudhry2018efficient},
lexicographic optimization \cite{gong2021Biojective},
machine learning fairness \cite{zafar2017fairness},
adversarial robustness~\cite{trillos2024attack},
neural architecture search~\cite{MR3948095},
as well as
ill-posed or over-parameterized machine learning~\cite{jiang2023conditional}, just to give a few examples.
In science and engineering,
this problem formulation appears in
compressed sensing~\cite{donoho2006compressed},
optimal control,
resource allocation~\cite{wogrin2020applications}, among many other settings. For a more exhaustive survey of examples and applications of such bi-level optimization problems, we refer the reader to~\cite{dempe2010optimality,liu2022investigating} as well as references therein.
Problems of the type~\eqref{eq:bilevel_opt} are commonly referred to as simple bi-level optimization problems~\cite{dempe2010optimality,beck2014first,sabach2017first,dutta2020algorithms,shehu2021inertial,jiang2023conditional,samadi2024achieving,wang2024near,cao2024accelerated}, and they comprise a relevant subclass of general bi-level optimization problems,
where in general the lower-level objective may be parametrized by some arguments of the upper-level objective function.

In several of the aforementioned applications, solving such problem may be viewed as enforcing a model selection principle in a setting where several models perform equally well w.r.t. the objective function~$L$,
while only one amongst them grants desirable properties beyond the lower-level objective. Such properties may be related to robustness, fairness, sparsity, efficacy, and are encoded through the upper-level objective function~$G$. Mathematically speaking, the set of feasible solutions to the upper-level problem~$G$ is given by the set of global minimizers of the lower-level objective function~$L$.
Throughout this paper, both the upper- and lower-level objectives $G$ and $L$, respectively, are assumed to be real-valued functions on $\bbR^d$. 
In accordance with this description, the lower-level objective function $L$ is typically a nonconvex function that admits multiple global minimizers~$\theta^*$. We denote $L$'s associated set of minimizers by $\Theta := \argmin_{\theta \in \bbR^d} L(\theta)$ and the value of $L$ at those points by $\minL$. We further assume that, as imposed by the problem formulation in \eqref{eq:bilevel_opt}, only one of the elements in $\Theta$ is optimal w.r.t.\@ the upper-level objective function $G$. 
That is, the function $G$ is assumed to have a unique minimizer within the set $\Theta$, which we henceforth denote by $\thetaG$. For an illustration of the problem setting, we refer to Figure~\ref{fig:CB2OIllustration} below.
\begin{figure}[!htb]
    \centering
    \includegraphics[trim=96 78 72 76,clip, width=0.8\linewidth]{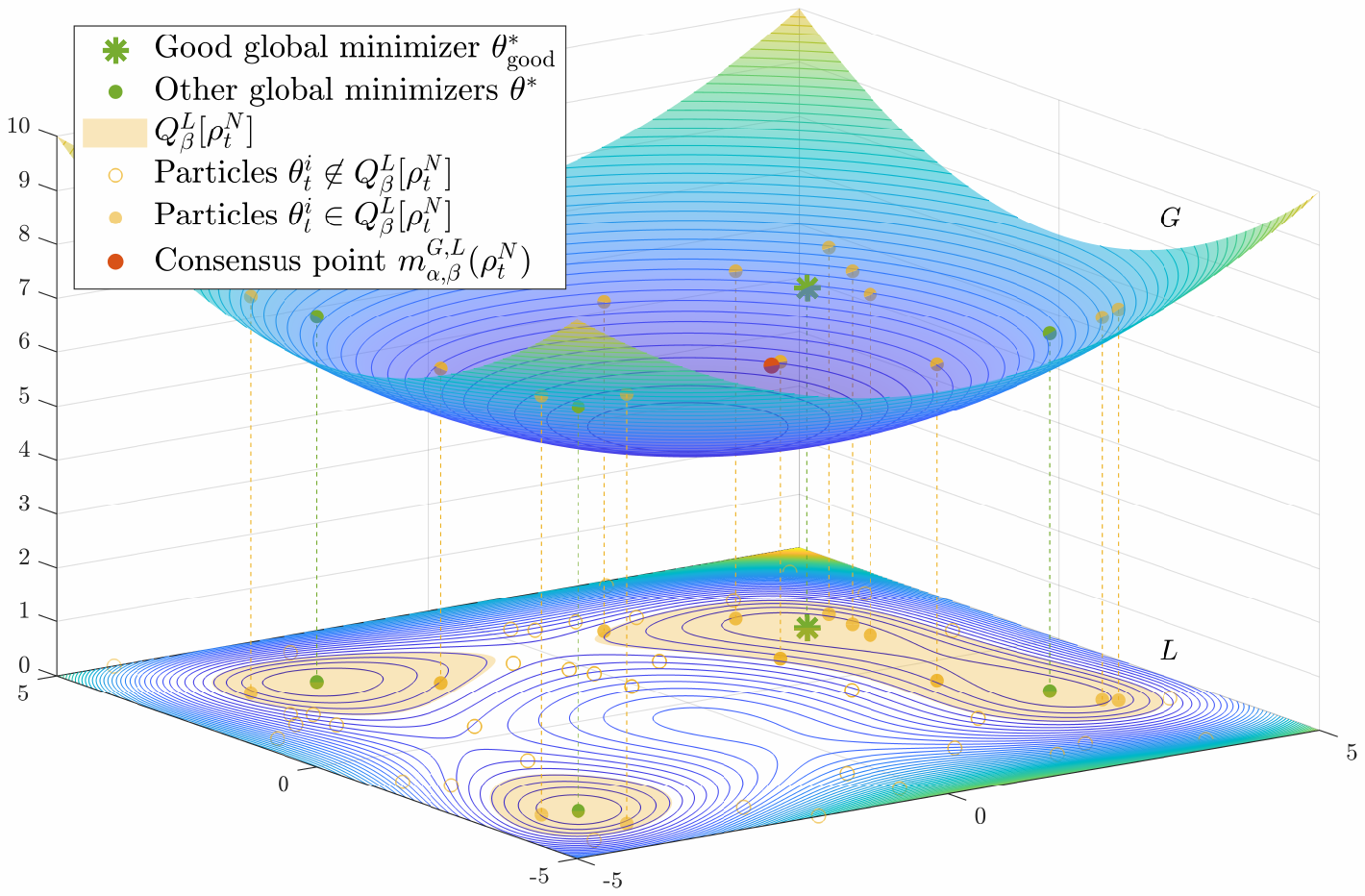}
    \caption{An illustration of the CB\textsuperscript{2}O algorithm~\eqref{eq:dyn_micro} and its working principles for solving bi-level optimization problems of the form~\eqref{eq:bilevel_opt}. 
    We depict a typical setting of~\eqref{eq:bilevel_opt} in two dimensions with the lower-level objective function~$L$ being a Himmelblau function (plotted as contours in the $xy$-plane) and with the upper-level objective function~$G$ being a parabola (plotted as a surface)\footnotemark. The set of global minimizers of $L$, i.e., the set $\Theta$, consists of the three green dots and the one green star that are plotted on the $xy$ plane. The green star identifies \textit{the} global minimizer~$\thetaG$ of $L$ which is optimal w.r.t.\@ the upper-level objective function~$G$ among the points in $\Theta$, and thereby is the solution to the bi-level optimization problem~\eqref{eq:bilevel_opt}. 
    The setting in this example satisfies the assumptions made later in the paper for our theoretical analysis.\\ At every point in time, CB\textsuperscript{2}O employs $N=50$ particles (depicted as yellow points, including both circles and filled dots on the $xy$-plane) and computes the consensus point~$\mAlphaBeta{\rho^N_t}$, which all particles are attracted to as they explore the space through noise (not depicted here). In order to compute the aforementioned consensus point~$\mAlphaBeta{\rho^N_t}$, each particle first evaluates the lower-level objective function $L$. For a chosen quantile parameter~$\beta=0.25$, the $\beta N$ w.r.t.\@ $L$ best positioned particles (depicted as the yellow filled dots) are selected. They belong to the set $\Qbeta{\rho^N_t}$ defined in \eqref{eq:Qbeta_FiniteParticles}, which serves as an approximation to the set~$\Theta$ of global minimizers of $L$. Based on those points, the consensus point $\mAlphaBeta{\rho^N_t}$ is computed, while the remaining particles (depicted as the yellow circles) are discarded from the computation of the consensus point. Figuratively speaking, those particles that belong to the quantile set $\Qbeta{\rho^N_t}$ are lifted to the upper-level objective function~$G$, which they evaluate in a second step. For a chosen weight parameter $\alpha=10$, the consensus point $\mAlphaBeta{\rho^N_t}$ is then computed as in \eqref{eq:ConsensusPoint_FiniteParticles}. It approximates the global minimizer~$\thetaG$ of the bi-level optimization problem~\eqref{eq:bilevel_opt}.}
    \label{fig:CB2OIllustration}
\end{figure}
\footnotetext{We ask the reader to appreciate the placement of the objective functions~$L$ and $G$ in Figure~\ref{fig:CB2OIllustration} in analogy with their intuitive interpretation, i.e., $L$ as the lower- and $G$ as the upper-level objective.}

Bi-level optimization problems of the form~\eqref{eq:bilevel_opt} come with two fundamental challenges.
First, there is typically no simple characterization or explicit formula for the set~$\Theta$ of global minimizers of the lower-level objective function~$L$.
Because of this, it is impractical to directly employ projection-based or projection-free optimization methods, since projecting onto or solving a minimization problem over such an implicitly defined feasible set is not straightforward.
Second, both the lower- and upper-level objective functions~$L$ and $G$ may potentially be nonconvex, as is typical in classical engineering or modern machine learning problems.
As a result, many existing methods developed for convex bi-level optimization problems might become less effective in a more challenging, non-convex setting (see also the discussions on related works in Section~\ref{sec:related_work_SBO}). The potential lack of convexity in one of the objectives also renders approaches such as unconstrained optimization of a weighted sum of the two objectives difficult to implement since it is in general challenging to tune weights to prevent the scheme from tilting towards arbitrary global minimizers of either $G$ or $L$. Related to this, we note that problem formulation~\eqref{eq:bilevel_opt} is invariant to outer composition of the lower objective with an increasing function. Under such transformations of the lower-objective, the geometry of the set $\Theta$ is unchanged, suggesting that a desirable algorithm should be stable to such perturbations of the problem. 



To address the aforementioned difficulties,
we design an interacting particle-based method, \textbf{C}onsensus-\textbf{B}ased \textbf{B}i-level \textbf{O}ptimization (CB\textsuperscript{2}O),
which is tailored to solving the optimization problem \eqref{eq:bilevel_opt}, i.e., finding the target good global minimizer $\thetaG$ of $L$ w.r.t.\@ $G$.
Our approach is inspired by consensus-based optimization (CBO) for standard optimization problems as introduced in \cite{pinnau2017consensus}, but enhanced by a suitable selection principle.
The method's working principles are described in what follows and illustrated in Figure~\ref{fig:CB2OIllustration}.
For this purpose, let us start by establishing some notation.
CB\textsuperscript{2}O methods employ a finite number of agents $\theta^1, \dots, \theta^N$,
which are formally stochastic processes, i.e., $\theta^i=(\theta^i_t)_{t\geq0}$ for $i=1,\dots,N$,
to explore the domain and to form a global consensus about the location of $\thetaG$ as time passes.
For user-specified parameters $\alpha,\beta, \lambda, \sigma > 0$,
the time-evolution of the $i$-th agent is defined by the stochastic differential equation (SDE)
\begin{equation}
\label{eq:dyn_micro}
\begin{split}
    d\theta_t^i
    &=
    -\lambda \left(\theta_t^i - \mAlphaBeta{\rho_t^N} \right)dt + \sigma D\left(\theta_t^i - \mAlphaBeta{\rho_t^N}\right) dB_t^i, \\
    \theta_0^i
    &\sim \rho_0,
\end{split}
\end{equation}
where $((B_t^i)_{t\geq0})_{i=1,\dots,N}$ are independent standard Brownian motions in $\bbR^d$.
The philosophy behind the design of the dynamics is in the spirit of many metaheuristics~\cite{blum2003metaheuristics,kennedy1995particle,aarts1989simulated},
which orchestrate an interaction between local improvement procedures and global strategies by combining deterministic and random decisions to create a process capable of escaping from local optima and performing a robust search in the solution space (see \cite[Section~1]{fornasier2021consensus} for more background on this quite popular class of optimization methods).
The dynamics of the particles $\theta^1, \dots, \theta^N$ in \eqref{eq:dyn_micro} are governed by two competing terms at each point in time.
A drift term drags each agent towards an instantaneous consensus point,
denoted by $\mAlphaBetanoarg$,
which is computed as a weighted average of the particles’ positions and acts as a momentaneous proxy for $\thetaG$. We will elaborate in a moment on how this consensus point is computed, given that it encapsulates the foundational innovations of the method and is therefore central to it. 
The second term is stochastic and with it we diffuse agents according to a scaled Brownian motion in $\bbR^d$, featuring the exploration of the energy landscape of the objective functions~$L$ and $G$.
Two commonly studied diffusion types are isotropic~\cite{pinnau2017consensus,carrillo2018analytical,fornasier2021consensus} and anisotropic~\cite{carrillo2019consensus,fornasier2021convergence} diffusion, which correspond to
\begin{equation} \label{eq:diffustion_types}
	D\!\left(\dummy\right) =
	\begin{cases}
		\N{\dummy}_2 \Id ,& \text{for isotropic diffusion,}\\
		\diag\left(\dummy\right)\!,& \text{for anisotropic diffusion},
			\end{cases}
\end{equation} 
where $\Id\in\bbR^{d\times d}$ is the identity matrix and $\diag:\bbR^{d} \rightarrow \bbR^{d\times d}$ the operator mapping a vector onto a diagonal matrix with the vector as its diagonal.
Ideally, as result of the described drift-diffusion mechanism, the agents should eventually achieve a near optimal global consensus, in the sense that the associated empirical measure $\rho^N_t:=\frac{1}{N}\sum_{i=1}^N\delta_{\theta^i_t}$ should converge in time to a Dirac delta at some point close to $\thetaG$.

In order to compute the consensus point~$\mAlphaBeta{\rho^N_t}$ at time $t$,
we first and foremost evaluate the lower-level objective function~$L$ for each of the $N$ particles, i.e., $\{L(\theta^i_t)\}_{i=1}^N$,
and arrange them in a non-decreasing order
\begin{equation}
    L (\theta_t^{\#1})
    \leq L (\theta_t^{\#2})
    \leq \cdots \leq L (\theta_t^{\#N}),
\end{equation}
where $\theta_t^{\#j}$ denotes the position of the particle with the $j$-th smallest lower-level objective function value at time $t$ within the set $\{L(\theta_t^i)\}_{i=1}^N$.
For $\beta \in (0,1]$, let us denote by $\lceil \beta N \rceil$ the smallest integer greater than or equal to $\beta N$.
We then define the $\beta$-quantile of the set $\{L(\theta_t^i)\}_{i=1}^N$ as
\begin{equation}\label{eq:qbeta_FiniteParticles}
    \qbeta{\rho^N_t}
    := L\big(\theta_t^{\#\lceil \beta N \rceil}\big).
\end{equation}
This means that $\qbeta{\rho^N_t}$ is the value for which $\lceil\beta N\rceil$ particles of $\{\theta^i\}_{i=1}^N$ attain a lower-level function value less than or equal to $\qbeta{\rho^N_t}$.
Intuitively speaking, the $\beta$-quantile $\qbeta{\rho_t^N}$ can be regarded as a proxy for $\minL$,
the minimum value of objective function $L$, given the currently available information amongst the particles.
This approximation becomes more and more accurate (still constrained by the currently available information) as $\beta$ becomes smaller and smaller.
The consensus point~$\mAlphaBeta{\rho_t^N}$ is now defined as a weighted average of precisely those particle positions, or put into formulas as 
\begin{equation}
\label{eq:ConsensusPoint_FiniteParticles}
\begin{split}
    \mAlphaBeta{\rho_t^N} := \sum_{\theta_t^i: L(\theta_t^i)\leq \qbeta{\rho_t^N}} \theta_t^i \frac{\omegaa(\theta_t^i)}{\sum_{\theta_t^j: L(\theta_t^j)\leq \qbeta{\rho_t^N}} \omegaa(\theta_t^j)},
    \qquad \text{with} \quad
    \omegaa(\theta) := \exp \left( -\alpha G(\theta) \right),
\end{split}
\end{equation}
where the $\beta$-quantile $\qbeta{\rho_t^N}$ of the full particle set $\{\theta_t^i\}_{i=1}^N$ w.r.t.\@ $L$ is as defined in \eqref{eq:qbeta_FiniteParticles}.
In order to facilitate the upcoming discussion and analysis,
let us define the sub-level set of the lower-level objective function $L$ w.r.t.\@ the $\beta$-quantile $\qbeta{\rho_t^N}$ defined in \eqref{eq:qbeta_FiniteParticles} as
\begin{equation}
\label{eq:Qbeta_FiniteParticles}
    \Qbeta{\rho_t^N}
    := \left\{ \theta \in \bbR^d : L(\theta) \leq \qbeta{\rho_t^N}  \right\}.
\end{equation}
In words, $\Qbeta{\rho_t^N}$ encompasses those regions of $\bbR^d$ where the lower-level objective function $L$ attains values less than or equal to the $\beta$-quantile $\qbeta{\rho_t^N}$.
Recalling that $\qbeta{\rho_t^N}$ can be thought of as a proxy for $\minL$, the set $\Qbeta{\rho_t^N}$ can therefore be regarded as an approximation of the neighborhood of the set $\Theta$ of global minimizers of $L$.
As the number $N$ of particles becomes large,
the particles are expected to effectively cover this neighborhood in full,
with the approximation becoming more and more refined as $\beta$ becomes smaller and smaller; see Remark~\ref{remark:beta0}. With this notation, the consensus point~$\mAlphaBeta{\rho_t^N}$ in \eqref{eq:ConsensusPoint_FiniteParticles} can be rewritten as
\begin{equation}\label{eq:ConsensusPoint_FiniteParticles_II}
    \mAlphaBeta{\rho_t^N} = \sum_{\theta_t^i: \theta_t^i \in \Qbeta{\rho_t^N}} \theta_t^i \frac{\omegaa(\theta_t^i)}{\sum_{\theta_t^j: \theta_t^j \in \Qbeta{\rho_t^N}} \omegaa(\theta_t^j)}.
\end{equation}
Note that the set $\Qbeta{\rho_t^N}$, and therefore also the consensus point $ \mAlphaBeta{\rho_t^N} $, is unchanged if $L$ is composed with an arbitrary increasing function and it therefore shares the same type of invariance with the set $\Theta$.

Unlike in standard CBO~\cite{pinnau2017consensus} for classical optimization,
where the consensus point is defined as a weighted average of \textit{all} particle positions,
our consensus point $\mAlphaBeta{\rho_t^N}$, as defined in \eqref{eq:ConsensusPoint_FiniteParticles_II}, is obtaiend as an average over a subset of particles from $\{\theta_t^i\}_{i=1}^N$,
namely those belonging to the sub-level set $\Qbeta{\rho_t^N}$.
While by construction the consensus point is supported on the set $\Qbeta{\rho_t^N}$ and computed from positions with good lower-level objective values,
this does not yet mean that the consensus point itself will lie within the set $\Qbeta{\rho_t^N}$ or have a good lower-level objective value~$L$ for some $\alpha$.
However, since the Laplace principle~\cite{dembo2009large,miller2006applied}
ensures that for any absolutely continuous probability distribution $\indivmeasure$ on $\bbR^d$ we have
\begin{align}
\label{eq:laplace_principle}
    \lim\limits_{\alpha\rightarrow \infty}\left(-\frac{1}{\alpha}\log\left(\int\omegaa(v)\,d\indivmeasure(v)\right)\right) = \inf\limits_{v \in \supp\indivmeasure}G(v),
\end{align}
the consensus point $\mAlphaBeta{\rho_t^N}$ can be regarded as a proxy for the global minimizer of the upper-level objective function~$G$ restricted to the set $\Qbeta{\rho_t^N}$.
This approximation improves as $\alpha$ becomes larger; see Remark~\ref{remark:alpha0}.
The combination and interplay between these two mechanisms, each of them related to the hyperparameters~$\beta$ and $\alpha$, respectively, ensure that the consensus point $\mAlphaBeta{\rho_t^N}$ is a proxy for the good global minimizer~$\thetaG$, i.e., the global minimizer of $G$ within the set $\Theta$ of global minimizers of $L$.

Apart from the crucial difference in the definition of the consensus point introduced above and the one appearing in the standard CBO setting (difference motivated by the specific goal of solving problem~\eqref{eq:bilevel_opt}),
the CB\textsuperscript{2}O system~\eqref{eq:dyn_micro} closely resembles the dynamics of the CBO method for classical optimization introduced in \cite{pinnau2017consensus}. For notational convenience, when defining the CB\textsuperscript{2}O system in a more generic way later on, and to better align with the conventional notation in the CBO literature \cite{pinnau2017consensus, carrillo2018analytical,fornasier2021consensus},
we introduce the measure
\begin{equation}
    \label{eq:Ibeta_FiniteParticles}
    \Ibeta{\rho_t^N} :=  \mathds{1}_{\Qbeta{\rho_t^N}} \rho_t^N,
\end{equation}
which is the (unnormalized) empirical measure of those particles of $\{\theta_t^i\}_{i=1}^N$ that belong to the quantile set $\Qbeta{\rho_t^N}$, i.e., the particles that are used to compute the consensus point.
With the definition of $\Ibeta{\rho_t^N}$,
we can write the consensus point $\mAlphaBeta{\rho_t^N}$ in integral form as
\begin{equation}
    \label{eq:ConsensusPoint_FiniteParticles_III}
    \mAlphaBeta{\rho_t^N}
    := \int \theta \frac{\omegaa(\theta)}{\|\omegaa\|_{L^1(\Ibeta{\rho_t^N})}} d\Ibeta{\rho_t^N}(\theta)
\end{equation}
with normalization factor $\Nnormal{\omegaa}_{L^1(\Ibeta{\rho_t^N})} := \int \omegaa(\theta) \,d\Ibeta{\rho_t^N}(\theta)$.
Let us remark in particular that $ \mAlphaBeta{\rho_t^N} = \mAlpha{\Ibeta{\rho_t^N}}$, where $\mAlphanoarg$ is just the definition of the consensus point in classical CBO for optimization, see, e.g., \cite[Equation~(1.5)]{fornasier2021consensus}.

To wrap up, by following the dynamics \eqref{eq:dyn_micro},
the particles are attracted towards a position, namely the consensus point~$\mAlphaBeta{\rho_t^N}$, which has a small value of the upper-level objective function~$G$ while simultaneously lying \textit{within} the regions where the lower-level objective function~$L$ is small. This intuition aligns with our goal of finding $\thetaG$, which is \textit{the} global minimizer of $L$ that has the smallest objective function value w.r.t.\@ $G$.
The mechanisms behind computing the consensus point~$\mAlphaBeta{\rho_t^N}$ are illustrated in Figure~\ref{fig:CB2OIllustration}.

A theoretical analysis of the CB\textsuperscript{2}O dynamics can either be done directly at the level of the microscopic interacting particle system~\eqref{eq:dyn_micro},
or, following the recent and popular vein of works \cite{carrillo2018analytical,carrillo2019consensus,fornasier2021consensus,fornasier2021convergence,grassi2021mean,riedl2024perspective,riedl2022leveraging}, by analyzing the macroscopic behavior of agents density~$\rho=(\rho_t)_{t\geq0}$ through a mean-field limit associated with the particle-based dynamics~\eqref{eq:dyn_micro}, which we will specify more rigorously shortly. Typically, this second perspective allows to leverage more agile deterministic calculus tools and leads generally to more expressive theoretical results that characterize the average agent behavior through the evolution of the measure~$\rho$.
This analytical perspective is justified by the mean-field limit or mean-field approximation,
which qualitatively or quantitatively captures the convergence of the microscopic system~\eqref{eq:dyn_micro} to the mean-field limit~\eqref{eq:dyn_macro} or associated Fokker-Planck Equation~\eqref{eq:fokker_planck} as the number~$N$ of agents grows; see Remark~\ref{rem:MFA} for more details.
For the mean-field limit of the interacting CB\textsuperscript{2}O method~\eqref{eq:dyn_micro}, i.e., as $N\rightarrow\infty$,
it is natural to postulate the stochastic process $\overbar{\theta} = (\overbar{\theta}_t)_{t\geq0}$ satisfying the self-consistent nonlinear nonlocal SDE 
\begin{equation}
\label{eq:dyn_macro}
\begin{split}
    d\overbar{\theta}_t
    &=
    -\lambda \left(\overbar{\theta}_t - \mAlphaBeta{\rho_t} \right) dt + \sigma D\left(\overbar{\theta}_t - \mAlphaBeta{\rho_t}\right) dB_t, \\
    \overbar{\theta}_0
    &\sim \rho_0,
\end{split}
\end{equation}
where $\rho_t := \mathrm{Law} (\overbar{\theta}_t)$.
Here, the definition of the consensus point $\mAlphaBeta{\rho_t}$ is analogous to the definition in \eqref{eq:ConsensusPoint_FiniteParticles_III} with $\rho_t$ replacing the empirical measure $\rho_t^N$; see also Section \ref{sec:CB2O}. 
The macroscopic dynamics~\eqref{eq:dyn_macro} can be regarded as the evolution of a typical particle in an interacting particle system of the form \eqref{eq:dyn_micro} when the number $N$ of particles is large.
The path $\rho \in \CC([0,T],\CP(\bbR^d))$ with $\rho_t = \rho(t) = \Law(\overbar{\theta}_t)$ satisfies the nonlinear nonlocal Fokker-Planck equation
\begin{equation}
	\label{eq:fokker_planck}
	\partial_t\rho_t
	= \lambda\divergence \left(\left(\theta -\mAlphaBeta{\rho_t}\right)\rho_t\right)
	+ \frac{\sigma^2}{2}\sum_{k=1}^d \partial_{kk} \left(D\left(\theta-\mAlphaBeta{\rho_t}\right)_{kk}^2\rho_t\right)
\end{equation}
in a weak sense (see Definition~\ref{def:weak_solution}).

\begin{remark}[Mean-field approximation for CB\textsuperscript{2}O]
    \label{rem:MFA}
    Analyzing the mean-field dynamics~\eqref{eq:dyn_macro} or the associated Fokker-Planck Equation~\eqref{eq:fokker_planck} in lieu of the interacting particle system~\eqref{eq:dyn_micro}
    is justified by establishing the convergence
    \begin{equation}
        \label{eq:MFconvergence}
        \rho^N \rightharpoonup \rho
        \in \CC([0,T],\CP(\bbR^d)), \quad N \rightarrow \infty,
    \end{equation}
    or, more quantitatively, by obtaining a mean-field approximation result of the form
    \begin{equation}
        \label{eq:MFA}
        \max_{i=1,\dots,N} \sup_{t\in[0,T]} \bbE\left[\N{\theta_t^i-\overbar{\theta}_t^i}_2^2\right]
        \leq CN^{-1},
    \end{equation}
    where $\overbar{\theta}_t^i$ denote $N$ i.i.d.\@ copies of the mean-field dynamics~\eqref{eq:dyn_macro}
    which are coupled to the processes $\theta_t^i$ by choosing the same initial conditions as well as same Brownian motion paths,
    see, e.g., the recent review papers \cite{chaintron2021propagation,chaintron2022propagation}.
    While we leave establishing the convergence~\eqref{eq:MFconvergence} or an estimate of the form~\eqref{eq:MFA} for the CB\textsuperscript{2}O dynamics as an interesting and challenging (due to the more intricate definition of the consensus point~$\mAlphaBetanoarg$) open problem for future research,
    analogous results for CBO in the setting of classical optimization have been provided~\cite{huang2021MFLCBO,fornasier2021consensus,gerber2023mean,huang2024uniform}.
    An exhaustive survey of this type of results in the setting of CBO can be found in \cite[Remark~1.2]{fornasier2021consensus} or, in more detail, in \cite[Section~3.1.2]{riedl2024perspective}.
\end{remark}

\begin{remark}[Choice of the hyperparameter~$\alpha$ in CBO-type methods]
    \label{rem:choicealphafinite}
    Although it might seem tempting to set the hyperparameter $\alpha$ to infinity in view of typical mean-field convergence statement for CBO-type methods, see, e.g., \cite[Theorem~3.7]{fornasier2021consensus} or Theorem~\ref{thm:main},
    this is not an advisable strategy in the finite particle regime.
    As elaborated on in \cite[Section~3.3]{fornasier2021consensus} and \cite[Section~3.1]{riedl2024perspective},
    the analytical framework leveraging a mean-field perspective decomposes the approximation error of the finite interacting particle dynamics~\eqref{eq:dyn_micro} into two central error contributions.
    First, the mean-field approximation error (see Remark~\ref{rem:MFA}), which captures how well the microscopic interacting particle system~\eqref{eq:dyn_micro} approximates the macroscopic mean-field limit~\eqref{eq:dyn_macro} and \eqref{eq:fokker_planck} as the number~$N$ of employed particles grows.
    Second, the approximation error of the global target minimizer~$\thetaG$ by the mean-field dynamics~\eqref{eq:dyn_macro}.
    While the latter error improves as $\alpha\rightarrow\infty$, the first error would degenerate, see \cite[Proposition 3.11]{fornasier2021consensus}, where the constant~$C$ in \eqref{eq:MFA} is shown to exponentially depend on $\alpha$.
    For this reason, in practical applications, a trade-off is required when it comes to the choice of the hyperparameter~$\alpha$, see \cite[Theorem~3.8]{fornasier2021consensus}.
\end{remark}

Before wrapping up this introduction by giving an overview of the contributions and organization of this paper, as well as provide pointers to relevant related literature, we first discuss in the subsequent remark an extension of the CB\textsuperscript{2}O dynamics where each particle  experiences an additional local gradient drift w.r.t.\@ the lower-level objective function~$L$.

\begin{remark}[CB\textsuperscript{2}O algorithm with additional gradient drift w.r.t.\@~$L$]
    \label{rem:gradient_drift}
    In several applications of real-world interest,
    it can be of numerical benefit to enhance the particle selection mechanisms present in the  CB\textsuperscript{2}O method by adding a local gradient drift into the dynamics. This extra term facilitates the detection of minimizers of the lower-level objective function~$L$, since individual particles can get closer to, at the very least, locally optimal points of the objective landscape of $L$. In concrete terms, the time-evolution of the $i$-th agent of this version of CB\textsuperscript{2}O is given by
    \begin{equation} \label{eq:dyn_micro_add_grad}
    \begin{split}
        d\theta_t^i
        &=
        -\lambda \left(\theta_t^i - \mAlphaBeta{\rho_t^N} \right)dt + \sigma D\left(\theta_t^i - \mAlphaBeta{\rho_t^N}\right) dB_t^i\\
        &\quad\, - \lambda_\nabla \nabla L(\theta_t^i) \,dt
        + \sigma_\nabla D\left(\nabla L(\theta_t^i)\right) dB_t^{\nabla,i},
        \\
        \theta_0^i
        &\sim \rho_0.
    \end{split}
    \end{equation}
    Note that a gradient drift w.r.t.\@ the upper-level objective function $G$ is typically not desirable, as $\thetaG$ is generally not a critical point of $G$ on the whole domain.
\end{remark}

\subsection{Contributions}

The contributions of this work are fourfold.
First of all, we propose a novel algorithm for solving bi-level optimization problems of the form~\eqref{eq:bilevel_opt} with a potentially nonconvex non-smooth lower-level objective function $L$ and a potentially nonconvex non-smooth upper-level objective function $G$.
Such problems are of widespread relevance in science and engineering, machine learning and artificial intelligence as well as economics and operations research.
We tackle this problem by proposing
consensus-based bi-level optimization, abbreviated as CB\textsuperscript{2}O,
which takes inspiration from the family of particle-based optimization algorithms, in particular from consensus-based optimization~\cite{pinnau2017consensus}, which in contrast has been designed for classical unconstrained optimization problems. Our method leverages a particle selection principle implemented through the suitable choice of a quantile of the lower-level objective function~$L$ to ensure that the constraint in \eqref{eq:bilevel_opt} is satisfied, while the remaining particles are weighted w.r.t.\@ a Gibbs-type measure in order to  ensure optimization in terms of the upper-level objective function~$G$. In stark contrast to a variety of other algorithms, particle-based or not, CB\textsuperscript{2}O is intrinsically designed to solve the problem~\eqref{eq:bilevel_opt}, rather than a relaxation thereof. In particular, the CB\textsuperscript{2}O dynamics are unchanged if the objective $L$ is composed with an increasing function. 

Second, and marking the first main theoretical contribution of the paper,
we rigorously prove in Theorem~\ref{thm:well_posedness} existence of solutions to the mean-field dynamics~\eqref{eq:dyn_macro} and the associated Fokker-Planck Equation~\eqref{eq:fokker_planck} of CB\textsuperscript{2}O under reasonable regularity assumptions on the initial distribution~$\rho_0$.
The design of the consensus point in \eqref{eq:ConsensusPoint_FiniteParticles_III} creates technical challenges of fundamental nature when establishing such result due to a lack of stability under Wasserstein perturbations of the measure, see Remark \ref{rem:counter_example}.
This renders the standard classical coupling method, which is typically used in the CBO literature \cite{carrillo2018analytical}, ineffective and in particular insufficient to prove existence and uniqueness of solutions to~\eqref{eq:dyn_macro} and \eqref{eq:fokker_planck}.
To resolve this, we prove with Proposition~\ref{thm:stabilityEst} the stability of the consensus point~\eqref{eq:ConsensusPoint_FiniteParticles_III} w.r.t.\@ a combination of Wasserstein and $L^2$ perturbations.
This requires the development of several new ideas that enhance the standard toolbox~\cite{carrillo2018analytical} used in papers discussing CBO-like dynamics and may be of independent interest beyond CB\textsuperscript{2}O and this paper.
More specifically, we resort to PDE considerations and work in appropriate spaces of functions to extend the classical Picard iteration method used in \cite{albi2017meanfield,fornasier2020consensus_hypersurface_wellposedness} to analyze standard CBO dynamics in bounded domains. 
With this, our paper contributes in novel ways to the CBO literature, and, more broadly, to the corpus of global optimization methods. 

Third, and constituting the other central theoretical contribution of our paper,
is a global convergence analysis of the CB\textsuperscript{2}O method in mean-field law~\cite{fornasier2021consensus}.
Following the analytical framework put forward by the authors of \cite{fornasier2021consensus,fornasier2021convergence,riedl2024perspective},
Theorem~\ref{thm:main} states that under mild assumptions about the initial distribution~$\rho_0$ and minimal assumptions on the objective functions~$L$ and $G$,
CB\textsuperscript{2}O converges in mean-field law to the good global minimizer~$\thetaG$ of the bi-level optimization problem~\eqref{eq:bilevel_opt}.
More rigorously, we prove that the law~$\rho_t$ solving the nonlinear nonlocal Fokker-Planck Equation~\eqref{eq:fokker_planck} converges w.r.t.\@ the Wasserstein distance exponentially fast to a Dirac delta located at $\thetaG$ provided suitable choices of the hyperparameters~$\beta$ and $\alpha$.
The influence of those choices is extensively discussed and illustrated in Remarks~\ref{remark:beta0} and \ref{remark:alpha0}, respectively.
Once again,
the design of the consensus point in \eqref{eq:ConsensusPoint_FiniteParticles_III} requires a novel treatment due to the additional particle selection principle. For this reason,
we first establish in Proposition~\ref{lem:laplace_LG} a quantitative ``quantiled" Laplace principle~(Q\textsuperscript{2}LP),
which combines a suitable choice of the hyperparameter~$\beta$, and through it a choice of the quantile~$\qbetanoarg$, with the quantitative Laplace principle~\cite[Proposition~4.5]{fornasier2021consensus}.
This technique may be of independent interest.

Fourth and lastly, we provide extensive numerical experiments that demonstrate the practicability and efficacy of the CB\textsuperscript{2}O method~\eqref{eq:dyn_micro}.
In Section~\ref{sec:numerics_COPT}, we treat the special instance of constrained global optimization, where we firstly compare CB\textsuperscript{2}O to particle-based methods designed for this task and then perform a thorough study of the influence of the various hyperparameters of the scheme. In Section~\ref{sec:numerics_SRL}, we then tackle a sparse representation learning task~\cite{gong2021Biojective} thereby demonstrating that the CB\textsuperscript{2}O method is also applicable in high-dimensional real-world scenarios. 
For an additional machine learning example, in the context of federated learning,
we refer to the paper \cite{trillos2024attack}.

\subsection{Related Works}
\label{sec:related_work}

\paragraph{Related works in bi-level optimization problems of the form~\eqref{eq:bilevel_opt}.}
\label{sec:related_work_SBO}

The convex case of bi-level optimization problems of the form~\eqref{eq:bilevel_opt} (also known as simple bi-level optimization problems), i.e., where both the upper- and lower-level objective functions~$G$ and $L$ in \eqref{eq:bilevel_opt} are assumed to be convex,
has been extensively studied in recent years. 
In what follows,
we discuss those works that are related to our method,
and refer the reader to \cite{giang2023projection,jiang2023conditional,chen2024penalty} for a more comprehensive review of the literature on this problem.

One standard approach for solving convex bi-level optimization problems of the form~\eqref{eq:bilevel_opt} is to recast the problem as a single-level unconstrained optimization problem by introducing a suitable Lagrangian; this approach is commonly referred to as regularization approach. Specifically, in \cite{tikhonov1977solutions} the authors propose to combine the upper- and lower-level objective functions in \eqref{eq:bilevel_opt} to obtain the unconstrained optimization problem $\argmin_{\theta \in \bbR^d}\, \{\LM L(\theta) +  G(\theta)\}$, with regularization parameter $\LM > 0$.
To address the issue of tuning $\LM$, works such as \cite{cabot2005proximal,solodov2007explicit, dutta2020algorithms} propose to iteratively update the regularization parameter in a suitable way, producing in this way a sequence $\{\LM_k\}_{k\geq 0}$. This approach has been further explored under various assumptions and different algorithmic choices in \cite{helou2017subgradient,malitsky2017primal,amini2019iterative, kaushik2021method,shen2023online,chen2024penalty}.
However, these approaches are dual methods inherently and thus are effective only when strong duality holds, thereby limiting their applicability to the convex case.
In contrast, our interacting particle system-based approach is a primal method,
which directly and inherently tackles the primal problem~\eqref{eq:bilevel_opt},
making it consequently more suitable for the nonconvex setting.

Another strategy is to approximate the set~$\Theta$ of global minimizers of $L$ with a surrogate set of simpler and more explicit form, before applying conventional projection-based or projection-free algorithms using this surrogate set \cite{beck2014first,jiang2023conditional,cao2024projection,cao2024accelerated}.
This approach is referred to as the sub-level set approach.
However, these surrogate sets are built based on the convexity of the lower-level objective function $L$, which limits their ability to accurately approximate the set $\Theta$ in nonconvex settings, thereby reducing the effectiveness of those methods.
Our CB\textsuperscript{2}O method~\eqref{eq:dyn_micro} shares some conceptual similarities with the sub-level set approach.
Specifically, the set~$\Qbeta{\rho_t^N}$, as defined in \eqref{eq:Qbeta_FiniteParticles}, can be regarded as an approximation of the neighborhoods of the set~$\Theta$ of global minimizers of the lower-level objective~$L$, in particular, in the mean-field regime with small $\beta$.
Unlike existing methods, however, the construction of $\Qbeta{\rho_t^N}$ does not depend on the convexity of the objective function $L$, making it, in consequence, applicable to the relevant class of nonconvex problems that are of interest to us.

To the best of our knowledge, \cite{gong2021Biojective} is the only paper that directly addresses nonconvex bi-level optimization problems of the form~\eqref{eq:bilevel_opt}.
That paper introduces a dynamic barrier constraint on the search direction at each iteration, leading to a simple gradient-based algorithm that balances the objective functions $G$ and $L$ through an adaptive combination of coefficients, thereby tracing a trajectory towards the optimal solution.
Notably, similarly to our approach, \cite{gong2021Biojective} is also a primal method, which is generally preferred for nonconvex problems as explained earlier. However, their theoretical results only guarantee convergence to stationary points of the nonconvex problem.
In contrast, we rigorously demonstrate that our method converges to the desired target global minimizer~$\thetaG$ with arbitrary precision, at least in the mean-field regime.

\paragraph{Related works in general consensus-based optimization.}
Consensus-based optimization methods for classical optimization have been introduced by the authors of \cite{pinnau2017consensus} motivated by the urge to develop and design a class of metaheuristic particle-based algorithms that are amenable to a rigorous mathematical convergence analysis.
The first serious and so far most promising global convergence frameworks,
leveraging a mean-field perspective,
were proposed in \cite{carrillo2018analytical,carrillo2019consensus} and \cite{fornasier2021consensus,fornasier2021convergence,riedl2022leveraging,riedl2024perspective}, respectively.
While the authors of the former analyze the variance of the mean-field dynamics, thereby establishing consensus formation, and consecutively bounding the distance between the found consensus and the global minimizer of the problem, the authors of the latter group of papers directly investigate the Wasserstein distance between the law of the mean-field dynamics and the global minimizer. To further understand the global convergence property in the finite particle regime, \cite{huang2021MFLCBO, fornasier2021convergence,fornasier2021consensus,gerber2023mean,huang2024uniform} study the mean-field approximation of the finite particle system, which, together with the global convergence of the mean-field system, yields convergence guarantees for the finite particle algorithm implemented in practice.
It is worth mentioning that a convergence analysis directly at the finite particle level has so far been addressed only in \cite{ko2022convergence,byeon2024discrete,bellavia2024discrete}, where \cite{bellavia2024discrete} adopts the Wasserstein-based analysis framework which we also follow in this work.

The applications and analyses of CBO-type algorithms have been extended to a variety of other optimization problems, including, e.g., stochastic optimization \cite{bonandin2024consensus},
multi-objective optimization~\cite{borghi2022consensus,borghi2022adaptive,borghi2023repulsion},
min-max problems~\cite{huang2022consensus, borghi2024particle},
the simultaneous search of multiple global minimizers~\cite{bungert2022polarized, fornasier2024polarized},
sampling~\cite{carrillo2022consensus} and the optimization over the space of probability distributions~\cite{borghi2024dynamics}.
Further applications of CBO have also been applied to high-dimensional
machine learning problems \cite{carrillo2019consensus,fornasier2020consensus_sphere_convergence} such as decentralized clustered federated learning \cite{carrillo2024fedcbo}.
CBO has also been shown to be related to the well-known particle swarm optimization method~\cite{cipriani2021zero,huang2022global} and stochastic gradient descent with specific noise~\cite{riedl2023all,riedl2023all2}.

The closely related work~\cite{herty2024multiscale} proposes a CBO-type method to solve general bi- and multi-level optimization problems. 
In their approach, multiple interacting populations of particles are employed,
each designated to optimize one level of the problem, with the individual populations evolving independently according to a CBO-type dynamics with different time-scaling. 
While the bi-level optimization problem that we study here can be seen as a specific instance of the type of problem considered in \cite{herty2024multiscale}, we highlight that their proposed method does not fully capture and apply to our setting.
Specifically, our typical scenario is that the lower-level objective function $L$ possesses multiple global minimizers, a scenario that the standard CBO dynamics (as applied to each level of optimization problem in \cite{herty2024multiscale}) cannot effectively handle as it is not designed for such setting. Furthermore, here we provide rigorous justification of the global convergence of our method in the mean-field regime, whereas \cite{herty2024multiscale} puts more emphasis on the algorithmic development.

\paragraph{Related works in constrained consensus-based optimization.}
For standard constrained optimization problems, which can be formulated as a specific instance of the bi-level optimization setting~\eqref{eq:bilevel_opt}, several variants of CBO-type methods have been proposed. 
These approaches can be broadly categorized into three main classes.

The first approach recasts the constrained optimization problem as an unconstrained problem by incorporating the constraint into the upper-level objective function via a penalization term.
The standard CBO \cite{pinnau2017consensus} is then applied to this modified problem \cite{borghi2021constrained,carrillo2021consensus,herty2024micromacro}.
However, penalization-based methods are effective primarily when the penalty is exact \cite{burke1991exact, borghi2021constrained}, a condition that does not generally hold for nonconvex constraint sets.

The second approach involves projection onto the constrained set, referred to projection-based methods.
For example,
\cite{bae2022constrained} proposes a predictor-corrector CBO method where particles are projected onto the feasible convex set at each time step.
Similarly,
\cite{fornasier2020consensus_hypersurface_wellposedness,fornasier2020consensus_sphere_convergence,fornasier2021anisotropic} design CBO dynamics intrinsic to manifolds such as spheres and tori, leveraging simple projection mechanisms.
Despite their effectiveness, these methods require the computation of the distance function $\text{dist}(\dummy, \Theta)$ to the constraint set $\Theta$, which becomes infeasible for complex constraint sets and may not be possible in nonconvex settings.

The third approach \cite{carrillo2022consensus,carrillo2024interacting} introduces a large gradient forcing term, derived from the constraint set, into the standard CBO dynamics \cite{pinnau2017consensus}.
This ensures rapid particle concentration on the constraint set, achieving high accuracy and stability when the constraint set is convex or otherwise well-behaved.
However, this method lacks convergence guarantees for general constraint sets and requires the evaluation of the inverse of the Hessian at each step \cite{carrillo2024interacting}, which is computationally infeasible for high-dimensional problems.

In addition to these main approaches, \cite{beddrich2024constrained} implements and explores a reflective boundary condition for CBO within compact convex domains.

In contrast to these existing methods, which typically rely on favorable geometric properties (e.g., convexity) of the constraint set, our CB\textsuperscript{2}O method, by leveraging a lower-level objective function $L$ whose set of minimizers coincide with the constraint set of interest, can be applied to general (and in particular nonconvex) constraint sets, achieving the accuracy and stable behavior at the same time. 

\paragraph{Related works in evolutionary algorithms for bi-level optimization problems.}
Due to the potential nonconvexity in both the upper- and lower-level objective functions,
evolutionary and metaheuristic algorithms have been proposed to address bi-level optimization problems.
Some of the pioneering methods date back to the late 1980s; see \cite{anandalingam1989artificial,mathieu1994genetic}.
Since then, various evolutionary algorithms tailored to bi-level optimization problems have been explored, see, e.g., \cite{yin2000genetic,oduguwa2002bi, li2006hierarchical, zhu2006hybrid, angelo2015study,jiang2021research}, amongst many others. 
For a comprehensive review, we refer the reader to \cite{sinha2017review}.
Despite their popularity and success, those metaheuristic approaches typically lack a rigorous mathematical foundation.
In contrast, for our CB\textsuperscript{2}O method, which falls within the class of metaheuristic algorithms, we can provide a rigorous mathematical justification for its ability to solve problem \eqref{eq:bilevel_opt} by  
building upon the works~\cite{carrillo2018analytical,carrillo2021consensus,fornasier2021consensus,fornasier2021convergence} and by introducing new analysis and mathematical constructions. We thus believe that this paper strengthens and advances the field of metaheuristic methods for bi-level optimization.

\subsection{Organization}
\label{subsec:orga}

In Section~\ref{sec:main} we present the main theoretical contributions of this paper.
To this end, Section~\ref{sec:CB2O} first reintroduces the CB\textsuperscript{2}O dynamics which is theoretically investigated in the remainder of the theoretical parts of this paper.
It should be regarded as a merely (for analytical purposes) regularized version of the CB\textsuperscript{2}O algorithm that has been described in the introduction. While in practice this regularized version does work well, the modifications are not necessary to obtain successful experimental results and for this reason we implement the simpler variant discussed in the introduction.
Section~\ref{sec:well_posedness} then provides an existence result as well as regularity estimates for solutions to the mean-field CB\textsuperscript{2}O dynamics.
The accompanying proofs of this section are deferred to Section~\ref{sec:well_posedness_proofs}.
In Section~\ref{sec:convergence} we state a global mean-field law convergence result for the CB\textsuperscript{2}O dynamics to the good global minimizer~$\thetaG$ of the bi-level optimization problem~\eqref{eq:bilevel_opt}.
We furthermore provide an elaborate discussion of the choices of the hyperparameters~$\beta$ and $\alpha$, obtaining in this way deeper insights into the design of our CB\textsuperscript{2}O algorithm, while additionally hinting at the mechanisms at play in the convergence proof.
The proof details of this section can be found in Section~\ref{sec:convergence_proofs}. Thereafter we turn towards the experimental validation of the proposed CB\textsuperscript{2}O method.
For this purpose, Section~\ref{sec:experiments} showcases a series of numerical experiments which demonstrate the practicability as well as efficiency of the CB\textsuperscript{2}O algorithm, whose implementation is formalized in Section~\ref{sec:numerics_alg}.
In Section~\ref{sec:numerics_COPT}, we consider the special instance of classical constrained global optimization in lower-dimensional settings, where we conduct an extensive experimental study of the influence of the various hyperparameters of our method on its performance.
Section~\ref{sec:numerics_SRL} is concerned with a sparse representation learning task. We use this experiment to demonstrate the applicability of the CB\textsuperscript{2}O method in high-dimensional modern applications with real data. For a further machine learning example in the setting of (clustered) federated learning,
we refer to the recent paper \cite{trillos2024attack},
where CB\textsuperscript{2}O is used to implement a defense mechanism against backdoor adversarial attacks. We wrap up the paper in Section~\ref{sec:conclusions}.

For the sake of reproducible research, we provide the code implementing the CB\textsuperscript{2}O algorithm introduced in this work and used to run the numerical experiments in Section~\ref{sec:experiments} in the GitHub repository \url{https://github.com/SixuLi/CB2O}.

\subsection{Notation}
\label{subsec:notation}
We use $\bbN_0^d$ to denote the set of $d$-dimension vectors with each element being a non-negative integer. 
For a function $f: \R^d \mapsto \R$, and $a = (a_1, \dots, a_d) \in \bbN_0^d$ with $|a| := \sum_{i=1}^d a_i$, we denote $f^{(a)} := \frac{\partial^{|\nu|}}{\partial_{\theta_1}^{a_1} \cdots \partial_{\theta_d}^{a_d}} f$.
For notation simplicity, we will use $f^{(l)}$ or $D^{l} f$ to denote $f^{(a)}$ for some $a \in \bbN_0^d$ with $|a| = l$.
Unless otherwise specified, we use $C$ or $\widetilde{C}$ to denote some generic constants.
Euclidean balls are denoted as \mbox{$B_{r}(\theta) := \{\tilde\theta \in \bbR^d\!:\! \Nnormal{\tilde\theta-\theta}_2 \leq r\}$}.
We define the distance between a point $\theta$ and set $S$ in $\R^d$ as $\dist(\theta, S) := \inf_{\tilde\theta \in S}\, \Nnormal{\tilde\theta-\theta}_2$.
Furthermore, we denote the neighborhood of set $S$ with radius $r$ as $\CN_r(S) := \{\tilde\theta \in \R^d : \dist(\tilde\theta, S) \leq r \}$.
For the space of continuous functions~$f:X\rightarrow Y$ we write $\CC(X,Y)$, with $X\subset\bbR^n$ and a suitable topological space $Y$.
For an open set $X\subset\bbR^n$ and for $Y=\bbR^m$ the spaces~$\CC^k_{c}(X,Y)$ and~$\CC^k_{b}(X,Y)$ consist of functions~$f\in\CC(X,Y)$ that are $k$-times continuously differentiable and have compact support or are bounded, respectively. 
We avoid making $Y$ explicit in the real-valued case.
Furthermore, following the conventional notation in the literature, we abbreviate $\CC^{1} ([0,T]; \CC^2(\R^d) )$ by $\CC^{2, 1} (\R^d \times [0,T])$.
The operators $\nabla$ and $\Delta$ denote the gradient and Laplace operators of a function on~$\bbR^d$.
The main objects of study in this paper are laws of stochastic processes, $\rho\in\CC([0,T],\CP(\bbR^d))$, where the set $\CP(\bbR^d)$ contains all Borel probability measures over $\bbR^d$.
With $\rho_t\in\CP(\bbR^d)$ we refer to the snapshot of such law at time~$t$.
In case we refer to some fixed distribution, we write~$\indivmeasure$.
Measures~$\indivmeasure \in \CP(\bbR^d)$ with finite $p$-th moment $\int \Nnormal{\theta}_2^p\,d\indivmeasure(\theta)$ are collected in $\CP_p(\bbR^d)$.
For any $1\leq p<\infty$, $W_p$ denotes the \mbox{Wasserstein-$p$} distance between two Borel probability measures~$\indivmeasure_1,\indivmeasure_2\in\CP_p(\bbR^d)$ (see, e.g., \cite{savare2008gradientflows}), which is defined by
\begin{align} \label{def:wassersteindistance}
	W_p(\indivmeasure_1,\indivmeasure_2) = \left(\inf_{\pi\in\Pi(\indivmeasure_1,\indivmeasure_2)}\int\Nnormal{\theta-\tilde\theta}_2^p\,d\pi(\theta,\tilde\theta)\right)^{1/p},
\end{align}
where $\Pi(\indivmeasure_1,\indivmeasure_2)$ denotes the set of all couplings of $\indivmeasure_1$ and $\indivmeasure_2$, i.e., the collection of all Borel probability measures over $\mathbb{R}^d\times\mathbb{R}^d$ with marginals $\indivmeasure_1$ and $\indivmeasure_2$. 
$\bbE(\indivmeasure)$ denotes the expectation of a random variable distributed according to a probability measure $\indivmeasure$.

\section{Consensus-Based Bi-Level Optimization (CB\texorpdfstring{\textsuperscript{2}}{2}O) and Main Theoretical Results}
\label{sec:main}

This section is dedicated to the presentation of the main theoretical contributions of this paper.
To this end, we first introduce in Section~\ref{sec:CB2O} the CB\textsuperscript{2}O dynamics that we theoretically investigate in the subsequent two sections and which shall be regarded as a merely (for analytical purposes) regularized version of the variant that has been described intuitively and instructively in the introduction; we note that, for simplicity, it is the version of the dynamics presented in the Introduction the one that is implemented towards the end of the paper in Section~\ref{sec:experiments}. In Section~\ref{sec:well_posedness}, we then provide an existence result for solutions to the mean-field CB\textsuperscript{2}O dynamics and establish regularity estimates thereof. The proofs of this section are deferred to Section~\ref{sec:well_posedness_proofs}.
After studying existence and regularity of solutions to the mean-field PDE, we discuss in Section~\ref{sec:convergence} the behavior of the CB\textsuperscript{2}O dynamics by stating a global convergence result of the CB\textsuperscript{2}O dynamics in mean-field law to the global minimizer of the bi-level optimization problem~\eqref{eq:bilevel_opt}.
Its proof details can be found in Section~\ref{sec:convergence_proofs}.

\subsection{Consensus-Based Bi-Level Optimization (CB\texorpdfstring{\textsuperscript{2}}{2}O)}
\label{sec:CB2O}

To solve a bi-level optimization problem of the form~\eqref{eq:bilevel_opt},
we propose an interacting particle system inspired by the class of consensus-based optimization methods~\cite{pinnau2017consensus,carrillo2018analytical, carrillo2021consensus,fornasier2021consensus,fornasier2021convergence,riedl2024perspective},
namely
\textbf{C}onsensus-\textbf{B}ased \textbf{B}i-level \textbf{O}ptimization (CB\textsuperscript{2}O).
It is described, as we restate in this concise overview here, by the system of stochastic differential equations~\eqref{eq:dyn_micro} given by
\begin{align*}
    d\theta_t^i
    &=
    -\lambda \left(\theta_t^i - \mAlphaBeta{\rho_t^N} \right)dt + \sigma D\left(\theta_t^i - \mAlphaBeta{\rho_t^N}\right) dB_t^i, \\
    \theta_0^i
    &\sim \rho_0 \quad\text{for all } i=1,\dots,N,
\end{align*}
where $((B_t^i)_{t\geq0})_{i=1,\dots,N}$ are independent standard Brownian motions in $\bbR^d$.
Moreover, the consensus point~$\mAlphaBetanoarg$ is defined for an arbitrary measure $\varrho\in\CP(\bbR^d)$ according to
\begin{equation}
    \label{eq:consensus_point}
    \mAlphaBeta{\varrho}
    :=
    \mAlpha{\Ibeta{\varrho}} := \int \theta \frac{\omegaa(\theta)}{\N{\omegaa}_{L^1(\Ibeta{\varrho})}} d\Ibeta{\varrho}(\theta),
    \quad \text{with}\quad
    \omegaa(\theta):= \exp \left(-\alpha G(\theta) \right).
\end{equation}
Here, the unnormalized measure $\Ibeta{\varrho}$ is given as 
\begin{align}
    \Ibeta{\varrho}\label{eq:I_beta}
    &:=
    \mathbbm{1}_{\Qbeta{\varrho}} \varrho,
\end{align}
where, for some fixed parameters $\delta_q>0$ and sufficiently large $R > 0$, the sub-level set $\Qbeta{\varrho}$ is defined as
\begin{align}
    \Qbeta{\varrho} \label{eq:Q_beta}
    &:=
    \left\{\theta\in B_R(0) :  L(\theta)\leq \frac{2}{\beta}\int_{\beta/2}^{\beta} \qa{\varrho}\,da +\delta_q \right\},
\end{align}
with the $a$-quantile function~$\qa{\varrho}$ of $\varrho$ under $L$ given by
\begin{align}
    \label{eq:q_beta}
    \qa{\varrho}
    &:=
    \inf \big\{ q \text{ s.t. } a \leq \varrho(L(\theta)\leq q) \big\}.
\end{align}
This latter definition generalizes the one of the quantile function~\eqref{eq:qbeta_FiniteParticles} from the setting of a finite number of particles
to a general distribution $\varrho \in \CP(\bbR^d)$.
As before,
the quantile function~$\qa{\varrho}$ from \eqref{eq:q_beta} returns the particular value such that the probability of a random variable with distribution $\varrho$ having an objective value w.r.t.\@ $L$ smaller or equal to $\qa{\varrho}$ is at least $a$.
The set $\Qbeta{\varrho}$, as defined in \eqref{eq:Q_beta},
denotes, analogously to the definition in \eqref{eq:Qbeta_FiniteParticles}, the set of all $\theta\in\bbR^d$ whose objective function value w.r.t.\@ the lower-level objective function~$L$ is, roughly speaking (see \ref{eq:Qbetaregularizationii} and \ref{eq:Qbetaregularizationiii} below), smaller than the value of the quantile function $\qbeta{\varrho}$ for a parameter $\beta$. For our later analytical studies, however, and compared to the analogous definition \eqref{eq:Qbeta_FiniteParticles} in the finite particle setting, the above definitions come with three modifications that can be interpreted as different ways of regularization.
\begin{enumerate}[label=(\arabic*),labelsep=10pt,leftmargin=35pt]
    \item \label{eq:Qbetaregularizationi}The set $\Qbeta{\varrho}$ is limited to a ball $B_R(0)$.
    Since our goal is to search $\thetaG$, we naturally assume that the radius $R>0$ is large enough so that $\thetaG \in B_R(0)$, i.e., in particular such that $R\geq\Nnormal{\thetaG}_2+r_R$ for some $r_R>0$.
    \item \label{eq:Qbetaregularizationii}Instead of choosing $\qbeta{\varrho}$ as the threshold parameter on the level of the lower-level objective function~$L$ in the definition of $\Qbeta{\varrho}$ in \eqref{eq:Q_beta},
    we use the smoothened threshold $\frac{2}{\beta}\int_{\beta/2}^{\beta} \qa{\varrho}\,da$.
    The theoretical analysis later on can be straightforwardly generalized to any regularized threshold of the form $\frac{1}{\vartheta\beta}\int_{(1-\vartheta)\beta}^{\beta} \qa{\varrho}\,da$ with $\vartheta\in(0,1)$, which can be made arbitrarily close to $\qbeta{\varrho}$.
    \item \label{eq:Qbetaregularizationiii}We introduce a small slack parameter~$\delta_q$ in \eqref{eq:Q_beta}.
\end{enumerate}
Despite these modifications to the original definitions in the introduction, the sub-level set $\Qbeta{\varrho}$ can still be interpreted as an approximation to the neighborhood of the set~$\Theta$ of global minimizers of the lower-level objective function~$L$. This is expected when the support of $\varrho$ covers $\Theta$ and $\beta$ and $\delta_q$ are sufficiently small. We emphasize again that while the above modifications are required for our theoretical analysis, in our practical implementations
we stick to the definitions in \eqref{eq:Qbeta_FiniteParticles} given that these definitions induce a scheme that is simpler to implement and has a comparable numerical performance as the modified scheme. At the theoretical level, the regularizations~\ref{eq:Qbetaregularizationi} and \ref{eq:Qbetaregularizationii} are introduced to establish the well-posedness of the mean-field CB\texorpdfstring{\textsuperscript{2}}{2}O dynamics that we present in Section~\ref{sec:well_posedness}. They are, on the other hand, not required for the convergence analysis presented in Section~\ref{sec:convergence}. The regularization~\ref{eq:Qbetaregularizationiii} is needed for the global convergence in mean-field law analysis.

\begin{remark}
With the definition of $\Qbeta{\varrho}$ as in \eqref{eq:Q_beta},
the consensus point~$\mAlphaBetanoarg$ as defined in \eqref{eq:consensus_point} is, due to \ref{eq:Qbetaregularizationi}, uniformly bounded,
i.e., satisfies $\Nnormal{\mAlphaBeta{\varrho}}_2 \leq R$ for any $\varrho \in \CP(\R^d)$.
\end{remark}

As discussed in the introduction,
our mathematical analysis takes a mean-field perspective.
Motivated by several mean-field approximation results in the setting of classical optimization, see Remark~\ref{rem:MFA},
it is natural to postulate that in the mean-field limit,
i.e., as the number of particles~$N$ tends to infinity,
the CB\textsuperscript{2}O dynamics \eqref{eq:dyn_micro} can be described by the self-consistent stochastic process~$\overbar{\theta} = (\overbar{\theta}_t)_{t\geq0}$ which solves the self-consistent nonlinear nonlocal SDE  \eqref{eq:dyn_macro} given by
\begin{equation*}
\begin{aligned}
    d\overbar{\theta}_t
    &=
    -\lambda \left(\overbar{\theta}_t - \mAlphaBeta{\rho_t} \right) dt + \sigma D\left(\overbar{\theta}_t - \mAlphaBeta{\rho_t}\right) dB_t,
\end{aligned}
\end{equation*}
where $\rho_t := \mathrm{Law} (\overbar{\theta}_t)$.
The associated law~$\rho$ with $\rho_t = \rho(t)$ must satisfy the Fokker-Planck Equation~\eqref{eq:fokker_planck} given by
\begin{equation*}
    \partial_t\rho_t
	= \lambda\divergence \left(\left(\theta -\mAlphaBeta{\rho_t}\right)\rho_t\right)
	+ \frac{\sigma^2}{2}\sum_{k=1}^d \partial_{kk} \left(D\left(\theta-\mAlphaBeta{\rho_t}\right)_{kk}^2\rho_t\right)
\end{equation*}
in the weak sense, defined as follows.

\begin{definition}[Weak solution to Fokker-Planck Equation~\eqref{eq:fokker_planck}]
\label{def:weak_solution}
Let $\rho_0 \in \CP(\bbR^d)$, $T > 0$.
We say $\rho\in\CC([0,T],\CP(\bbR^d))$ satisfies the Fokker-Planck Equation~\eqref{eq:fokker_planck} with initial condition $\rho_0$ in the weak sense in the time interval $[0,T]$,
if for all $\phi \in \CC_c^{\infty}(\bbR^d)$ and all $t \in (0,T)$ 
	\begin{equation} \label{eq:weak_solution_identity}
	\begin{aligned}
		\frac{d}{dt}\int \phi(\theta) \,d\rho_t(\theta)
		&= \lambda\int \sum_{k=1}^d (\theta - \mAlphaBeta{\rho_t})_k \partial_k \phi(\theta) \, d\rho_t(\theta)\\
		&\quad\, + \frac{\sigma^2}{2} \int \sum_{k=1}^d D\!\left(\theta-\mAlphaBeta{\rho_t}\right)_{kk}^2  \partial^2_{kk} \phi(\theta) \,d\rho_t(\theta),
	\end{aligned}
	\end{equation}
	and $\lim_{t\rightarrow 0}\rho_t = \rho_0$ in the weak sense of probability distributions. 
\end{definition}

Our main theoretical contributions, which we will present in the remainder of this section,
are split into two key theorems.
First, we discuss the existence of regular solutions to the mean-field system \eqref{eq:dyn_macro} in Theorem \ref{thm:well_posedness}.
Second, we discuss in Theorem \ref{thm:main} the long-time behavior properties of the mean-field PDE \eqref{eq:fokker_planck} and show that, under some mild assumptions on initialization and the correct tuning of hyperparameters, the measure $\rho_t$ concentrates around the target global minimizer $\thetaG$.

For the remainder of the theoretical parts of this paper we limit our attention to the case of CB\textsuperscript{2}O with isotropic diffusion, i.e., $D\!\left(\dummy\right)=\N{\dummy}_2 \Id$ in \eqref{eq:diffustion_types}.

\subsection{Well-Posedness of the Mean-Field CB\texorpdfstring{\textsuperscript{2}}{2}O Dynamics}
\label{sec:well_posedness}

Let us start the presentation of our main theoretical contributions with an existence result for solutions to the mean-field CB\textsuperscript{2}O dynamics~\eqref{eq:dyn_macro} and the associated Fokker-Plank Equation~\eqref{eq:fokker_planck}.
We moreover establish their regularity for objective functions~$L$ and $G$ that satisfy the following assumptions.

\begin{assumption}\label{asm:well-posedness}

Throughout we are interested in objective functions $L \in \CC(\bbR^d)$ and $G \in \CC(\bbR^d)$
    for which
\begin{enumerate}[label=W\arabic*,labelsep=10pt,leftmargin=35pt]
    \item\label{asm:minimizers}
   $\minL := \inf_{\theta \in \R^d} L(\theta) > -\infty$,
    \item\label{asm:LipObjL}
    there exists $C_1 > 0$ such that
    \begin{equation}
        \abs{L(\theta) - L(\theta')} \leq C_1 \N{\theta - \theta'}_2 \qquad \text{for all} \; \theta, \theta' \in \R^d,
    \end{equation}
    \item \label{asm:LipschitzDistrib} $L$'s distribution function $q \in [0,\infty) \mapsto \mathrm{Vol}(\{ \theta\in B_R(0):  L(\theta) \leq q \})$ is $C_L$-Lipschitz continuous, where $R$ is the radius of the ball in \eqref{eq:Q_beta}, see also \ref{eq:Qbetaregularizationi},
    \item\label{asm:lowerBound_G}  $\underbar{G} := \inf_{\theta \in \bbR^d} G(\theta) > -\infty$,

    \item\label{asm:upperBound_G} either $\overbar\CE:=\sup_{\theta \in \bbR^d}\CE(\theta) < \infty$, or there exist constants $C_2,C_3 > 0$ such that
		\begin{equation} \label{asm:quadratic_growth_2}
			\CE(\theta) - \minobj
			\geq C_2\N{\theta}_2^2
			\quad \text{for all } \N{\theta}_2 \geq C_3,
		\end{equation}
    \item\label{asm:growthBound_G}
    there exists $C_4 > 0$ such that
    \begin{equation}
        G(\theta) - \underbar{G} \leq C_4 (1 + \N{\theta}_2^2) \quad \text{for all } \theta \in \R^d.
    \end{equation}
\end{enumerate}
\end{assumption}

\textit{Lower-level objective function $L$:} 
Assumption~\ref{asm:minimizers} imposes that the lower-level objective function~$L$ is bounded from below by $\minL$.
Assumption~\ref{asm:LipObjL} requires that $L$ is Lipschitz continuous with constant $C_1$.  Assumption~\ref{asm:LipschitzDistrib} is an assumption on regularity of level sets of $L$ that we impose to rule out lower-level objective functions $L$ with level sets $\{ \theta:L(\theta)=q \}$ that are full-dimensional.

\textit{Upper-level objective function $G$:} 
Assumption~\ref{asm:lowerBound_G} requires that also the upper-level objective function~$G$ is lower-bounded.
Assumption~\ref{asm:upperBound_G} requires that $G$ is either bounded from above or has at least quadratic growth in the farfield. Assumption~\ref{asm:growthBound_G}, on the other hand, requires that $G$ also has bounded growth-rate globally.




In our first main theoretical result,
we establish the existence of solutions to the mean-field CB\textsuperscript{2}O dynamics~\eqref{eq:dyn_macro} and the associated nonlinear nonlocal Fokker-Planck Equation~\eqref{eq:fokker_planck}.
The proof of uniqueness of such solution is challenging, and is left for future investigation (see discussions in Section \ref{sec:conclusions}).

\begin{theorem}[Existence of regular solutions to the mean-field CB\textsuperscript{2}O dynamics~\protect{\eqref{eq:dyn_macro}} and~\protect{\eqref{eq:fokker_planck}}]
\label{thm:well_posedness}
    Let $L \in \CC(\bbR^d)$ and $G \in \CC(\bbR^d)$ satisfy \ref{asm:minimizers}--\ref{asm:growthBound_G}.
    Moreover, for $l \geq 0$, let $\rho_0 \in H^{l+2}(\R^d) \cap L^{\infty} (\R^d) \cap \CP_4(\bbR^d) $ be such that $\thetaG \in \supp{\rho_0}$.
    Then, for $T > 0$, there exists a nonlinear process $\overbar{\theta} \in \CC \left( [0,T], \R^d\right)$ satisfying the mean-field dynamics \eqref{eq:dyn_macro} in the strong sense whose associated law $\rho=\Law(\overbar\theta)$ is a weak solution (Definition \ref{def:weak_solution}) to the Fokker-Planck Equation~\eqref{eq:fokker_planck} and has regularity    \begin{equation}\label{eq:Sol_regularity}
        \rho \in W^{1,\infty} \left( [0,T], H^l(\R^d) \right) \cap L^{\infty} ([0,T], L^{\infty}(\R^d))  \cap \CC \left([0,T], \CP_2(\R^d) \right) \, .
    \end{equation}
    Moreover, $\rho_t \in \CP_4(\R^d)$ and the mapping $t \mapsto \mAlphaBeta{\rho_t}$ is $\frac{1}{2}$-Hölder continuous for $t \in [0,T]$.
\end{theorem}

A few comments about Theorem~\ref{thm:well_posedness} are in order.
First and foremost, it is important to highlight that, compared to the classical well-posedness proof for CBO~\cite{pinnau2017consensus}, which can be found for instance in \cite{carrillo2018analytical,fornasier2021consensus},
the proof of Theorem~\ref{thm:well_posedness} is substantially more involved and intricate.
This added complexity arises from the lack of Wasserstein stability for the consensus point~\eqref{eq:consensus_point}.
For standard CBO, one has by \cite[Lemma~3.2]{carrillo2018analytical} that $\Nbig{\mAlpha{\varrho}-\mAlpha{\widetilde{\varrho}}}_2\leq CW_2(\varrho,\widetilde{\varrho})$.
Unfortunately, as we demonstrate in Remark~\ref{rem:counter_example} by giving an instructive counterexample, this is \textit{not} the case for our consensus point $\mAlphaBeta{\varrho}$. In contrast, we establish in Proposition~\ref{thm:stabilityEst} the following stability estimate 
\[\Nbig{\mAlphaBeta{\varrho} - \mAlphaBeta{\widetilde{\varrho}}}_2\leq C \left( \|\varrho - \widetilde{\varrho}\|_{L^2(B_R(0))} + W_2 \left( \varrho, \widetilde{\varrho} \right) \right),\]
which holds for probability measures that are absolutely continuous with respect to the Lebesgue measure. As a result, we must pursue a different approach to the standard fixed-point argument used in \cite{carrillo2018analytical}. Specifically, we extend the Picard's iteration argument from \cite{albi2017meanfield, fornasier2020consensus_hypersurface_wellposedness}, originally applied to the dynamics restricted to a bounded domain, to one applicable over the entire space $\R^d$. 
This requires integrating analyses at both the PDE and SDE levels and working within suitable spaces of functions; we believe that our analysis is thus of independent interest.

Second, the assumption $\thetaG\in\supp{\rho_0}$ about the initial configuration $\rho_0$ is not really a restriction,
as it would anyhow hold immediately for $\rho_t$ for any $t > 0$ in view of the diffusive character of
the dynamics~\eqref{eq:dyn_macro}, see also \cite[Remark~4.8]{fornasier2021consensus}.
It is an assumption of technical nature that ensures that the consensus point is always well-defined for $\rho$ satisfying the Fokker-Planck Equation~\eqref{eq:fokker_planck}.
More details are to be found in Section~\ref{sec:well_posedness_proof} when proving Theorem~\ref{thm:well_posedness}.

Third, it is important to highlight that the Picard iteration approach that we follow to prove existence of solutions to the mean-field PDE does not allow us to deduce uniqueness. It is reasonable to expect that under additional regularity assumptions one may be able to prove uniqueness within a class of regular solutions of the mean field PDE. We will explore this question in future work.

\begin{remark} \label{rem:test_functions_redefine}
    Analogously to \cite{fornasier2021consensus} for standard CBO,
    the proof of Theorem~\ref{thm:well_posedness} implies that, 
    for a weak solution $\rho \in \CC \left([0,T], \CP_4(\R^d) \right)$ to the Fokker-Planck Equation~\eqref{eq:fokker_planck},
    we can extend the standard test function space $\CC_c^{\infty}(\bbR^d)$ to the space $\CC^2_*(\bbR^d)$, where
    \begin{equation}
        \label{eq:CCstar}
        \CC_*^2 (\R^d)
        := \left\{ \phi \in \CC^2(\R^d): \|\nabla \phi(\theta)\| \leq C \left(1 + \|\theta\|_2 \right) \text{ and } \sup_{\theta \in \R^d} |\Delta \phi (\theta)| < \infty \right\}.
    \end{equation}
\end{remark}


\subsection{Global Convergence of the CB\texorpdfstring{\textsuperscript{2}}{2}O Dynamics in Mean-Field Law}
\label{sec:convergence}

We now present the main result about the global convergence of the CB\textsuperscript{2}O dynamics~\eqref{eq:dyn_macro} in mean-field law for a lower-level objective function~$L$ and an upper-level objective function~$G$ that satisfy the following assumptions. We suggest to the reader to take another look at Figure~\ref{fig:CB2OIllustration} to recall a typical setting for a bi-level optimization problem of the form \eqref{eq:bilevel_opt} in order to convince themselves that the assumptions on the objective functions~$L$ and $G$ that we make below are reasonable.

\begin{assumption}\label{def:assumptions}
	Throughout we are interested in objective functions $L \in \CC(\bbR^d)$ and $G \in \CC(\bbR^d)$,
    for which
	\begin{enumerate}[label=A\arabic*,labelsep=10pt,leftmargin=35pt]
		\item\label{asm:thetaG} there exists a unique $\thetaG\in\Theta:=\argmin_{\theta\in\bbR^d} L(\theta)$ with $\underbar{L}:=L(\thetaG)=\inf_{\theta\in\bbR^d} L(\theta)$ such that \begin{align}
            G(\thetaG) = \inf_{\theta^*\in\Theta} G(\theta^*),
		\end{align}
        \item\label{asm:LipL} there exist $h_L,R^H_{L}>0$, and $H_L<\infty$ such that
        \begin{align}
            L(\theta)-\underbar{L}
            \leq H_L\N{\theta-\thetaG}_2^{h_L}
            \quad \text{ for all } \theta\in B_{R^H_{L}}(\thetaG),
        \end{align}
        \item\label{asm:icpL} there exist $L_\infty,R_L,\eta_L > 0$, and $\nu_L \in (0,\infty)$ such that
        \begin{subequations}
		\begin{align}
			\label{eq:icpL_1}
			\dist(\theta,\Theta)
            &\leq \frac{1}{\eta_L}\left(L(\theta)-\underbar{L}\right)^{\nu_L} \quad \text{ for all } \theta \in \CN_{R_L}(\Theta),\\
            \label{eq:icpL_2}
            L(\theta)-\underbar{L}
            &> L_\infty \quad \text{ for all } \theta \in \big(\CN_{R_L}(\Theta)\big)^c,
		\end{align}
        \end{subequations}
        \item\label{asm:LipG} there exist $h_G,R^H_{G}>0$, and $H_G<\infty$ such that
        \begin{align}
            G(\theta)-G(\thetaG)
            \leq H_G\N{\theta-\thetaG}_2^{h_G}
            \quad \text{ for all } \theta\in B_{R^H_{G}}(\thetaG),
        \end{align}
        \item\label{asm:icpG} there exist $G_{\infty},R_G,\eta_G > 0$, and $\nu_G \in (0,\infty)$ such that for all $r_G\leq R_G$ there exists $\thetaGtilde\in B_{r_G}(\thetaG)$ such that
        \begin{subequations}
		\begin{align}
			\label{eq:icpG_1}
			\Nbig{\theta-\thetaGtilde}_2
            &\leq \frac{1}{\eta_G}\left(G(\theta)-G(\thetaGtilde)\right)^{\nu_G} \quad \text{ for all } \theta \in B_{r_G}(\thetaG),\\
			\label{eq:icpG_2}
			G(\theta)-G(\thetaGtilde)
            &> G_\infty \quad \text{ for all } \theta \in \CN_{r_G}(\Theta)\backslash B_{r_G}(\thetaG).
		\end{align}
        \end{subequations}
    \end{enumerate}
\end{assumption}

In Assumption~\ref{asm:thetaG}, we clearly state that we assume there is a unique minimizer to the bilevel optimization problem \eqref{eq:bilevel_opt}. Typically, $L$ has more than one global minimizer, and it may even take the minimal value at an entire set of any shape in $\bbR^d$ subject to the additional assumptions~\ref{asm:LipL}--\ref{asm:icpL}.
All those global minimizers have objective value $\underbar{L}$.
Amongst them, $\thetaG$ is {\it the} global minimizer performing best w.r.t.\@ the robustness-promoting upper-level objective function~$G$, i.e., the good global minimizer. 
Let us recall at this point that we additionally assumed in the introduction that $\thetaG \in B_R(0)$, where $R>0$ is supposed to be sufficiently large such that $R\geq\Nbig{\thetaG}_2+r_R$ for $r_R>0$ as defined earlier.

\textit{Lower-level objective function $L$:} Assumption~\ref{asm:LipL} imposes that the objective function $L$ is $h_L$-H\"older continuous with constant $H_L$ locally around $\thetaG$.
The lower-level objective function $L$ moreover satisfies a tractability condition on its landscape in the form of Assumption~\ref{asm:icpL}.
In the neighborhood of the set of global minimizers~$\Theta$, i.e., on the set~$\CN_{R_L}(\Theta)$, the first part of \ref{asm:icpL}, Equation~\eqref{eq:icpL_1}, requires the local coercivity of $L$ w.r.t.\@ $\Theta$, thereby ensuring that the objective function values of $L$ on the level sets give instructive information about the distance of the associated points in the level sets to the set of global minimizers.
Such condition is known as an inverse continuity property~\cite{fornasier2021consensus,fornasier2021convergence,fornasier2020consensus_sphere_convergence,riedl2022leveraging}.
Compared to the assumptions made in those references, however, \ref{asm:icpL} is more local in nature by being an assumption on the distance to the whole set~$\Theta$, instead of to a single point.
This in particular warrants the existence of multiple global minimizers of the lower-level objective function~$L$.
For similar assumptions in the literature, we refer to quadratic growth or error bound conditions, as, e.g., discussed in \cite{anitescu2000degenerate, necoara2019linear,bolte2017error}.
Away from the neighborhood~$\CN_{R_L}(\Theta)$ of the set of global minimizers~$\Theta$, the second part of \ref{asm:icpL}, Equation~\eqref{eq:icpL_2}, merely ensures that there is a sufficient gap in the objective function values to the minimal value $\underbar{L}$,
thereby preventing in particular that $L(\theta) \approx\underbar{L}$ for some $\theta\in\bbR^d$ considerably far from $\Theta$.
For an illustration of this assumption, both Figures~\ref{fig:CB2OIllustration} and \ref{fig:A5_Illustration} may be consulted.
\begin{figure}[!htb]
    \centering
    \includegraphics[trim=37 258 12 250,clip,width=0.6\textwidth]{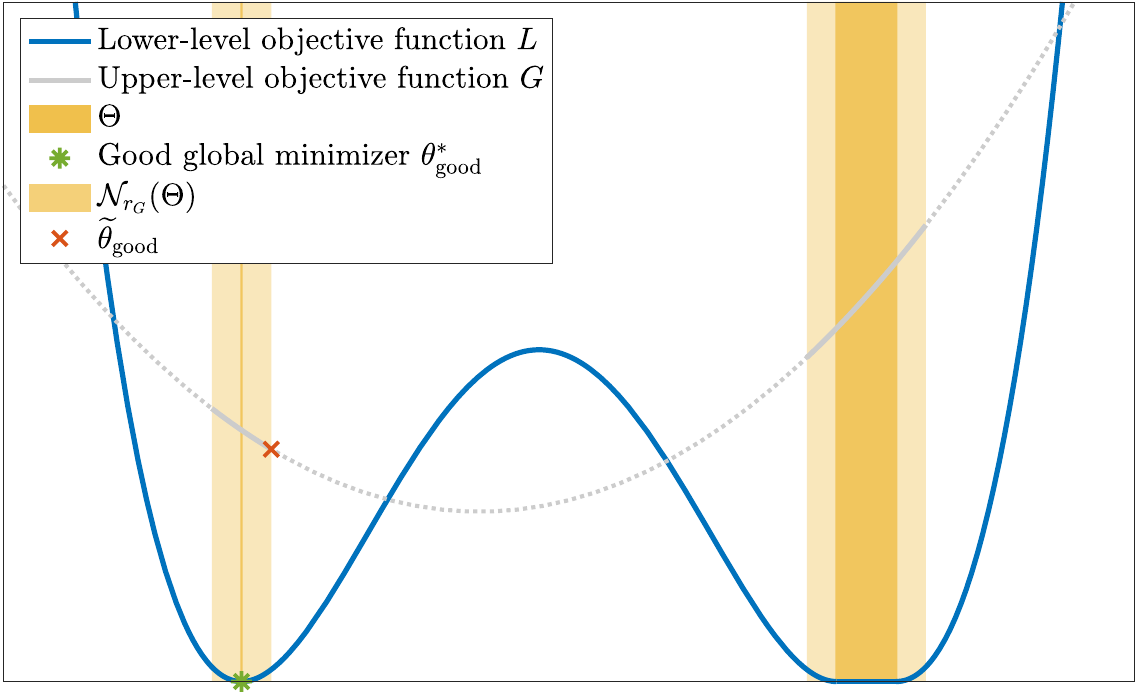}
    \caption{Illustration of Assumptions~\ref{asm:icpL} and \ref{asm:icpG} on the lower-level objective function~$L$ and the upper level objective function~$G$.}
    \label{fig:A5_Illustration}
\end{figure}

\textit{Upper-level objective function $G$:} 
Assumption~\ref{asm:LipG} imposes that the objective function $G$ is $h_G$-H\"older continuous with constant~$H_G$ locally around $\thetaG$.
In addition, Assumption~\ref{asm:icpG} requires the upper-level objective function $G$ to admit an inverse continuity condition~\cite{fornasier2021consensus,fornasier2021convergence,fornasier2020consensus_sphere_convergence,riedl2022leveraging}, which has to hold exclusively on a neighborhood of the set of global minimizers~$\Theta$ of $L$.
More precisely, \ref{asm:icpG} imposes an inverse continuity condition for all neighborhoods~$\CN_{r_G}(\Theta)$ with $r_G\leq R_G$ for some (potentially sufficiently small) $R_G>0$.
Outside the set~$\CN_{R_G}(\Theta)$, the behavior of the function~$G$ is of no relevance.
On any such neighborhood $\CN_{r_G}(\Theta)$, 
the minimizer of the upper-level objective function $G$, which we denote by $\thetaGtilde$ (the reader may want to convince themselves that $\thetaGtilde$ is typically not $\thetaG$; see Figure \ref{fig:A5_Illustration} for an illustration), is required to lie close to $\thetaG$; to be precise, within the ball~$B_{r_G}(\thetaG)\subset\CN_{r_G}(\Theta)$ of radius~$r_G$ around $\thetaG$.
On this ball, the objective function $G$ is locally coercive around $\thetaGtilde$ as required by the first condition of \ref{asm:icpG}, Equation~\eqref{eq:icpG_1}.
Outside of $B_{r_G}(\thetaG)$, but for $\theta\in\CN_{r_G}(\Theta)$, the second condition of \ref{asm:icpG}, Equation~\eqref{eq:icpG_2}, ensures that $G(\theta)\approx G(\thetaGtilde)=\min_{\tilde\theta\in\CN_{r_G}(\Theta)} G(\tilde\theta)$ does not occur.
That is, far away from $\thetaGtilde$ but within the neighborhood $\CN_{r_G}(\Theta)$, the value $G(\thetaGtilde)$ is not attained.
This resembles the standard inverse continuity condition~\cite{fornasier2021consensus,fornasier2021convergence,fornasier2020consensus_sphere_convergence,riedl2022leveraging}, however, restricted to the neighborhood $\CN_{r_G}(\Theta)$ and w.r.t.\@ the point $\thetaGtilde$.
More intuitively speaking, recall that $\thetaG$, by definition of the bi-level optimization problem~\eqref{eq:bilevel_opt}, is the global minimizer of the upper-level objective function~$G$ on the set $\Theta$.
Hence, \ref{asm:icpG} can be regarded as a stability condition of the global minimizer of $G$ for neighborhoods of $\Theta$.
It extends the aforementioned optimality property to those neighborhoods~$\CN_{r_G}(\Theta)$ of $\Theta$ with $r_G\leq R_G$, while allowing at the same time an error of order $r_G$ when wishing to approximate $\thetaG$.
This error is necessary and natural as $G$ typically admits (in realistic scenarios) on the set $\CN_{r_G}(\Theta)$ a global minimizer $\thetaGtilde\not=\thetaG$.
\ref{asm:icpG} ensures that $\thetaGtilde$ is $r_G$-close to $\thetaG$, hence excluding cases where $\thetaGtilde$ might have jumped to, e.g., a separate component of the set $\CN_{r_G}(\Theta)$.
Here, we again refer for an illustration to both Figures~\ref{fig:CB2OIllustration} and \ref{fig:A5_Illustration}.

We are now ready to state the main result about the global convergence of the mean-field CB\textsuperscript{2}O dynamics~\eqref{eq:fokker_planck}.
The proof is deferred to Section~\ref{sec:convergence_proofs}.
\begin{theorem}[Convergence of the mean-field CB\textsuperscript{2}O dynamics~\protect{\eqref{eq:dyn_macro}} and~\protect{\eqref{eq:fokker_planck}}]
    \label{thm:main}
    Let $L \in \CC(\bbR^d)$ and $G \in \CC(\bbR^d)$ satisfy \ref{asm:thetaG}--\ref{asm:icpG}.
	Moreover, let $\rho_0 \in \CP_4(\bbR^d)$ be such that $\thetaG\in\supp{\rho_0}$.
	Fix any $\varepsilon \in (0,W_2^2(\rho_0,\delta_{\thetaG})/2)$ and $\vartheta \in (0,1)$,
    choose parameters $\lambda,\sigma > 0$ with $2\lambda > d\sigma^2$, and
    define the time horizon
	\begin{align} \label{eq:end_time_star_statement}
		T^* := \frac{1}{(1-\vartheta)\big(2\lambda-d\sigma^2\big)}\log\left(\frac{W_2^2(\rho_0,\delta_{\thetaG})/2}{\varepsilon}\right).
	\end{align}
	Then,
	for $\delta_q>0$ in \eqref{eq:Q_beta} sufficiently small (see Remark~\ref{remark:deltaq}),
	there exist $\beta_0 > 0$ and $\alpha_0 > 0$ (see Remarks~\ref{remark:beta0} and~\ref{remark:alpha0}, as well as Figures~\ref{fig:CB2OIllustration_beta} and~\ref{fig:CB2OIllustration_alpha}),
    depending (among problem-dependent quantities) on $R$, $\delta_q$, $\varepsilon$ and $\vartheta$,
    such that for all $\beta < \beta_0$ and $\alpha > \alpha_0$,
    if $\rho \in \CC([0,T^*], \CP_4(\bbR^d))$ is a weak solution to the Fokker-Planck Equation~\eqref{eq:fokker_planck} on the time interval $[0,T^*]$ with initial condition $\rho_0$ and if the mapping $t \mapsto \mAlphaBeta{\rho_t}$ is continuous for $t\in[0,T^*]$,
    it holds 
    \begin{align}
        \label{eq:end_time_statement}
        W_2^2(\rho_T,\delta_{\thetaG})/2 = \varepsilon
        \quad\text{ with }\quad
        T\in\left[\frac{1-\vartheta}{(1+\vartheta/2)}\;\!T^*,T^*\right].
    \end{align}
	Furthermore, on the time interval $[0,T]$, $W_2^2(\rho_t,\delta_{\thetaG})$ decays at least exponentially fast.
    More precisely,
	\begin{align} \label{eq:thm:global_convergence_main:V}
		W_2^2(\rho_t,\delta_{\thetaG})
        \leq W_2^2(\rho_0,\delta_{\thetaG}) \exp\left(-(1-\vartheta)\left(2\lambda- d\sigma^2\right) t\right)
	\end{align}
    for all $t\in[0,T]$.
\end{theorem}

A few comments about Theorem~\ref{thm:main} are in order.
For a discussion regarding the choices of the central hyperparameters $\beta$ and $\alpha$ we refer to the in-depth discussions below in Remarks~\ref{remark:beta0} and~\ref{remark:alpha0}, respectively.
First we note, as already noted for Theorem~\ref{thm:well_posedness}, that the assumption $\thetaG\in\supp{\rho_0}$ about the initial configuration $\rho_0$ is not really a restriction,
as it would anyhow hold immediately for $\rho_t$ for any $t > 0$ in view of the diffusive character of
the dynamics~\eqref{eq:dyn_macro}; see also \cite[Remark~4.8]{fornasier2021consensus}.
In fact, this has been rigorously proven recently in \cite{fornasier2025regularity}.
Second, the dimensionality dependence of the condition $2\lambda > d\sigma^2$ on the drift and noise parameter may be removed by using anisotropic noise~\cite{carrillo2019consensus,fornasier2021convergence}, which would require that $2\lambda > \sigma^2$.
Such conditions may be further relaxed by truncating the noise; see \cite{fornasier2023consensus}.
Third, the rate of convergence in \eqref{eq:thm:global_convergence_main:V} is at least $(1-\vartheta)(2\lambda-d\sigma^2)$ and, as can be seen from the proof, at most $(1+\vartheta/2)(2\lambda-d\sigma^2)$.
Thus, at the cost of taking $\alpha\rightarrow\infty$ to allow for $\vartheta\rightarrow0$, it can be made arbitrarily close to the rate $(2\lambda-d\sigma^2)$.
This at the same time renders more precise the time horizon $T$ in \eqref{eq:end_time_statement} at which time the desired $\varepsilon$ accuracy is reached for the Wasserstein distance $W_2^2(\rho_T,\delta_{\thetaG})/2$.
Fourth, we require in the statement of the theorem that $\rho \in \CC([0,T^*], \CP_4(\bbR^d))$ has to be a weak solution to the Fokker-Planck Equation~\eqref{eq:fokker_planck} with initial condition $\rho_0$ \textit{and} the mapping $t \mapsto \mAlphaBeta{\rho_t}$ must be continuous.
Theorem~\ref{thm:well_posedness} in the preceding section proves that solutions to the mean-field Fokker Planck equation with the desired regularity and for which the map $t \mapsto \mAlphaBeta{\rho_t}$ is continuous indeed exist under the assumptions stated in this result. Those assumptions are therefore sufficient, but they may not be necessary and
in particular, the statement of the above theorem may be valid under a different set of assumptions as long as the aforementioned requirements about the regularity of $\rho$ as well as the continuity of the consensus point mapping are satisfied.

Let us now highlight how the hyperparameters~$\beta$ and $\alpha$ influence the behavior of the CB\textsuperscript{2}O dynamics~\eqref{eq:dyn_macro} and thereby the CB\textsuperscript{2}O algorithm~\eqref{eq:dyn_micro}.
The illustrations and insights that we provide here hint at how we later rigorously prove in Section~\ref{sec:convergence_proofs} our main convergence statement, Theorem~\ref{thm:main}.

We start with the parameter~$\delta_q$, which is of mere proof-technical purpose and is convenient due to its regularizing effect.

\begin{remark}[Choice of $\delta_q$]
	\label{remark:deltaq}
	The requirement $\delta_q>0$ sufficiently small in Theorem~\ref{thm:main} means that
    \begin{align}
    	\delta_{q}
    	\leq \frac{1}{2} \min\left\{L_\infty,(\eta_Lr_{G,\varepsilon})^{1/\nu_L}\right\}
    \end{align}
    where $r_{G,\varepsilon}$ is as defined explicitly in \eqref{eq:rGeps}.
\end{remark}

Let us now address the choice and influence of the hyperparameter~$\beta$.

\begin{remark}[Choice of hyperparameter $\beta$]
	\label{remark:beta0}
    The hyperparameter $\beta$ is related to the lower-level objective function~$L$ and steers (in the finite particle terminology) the particle selection. 
    In the mean-field interpretation, it analogously restricts the particle density~$\varrho$ to its most favorable regions w.r.t.\@ $L$.
    Thereby, $\beta$ warrants that CB\textsuperscript{2}O is successful in identifying the set~$\Theta$ of global minimizers of $L$, which is a prerequisite to eventually succeeding in determining a point in the proximity of $\Theta$ that is best w.r.t. the upper-level objective function~$G$.
    Note that at this point and with the parameter $\beta$ only we can not drive the dynamics towards$\thetaG$.
    Rather, our aim at this stage is to more and more accurately identify $\Theta$ as $\beta$ becomes smaller and smaller.
    
    This selection principle is implemented through the $\beta$-quantile $\qbeta{\varrho}$ and the associated level set $\Qbeta{\varrho}$, which identifies those regions of the domain of $L$ that have the best (lowest) lower-level objective function value.
    Only the parts of the density supported on this set are then taken into account when computing the consensus point~$\mAlphaBeta{\varrho} = \mAlpha{\Ibeta{\varrho}}$ in \eqref{eq:consensus_point},
    where $\Ibeta{\varrho} = \mathds{1}_{\Qbeta{\varrho}}\varrho$ denotes the associated restricted measure which exerts precisely this selection principle.
    Phrased differently, the consensus point is exclusively computed on the quantile set $\Qbeta{\varrho}$\footnote{A word of caution: Note that this does not yet ensure that the consensus point~$\mAlphaBeta{\varrho}$ itself lies, even if computed on it, within the set $\Qbeta{\varrho}$, let alone that it approximates $\thetaG\in\Theta$.
    That we are indeed able to achieve the latter two properties is the topic of Remark~\ref{remark:alpha0} and thanks to the inverse continuity property~\ref{asm:icpG} on $G$ and the hyperparameter~$\alpha$ as we discuss there.}.

    With the set $\Qbeta{\varrho}$, and thus through the hyperparameter~$\beta$, we can identify a set which is expected to approximate $\Theta$ up to an accuracy level that is controllable through $\beta$.
    The quality of the approximation of the set $\Theta$ with the set $\Qbeta{\varrho}$ improves as $\beta\rightarrow0$ (given some regularity of $L$).
    As we let $\beta$ become smaller and smaller,
    by construction, $\Qbeta{\varrho}$ is induced to contain only points with objective function value closer and closer to $\underbar{L}$, narrowing down the set $\Qbeta{\varrho}$ towards $\Theta$ in the process.
    As soon as $\beta=0$, $\Qbeta{\varrho}=\Theta$.
    This behavior is illustrated in Figure~\ref{fig:CB2OIllustration_beta}, where we plot the set $\Qbeta{\varrho}$, i.e., the support of the measure~$\Ibeta{\varrho} = \mathds{1}_{\Qbeta{\varrho}}\varrho$ (provided that $\varrho$ were supported everywhere),
    for decreasing (from left to right) values of $\beta$.
    \begin{figure}[!htb]
        \centering
        \includegraphics[trim=118 228 88 222,clip,width=1\linewidth]{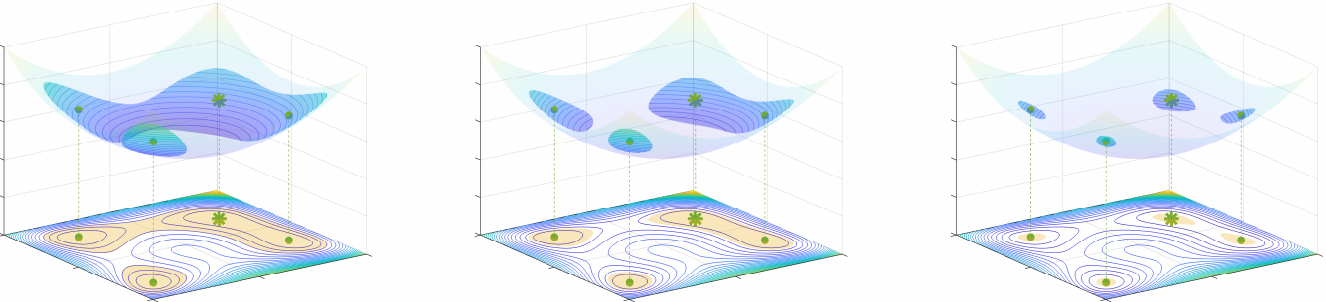}
        \caption{The influence of the hyperparameter~$\beta$. By decreasing the value of $\beta$ from left to right, we shrink the quantile set $\Qbeta{\varrho}$ (illustrated as yellow shadows on the $xy$ plane and projected onto the upper-level objective), thereby ensuring an increasingly finer approximation of the set $\Theta$ of global minimizers of $L$.}
        \label{fig:CB2OIllustration_beta}
    \end{figure}
    
    In order to technically approximate the set $\Theta$ through $\Qbeta{\varrho}$ sufficiently well,
    we need to choose in Theorem~\ref{thm:main} the hyperparameter~$\beta$ sufficiently small in the sense that
    $\beta<\beta_0$ with
    \begin{equation} \label{eq:beta0}
        \beta_0
        := \frac{1}{2}\rho_0\big(B_{r_{H,\varepsilon}/2}(\thetaG)\big) \exp(-p_{H,\varepsilon}T^*),
    \end{equation}
    where $r_{H,\varepsilon}$ is as in \eqref{eq:rHeps} and where $p_{H,\varepsilon}$ is defined as discussed after \eqref{eq:beta}.
\end{remark}

Being now able to approximate the set $\Theta$ with $\Qbeta{\varrho}$ through the choice of $\beta$, our next task is to locate $\thetaG$ within this set.
For this we need to discuss the choice and influence of the hyperparameter~$\alpha$.

\begin{remark}[Choice of hyperparameter $\alpha$]
    \label{remark:alpha0}
    The hyperparameter $\alpha$ is related to the upper-level objective function~$G$ and is concerned with the actual approximation of the good global minimizer~$\thetaG$ within the set $\Qbeta{\varrho}$.
    To this end,
    we actually don't aim at approximating $\thetaG$ directly but instead intend to find the global minimizer~$\thetaGtilde$ of $G$ restricted to the set $\Qbeta{\varrho}$.
    However, granted the stability of the global optimizer of $G$ for neighborhoods of $\Theta$, which is ensured by the inverse continuity property~\ref{asm:icpG}, $\thetaGtilde$ is close to $\thetaG$.
    This closeness has been controlled and quantified through the choice of $\beta$ in Remark~\ref{remark:beta0}.
    Thus, it is sufficient to approximate $\thetaGtilde$ when computing the consensus point~\eqref{eq:consensus_point}.

    In order to do so, we leverage a quantitative nonasymptotic variant of the classical Laplace principle~\cite{dembo2009large,miller2006applied}, namely \cite[Proposition~4.5]{fornasier2021consensus}, which ensures that $\mAlphaBeta{\varrho}=\mAlpha{\Ibeta{\varrho}}$ can approximate $\thetaGtilde=\argmin_{\theta\in\Qbeta{\varrho}} G(\theta)$ arbitrarily well provided a large enough choice of the hyperparameter~$\alpha$.
    In the mean-field limit, the quality of this approximation improves as $\alpha\rightarrow\infty$.
    As we let $\alpha$ become larger and larger, $\mAlphaBeta{\varrho}$ better and better approximates $\thetaGtilde$.
    As $\alpha\rightarrow\infty$, $\mAlphaBeta{\varrho}\rightarrow\thetaGtilde$.
    We illustrate this behavior in Figure~\ref{fig:CB2OIllustration_alpha}, where we depict the consensus point~$\mAlphaBeta{\varrho}$ for increasing (from left to right) values of $\alpha$.
    \begin{figure}[!htb]
        \centering
        \includegraphics[trim=118 228 88 222,clip,width=1\linewidth]{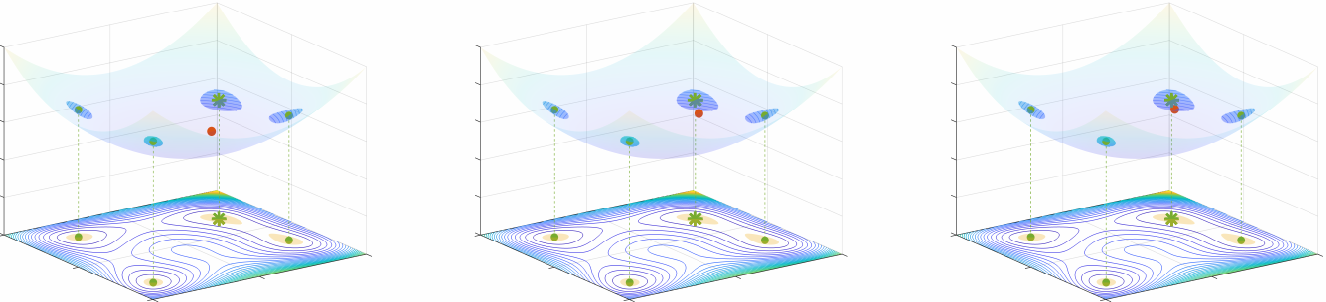}
        \caption{The influence of the hyperparameter~$\alpha$. By increasing the value of $\alpha$ from left to right, we improve on the approximation of $\thetaGtilde$, i.e., how well $\mAlphaBeta{\varrho}$, depicted as the orange dot in all three plots, eventually approximates $\thetaGtilde$.}
        \label{fig:CB2OIllustration_alpha}
    \end{figure}

    In order to technically approximate $\thetaGtilde$ (and thus $\thetaG$ by Remark~\ref{remark:beta0} and \ref{asm:icpG}) through $\mAlphaBeta{\varrho}$ sufficiently well,
    we need to choose in Theorem~\ref{thm:main} the hyperparameter~$\alpha$ sufficiently large in the sense that
    $\alpha>\alpha_0$ with
    \begin{equation} \label{eq:alpha0}
    \begin{split}
        \alpha_0
		&:= \frac{1}{u_\varepsilon}\Bigg(\log\left(\frac{4\sqrt{2}}{c\left(\vartheta,\lambda,\sigma\right)}\right)-\log\left(\rho_{0}\big(B_{r_\varepsilon/2}(\thetaG)\big)\right) + \max\left\{\frac{1}{2},\frac{p_\varepsilon}{(1-\vartheta)\big(2\lambda-d\sigma^2\big)}\right\}\log\left(\frac{\CV(\rho_0)}{\varepsilon}\right)\!\Bigg),
    \end{split}
    \end{equation}
    where $c(\vartheta,\lambda,\sigma)$ is as in \eqref{eq:c}, where $u_\varepsilon$ and $r_\varepsilon$ are as in \eqref{eq:ueps} and \eqref{eq:reps}, respectively, and where $p_\varepsilon$ is defined as discussed after \eqref{eq:alpha}.
\end{remark}

To wrap up the discussions about the choices of the hyperparameters~$\beta$ and $\alpha$, let us point out their interplay in the convergence proof, which is presented in Section~\ref{sec:convergence_proofs}. When deriving time-evolution inequalities for $W_2^2(\rho_t,\delta_{\thetaG})$, as we do in Lemma~\ref{lem:evolution_V}, the central quantity to be controlled by suitable choices of the hyperparameters is $\Nbig{\mAlphaBeta{\rho_t}-\thetaG}_2$.
To obtain this control, we employ the quantitative quantiled Laplace principle in Proposition~\ref{lem:laplace_LG}, where we identify two error contributions. The first of which comes from the approximation accuracy of the set $\Theta$ of global minimizers of the lower-level objective function $L$ by the set $\Qbeta{\rho_t}$.
The second contribution arises from approximating, within this set, the global minimizer~$\thetaGtilde$ of the upper-level objective function $G$ by $\mAlphaBeta{\rho_t}$.
In summary, we have a decomposition of the error $\Nbig{\mAlphaBeta{\rho_t}-\thetaG}_2$ of the form
\begin{equation}
    \label{eq:conspoint_errordecomp}
    \Nbig{\mAlphaBeta{\rho_t}-\thetaG}_2
    \leq \underbrace{\Nbig{\mAlphaBeta{\rho_t}-\thetaGtilde}_2}_{\substack{\text{small by choice of $\alpha$ since} \\\text{$\thetaGtilde=\argmin_{\theta\in\Qbeta{\rho_t}} G(\theta)$}}} + \underbrace{\Nbig{\thetaGtilde-\thetaG}_2}_{\substack{\text{small by choice of $\beta$ since}\\\text{$\thetaGtilde\in\Qbeta{\rho_t}$}}}.
\end{equation}
Let us point out, though, that the decomposition in the proof of Proposition~\ref{lem:laplace_LG} in Section~\ref{sec:laplace_p} is slightly more implicit.
However, \eqref{eq:conspoint_errordecomp} does paint the correct picture for an understanding of the interplay between the hyperparameters~$\beta$ and $\alpha$.
For more technical details we refer to Remark~\ref{rem:error_decomp_laplace}.

Before closing this section about the main results of the paper and providing the proof details as well as numerical experiments thereafter,
let us continue the discussion from Remark~\ref{rem:gradient_drift} on integrating a gradient drift w.r.t.\@ the lower-level objective function~$L$ into the CB\textsuperscript{2}O algorithm~\eqref{eq:dyn_micro} as well as the associated CB\textsuperscript{2}O mean-field dynamics~\eqref{eq:dyn_macro}.
For standard CBO~\cite{pinnau2017consensus,carrillo2018analytical,fornasier2021consensus,fornasier2021convergence}, this has been done and analyzed in \cite{riedl2022leveraging}.

\begin{remark}[Mean-field law convergence of the CB\textsuperscript{2}O dynamics~\eqref{eq:dyn_micro_add_grad} with additional gradient drift w.r.t.\@~$L$]
    As discussed in Remark~\ref{rem:gradient_drift}, it is reasonable and in some applications advisable to add an gradient drift w.r.t.\@~$L$ in the CB\textsuperscript{2}O algorithm~\eqref{eq:dyn_micro_add_grad}, i.e., $\lambda_\nabla\not=0$.
    In such case, the lower-level objective function $L$ naturally needs to be continuously differentiable and we additionally require that $L \in\CC^1(\bbR^d)$ is such that
	\begin{enumerate}[label=A\arabic*,labelsep=10pt,leftmargin=35pt]
		\setcounter{enumi}{5}
		\item\label{asm:smoothnessL}
        there exists $C_{\nabla L} > 0$ such that
		\begin{align}
            \N{\nabla L(\theta)-\nabla L(\theta')}_2 \leq C_{\nabla L}\N{\theta-\theta'}_2 \quad \text{for all } \theta,\theta' \in\bbR^d.
		\end{align}
	\end{enumerate}
    Assumption~\ref{asm:smoothnessL} is a standard smoothness assumption on $L$ which requires $\nabla L$ to be Lipschitz continuous.

    Under this additional assumption, the convergence statement of Theorem~\ref{thm:main} can be extended to the mean-field dynamics of \eqref{eq:dyn_micro_add_grad} by following \cite{riedl2022leveraging}.
\end{remark}

\section{Proof Details for Section~\ref{sec:well_posedness}}
\label{sec:well_posedness_proofs}

In this section, we provide the details of the proof of Theorem \ref{thm:well_posedness}, the existence of regular solutions to the mean-field dynamics \eqref{eq:dyn_macro}. In Section \ref{sec:existence_proof_sketch}, we present a proof sketch to outline the main steps.
In Section \ref{sec:conspoint_stability}, we provide a concrete example showing that, different from the consensus points defined in the standard CBO setting \cite{carrillo2018analytical}, our consensus point $\mAlphaBetanoarg$ is not stable under Wasserstein perturbations.
This lack of stability forces us to develop a stability estimate for $\mAlphaBetanoarg$ with respect to more restrictive perturbations as presented in Proposition \ref{thm:stabilityEst}.
In Section \ref{sec:well_posedness_proof}, we complete the proof of Theorem \ref{thm:well_posedness}.

\subsection{Proof Sketch}\label{sec:existence_proof_sketch}
Suppose that $\rho = \mathrm{Law} (\overbar{\theta})$ is a solution to the Fokker-Planck Equation~\eqref{eq:fokker_planck} satisfying $\rho \in \CC \left([0,T], \R^d \right)$. Then we may apply standard theory of SDEs (see, e.g., \cite[Chapter~6.2]{arnold1974stochasticdifferentialequations}) to conclude that there exists a unique strong solution to the mean-field dynamics \eqref{eq:dyn_macro}.
Therefore, in what follows we mainly prove that there exists a solution to the PDE \eqref{eq:fokker_planck} with the desired regularity.

The proof is based on the Picard's iteration method. In what follows, we denote by $\rho^n_t\in\CP(\bbR^d)$ the measure at time $t$ of the $n$-th iterate in the Picard-type recursion. Precisely, we set $\rho^0_t := \rho_0$ for all $t>0$, i.e.,  $\rho^{n=0}_{t}(\theta) = \rho_0(\theta)$ for all $\theta\in\bbR^d$ and for all $t>0$.
Then, having defined $\rho^n$ for a given $n\in \mathbb{N}$, we let $\rho^{n+1}$ be the solution to the \textit{linear} Fokker-Planck equation
\begin{equation}
	\partial_t\rho^{n+1}_t
	= \lambda\nabla \cdot \left(\left(\theta -\mAlphaBetaRED{\rho^n_t}\right)\rho^{n+1}_t\right)
	+ \frac{\sigma^2}{2}\Delta \left(\N{\theta-\mAlphaBetaRED{\rho^n_t}}_2^2\rho^{n+1}_t\right)
 \label{eqn:PicardIteration}
\end{equation}
with initial data $\rho^{n+1}_0 = \rho_0 \in H^{l+2}(\R^d) \cap L^{\infty}(\R^d) \cap \CP_4(\R^d)$. 
We will show by induction that 
\begin{equation}\label{eq:regularity}
\begin{aligned}
&\{\rho^{n+1}\}_{n\in \mathbb{N}} \subseteq W^{1, \infty} ([0,T], H^l(\R^d)) \cap L^{\infty} ([0,T] \times \R^d) \cap \CC ([0,T], \CP_2(\R^d)) \, .
\end{aligned}
\end{equation}
Moreover, we prove that the sequence $\{\rho^{n+1}\}_n$ has uniform bounded norms in each of those spaces.

Secondly, by the Aubin-Lions lemma and a diagonal argument, we are able to conclude that there exists a subsequence $\{ \rho^{n_k} \}_{k\in \NN}$ of $\{ \rho^n \}_{n \in \NN}$ and a limiting function $\rho$ such that $\rho^{n_k}$ strongly converges to $\rho$ in $\CC ([0,T], \Lloc(\R^d))$. 
Since $\{\rho^{n_k}\}_k$ has uniformly bounded norms, we can further conclude that the limiting function $\rho$ satisfies the similar regularity of $\{\rho^{n_k}\}_k$, specified in \eqref{eq:Sol_regularity}.

Lastly, we verify that $\rho$ is a solution to the Equation~\eqref{eq:fokker_planck} in the weak sense (Definition \ref{def:weak_solution}). 
In Equation~\eqref{eqn:PicardIteration}, by $m(\varrho)$ we mean $\mAlphaBeta{\varrho}$ as defined in Section \ref{sec:CB2O}. We have simplified the notation for the reader's convenience and also to highlight that our approach in the proof of Theorem \ref{thm:well_posedness}, which is presented in Section \ref{sec:well_posedness_proof}, continues to work for general functions $\mAlphaBetaRED{\dummy}$ satisfying a suitable stability estimate. 

\subsection{Stability Estimate for the Consensus Point~\texorpdfstring{$\mAlphaBetanoarg$}{}}
\label{sec:conspoint_stability}

A stability estimate for the consensus point~$\mAlphaBeta{\varrho}$ w.r.t.\@ perturbations of the measure~$\varrho$ is crucial for proving the well-posedness of the mean-field dynamics~\eqref{eq:dyn_macro} and \eqref{eq:fokker_planck},
as well as for establishing results about a mean-field approximation~\cite[Section~3.3 and 5]{fornasier2021consensus}.
For the standard CBO system~\cite{pinnau2017consensus, carrillo2018analytical} used for simple unconstrained optimization,
the consensus point is stable w.r.t.\@ Wasserstein perturbations~\cite[Lemma~3.2]{carrillo2018analytical}, allowing for an estimate of the form~\eqref{eq:WassStability}.
This stability enables an effective application of the coupling method as worked out in \cite{carrillo2018analytical} to prove existence (and uniqueness) of solutions of the corresponding mean-field PDE.

However, due to the particle selection principle implemented through the choice of quantile of the lower-level objective function~$L$, our consensus point $\mAlphaBeta{\varrho}$, as defined in \eqref{eq:consensus_point},
is no longer stable w.r.t.\@ Wasserstein perturbations of the measure~$\varrho$. We demonstrate this through a counterexample in Remark~\ref{rem:counter_example}. Indeed, we show that $\mAlphaBetanoarg$ cannot be Lipschitz continuous w.r.t.\@ the Wasserstein distance (alone). In order to resolve this issue, we provide in Proposition~\ref{thm:stabilityEst} a new stability estimate for our consensus point $\mAlphaBetanoarg$ by showing that $\mAlphaBetanoarg$ is stable w.r.t.\@ a combination of Wasserstein and $L^2$ perturbations.

\begin{remark}[Lack of stability of the consensus point~$\mAlphaBetanoarg$ w.r.t.\@ Wasserstein perturbations]
\label{rem:counter_example}
    We show that there \textit{cannot} exist a constant $c > 0$ such that
    \begin{equation}
        \label{eq:WassStability}
        \N{\mAlphaBeta{\varrho} - \mAlphaBeta{\widetilde{\varrho}}}_2
        \leq c W_2 \left(\varrho, \widetilde{\varrho} \right),
    \end{equation}
    for all $\varrho, \widetilde{\varrho} \in \CP_2(\R^2)$. For this two dimensional counterexample, consider the setting where the objective functions are $L(\theta) = \N{\theta}_2$ and $G(\theta) = \N{\theta}_2$. Moreover, let us set the quantile hyperparameter $\beta$ to be $\beta = 0.3$. Consider measures~$\varrho, \widetilde{\varrho} \in \CP(\R^2)$ 
    as depicted in Figure~\ref{fig:counter_example}.
    The mass of the measure $\varrho$ is distributed equally over two concentric circles (depicted in gray in Figure~\ref{fig:counter_example_left}) centered around the origin. The perturbed measure $\widetilde{\varrho}$ is obtained by shifting a quarter of the total mass of the measure $\varrho$ from the right side of the outer circle to the right by a distance of $s$ (depicted in blue in Figure~\ref{fig:counter_example_right}).
    The $2$-Wasserstein distance between $\varrho$ and $\widetilde{\varrho}$ is clearly $W_2\left(\varrho, \widetilde{\varrho}\right) = \frac{1}{2} s$. On the other hand, according to definition \eqref{eq:consensus_point}, the consensus point $\mAlphaBeta{\widetilde{\varrho}}$ (depicted as a purple dot in Figure~\ref{fig:counter_example_right}) moves from $\mAlphaBeta{\varrho}$  (depicted as an orange dot in Figure~\ref{fig:counter_example_left} and \ref{fig:counter_example_right}) to the left.
    The distance between these two points is $\Nbig{\mAlphaBeta{\varrho} - \mAlphaBeta{\widetilde{\varrho}}}_2 =c_{\alpha}$, where $c_{\alpha}$ is a constant that depends on the hyperparameter $\alpha$ but not on $s$. More details on this example are provided in the caption of Figure \ref{fig:counter_example}. 
    
    Therefore, in this case, there cannot exist a constant $c > 0$ such that \eqref{eq:WassStability} holds true.
    \definecolor{blue_intuition}{RGB}{0,114,189}
    \definecolor{orange_intuition}{RGB}{201,92,47}
    \definecolor{purple_intuition}{RGB}{148,37,51}
    \begin{figure}[!htb]
    	\centering
    	\subcaptionbox{\label{fig:counter_example_left}Measure~{\color{gray}$\varrho$} and consensus point~{\color{orange_intuition}$\mAlphaBeta{{\color{gray}\varrho}}$}}{\includegraphics[trim=0 0 0 0 ,clip,width=0.4\textwidth]{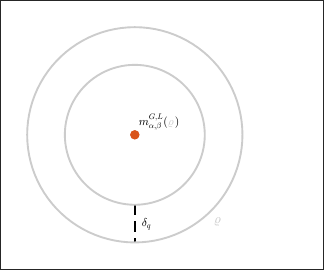}
    	}
    \hspace{2em}
    	\subcaptionbox{\label{fig:counter_example_right}Measure~{\color{blue_intuition}$\widetilde\varrho$} and consensus point~{\color{purple_intuition}$\mAlphaBeta{{\color{blue_intuition}\widetilde\varrho}}$}}{\includegraphics[trim=0 0 0 0 ,clip,width=0.4\textwidth]{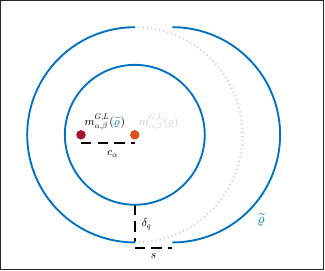}
    	}
    \caption{An illustration that the consensus point~$\mAlphaBeta{\varrho}$ as defined in \eqref{eq:consensus_point} cannot be stable w.r.t.\@ Wasserstein perturbations of the measure~$\varrho$.
    For this counterexample, consider the setting where the objective functions are $L(\theta) = \N{\theta}_2$ and $G(\theta) = \N{\theta}_2$.
    Moreover, let us set the quantile hyperparameter $\beta = 0.3$.
    On the left, in Figure~\ref{fig:counter_example_left} we depict the measure $\varrho$ together with its corresponding consensus point $\mAlphaBeta{\varrho}$. 
    The two concentric circles are level sets of the function $L$ and 
    we assume that the mass of measure $\varrho$ is equally split between the two circles and uniformly distributed over the circles.
    By construction, these two level sets correspond to the function values $\frac{2}{\beta} \int_{\beta/2}^{\beta} q_a^{L} da$ and $\frac{2}{\beta} \int_{\beta/2}^{\beta} q_a^{L} da + \delta_q$, respectively, where $q_a^L$ is the same quantile value for $\varrho$ or $\tilde \varrho$, since $a \leq \beta=0.3$. Consequently, if $R$ is set to be large enough, the measure $\Ibeta{\varrho} $ coincides with $\varrho$ and thus, based on \eqref{eq:consensus_point}, the consensus point~$\mAlphaBeta{\varrho}$ lies at the center of the circles. On the right, in Figure~\ref{fig:counter_example_right}, we construct the measure $\widetilde\varrho$ and plot it together with its corresponding consensus point $\mAlphaBeta{\widetilde\varrho}$. 
    The perturbed measure $\widetilde{\varrho}$ is obtained by shifting a quarter of the total mass of the measure $\varrho$ from the right side of the outer circle to the right by a distance of $s$ (depicted in blue in Figure~\ref{fig:counter_example_right}).
    By definition, the measure $\Ibeta{\widetilde{\varrho}}$ now includes the mass on the inner circle as well as the mass on the left side of the outer circle, but not the mass that has been shifted to the right.
    Consequently, the consensus point $\mAlphaBeta{\widetilde{\varrho}}$ shifts to the left, as shown in Figure~\ref{fig:counter_example_right}.
    The size~$c_\alpha$ of this shift, however, is independent of $s$.} 
    \label{fig:counter_example}
    \end{figure}
\end{remark}

\begin{proposition}[Stability estimate for the consensus point~$\mAlphaBetanoarg$]\label{thm:stabilityEst}
Let $L \in \CC(\bbR^d)$ and $G \in \CC(\bbR^d)$ satisfy \ref{asm:minimizers}--\ref{asm:growthBound_G}.
Moreover, let $\varrho, \widetilde{\varrho} \in \CP_2(\R^d) \cap L^{\infty}(B_R(0))$ with $\int_{\R^d} \N{\theta}_2^2 \varrho d\theta, \int_{\R^d} \N{\theta}_2^2 \widetilde{\varrho} d\theta \leq K$, and such that $\varrho (\Qbeta{\varrho} ), \widetilde{\varrho} ( \Qbeta{\widetilde{\varrho}} ) \geq c$ for constants $K < \infty$ and $c > 0$.
Then we have the stability estimate
\begin{equation}
    \N{\mAlphaBeta{\varrho} - \mAlphaBeta{\widetilde{\varrho}}}_2
    \leq C \left( \|\varrho - \widetilde{\varrho}\|_{L^2(B_R(0))} + W_2 \left( \varrho, \widetilde{\varrho} \right) \right),
\end{equation}
where the constant $C$ only depends on $R$, $K$, $c_K$  as defined in Lemma~\ref{lem:upperBound} (which depends on $c$), and $\widetilde{C}$ as derived in Lemma~\ref{lem:auxLemStability}. 
\end{proposition}

\begin{remark}
By assuming that $\varrho, \widetilde{\varrho}$ satisfy $\varrho ( \Qbeta{\varrho} ), \widetilde{\varrho} (\Qbeta{\widetilde{\varrho}} ) \geq c$, we in particular ensure that the consensus points $\mAlphaBeta{\varrho}$, $ \mAlphaBeta{\widetilde{\varrho}}$ are well-defined.
\end{remark}

Before providing the proof of Proposition~\ref{thm:stabilityEst}, we first state and prove an auxiliary lemma that we use in what follows. 

\begin{lemma}\label{lem:auxLemStability}
Let $L \in \CC(\bbR^d)$ satisfy \ref{asm:minimizers}--\ref{asm:LipschitzDistrib}.
Then we have, for any bounded function $f: \R^d \rightarrow \R^m$ such that $\sup_{\theta \in \R^d} \N{f(\theta)}_2 \leq M$ with constant $M > 0$, and $\varrho, \widetilde{\varrho} \in \CP(\R^d) \cap L^{\infty}(B_R(0))$, the inequality
\begin{equation}
\N{\int_{\Qbeta{\varrho}} f \varrho d\theta - \int_{\Qbeta{\widetilde{\varrho}}} f \widetilde{\varrho}  d\theta}_2 \leq \widetilde{C} \left(\|\varrho - \widetilde{\varrho}\|_{L^2(B_R(0))} + W_2\left(\varrho, \widetilde{\varrho} \right) \right),
\end{equation}
where the constant $\widetilde{C}$ only depends on $R$, $M$, $\|\widetilde{\varrho}\|_{L^{\infty}(B_R(0))}$ and $C_L$ (i.e., the level set constant that we assumed for $L$).
\end{lemma}

\begin{proof}
We can compute
\begin{equation}
\begin{aligned}
\left\| \int_{\Qbeta{\varrho}} f \varrho d\theta - \int_{\Qbeta{\widetilde{\varrho}}} f \widetilde{\varrho}  d\theta \right\|_2
& \leq \left\| \int_{\Qbeta{\varrho}} f \left(\varrho - \widetilde{\varrho} \right)d\theta \right\|_2 + \left\| \int_{\Qbeta{\varrho}} f \widetilde{\varrho}d\theta - \int_{\Qbeta{\widetilde{\varrho}}} f \widetilde{\varrho} d\theta \right\|_2\\
&=: T_1 + T_2,
\end{aligned}
\end{equation}
where $T_1$ and $T_2$ are defined implicitly.
For the term $T_1$, we have
\begin{equation}
\begin{aligned}
T_1 := \left\| \int_{\Qbeta{\varrho}} f \left(\varrho - \widetilde{\varrho} \right)d\theta \right\|_2 &\leq \left(\int_{\Qbeta{\varrho}} \left\| f(\theta) \right\|_2^2  d\theta \right)^{\frac{1}{2}} \left( \int_{\Qbeta{\varrho}} | \varrho(\theta) - \widetilde{\varrho}(\theta) |^2 d\theta\right)^{\frac{1}{2}}\\
&\leq \left(\int_{B_R(0)} \left\| f(\theta) \right\|_2^2  d\theta \right)^{\frac{1}{2}} \left( \int_{B_R(0)} | \varrho(\theta) - \widetilde{\varrho}(\theta) |^2 d\theta\right)^{\frac{1}{2}}\\
&\leq M \mathrm{Vol}(B_R(0))^{\frac{1}{2}} \|\varrho - \widetilde{\varrho}\|_{L^2(B_R(0))}, 
\end{aligned}
\end{equation}
where the first inequality follows from Cauchy-Schwarz inequality, the second inequality uses the fact that $\Qbeta{\varrho} \subseteq B_R(0)$ by definition \eqref{eq:Q_beta}, and the last inequality comes from $\sup_{\theta \in \R^d} \|f(\theta)\|_2 \leq M$.
For the second term $T_2$, without loss of generality, let us assume $\int_{\beta/2}^{\beta} \qa{\widetilde{\varrho}} da \leq \int_{\beta/2}^{\beta} \qa{\varrho} da$, and further denote
\begin{equation}
    \mathcal{Q}_{\beta}[\varrho, \widetilde{\varrho}] := \left\{\theta \in B_R(0) : \frac{2}{\beta}\int_{\beta/2}^{\beta} \qa{\widetilde{\varrho}} da + \delta_q \leq L(\theta) \leq \frac{2}{\beta} \int_{\beta/2}^{\beta} \qa{\varrho} da + \delta_q \right\} .
\end{equation}
Then, again by the definition of $\Qbeta{\dummy}$ in \eqref{eq:Q_beta}, we can compute
\begin{equation}
\begin{aligned}
T_2 &:= \left\| \int_{\Qbeta{\varrho}} f \widetilde{\varrho}d\theta - \int_{\Qbeta{\widetilde{\varrho}}} f \widetilde{\varrho} d\theta \right\|_2 \\
& = \left\| \int_{\{\theta \in B_R(0):\; L(\theta) \leq \frac{2}{\beta} \int_{\beta/2}^{\beta} \qa{\varrho} + \delta_q\}} f\widetilde{\varrho} d\theta - \int_{\{\theta \in B_R(0):\; L(\theta) \leq \frac{2}{\beta} \int_{\beta/2}^{\beta} \qa{\widetilde{\varrho}} + \delta_q\}} f\widetilde{\varrho} d\theta\right\|_2\\
&\leq M \left\| \widetilde{\varrho} \right\|_{L^{\infty}(\R^d)}  \int_{\mathcal{Q}_{\beta} [\varrho, \widetilde{\varrho}]} 1 d\theta \\
&\leq M \left\| \widetilde{\varrho} \right\|_{L^{\infty}(\R^d)} C_L \frac{2}{\beta} \left| \int_{\beta/2}^{\beta}\qa{\varrho} - \int_{\beta/2}^{\beta} \qa{\widetilde{\varrho}} \right|\\
&\leq M \left\| \widetilde{\varrho} \right\|_{L^{\infty}(\R^d)} C_{\beta, L} C_L W_2 \left(\varrho, \widetilde{\varrho} \right).
\end{aligned}
\end{equation}
Here the second-to-last inequality uses Assumption \ref{asm:LipschitzDistrib}, and the last inequality employs Lemma~\ref{lem:1dWass}. We conclude the proof by combining the bounds for both terms $T_1$ and $T_2$.
\end{proof}

With this auxiliary result, we can now prove Proposition~\ref{thm:stabilityEst}.

\begin{proof}[Proof of Proposition~\ref{thm:stabilityEst}]
Let $g(\theta; \varrho) := {\theta \omegaa(\theta)}/{\left\| \omegaa\right\|_{L^1(\Ibeta{\varrho})}}$, where $\omegaa(\dummy)$ and $\Ibeta{\dummy}$ are defined as in \eqref{eq:consensus_point} and \eqref{eq:I_beta}, respectively. From the definition of the consensus point in \eqref{eq:consensus_point}, we obtain
\begin{equation}
\begin{aligned}
\left\| \mAlphaBeta{\varrho} - \mAlphaBeta{\widetilde{\varrho}} \right\|_2 &= \left\| \int_{\Qbeta{\varrho}} g(\theta; \varrho) \varrho d\theta - \int_{\Qbeta{\widetilde{\varrho}}} g(\theta; \widetilde{\varrho}) \widetilde{\varrho} d\theta \right\|_2\\
&\leq \left\|\int_{\Qbeta{\varrho}} \left(g(\theta; \varrho) - g(\theta; \widetilde{\varrho}) \right) \varrho d\theta \right\|_2 +  \left\| \int_{\Qbeta{\varrho}} g(\theta; \widetilde{\varrho}) \varrho d\theta - \int_{\Qbeta{\widetilde{\varrho}}} g(\theta; \widetilde{\varrho}) \widetilde{\varrho} d\theta\right\|_2\\
&=: T_1 + T_2.
\end{aligned}
\end{equation}
For the term $T_1$, we have
\begin{equation}
\begin{aligned}
T_1 &= \left\|\int_{\Qbeta{\varrho}} \left(g(\theta; \varrho) - g(\theta; \widetilde{\varrho}) \right) \varrho d\theta \right\|_2 \\
&= \left\| \int_{\Qbeta{\varrho}} \theta \omegaa(\theta)\left( \frac{1}{\|\omegaa\|_{L^1(\Ibeta{\varrho})}} - \frac{1}{\|\omegaa\|_{L^1(\Ibeta{\widetilde{\varrho}})}} \right) \varrho d\theta \right\|_2\\
&\leq \int_{\Qbeta{\varrho}} \|\theta\|_2 \frac{\omegaa(\theta)}{\|\omegaa\|_{L^1(\Ibeta{\varrho})} \|\omegaa\|_{L^1(\Ibeta{\widetilde{\varrho}})}} \varrho d\theta \left|\int_{\Qbeta{\varrho}} \omegaa(\theta') \varrho d\theta' - \int_{\Qbeta{\widetilde{\varrho}}} \omegaa(\theta') \widetilde{\varrho} d\theta' \right| \\
&\leq \int_{\Qbeta{\varrho}} \|\theta\|_2 \frac{ \exp\left(-2\alpha \minobj \right)}{\|\omegaa\|_{L^1(\Ibeta{\varrho})} \|\omegaa\|_{L^1(\Ibeta{\widetilde{\varrho}})}} \varrho d\theta \widetilde{C} \left(\|\varrho - \widetilde{\varrho}\|_{L^2(B_R(0))} + W_2\left(\varrho, \widetilde{\varrho} \right) \right) \\
&\leq c_K^2 \widetilde{C} \int_{\Qbeta{\varrho}} \|\theta\|_2 \varrho d\theta \left(\|\varrho - \widetilde{\varrho}\|_{L^2(B_R(0))} + W_2\left(\varrho, \widetilde{\varrho} \right) \right)\\
&\leq \sqrt{K} c_K^2 \widetilde{C} \left(\|\varrho - \widetilde{\varrho}\|_{L^2(B_R(0))} + W_2\left(\varrho, \widetilde{\varrho} \right) \right).
\end{aligned}
\end{equation}
Here, the second inequality comes from $\sup_{\theta \in \R^d} \left| \omegaa(\theta) \right| \leq \exp(-\alpha \minobj)$ and Lemma \ref{lem:auxLemStability}.
The third inequality uses Lemma \ref{lem:upperBound} in combination with the assumption $\varrho (\Qbeta{\varrho}), \widetilde{\varrho} (\Qbeta{\widetilde{\varrho}}) \geq c > 0$ and the last inequality holds thanks to Cauchy-Schwarz inequality and the assumption $\int_{\R^d} \|\theta\|_2^2 \varrho d\theta \leq K$.

Regarding term $T_2$, by Lemma \ref{lem:upperBound} we have
\begin{equation}
    \sup_{\theta \in B_R(0)} \left\|g(\theta; \widetilde{\varrho})\right\|_2 = \sup_{\theta \in B_R(0)} \left\|\frac{\theta \omegaa(\theta)}{\left\| \omegaa\right\|_{L^1(\Ibeta{\widetilde{\varrho}})}}\right\|_2 \leq R c_K.
\end{equation}
Applying Lemma \ref{lem:auxLemStability} again yields
\begin{equation}
\begin{split}
    T_2 &:= \left\| \int_{\Qbeta{\varrho}} g(\theta; \widetilde{\varrho}) \varrho d\theta - \int_{\Qbeta{\widetilde{\varrho}}} g(\theta; \widetilde{\varrho}) \widetilde{\varrho} d\theta\right\|_2\\
&\leq R c_K \widetilde{C} \left(\|\varrho - \widetilde{\varrho}\|_{L^2(B_R(0))} + W_2\left(\varrho, \widetilde{\varrho} \right) \right).
\end{split}
\end{equation}
We conclude the proof by combining the estimates for both terms $T_1$ and $T_2$. 
\end{proof}

\subsection{Proof of Theorem~\ref{thm:well_posedness}}
\label{sec:well_posedness_proof}
Before we present the main proof, let us first recall some preliminary results from \cite{fornasier2025regularity}, which studies the regularity of weak solutions to the following type of linear PDEs with diffusion coefficients that may not satisfy a uniform ellipticity condition:
\begin{equation}\label{eq:gen_pde}
\begin{aligned}
\partial_t \mu_t &= \nabla \cdot \left(F_t \nabla \mu_t \right) - \nabla \cdot \left( J_t \mu_t \right) + A \mu_t,\\
\mu_0 &= \varrho \in H^{l+2} (\R^d), \quad \forall l \geq 0 \, .
\end{aligned}
\end{equation}
Here, $F: [0,T] \times \R^d \mapsto \R_{\geq 0}$ is a scalar function, $J: [0,T] \times \R^d \mapsto \R^d$ is a vector valued function, and $A$ is a constant.
Next, we define a class of test functions determined by $F$,
\begin{equation}\label{eq:function_class}
\begin{aligned}
\Gamma_F := \bigg\{ \phi: \; &\forall T > 0, \phi \in \CC^{2,1} ([0,T] \times \R^d), \phi \in W^{1, \infty} ([0,T]; H^l(\R^d)),\\
&\int_{\R^d} F_t^2 \|\phi^{(l)}\|^2 d\theta < \infty, \forall t \in [0,T], l \geq 0,\, \text{and} \int_0^T \int_{\R^d} F_t^3 \|\phi^{(l)}\|^2 d\theta dt < \infty, \forall l \geq 0  \bigg\},
\end{aligned}
\end{equation}
which, in turn, we use to introduce a notion of weak solution to \eqref{eq:gen_pde}.
\begin{definition}[Definition 2.4 in \cite{fornasier2025regularity}]\label{def:weak_sol_def_2}
We say $\mu$ is a weak solution to equation \eqref{eq:gen_pde} if for every $t\in [0,\infty)$ $\mu_t$ is a Radon measure and for any test function $\phi \in \Gamma_F$ we have
\begin{equation*}
\begin{aligned}
&\int_0^t \int_{\R^d} \left( \partial_{\tau} \phi_{\tau}(\theta) + \nabla \cdot \left(F_{\tau}(\theta) \nabla \phi_{\tau}(\theta) \right) -   J_{\tau}(\theta) \cdot \nabla \phi_{\tau}(\theta)  + A \phi_{\tau}(\theta) \right) \mu_{\tau} (\theta) d\theta d\tau \\
&\qquad = \int_{\R^d} \mu_t (\theta) \phi_t(\theta) d\theta - \int_{\R^d} \varrho(\theta) \phi_0(\theta) d\theta,
\end{aligned}
\end{equation*}
for any $t \in [0, \infty)$.
\end{definition}
Lastly, we state one of the main theorems in \cite{fornasier2025regularity}, which is an important preliminary result used in our proof of Theorem \ref{thm:well_posedness}.
\begin{proposition}[Theorem 2.5 in \cite{fornasier2025regularity}]\label{prop:well_pose_gen_pde}
If Assumptions 2.6, 2.7, and 2.17 in \cite{fornasier2025regularity} are satisfied, then equation \eqref{eq:gen_pde} has a unique weak solution in the sense of Definition \ref{def:weak_sol_def_2}, and this weak solution belongs to the function class $\Gamma_F$ defined in \eqref{eq:function_class}.
\end{proposition}

We now have all the necessary tools at hand to present a detailed proof of the existence of solutions to the mean-field CB\textsuperscript{2}O dynamics~\eqref{eq:fokker_planck}.
For notational convenience, we abbreviate the consensus point~\eqref{eq:consensus_point} by $\mAlphaBetaRED{\dummy} = \mAlphaBeta{\dummy}$.

\begin{proof}[Proof of Theorem~\ref{thm:well_posedness}]
First note that if $\rho = \mathrm{Law} (\overbar{\theta})$ is a solution to the Fokker-Planck Equation~\eqref{eq:fokker_planck} satisfying $\rho \in \CC \left([0,T], \R^d \right)$, then we may apply standard theory of SDEs (see, e.g., Chapter 6.2 of \cite{arnold1974stochasticdifferentialequations}) to conclude that there exists a unique strong solution to the mean-field dynamics \eqref{eq:dyn_macro}.
Therefore, in what follows we focus on proving that there exists a solution to the PDE \eqref{eq:fokker_planck} with the desired regularity.


\textbf{Step 1:}
Let $\rho^0_t = \rho_0$ for all $t>0$, i.e.,  $\rho^{n=0}_{t}(\theta) = \rho_0(\theta)$ for all $\theta\in\bbR^d$ and for all $t>0$.
For $n\geq 0$,
let $\rho^{n+1}$ be the solution to the following linear Fokker-Planck equation
\begin{equation}\label{eq:linearFK}
\partial_t\rho^{n+1}_t
= \lambda\nabla \cdot \left(\left(\theta -\mAlphaBetaRED{\rho^n_t}\right)\rho^{n+1}_t\right)
+ \frac{\sigma^2}{2}\Delta \left(\N{\theta-\mAlphaBetaRED{\rho^n_t}}_2^2\rho^{n+1}_t\right)
\end{equation}
with initial data $\rho^{n+1}_0 = \rho_0 \in H^{l+2}(\R^d) \cap L^{\infty} (\R^d) \cap \CP_4(\R^d)$.
Before we make the induction assumptions on $\rho^n$, we first rewrite \eqref{eq:linearFK} in the form of \eqref{eq:gen_pde}, i.e.,
\begin{equation}\label{eq:rho_nplus1_gen_form}
\partial_t \rho_t^{n+1} = \nabla \cdot \left( F_t^n \nabla \rho_t^{n+1} \right) - \nabla \cdot \left( J_t^n \rho_t^{n+1} \right) + A \rho_t^{n+1},
\end{equation}
where, for $t \in [0,T]$,
\begin{equation}\label{eq:F_J_A}
F_t^n (\theta) := \frac{\sigma^2}{2} \|\theta - \mAlphaBetaRED{\rho_t^n}\|_2^2, \quad \quad J_t^n(\theta) := (\lambda + \sigma^2) (\theta - \mAlphaBetaRED{\rho_t^n}), \quad  A := d(\lambda + \sigma^2) \,.
\end{equation}
Now we make the following induction assumptions on $\rho^n$: (for simplicity, we use $C,c > 0$ to denote constants that are \textit{independent} of $n$)
\begin{enumerate}[label=(\roman*),labelsep=10pt,leftmargin=35pt]
    \item\label{asm:induction_asm_I}
    $\rho_t^{n} (\Qbeta{\rho_t^n} ) \geq c$ with a constant $c > 0$ that is independent of $n$ or $t \in [0,T]$.
    This ensures that $\{\mAlphaBetaRED{\rho_t^n}\}_{t \in [0,T]}$ is well-defined;
    
    \item\label{asm:induction_asm_II} $\rho^n \in \Gamma_{F^{n-1}} \cap L^{\infty} ([0,T]\times \R^d) $ with $\Gamma_{F^{n-1}}$ defined in \eqref{eq:function_class} based on function $F^{n-1}$ introduced in \eqref{eq:F_J_A}.
    Moreover, $\rho^n$ satisfies
    \begin{equation*}
        \|\rho^n \|_{W^{1, \infty} ([0,T]; H^{l} (\R^d))} \leq C \|\rho_0\|_{H^{l+2}(\R^d)}, \quad \text{and} \quad \|\rho^n\|_{L^{\infty} ([0,T] \times \R^d)} \leq C \|\rho_0\|_{L^{\infty}(\R^d)};
    \end{equation*}

    \item\label{asm:induction_asm_III}
    $\rho^n \in \CC ([0,T]; \CP_2(\R^d))$  with $\sup_{t \in [0,T]} \int_{\R^d} \|\theta\|^4  d\rho_t^n(\theta) \leq C$;
    
    \item\label{asm:induction_asm_IV} $\rho^n$ satisfies $\| \mAlphaBetaRED{\rho_t^n} - \mAlphaBetaRED{\rho_s^n} \|_2 \leq C |t-s|^{1/2}$ for any $0 \leq s, t \leq T$.
\end{enumerate}

We verify that $\rho^{n+1}$ satisfies the induction assumptions \ref{asm:induction_asm_I}-\ref{asm:induction_asm_IV}.

For the induction assumption \ref{asm:induction_asm_I}, we first notice that, by the definition of $\Qbeta{\dummy}$ in \eqref{eq:Q_beta} and Assumption \ref{asm:LipL}, 
\begin{equation}\label{eq:aux_lower_bound_1}
\begin{aligned}
\rho_t^{n+1} \left( \Qbeta{\rho_t^{n+1}} \right) &= \rho_t^{n+1} \left(\left\{\theta\in B_R(0) :  L(\theta)\leq \frac{2}{\beta}\int_{\beta/2}^{\beta} \qa{\rho_t^{n+1}}\,da +\delta_q \right\} \right)\\
&\geq  \rho_t^{n+1} \left(\left\{\theta \in B_R(0): L(\theta) \leq \minL + \delta_q \right\} \right)\\
&\geq \rho_t^{n+1} \left( B_{r_q} (\thetaG) \right), \qquad \qquad \text{with} \; \; r_q = \left( \frac{\delta_q}{H_L} \right)^{1/h_L} .
\end{aligned}
\end{equation}
Here the first inequality holds because $\frac{2}{\beta}\int_{\beta/2}^{\beta} \qa{\varrho}\,da \geq \minL$ for all $\beta \in [0, 1]$, and the second inequality uses the fact that $L(\theta) \leq \minL + \delta_q$ for all $\theta \in B_{r_q} (\thetaG)$.
Moreover, following the argument described in the proof of Proposition \ref{lem:lower_bound_probability}, we can infer that the amount of mass in the ball $B_{r_q}(\thetaG)$ w.r.t. $\rho^{n+1}$ will decrease at most exponentially fast, i.e.,
\begin{equation}\label{eq:aux_lower_bound_2}
    \rho_t^{n+1} \left( B_{r_q} (\thetaG) \right) \geq \left(\int \phi_{r_q} (\theta) d\rho_0^{n+1}(\theta) \right) \exp \left( -p t \right) = \left(\int \phi_{r_q} (\theta) d\rho_0(\theta) \right) \exp \left( -p t \right), \qquad \text{for}\;\; t \in [0,T],
\end{equation}
where $\phi_{r_q}$ and $p$ are as in \eqref{eq:mollifier} and \eqref{eq:lower_bound_probability_rate}, respectively; we note that this $p$ can be chosen to depend only on hyperparameters of our model and not on $t \in [0,T]$ or $n$. Because $\rho_0$ satisfies that $\thetaG \in \supp{\rho_0}$, in other words, $\int \phi_{r_q} (\theta) d\rho_0(\theta) > 0$, we can thus deduce that 
\begin{equation}
\label{eq:LowerBoundDenominator}
\rho_t^{n+1} \big( \Qbeta{\rho_t^{n+1}} \big) \geq c
\end{equation}
for a constant $c>0$ that is \textit{independent} of $n$ or $t \in [0,T]$ by combining the inequalities \eqref{eq:aux_lower_bound_1} and \eqref{eq:aux_lower_bound_2}. 
Furthermore, recall that $\mAlphaBetaRED{\rho_t^{n+1}}$ was introduced as
\begin{equation}
    \mAlphaBetaRED{\rho_t^{n+1}} = \int \theta \frac{\omegaa(\theta)}{\left\| \omegaa \right\|_{L^1(\Ibeta{\rho_t^{n+1}})}} d\Ibeta{\rho_t^{n+1}} (\theta),
    \label{eqn:ConsensuPointAux}
\end{equation}
with $\Ibeta{\rho_t^{n+1}} := \mathbbm{1}_{\Qbeta{\rho_t^{n+1}}} \rho_t^{n+1}$ as defined in \eqref{eq:I_beta}. Therefore, $\rho_t^{n+1} ( \Qbeta{\rho_t^{n+1}}) > 0$ ensures that the above formula is well-defined.
This shows that $\rho^{n+1}$ satisfies the induction assumption \ref{asm:induction_asm_I}.

Next, we check that the induction assumption \ref{asm:induction_asm_II} holds for $\rho_t^{n+1}$.
To apply Proposition \ref{prop:well_pose_gen_pde}, we need to check that $F_t^n$ and $J_t^n$, as defined in \eqref{eq:F_J_A}, satisfy the required assumptions.
Among them, we verify that $F_t^n(\theta)$ and $J_t^n(\theta)$ are locally Hölder continuous in time $t$ with Hölder coefficient $\gamma \in [0,1]$ for any fixed $\theta \in \R^d$.
The remaining required assumptions hold for $F_t^n(\theta)$ and $J_t^n(\theta)$ following similar arguments as in \cite[ Remark 2.18]{fornasier2025regularity}. The Hölder continuity follows from the direct computation:
\begin{equation*}
\begin{aligned}
|F_t^n(\theta) - F_s^n(\theta)| &= \frac{\sigma^2}{2} \left| \| \theta - \mAlphaBetaRED{\rho_t^n} \|_2 - \| \theta - \mAlphaBetaRED{\rho_s^n} \|_2 \right| \\
&\leq \frac{\sigma^2}{2} \left\| 2\theta + \mAlphaBetaRED{\rho_t^n} + \mAlphaBetaRED{\rho_s^n} \right\|_2 \|\mAlphaBetaRED{\rho_t^n} - \mAlphaBetaRED{\rho_s^n}\|_2 \\
&\leq \sqrt{\frac{3}{2}}\sigma^2 (\sqrt{2}\|\theta\|_2 + R) C |t-s|^{1/2} \, ;
\end{aligned}
\end{equation*}
in the last inequality, we applied Young's inequality, the fact that $\mAlphaBetaRED{\rho_t^n}, \mAlphaBetaRED{\rho_s^n} \in B_R(0)$, and the induction assumption \ref{asm:induction_asm_IV} for $\rho^n$.
With similar computations, we can also show that $J_t^n(\theta)$ is locally $\frac{1}{2}$-Hölder continuous for any fixed $\theta \in \R^d$.
Thus, by Proposition \ref{prop:well_pose_gen_pde}, we know that $\rho_t^{n+1}$ is well defined as the unique solution to the PDE \eqref{eq:rho_nplus1_gen_form} and satisfies the regularity conditions in the definition of $\Gamma_{F^n}$.
With similar computations as in the proof of  \cite[Proposition \ref{prop:well_pose_gen_pde}]{fornasier2025regularity}, we can further show that 
\begin{equation}\label{eq:uniform_norm_bound}
\|\rho^{n+1}\|_{W^{1, \infty} ([0,T]; H^l(\R^d))} \leq C(T, l, \lambda, \sigma) \|\rho_0\|_{H^{l+2}(\R^d)} \, .
\end{equation}

Next, we show that $\rho_t^{n+1} \in L^{\infty} ([0,T] \times \R^d)$.
Since we know that $\rho_t^{n+1} \in \Gamma_{F^n}$, it follows that for every $2 \leq p < \infty$ we can use $(p_t^{n+1})^{p-1}$ as a test function for the linear Fokker-Planck \eqref{eq:rho_nplus1_gen_form} and get 
\begin{equation*}
\begin{aligned}
\frac{1}{p} \frac{d}{dt} \|\rho_t^{n+1}\|_{L^p(\R^d)}^p &= \int_{\R^d} \partial_t \rho_t^{n+1} (\rho_t^{n+1})^{p-1} d\theta\\
&= \int_{\R^d} \left( \nabla \cdot \left( F_t^n(\theta) \nabla (\rho_t^{n+1})^{p-1} \right) - \nabla \cdot (J_t^n(\theta) (\rho_t^{n+1})^{p-1}) + A (\rho_t^{n+1})^{p-1} \right) \rho_t^{n+1} d\theta\\
&= -\int_{\R^d} F_t^n(\theta) \left( \nabla (\rho_t^{n+1})^{p-1} \right)^T \nabla \rho_t^{n+1} d\theta\\
&\quad + \int_{\R^d} J_t^n(\theta)^T \nabla \rho_t^{n+1} (\rho_t^{n+1})^{p-1} d\theta + A \|\rho_t^{n+1}\|_{L^p(\R^d)}^p\\
&= - (p-1) \int_{\R^d} F_t^n(\theta) \|\nabla \rho_t^{n+1}\|^2_2 (\rho_t^{n+1})^{p-2} d\theta \\
& \quad + \frac{1}{p}\int_{\R^d} J_t^n(\theta)^T  \nabla (\rho_t^{n+1})^p d\theta + A \|\rho_t^{n+1}\|_{L^p(\R^d)}^p\\
&\leq -\frac{1}{p} \int_{\R^d} \nabla \cdot J_t^n(\theta) (\rho_t^{n+1})^p d\theta + A \|\rho_t^{n+1}\|_{L^p(\R^d)}^p\\
&= \frac{p-1}{p} d (\lambda + \sigma^2) \|\rho_t^{n+1}\|_{L^p(\R^d)}^p \, .
\end{aligned}
\end{equation*}
Here, the third equality comes from integration by parts, the first inequality uses the fact that $F_t^n(\theta) = \frac{\sigma^2}{2} \|\theta - \mAlphaBetaRED{\rho_t^n}\|^2_2 \geq 0$ and integration by parts, and in the last equality we simply plug in the definitions of $J_t^n$ and $A$ as introduced in \eqref{eq:rho_nplus1_gen_form}. Using Grönwall's inequality, we obtain
\begin{equation*}
    \left\| \rho_t^{n+1} \right\|_{L^p(\R^d)} \leq \exp \left( \frac{p-1}{p} d (\lambda + \sigma^2) t \right) \left\| \rho_0 \right\|_{L^p(\R^d)} .
\end{equation*}
Thus, for any $t \in [0,T]$ we have
\begin{equation*}
\begin{aligned}
\left\| \rho_t^{n+1} \right\|_{L^{\infty}(\R^d)} = \lim_{p \rightarrow \infty} \left\| \rho_t^{n+1} \right\|_{L^p(\R^d)} &\leq \lim_{p \rightarrow \infty} \exp \left( \frac{p-1}{p} d (\lambda + \sigma^2) t \right) \left\| \rho_0 \right\|_{L^p(\R^d)}\\
&= \exp \left( d (\lambda + \sigma^2) t \right) \left\| \rho_0 \right\|_{L^{\infty}(\R^d)}\\
&\leq \exp \left( d (\lambda + \sigma^2) T \right) \left\| \rho_0 \right\|_{L^{\infty}(\R^d)} .
\end{aligned}
\end{equation*}
Form the above we conclude $\rho^{n+1} \in L^{\infty}([0,T] \times \R^d)$, showing that $\rho^{n+1}$ satisfies the induction assumption \ref{asm:induction_asm_II}.

To verify the induction assumption \ref{asm:induction_asm_III} for $\rho_t^{n+1}$, we denote the SDE corresponding to the linear Fokker-Planck Equation~\eqref{eq:linearFK} as
\begin{equation}\label{eq:sdePicard}
d\theta_t^{n+1} = - \lambda \left(\theta_t - \m{\rho_t^n} \right)dt + \sigma \N{\theta_t^{n+1} - \m{\rho_t^n}}_2 dB_t .
\end{equation}
Note that by the definition of $\mAlphaBetaRED{\dummy}$ in \eqref{eq:consensus_point}, $\N{\mAlphaBetaRED{\rho_t^n}}_2 \leq R$ for all $t \in [0,T]$.
Therefore, by standard theory of SDEs (see, e.g., Chapter 7 in \cite{arnold1974stochasticdifferentialequations}), a fourth-order moment estimate for solutions to \eqref{eq:sdePicard} takes the form
\begin{equation}
\E \N{\theta_t^{n+1}}_2^4 \leq \left(1 + \E \N{\theta_0^{n+1}}_2^4 \right) e^{at},
\end{equation}
for some constant $a > 0$ (independent of $n$ or $t$). In particular, 
\begin{equation}\label{eq:moment_bound}
    \sup_{t \in [0,T]} \int_{\R^d} \N{\theta_t^{n+1}}_2^4 d\rho_t^{n+1} \leq K
\end{equation}
for a constant $K < \infty$ that is \textit{independent} of $n$. This shows $\rho^{n+1}_t \in \CP_4(\R^d)$ for all $ t \in [0,T]$.
Furthermore, from the SDE \eqref{eq:sdePicard}, the Itô isometry yields
\begin{equation}
\label{eqn:ConvergenceWasserstein}
\begin{aligned}
\E \left\|\theta_t^{n+1} - \theta_s^{n+1} \right\|_2^2 & \leq 2 \lambda^2 |t - s| \int_s^t \left\|\theta_{\tau}^{n+1} - \mAlphaBetaRED{\rho_{\tau}^n} \right\|_2^2 d\tau + 2\sigma^2 \int_s^t \left\| \theta_{\tau}^{n+1} - \mAlphaBetaRED{\rho_{\tau}^n} \right\|_2^2 d\tau\\
&\leq 4 \left(\lambda^2 T + \sigma^2 \right) \left(K + R^2 \right) |t-s|\\
&=: b|t-s|,
\end{aligned}
\end{equation}
and, therefore, $W_2\left(\rho_t^{n+1}, \rho_s^{n+1} \right) \leq b |t-s|^{\frac{1}{2}}$, for some constant $b > 0$ that is \textit{independent} of $n \in \bbN$.
This implies that $\rho_t^{n+1} \in \CC \left([0,T], \CP_2(\R^d) \right)$, and shows that $\rho_t^{n+1}$ fulfills the induction assumption \ref{asm:induction_asm_III}.

Lastly, we verify that $\rho^{n+1}$ also meets the induction assumption \ref{asm:induction_asm_IV}.
For any $0 \leq s < t \leq T$, from the above analysis, we know that $\rho_s^{n+1}, \rho_t^{n+1}$ satisfy the assumption requirements for the stability estimate for the consensus point $\mAlphaBetaRED{\dummy}$ stated in Proposition \ref{thm:stabilityEst}.
In particular,
\begin{equation}\label{eq:Holder_cont}
\begin{aligned}
\left\| \mAlphaBetaRED{\rho_t^{n+1}} - \mAlphaBetaRED{\rho_s^{n+1}}  \right\|_2 &\leq \widetilde{C} \left( W_2 \left( \rho_t^{n+1}, \rho_s^{n+1} \right) + \|\rho_t^{n+1} - \rho_s^{n+1}\|_{L^2(B_R(0))} \right),
\end{aligned}
\end{equation}
where the constant $\widetilde{C} > 0$ is independent of $n$. 

Moreover, we can compute
\begin{equation}
\begin{aligned}
\|\rho_t^{n+1} - \rho_s^{n+1}\|_{L^2(B_R(0))} &= \left(\int_{B_R(0)} |\rho_t^{n+1}(\theta) - \rho_s^{n+1}(\theta) |^2 d\theta  \right)^{1/2} \\
&= \left( \int_{B_R(0)} \left| \int_s^t \partial_{\tau} \rho_{\tau}^{n+1} (\theta) d\tau \right|^2 d\theta \right)^{1/2} \\
&\leq \|\partial_{\tau} \rho_{\tau}^{n+1}\|_{L^2 ([s,t] \times B_R(0))} |t-s|^{1/2}\\
&\leq \|\rho^{n+1}\|_{W^{1,\infty} ([0,T]; H^l(\R^d))} |t-s|^{1/2}\\
&\leq C(T, l, \lambda, \sigma) \|\rho_0\|_{H^{l+2}(\R^d)}  |t-s|^{1/2} \, .
\end{aligned}
\end{equation}
Here, in the first inequality we apply the Cauchy-Schwarz inequality, and the last inequality makes use of inequality \eqref{eq:uniform_norm_bound}.
Integrating the above estimate and $W_2\left(\rho_t^{n+1}, \rho_s^{n+1} \right) \leq b |t-s|^{\frac{1}{2}}$, derived in \eqref{eqn:ConvergenceWasserstein}, into \eqref{eq:Holder_cont}, we obtain that
\begin{equation*}
\left\| \mAlphaBetaRED{\rho_t^{n+1}} - \mAlphaBetaRED{\rho_s^{n+1}}  \right\|_2 \leq C |t - s|^{1/2}
\end{equation*}
for a constant $C > 0$ that is independent of $n$.
With this we conclude that $\rho^{n+1}$ indeed satisfies the induction assumption \ref{asm:induction_asm_IV}.

\textbf{Step 2:} 
From Step 1, we have shown that the sequence $\{\rho^n\}_n \subset W^{1, \infty} ([0,T]; H^l(\R^d))$ has a uniformly bounded (i.e., independent of $n$) norm in this space. 
Consequently, for any fixed $r > 0$, we know that $\{\rho^n\}_n$ is bounded in $L^{\infty} ([0,T]; H^1 (B_r(0)))$ and $\{\partial_t \rho^n\}_n$ is bounded in $ L^{\infty} ([0,T]; H^{-1} (B_r(0)))$ (since $H^1({B_r(0)}) \subset H^{-1} (B_r(0))$).
With the further facts $H^1 (B_r(0)) \subset \subset L^2 (B_r(0))$ (compactly embedded) and $L^2 (B_r(0)) \hookrightarrow H^{-1} (B_r(0))$ (continuously embedded), we can use the Aubin-Lions lemma \cite{aubin1963theoreme, lions1969quelques} to conclude that there exists a subsequence $\rho^{n_k(1)}$ and a function $g_r \in \CC \left([0,T], L^2(B_r(0)) \right)$ such that
\begin{equation}
    \rho^{n_k(1)} \rightarrow g_r \; \text{in} \; \CC \left([0,T], L^2(B_r(0)) \right) \quad \text{as} \;\; k \rightarrow \infty .
\end{equation}
Repeating the above argument, from the subsequence $\{\rho^{n_k(1)}\}$ we can extract a further subsequence $\{\rho^{n_k(2)}\}$ converging to  an element $g_{2r}$ in $  \CC \left([0,T], L^2(B_{2r}(0))\right)$. We note that it must be the case that $g_{2r} (t, \theta) = g_{r}(t, \theta)$ for all $(t,\theta) \in [0,T] \times B_r(0)$. 
Proceeding in an inductive fashion, we can create a subsequence $\rho^{n_k(j)}$ of the sequence $\rho^{n_{k}(j-1)}$ such that $\rho^{n_k(j)} \rightarrow g_{jr}$ in $\CC \left([0,T], L^2(B_{jr}(0))\right)$; this construction ensures that $\rho^{n_k(j)}$ converges to $g_{hr}$ in $\CC \left([0,T], L^2(B_{hr}(0)) \right)$ for all $h \leq j$.
Moreover, $g_{jr} (t, \theta) = g_{hr}(t, \theta)$ for $(t, \theta) \in [0,T] \times B_{hr}(0)$ for every $h \leq j$.
Taking $n_k := n_k(k)$ and defining $\rho \in \CC \left([0,T]; L^2_{loc}(\R^d) \right)$ as $\rho(t,\theta) = g_{j r}(t,\theta)$ for all $(t,\theta) \in [0,T] \times B_{jr}(0)$ and $j \in \mathbb{N}_+$, we then have
\begin{equation}\label{eq:convInL2}
\rho^{n_k} \rightarrow \rho|_{B_{jr}(0)} \; \text{in} \; \CC \left([0,T], L^2(B_{jr}(0))\right) \quad \text{as} \;\; k\rightarrow \infty,
\end{equation}
for all $j \in \mathbb{N}_+$. 
This implies that
\begin{equation}\label{eq:localConvAETime}
\rho^{n_k} \rightarrow \rho \,  \text{ strongly in } \, \CC ([0,T], \Lloc(\R^d)) \, .  
\end{equation}

Next, we show that the limiting function $\rho$ satisfies the following properties (similar to the induction assumptions \ref{asm:induction_asm_I}-\ref{asm:induction_asm_IV}):
\begin{enumerate}[label=(\Roman*),labelsep=10pt,leftmargin=35pt]
    \item\label{pp:lowerBoundmass} 
    $\rho (\Qbeta{\rho_t}) \geq c$ for all $t \in [0,T]$ with constant $c > 0$ defined in \eqref{eq:LowerBoundDenominator}.
    This ensures that $\{\mAlphaBetaRED{\rho_t}\}_{t \in [0,T]}$ is well-defined;

    \item\label{pp:regularity} 
    $\rho$ satisfies:
    \begin{equation*}
        \rho \in \CC ([0,T]; L^2(\R^d)) \cap W^{1,\infty}([0,T]; H^l(\R^d)) \cap L^{\infty} ([0,T]; L^{\infty}(\R^d));
    \end{equation*}

    \item\label{pp:boundFourthMom}
    $\rho \in \CC ([0,T]; \CP_2(\R^d))$ and $\sup_{t \in [0,T] }\int_{\R^d} \|\theta\|^4 d\rho(\theta_t) \leq K$ with a constant $K > 0$ defined in \eqref{eq:moment_bound};

    \item\label{pp:CSP_continuity}
    $\rho$ satisfies $\|\mAlphaBetaRED{\rho_t} - \mAlphaBetaRED{\rho_s}\|_2 \leq C |t - s|^{1/2}$ for any $0 \leq s, t \leq T$.
\end{enumerate}

For property \ref{pp:lowerBoundmass}, with similar calculations as in \eqref{eq:aux_lower_bound_1}, we know that $\rho_t (\Qbeta{\rho_t} ) \geq \rho_t ( B_{r_q} (\thetaG) )$, where $r_q$ is defined in \eqref{eq:aux_lower_bound_1}.
Furthermore, we can compute that
\begin{equation}\label{eq:lower_bound_mass_rho_t}
\begin{aligned}
\rho_t \left( B_{r_q} (\thetaG) \right) &\geq \rho_t^{n_k} \left( B_{r_q} (\thetaG) \right) - \left| \rho_t^{n_k} \left( B_{r_q} (\thetaG) \right) - \rho_t \left( B_{r_q} (\thetaG) \right) \right| \\
&\geq \rho_t^{n_k} \left(B_{r_q}(\thetaG) \right) - \left( \mathrm{Vol} \left(B_{r_q}(\thetaG) \right) \right)^{\frac{1}{2}} \left\| \rho_t^{n_k} - \rho_t \right\|_{L^2 (B_{r_q} (\thetaG))}\\
&\geq c - \left( \mathrm{Vol} \left(B_{r_q}(\thetaG) \right) \right)^{\frac{1}{2}} \left\| \rho_t^{n_k} - \rho_t \right\|_{L^2 (B_{r_q} (\thetaG))} \, .
\end{aligned}
\end{equation}
Here the first inequality uses the inverse triangle inequality, the second inequality is an application of Cauchy-Schwarz, and the last inequality utilizes the fact that $\rho_t^{n_k} (B_{r_q} (\thetaG) ) \geq c$ for all $ k \geq 0$ and $t \in [0,T]$ as shown in \eqref{eq:LowerBoundDenominator}.
By letting $k \rightarrow \infty$ in \eqref{eq:lower_bound_mass_rho_t} and using $L^2_{\mathrm{Loc}} (\R^d)$ convergence property \eqref{eq:localConvAETime} with $\Omega = B_{r_q} (\thetaG)$, we verify that $\rho_t ( B_{r_q} (\thetaG) ) \geq c$ with a constant $c > 0$ for all $t \in [0,T]$.
This implies that $\rho_t (\Qbeta{\rho_t} ) \geq c$, and further indicates that $\mAlphaBetaRED{\rho_t}$ is well-defined for all $t \in [0,T]$.

To verify that $\rho$ satisfies property \ref{pp:regularity}, we make use of the fact that $\rho^{n_k} \rightarrow \rho$ strongly in $\CC ([0,T]; \Lloc(\R^d))$, and $\{\rho^{n_k}\}_{n_k} \subset W^{1, \infty} ([0,T]; H^l(\R^d)) \cap L^{\infty} ([0,T] \times \R^d)$ has uniform bounded norms in these two spaces, respectively.
The proof follows from standard functional analysis arguments, and we defer the details to Lemma \ref{lem:limiting_func_Sob_reg} and Lemma \ref{lem:limiting_func_L_infty_reg} in the Appendix.

For the property \ref{pp:boundFourthMom} of $\rho$, since for all $t \in [0,T]$ we have that $\rho_t^{n_k}$ strongly converges to $\rho_t$ in $\Lloc(\R^d)$ as $n \rightarrow \infty$ and $\sup_n \int \|\theta\|_2^4 d\rho^n_t(\theta) \leq K$ with a constant $K$ defined in \eqref{eq:moment_bound}, we can show that $\rho^{n_k}_t$ converges to $\rho_t$ in $W_2$-distance and moreover $\int \|\theta\|_2^4 d\rho_t(\theta) \leq \liminf_{k \rightarrow \infty} \int \|\theta\|^4_2d \rho_t^{n_k} \leq K$.
We defer the details of these facts to Lemma \ref{lem:WassConv} and Lemma \ref{lem:semiContinuity} in the Appendix. Then, by $\rho^{n_k} \in \CC ([0,T]; \CP_2 (\R^d))$ and triangle inequality, we have, for any $0 \leq s, t \leq T$,
\begin{equation*}
\lim_{t \rightarrow s} W_2 (\rho_t, \rho_s) \leq \lim_{t \rightarrow s} W_2 (\rho_t, \rho^{n_k}_t) + W_2 (\rho^{n_k}_t, \rho^{n_k}_s) + W_2 (\rho^{n_k}_s, \rho_s) \rightarrow 0 \qquad \text{as} \; k \rightarrow +\infty \, .
\end{equation*}
Thus, $\rho \in \CC ([0,T]; \CP_2(\R^d))$, and we finish verifying property \ref{pp:boundFourthMom} of $\rho$.

To justify that $\rho$ has property \ref{pp:CSP_continuity}, note that for any $0 \leq s, t \leq T$, $\rho_s$ we can compute that
\begin{equation}\label{eq:aux_triangle_ineq}
\begin{aligned}
\| \mAlphaBetaRED{\rho_t} - \mAlphaBetaRED{\rho_s} \|_2 &\leq \sqrt{3} \left( \|\mAlphaBetaRED{\rho_t} - \mAlphaBetaRED{\rho_t^{n_k}}\|_2 + \|\mAlphaBetaRED{\rho_t^{n_k}} - \mAlphaBetaRED{\rho_s^{n_k}}\|_2 + \|\mAlphaBetaRED{\rho_t^{n_s}} - \mAlphaBetaRED{\rho_s}\|_2 \right)\\
&\leq C|t-s|^{1/2} + \sqrt{3} \left( \|\mAlphaBetaRED{\rho_t^{n_k}} - \mAlphaBetaRED{\rho_t}\|_2 + \|\mAlphaBetaRED{\rho_t^{n_s}} - \mAlphaBetaRED{\rho_s}\|_2\right)\\
&\leq C|t-s|^{1/2}\\
& \quad + \widetilde{C} \left(\|\rho_t^{n_k} - \rho_t\|_{L^2(B_R(0))} + W_2 (\rho_t, \rho_t^{n_k}) + \|\rho_s^{n_k} - \rho_s\|_{L^2(B_R(0))} + W_2 (\rho_s, \rho_s^{n_k}) \right) \, .
\end{aligned}
\end{equation}
Here the first inequality applies Young's inequality, the second inequality applies the induction property \ref{asm:induction_asm_IV} of $\rho^{n_k}$, and the last inequality makes use of the stability estimate of the consensus point in Proposition \ref{thm:stabilityEst}.
We further note that the constants $C, \widetilde{C} > 0$ are \textit{independent} from $n_k$ as shown in the induction properties \ref{asm:induction_asm_I}-\ref{asm:induction_asm_IV}.
Then by letting $k \rightarrow \infty$ and noticing that $\rho^{n_k}$ strongly converges to $\rho$ in $\CC ([0,T]; \Lloc(\R^d))$, and $\rho_{\tau}^{n_k}$ converges to $\rho_{\tau}$ in $W_2$ for all $\tau \in [0,T]$, we conclude that $\rho$ satisfies the property \ref{pp:CSP_continuity}.

\textbf{Step 3:}
Next we show that the limiting function $\rho$ solves the Equation~\eqref{eq:weak_solution_identity} in the sense of Definition \ref{def:weak_solution}.

First of all, we verify that for any fixed $r > 0$,
\begin{subequations}
\begin{equation}\label{eq:convTerm1}
\int_0^T \int_{B_r(0)} \left\|  \mAlphaBetaRED{\rho_t^{n_k}} \rho_t^{n_{k+1}} -  \mAlphaBetaRED{\rho_t} \rho_t \right\|_2^2 d\theta dt \rightarrow 0 \quad \text{as} \;\; k \rightarrow \infty,
\end{equation}
\begin{equation}\label{eq:convTerm2}
\int_0^T \int_{B_r(0)} \left| \left\| \theta - \mAlphaBetaRED{\rho_t^{n_k}} \right\|_2^2 \rho_t^{n_{k+1}} - \left\| \theta - \mAlphaBetaRED{\rho_t} \right\|_2^2 \rho_t\right|^2 d\theta dt \rightarrow 0  \quad \text{as} \;\; k \rightarrow \infty \; .
\end{equation}
\end{subequations}
For term \eqref{eq:convTerm1}, one can compute
\begin{equation}\label{eq:convTerm1Varify}
\begin{aligned}
& \lim_{k \rightarrow \infty} \int_0^T \int_{B_r(0)} \left\|  \mAlphaBetaRED{\rho_t^{n_k}} \rho_t^{n_{k+1}} - \mAlphaBetaRED{\rho_t} \rho_t \right\|_2^2 d\theta dt \\
&\qquad\leq 2 \lim_{k \rightarrow \infty}\int_0^T \int_{B_r(0)} \left\| \mAlphaBetaRED{\rho_t^{n_k}} - \mAlphaBetaRED{\rho_t} \right\|_2^2  \rho_t^2 d\theta dt + 2 \lim_{k \rightarrow \infty}\int_0^T \int_{B_r(0)} \left\| \mAlphaBetaRED{\rho_t^{n_k}}\right\|_2^2  \left| \rho_t^{n_{k+1}} - \rho_t \right|^2 d\theta dt\\
&\qquad\leq 2 \int_0^T \lim_{k \rightarrow \infty}  \left\| \mAlphaBetaRED{\rho_t^{n_k}} - \mAlphaBetaRED{\rho_t} \right\|_2^2 \left(\int_{B_r(0)} \rho_t^2 d\theta \right) dt + 2 R^2 \lim_{k \rightarrow \infty}  \left\| \rho^{n_{k+1}} - \rho \right\|^2_{L^2([0,T] \times B_r(0))}\\
&\qquad\lesssim \int_0^T \lim_{k \rightarrow \infty}  \left( \left\|\rho_t^{n_k} - \rho_t \right\|_{L^2(B_R(0))} + W_2\left( \rho_t^{n_k}, \rho_t \right)\right) \left\| \rho_t \right\|^2_{L^2(B_r(0))} dt\\
&\qquad\quad + 2 R^2 \lim_{k \rightarrow \infty}  \left\| \rho^{n_{k+1}} - \rho \right\|^2_{L^2([0,T] \times B_r(0))}\\
&\qquad= 0 .
\end{aligned}
\end{equation}
The second inequality comes from the fact that $\left\| \mAlphaBetaRED{\rho_t^{n_k}} - \mAlphaBetaRED{\rho_t}\right\|_2^2 \leq 4 R^2$ and the dominated convergence theorem. The third inequality uses the stability estimate of consensus point in Proposition \ref{thm:stabilityEst} with the fact that $\varrho_t^{n_k}, \rho_t \in L_{\text{loc}}^{\infty} (\R^d)$ and $\rho_t^{n_k} (\Qbeta{\rho_t^{n_k}} ), \rho_t (\Qbeta{\rho_t} ) \geq c > 0$ for all $t \in [0,T]$ and all $k$.
The last equality uses \eqref{eq:localConvAETime} and Lemma \ref{lem:WassConv}, and \eqref{eq:convInL2}, i.e.
\begin{equation}
\lim_{k \rightarrow \infty}  \left( \left\|\rho_t^{n_k} - \rho_t \right\|_{L^2(B_R(0))} + W_2\left( \rho_t^{n_k}, \rho_t \right)\right) = 0 \quad \text{and} \quad \lim_{k \rightarrow \infty} \left\| \rho^{n_{k+1}} - \rho \right\|^2_{L^2([0,T] \times B_r(0))} = 0.
\end{equation}

For term \eqref{eq:convTerm2}, we have
\begin{equation}
\begin{aligned}
&\int_0^T \int_{B_r(0)} \left| \left\| \theta - \mAlphaBetaRED{\rho_t^{n_k}} \right\|_2^2 \rho_t^{n_{k+1}} - \left\| \theta - \mAlphaBetaRED{\rho_t} \right\|_2^2 \rho_t\right|^2 d\theta dt\\
&= \int_0^T \int_{B_r(0)} \left| \|\theta\|_2^2 \left(\rho_t^{n_{k+1}} \!-\! \rho_t \right) \!-\! 2 \theta^T \left( \mAlphaBetaRED{\rho_t^{n_k}} \rho_t^{n_{k+1}} \!-\! \mAlphaBetaRED{\rho_t} \rho_t \right) \!+\! \left( \left\|\mAlphaBetaRED{\rho_t^{n_k}} \right\|_2^2 \rho_t^{n_{k+1}} \!-\! \left\|\mAlphaBetaRED{\rho_t} \right\|_2^2 \rho_t \right) \right|^2 d\theta dt\\
&\lesssim r^4 \left\| \rho^{n_{k+1}} - \rho \right\|^2_{L^2([0,T] \times B_r(0))} + r^2 \int_0^T \int_{B_r(0)} \left\|  \mAlphaBetaRED{\rho_t^{n_k}} \rho_t^{n_{k+1}} - \mAlphaBetaRED{\rho_t} \rho_t \right\|_2^2 d\theta dt\\
&\quad\, +  \int_0^T \int_{B_r(0)} \left|  \left\|\mAlphaBetaRED{\rho_t^{n_k}} \right\|_2^2 \rho_t^{n_{k+1}} - \left\|\mAlphaBetaRED{\rho_t}\right\|_2^2 \rho_t \right|^2 d\theta dt.
\end{aligned}
\end{equation}
In the above inequality, the first two terms converge to $0$ as $k \rightarrow \infty$ by \eqref{eq:convInL2} and \eqref{eq:convTerm1}.
For the last term, we can compute that
\begin{equation}
\begin{aligned}
&\int_0^T \int_{B_r(0)} \left|  \left\|\mAlphaBetaRED{\rho_t^{n_k}} \right\|_2^2 \rho_t^{n_{k+1}} - \left\|\mAlphaBetaRED{\rho_t}\right\|_2^2 \rho_t \right|^2 d\theta dt\\
&\quad\,\leq 2 \int_0^T \int_{B_r(0)} \left| \left\|\mAlphaBetaRED{\rho_t^{n_k}} \right\|_2^2 - \left\|\mAlphaBetaRED{\rho_t} \right\|_2^2 \right\|_2^2 \rho_t^2 d\theta dt + 2 \int_0^T \int_{B_r(0)} \left\| \mAlphaBetaRED{\rho_t^{n_k}} \right\|_2^4 \left\| \rho_t^{n_{k+1}} - \rho_t \right\|_2^2 d\theta dt\\
&\quad\,\leq 2  \int_0^T \int_{B_r(0)} \left| \left( \mAlphaBetaRED{\rho_t^{n_k}} + \mAlphaBetaRED{\rho_t} \right)^T \left( \mAlphaBetaRED{\rho_t^{n_k}} - \mAlphaBetaRED{\rho_t} \right) \right|^2 \rho_t^2 d\theta dt + 2R^4 \left\| \rho^{n_{k+1}} - \rho \right\|^2_{L^2([0,T] \times B_r(0))}\\
&\quad\,\leq 8R^2 \int_0^T \int_{B_r(0)} \left\| \mAlphaBetaRED{\rho_t^{n_k}} - \mAlphaBetaRED{\rho_t} \right\|_2^2  \rho_t^2 d\theta dt + 2R^4 \left\| \rho^{n_{k+1}} - \rho \right\|^2_{L^2([0,T] \times B_r(0))}\\
&\quad\,\rightarrow 0\text{ as }k \rightarrow \infty. \quad (\text{by \eqref{eq:convInL2} and \eqref{eq:convTerm1Varify}})
\end{aligned}
\end{equation}
Therefore, we have verified that \eqref{eq:convTerm2} holds.
Furthermore, we can easily show that for any test function $\phi \in \CC_c^2 (\R^d)$,
\begin{equation}
\lim_{t \rightarrow 0} \lim_{k \rightarrow \infty} \int_{B_r(0)} \rho^{n_{k+1}}(\theta,t) \phi(\theta) d\theta = \int_{B_r(0)}\rho(\theta, 0) \phi(\theta) d\theta,
\end{equation}
\begin{equation}
\lim_{t \rightarrow T} \lim_{k \rightarrow \infty} \int_{B_r(0)} \rho^{n_{k+1}}(\theta,t) \phi(\theta) d\theta = \int_{B_r(0)}\rho(\theta, T) \phi(\theta) d\theta
\end{equation}
and
\begin{equation}
\int_0^T \int_{B_r(0)} \left\| \theta \rho_t^{n_{k+1}} - \theta\rho_t  \right\|_2^2 d\theta dt \rightarrow 0 \;\; \text{as} \;\; k \rightarrow \infty,
\end{equation}
due to $\rho^{n_k} \in \CC\left([0,T]; L^2(\R^d)\right)$ and \eqref{eq:convInL2}. 
Hence we have shown that the limiting function $\rho$ is a solution to the Equation~\eqref{eq:fokker_planck} in the weak sense (Definition \ref{def:weak_solution}). 
\end{proof}



    

    

\section{Proof Details for Section~\ref{sec:convergence}}
\label{sec:convergence_proofs}

In this section, we provide the proof details for the global convergence result of CB\textsuperscript{2}O in mean-field law, Theorem~\ref{thm:main}.
Therefore, we present in Section~\ref{sec:sketch} a proof sketch outlining the main ideas and giving an overview of the crucial steps.
Sections~\ref{sec:time_evolution_V}--\ref{sec:lower_bound_prob_mass} then provide with Lemma~\ref{lem:evolution_V}, Proposition~\ref{lem:laplace_LG}, and Proposition~\ref{lem:lower_bound_probability} the individual statements that are required and central in the overall proof of Theorem~\ref{thm:main}, which we then conclude in Section~\ref{sec:proof_main}.

For the sake of convenience, we introduce the notation $\CV(\rho_t) = \frac{1}{2}W_2^2(\rho_t,\delta_{\thetaG})$, i.e., 
\begin{equation}
	\label{eq:V}
    \CV(\rho_t) = \frac{1}{2}\int\N{\theta-\thetaG}_2^2d \rho_t(\theta).
\end{equation}

The proof framework follows the one introduced in \cite{fornasier2021consensus,fornasier2021convergence,riedl2022leveraging,riedl2024perspective}.

\subsection{Proof Sketch}
\label{sec:sketch}

In order to show the convergence of the mean-field CB\textsuperscript{2}O dynamics~\eqref{eq:fokker_planck} to the desired solution $\thetaG$ of the bi-level optimization problem~\eqref{eq:bilevel_opt},
our eventual goal is to establish that a solution $\rho$ of \eqref{eq:fokker_planck} approximates arbitrarily well a Dirac measure located at $\thetaG$ w.r.t.\@ the Wasserstein-$2$ distance~$W_2$.
More rigorously, we wish to show that for any $\varepsilon>0$, when choosing suitably the hyperparameters~$\alpha$, $\beta$, and $\delta_q$ of the CB\textsuperscript{2}O dynamics~\eqref{eq:fokker_planck},
it holds
\begin{equation}
    W_2^2(\rho_T,\delta_{\thetaG}) = 2\CV(\rho_T)\leq\varepsilon
    \quad\text{for}\quad
    T = T(\varepsilon)
\end{equation}
sufficiently large.
Let us recall here that $R$ has to be sufficiently large from the beginning fulfilling $R\geq\Nbig{\thetaG}_2+r_R$ for some $r_R>0$.

The main idea to achieve the former is to show that the functional $\CV(\rho_t)$ as defined in \eqref{eq:V} satisfies for any $\vartheta\in(0,1)$ differential inequalities of the form
\begin{align} \label{eq:proof_sketch_ODE_ineq}
	-(1+\vartheta/2)\left(2\lambda-d\sigma^2\right)\CV(\rho_t)
    \leq\frac{d}{dt}\CV(\rho_t)
	\leq -(1-\vartheta)\left(2\lambda-d\sigma^2\right)\CV(\rho_t)
\end{align}
until the time $T = T^*$ or until $\CV(\rho_T) \leq \varepsilon$.
In the case $T=T^*$, it is easy to check that the definition of $T^*$ in \eqref{eq:end_time_star_statement} in Theorem~\ref{thm:main} implies $\CV(\rho_{T^*}) \leq \varepsilon$, thus giving $\CV(\rho_T)\leq \varepsilon$ in either case.

As a first step towards \eqref{eq:proof_sketch_ODE_ineq} we derive a differential inequality for the evolution of $\CV(\rho_t)$ by using the dynamics~\eqref{eq:fokker_planck} of $\rho$.
More concisely, by using the weak form~\eqref{eq:weak_solution_identity} of \eqref{eq:fokker_planck} with test function $\phi(\theta) = 1/2\Nbig{\theta-\thetaG}_2^2$,
we derive in Lemma~\ref{lem:evolution_V} the upper bound
\begin{equation} \label{eq:proof_sketch_ODE_ineq_2}
\begin{aligned}
	\frac{d}{dt}\CV(\rho_t)
		\leq
		&-\left(2\lambda - d\sigma^2\right) \CV(\rho_t)
	    + \sqrt{2}\left(\lambda + d\sigma^2\right) \sqrt{\CV(\rho_t)} \N{\mAlphaBeta{\rho_t}-\thetaG}_2\phantom{.} \\
	    &+ \frac{d\sigma^2}{2} \N{\mAlphaBeta{\rho_t}-\thetaG}_2^2
\end{aligned}
\end{equation}
and analogously a lower bound on $\frac{d}{dt}\CV(\rho_t)$. 

To find suitable bounds on the second and third term in \eqref{eq:proof_sketch_ODE_ineq_2}, we need to control the quantity $\Nbig{\mAlphaBeta{\rho_t}-\thetaG}_2$ through appropriate choices of the hyperparameters $\alpha$, $\beta$, and $\delta_q$.
Before describing how to do this, let us recall that as of \eqref{eq:consensus_point} we have
\begin{equation}
    \mAlphaBeta{\rho_t} = \mAlpha{\Ibeta{\rho_t}}.
\end{equation}
Control on the discrepancy $\Nbig{\mAlphaBeta{\rho_t}-\thetaG}_2$ is now obtained by the quantitative quantiled Laplace principle (Q\textsuperscript{2}LP) developed in Proposition~\ref{lem:laplace_LG}.
It combines a quantile choice with the quantitative Laplace principle~\cite[Proposition~4.5]{fornasier2021consensus}.
The former, i.e., the quantitative quantile choice, is associated with a suitably small choice of $\beta>0$, to be precise, such that
\begin{equation}
    q_{\beta}[\varrho]+\delta_q\leq\underbar{L}+\min\{L_\infty,(\eta_Lr_G)^{1/\nu_L}\}
\end{equation}
holds for some small $r_G\in(0,R_G]$.
This warrants that, under the inverse continuity property~\ref{asm:icpL} on $L$ w.r.t.\@ the set $\Theta$ of all global minimizers of $L$, the support of the measure~$\Ibeta{\rho_t}$ approximates well (by being contained in) a small local neighborhood around the set~$\Theta$, denoted by $\CN_{r_G}(\Theta)$.
With the inverse continuity property~\ref{asm:icpG} on $G$ being valid on $\CN_{r_G}(\Theta)$ and thus on the support of $\Ibeta{\rho_t}$, we can imagine a decomposition (in the proof below, this decomposition will be more implicit, see Remark~\ref{rem:error_decomp_laplace})  of the error as in \eqref{eq:conspoint_errordecomp}, i.e., 
\begin{align} \label{eq:decomposition_sketch}
\begin{aligned}
    \N{\mAlphaBeta{\rho_t}-\thetaG}_2
    &= \NBig{\mAlpha{\Ibeta{\rho_t}}-\thetaG}_2\\
    &\leq \N{\mAlpha{\Ibeta{\rho_t}}-\thetaGtilde}_2 + \N{\thetaGtilde-\thetaG}_2\\
    &\leq \N{\mAlpha{\Ibeta{ \rho_t}}-\thetaGtilde}_2 + r_G,
\end{aligned}
\end{align}
where $\thetaGtilde\in B_{r_G}(\thetaG)$ exists as of \ref{asm:icpG}, thus giving the bound on the second term in \eqref{eq:decomposition_sketch}.
For an illustration, we ask the reader to consult once more Figure~\ref{fig:A5_Illustration} above and imagine the set $\CN_{r_G}(\Theta)$ to coincide with the support of $\Ibeta{\rho_t}$.
Since by assumption \ref{asm:icpG}, $\thetaGtilde$ globally minimizes the objective function $G$ restricted to $\CN_{r_G}(\Theta)$ and thus the support of $\Ibeta{\rho_t}$,
the quantitative Laplace principle can be employed to get control over the first term in \eqref{eq:decomposition_sketch}.
This is since the Laplace principle, for a sufficiently large choice of $\alpha>0$, ensures that $\mAlphaBeta{\rho_t} = \mAlpha{\Ibeta{\rho_t}}$ does approximate $\thetaGtilde$.
More precisely, under the inverse continuity property~\ref{asm:icpG} on $G$, we can show
\begin{align}
    \N{\mAlphaBeta{\rho_t} - \thetaGtilde}_2
    &\lesssim
    \ell(r_G,r)
    + \frac{\exp\left(-\alpha r\right)}{\rho_t\big(B_r(\thetaG)\big)}
\end{align}
for sufficiently small $r>0$.
Here, $\ell$ is a strictly positive but monotonically decreasing functions with $\ell(r_G,r)\rightarrow 0$ as $r_G,r\rightarrow 0$.
As long as we can guarantee that $\rho_t(B_{r}(\thetaG)) > 0$,
we achieve $\Nbig{\mAlphaBeta{\rho_t}-\thetaG}_2\lesssim \ell(r_G,r)$, which can be made arbitrarily small by suitable choices of~$r\ll 1$ as well as $r_G \ll 1$,
which in turn require sufficiently small choices of $\beta$ and $\delta_q$, as well as a sufficiently large choice of $\alpha$.

It thus remains to ensure that $\rho_t(B_{r}(\thetaG)) > 0$ for all $r > 0$.
To do so, we employ \cite[Proposition~4.6]{fornasier2021consensus}, see also Proposition~\ref{lem:lower_bound_probability} below for a more detailed explanation,
showing that the initial mass $\rho_0(B_{r}(\thetaG)) > 0$ can, thanks to an active Brownian motion term, i.e., $\sigma > 0$, decay at most at an exponential rate for any $r > 0$, therefore remaining strictly positive during finite time $[0,T]$.
\color{black}

\subsection{Time-Evolution of \texorpdfstring{$\CV(\rho_t)$}{Squared Wasserstein Distance}}
\label{sec:time_evolution_V}

In this section,
we derive evolution inequalities from above and below for the functional $\CV(\rho_t)$ defined in \eqref{eq:V}.

\begin{lemma}[Time-evolution of $\CV(\rho_t)$, cf.\@{\cite[Lemmas~4.1 and 4.3]{fornasier2021consensus}}]
    \label{lem:evolution_V}
    Let $L:\bbR^d\rightarrow\bbR$ and $G:\bbR^d\rightarrow\bbR$.
    Fix $\alpha,\beta,R,\delta_q,\lambda,\sigma > 0$.
	Moreover, let $T>0$ and let $\rho \in \CC([0,T], \CP_4(\bbR^d))$ be a weak solution to the Fokker-Planck Equation~\eqref{eq:fokker_planck} in the sense of Definition~\ref{def:weak_solution}.
	Then the functional $\CV(\rho_t)$ satisfies
	\begin{equation}
	    \label{eq:evolution_of_objective}
	\begin{split}
	    \frac{d}{dt}\CV(\rho_t)
		\leq
		&-\left(2\lambda - d\sigma^2\right) \CV(\rho_t)
	    + \sqrt{2}\left(\lambda + d\sigma^2\right) \sqrt{\CV(\rho_t)} \N{\mAlphaBeta{\rho_t}-\thetaG}_2\phantom{.} \\
	    &+ \frac{d\sigma^2}{2} \N{\mAlphaBeta{\rho_t}-\thetaG}_2^2
	\end{split}
	\end{equation}
    as well as
    \begin{equation}
        \label{eq:evolution_of_objective_lower}
        \frac{d}{dt}\CV(\rho_t)
        \geq
        -\left(2\lambda - d\sigma^2\right) \CV(\rho_t)
        - \sqrt{2}\left(\lambda + d\sigma^2\right) \sqrt{\CV(\rho_t)} \N{\mAlphaBeta{\rho_t}-\thetaG}_2.
    \end{equation}
\end{lemma}

\begin{proof}
    The proof follows the lines of \cite[Lemmas~4.1 and 4.2]{fornasier2021consensus}.

	To start with, we recall that the function $\theta\mapsto \phi(\theta) := 1/2\Nnormal{\theta-\thetaG}_2^2$ is in the function space $\CC_*^2(\bbR^d)$ and that $\rho$ satisfies the weak solution identity~\eqref{eq:weak_solution_identity} in Definition~\ref{def:weak_solution} for all test functions in $\CC_*^2(\bbR^d)$, see Remark~\ref{rem:test_functions_redefine}.
	Hence,
	by applying \eqref{eq:weak_solution_identity} with $\phi$ as defined above,
    we obtain for the time evolution of $\CV(\rho_t)$ that
	\begin{align}
		\frac{d}{dt} \CV(\rho_t)
		&= \underbrace{-\lambda \int \langle  \theta-\thetaG, \theta-\mAlphaBeta{\rho_t}\rangle \,d\rho_t(\theta)}_{=:T_1} + \underbrace{\frac{d\sigma^2}{2} \int \N{\theta-\mAlphaBeta{\rho_t}}_2^2 d\rho_t(\theta)}_{=:T_2},
	\end{align}
	where we used $\partial_k \phi(\theta) = (\theta-\thetaG)_k$ and $\partial^2_{kk} \phi(\theta) = 1$ for all $k\in\{1,\dots,d\}$.
	Expanding the right-hand side of the scalar product in the integrand of $T_1$ by subtracting and adding $\thetaG$ yields
	\begin{equation}
	\begin{split}
	    T_1
		&=-\lambda \int \langle \theta-\thetaG, \theta-\thetaG\rangle \,d\rho_t(\theta) + \lambda  \left\langle \int (\theta-\thetaG) \,d\rho_t(\theta), \mAlphaBeta{\rho_t} - \thetaG\right\rangle\\
		&\leq -2\lambda \CV(\rho_t) + \lambda \N{\bbE(\rho_t) - \thetaG}_2 \N{\mAlphaBeta{\rho_t}-\thetaG}_2,
	\end{split}
	\end{equation}
	with the last step involving Cauchy-Schwarz inequality.
	Analogously,
    again by subtracting and adding $\thetaG$ in the integrand of the term $T_2$,
    we have with Cauchy-Schwarz inequality the upper bound
	\begin{equation}
	\begin{split}
		T_2
		&= \frac{d\sigma^2}{2} \left(\int \N{\theta-\thetaG}_2^2 d\rho_t(\theta) - 2 \left\langle \int (\theta-\thetaG) \,d\rho_t(\theta), \mAlphaBeta{\rho_t} - \thetaG\right\rangle+  \N{\mAlphaBeta{\rho_t}-\thetaG}_2^2\right)\\
		& \leq d\sigma^2 \left(\CV(\rho_t) + \int \N{\theta-\thetaG}_2 d\rho_t(\theta)\N{\mAlphaBeta{\rho_t}-\thetaG}_2 + \frac{1}{2}\N{\mAlphaBeta{\rho_t}-\thetaG}_2^2\right).
	\end{split}
	\end{equation}
	The first part of the statement, namely the upper bound in~\eqref{eq:evolution_of_objective}, now follows after noting that by Jensen's inequality it holds $\Nnormal{\bbE(\rho_t) - \thetaG}_2 \leq \int \Nnormal{\theta-\thetaG}_2 \,d\rho_t(\theta) \leq \sqrt{2\CV(\rho_t)}$.

    What concerns the second part, i.e., the lower bound in \eqref{eq:evolution_of_objective_lower}, we have by Cauchy-Schwarz inequality also
    \begin{equation}
            \pm \left\langle \int (\theta-\thetaG) \,d\rho_t(\theta), \mAlphaBeta{\rho_t} - \thetaG\right\rangle
            \geq -\N{\bbE(\rho_t) - \thetaG}_2 \N{\mAlphaBeta{\rho_t}-\thetaG}_2.
    \end{equation}
    Thus, 
    \begin{equation}
    \begin{split}
        T_1
		&=-\lambda \int \langle \theta-\thetaG, \theta-\thetaG\rangle \,d\rho_t(\theta) + \lambda  \left\langle \int (\theta-\thetaG) \,d\rho_t(\theta), \mAlphaBeta{\rho_t} - \thetaG\right\rangle\\
		&\geq -2\lambda \CV(\rho_t) - \lambda \N{\bbE(\rho_t) - \thetaG}_2 \N{\mAlphaBeta{\rho_t}-\thetaG}_2
    \end{split}
    \end{equation}
    as well as
    \begin{equation}
	\begin{split}
		T_2
		&= \frac{d\sigma^2}{2} \left(\int \N{\theta-\thetaG}_2^2 d\rho_t(\theta) - 2 \left\langle \int (\theta-\thetaG) \,d\rho_t(\theta), \mAlphaBeta{\rho_t} - \thetaG\right\rangle+  \N{\mAlphaBeta{\rho_t}-\thetaG}_2^2\right)\\
		& \geq d\sigma^2 \left(\CV(\rho_t) -\N{\bbE(\rho_t) - \thetaG}_2 \N{\mAlphaBeta{\rho_t}-\thetaG}_2\right),
	\end{split}
	\end{equation}
    where we used additionally that $\Nnormal{\mAlphaBeta{\rho_t}-\thetaG}_2^2\geq0$.
    The lower bound \eqref{eq:evolution_of_objective_lower} now follows after recalling that $\Nnormal{\bbE(\rho_t) - \thetaG}_2 \leq \sqrt{2\CV(\rho_t)}$ by Jensen's inequality.
\end{proof}

\subsection{Quantitative Quantiled Laplace Principle~(Q\texorpdfstring{\textsuperscript{2}}{2}LP)}
\label{sec:laplace_p}

In order to leverage the differential inequalities from Lemma~\ref{lem:evolution_V} to establish the exponential convergence of the functional~$\CV(\rho_t)$,
we need to control the term $\Nbig{\mAlphaBeta{\varrho} - \thetaG}_2$ through the choices of the hyperparameters $\delta_q$, $\beta$ and $\alpha$, which is the content of the following statement, which we refer to as the quantitative quantiled Laplace principle~(Q\textsuperscript{2}LP),
as it combines a quantitative quantile choice with the quantitative Laplace principle~\cite[Proposition~4.5]{fornasier2021consensus}.

\begin{proposition}[Quantitative quantiled Laplace principle]
	\label{lem:laplace_LG}
    Let $\varrho \in \CP(\bbR^d)$,
    fix $\alpha>0$,
    let $r_G\in(0,\min\{R_G,R^H_{G},(G_{\infty}/(2H_G))^{1/h_G}\}]$ and $\delta_q\in(0,\min\{L_\infty,(\eta_Lr_G)^{1/\nu_L}\}/2]$.
    For any $r>0$ define $G_r:=\sup_{\theta\in B_{r}(\thetaG)}G(\theta)-G(\thetaG)$.
    Then, under
    the inverse continuity property~\ref{asm:icpL} on $L$,
    the H\"older continuity assumption~\ref{asm:LipL} on $L$,
    the inverse continuity property~\ref{asm:icpG} on $G$ and the H\"older continuity assumption~\ref{asm:LipG} on $G$,
    and provided that there exists
    \begin{equation}
        \label{eq:beta_QQLP}
    	\beta\in(0,1)
    	\quad\text{satisfying}\quad
    	q^L_{\beta}[\varrho]+\delta_q\leq\underbar{L}+\min\{L_\infty,(\eta_Lr_G)^{1/\nu_L}\},
    \end{equation}
    for any $r\in(0,\min\{r_R,r_G,R^H_{L},(\delta_q/H_L)^{1/h_L}\}]$ and $u > 0$ such that $u + G_{r} + H_Gr_G^{h_G} \leq G_{\infty}$,
    we have
    \begin{equation}
        \label{eq:lem:laplace_LG}
    \begin{split}
        \N{\mAlphaBeta{\varrho} - \thetaG}_2
        &\leq
        \frac{(u+G_r + H_Gr_G^{h_G})^{\nu_G}}{\eta_G} +\frac{\exp\left(-\alpha u\right)}{\varrho\big(B_r(\thetaG)\big)}\!\left(\int\N{\theta-\thetaG}_2 d\Ibeta{\varrho}(\theta)\right).
    \end{split}
    \end{equation}
    We furthermore have $B_r(\thetaG)\subset\Qbeta{\varrho}\subset\CN_{r_G}(\Theta)$.
\end{proposition}

%
%

Before providing the proof of Proposition~\ref{lem:laplace_LG},
let us discuss the bound \eqref{eq:lem:laplace_LG} in view of the error decomposition $\Nbig{\mAlphaBeta{\varrho}-\thetaG}_2
    \leq \Nbig{\mAlphaBeta{\varrho}-\thetaGtilde}_2 + \Nbig{\thetaGtilde-\thetaG}_2$.

\begin{remark}
    \label{rem:error_decomp_laplace}
    The bound
    \begin{equation*}
    \begin{split}
        \N{\mAlphaBeta{\varrho} - \thetaG}_2
        &\leq
        \frac{(u+G_r + H_Gr_G^{h_G})^{\nu_G}}{\eta_G} +\frac{\exp\left(-\alpha u\right)}{\varrho\big(B_r(\thetaG)\big)}\!\left(\int\N{\theta-\thetaG}_2 d\Ibeta{\varrho}(\theta)\right)
    \end{split}
    \end{equation*}
    from~\eqref{eq:lem:laplace_LG} in Proposition~\ref{lem:laplace_LG} comprises two levels of approximation that correspond to the two error terms in the decomposition $\Nbig{\mAlphaBeta{\varrho}-\thetaG}_2
    \leq \Nbig{\mAlphaBeta{\varrho}-\thetaGtilde}_2 + \Nbig{\thetaGtilde-\thetaG}_2$ from~\eqref{eq:lem:laplace_LG}, and are thus related to the two hyperparameters $\beta$ and $\alpha$, for which we refer to the discussion in Section~\ref{sec:convergence} after Theorem~\ref{thm:main}.
    
    Let us discuss now which part of the bound~\eqref{eq:lem:laplace_LG} corresponds to which term and can be controlled through which hyperparameter.
    For this purpose, imagine taking, for $\beta$ fixed,
    $\alpha\rightarrow\infty$.
    Recall that by the Laplace principle $\mAlphaBeta{\varrho}=\mAlpha{\Ibeta{\varrho}}$ is expected to converge to $\thetaGtilde=\argmin_{\theta\in\Qbeta{\varrho}}G(\theta)$ as $\alpha\rightarrow\infty$.
    Indeed, taking $\alpha\rightarrow\infty$ in the bound \eqref{eq:lem:laplace_LG}, allows for $u\rightarrow0$ and $r\rightarrow0$, thus yielding the error bound $\Nbig{\thetaGtilde-\thetaG}_2\leq(H_Gr_G^{h_G})^{\nu_G}/\eta_G$.
    This term can be controlled through the choice of the hyperparameter~$\beta$ since, according to \eqref{eq:beta_QQLP}, we need to find $\beta\in(0,1)$ sufficiently small such that $q^L_{\beta}[\varrho]+\delta_q\leq\underbar{L}+\min\{L_\infty,(\eta_Lr_G)^{1/\nu_L}\}$ holds.
    Feasible choices of $\beta$ are described during the proof of Theorem~\ref{thm:main} in Section~\ref{sec:convergence_proofs}.
    
    The remaining error contributions in \eqref{eq:lem:laplace_LG} can be, as already alluded to, controlled through the choice of the hyperparameter~$\alpha$.
    Imagining setting $\beta=0$, \eqref{eq:lem:laplace_LG} reduces to the classical quantitative Laplace principle~\cite[Proposition~4.5]{fornasier2021consensus}.
\end{remark}

\begin{proof}[Proof of Proposition~\ref{lem:laplace_LG}]
    \textbf{Preliminaries.}
    To begin with, notice that it holds
    \begin{equation} \label{eq:qbeta_monotonicity}
        \qbetahalf{\varrho}
        \leq \frac{2}{\beta}\int_{\beta/2}^{\beta} \qa{\varrho}\,da
        \leq \qbeta{\varrho},
    \end{equation}
    since the $a$-quantile~$\qa{\varrho}$ as defined in \eqref{eq:q_beta} is monotonously increasing in $a$.

    \textbf{Set inclusion $B_r(\thetaG)\subset\Qbeta{\varrho}$.}
    We now start by establishing that with the choice of $r$ it holds $B_r(\thetaG)\subset\Qbeta{\varrho}$, which will be useful throughout the proof.
    First, since $r\leq r_R$, we have $B_r(\thetaG)\subset B_R(0)$, since for all $\theta\in B_r(\thetaG)$ it holds $\N{\theta}_2 \leq \Nbig{\theta-\thetaG}_2 + \Nbig{\thetaG}_2 \leq r + \Nbig{\thetaG}_2 \leq r_R + \Nbig{\thetaG}_2 \leq R$ by assumption on $R$ from Section~\ref{sec:CB2O}.
    Moreover, since $r\leq R^H_{L}$, employing the H\"older continuity~\ref{asm:LipL} of $L$ yields for all $\theta\in B_r(\thetaG)$ that
	\begin{equation}
		L(\theta)-\underbar{L}
        \leq H_L\N{\theta-\thetaG}_2^{h_L}
        \leq H_Lr^{h_L}
        \leq \delta_q
        \leq \frac{2}{\beta}\int_{\beta/2}^{\beta} \qa{\varrho}\,da-\underbar{L}+\delta_q,
	\end{equation}
	where we further used that $r\leq(\delta_q/H_L)^{1/h_L}$ in the next-to-last step and that obviously $\qa{\varrho}\geq\underbar{L}$ for any $a\in(0,1]$ in the last.
    Then $L(\theta)\leq\frac{2}{\beta}\int_{\beta/2}^{\beta} \qa{\varrho}\,da+\delta_q$ together with the first observation made implies $\theta\in\Qbeta{\varrho}$, ergo $B_r(\thetaG)\subset\Qbeta{\varrho}$.

	\textbf{Set inclusion $\Qbeta{\varrho}\subset\CN_{r_G}(\Theta)$.}
	We notice
    that with \eqref{eq:icpL_1} of the inverse continuity property~\ref{asm:icpL} on the function $L$ w.r.t.\@ the set $\Theta$
    it holds
    \begin{equation}
        \label{eq:proof:lem:laplace_LG:0}
        \dist(\theta,\Theta)
        \leq \frac{1}{\eta_L}\left(L(\theta)-\underbar{L}\right)^{\nu_L}
        \leq \frac{1}{\eta_L}\left(\left(\frac{2}{\beta}\int_{\beta/2}^{\beta} \qa{\varrho}\,da+\delta_q\right) -\underbar{L}\right)^{\nu_L}
        \quad \text{for all }
        \theta\in\CN_{R_L}(\Theta)\cap\Qbeta{\varrho},
    \end{equation}
    where for the latter inequality we recall that by definition of $\Qbeta{\varrho}$ it holds $L(\theta)\leq\frac{2}{\beta}\int_{\beta/2}^{\beta} \qa{\varrho}\,da+\delta_q$ for all $\theta\in\Qbeta{\varrho}$.
    Together with the observation~\eqref{eq:qbeta_monotonicity} we may further bound
    \begin{equation}
        \label{eq:proof:lem:laplace_LG:1}
        \dist(\theta,\Theta)
        \leq \frac{1}{\eta_L}\left(\left(\qbeta{\varrho}+\delta_q\right)-\underbar{L}\right)^{\nu_L}
        \quad \text{for all }
        \theta\in\CN_{R_L}(\Theta)\cap\Qbeta{\varrho}.
    \end{equation}
    For $\theta\in\Qbeta{\varrho}$ it furthermore holds, again by \eqref{eq:qbeta_monotonicity}, and by assumption on $\beta$ that
    \begin{equation}
        L(\theta)\leq\frac{2}{\beta}\int_{\beta/2}^{\beta} \qa{\varrho}\,da+\delta_q\leq\qbeta{\varrho}+\delta_q\leq\underbar{L}+L_\infty.
    \end{equation}
    Since, on the other hand, for all $ \theta\in\big(\CN_{R_L}(\Theta)\big)^c$ it holds $L(\theta)>\underbar{L}+L_\infty$ by \eqref{eq:icpL_2} of the inverse continuity property~\ref{asm:icpL},
    we have $\Qbeta{\varrho}\subset\CN_{R_L}(\Theta)$.
    Thus \eqref{eq:proof:lem:laplace_LG:1} holds on $\Qbeta{\varrho}$.
    Moreover, with the assumption on $\beta$ we conclude from \eqref{eq:proof:lem:laplace_LG:1} and the latter observation that
    \begin{equation}
        \label{eq:proof:lem:laplace_LG:2}
        \dist(\theta,\Theta)
        \leq r_G
        \quad \text{for all }
        \theta\in\Qbeta{\varrho},
    \end{equation}
    i.e., $\Qbeta{\varrho}\subset\CN_{r_G}(\Theta)$.
    
    With $r_G\leq R_G$ the inverse continuity property~\ref{asm:icpG} on $G$ holds.
    Since $\Qbeta{\varrho}\subset\CN_{r_G}(\Theta)$ as established in \eqref{eq:proof:lem:laplace_LG:2}, it holds in particular on the set $\Qbeta{\varrho}$, on which $\mAlphaBetanoarg$ is computed.

    \textbf{Main proof.}
    In order to control the term $\Nbig{\mAlphaBeta{\varrho} - \thetaG}_2$,
    we can follow the proof of the quantitative Laplace principle \cite[Propostion~4.5]{fornasier2021consensus} and \cite[Proposition~1]{fornasier2021convergence} with the necessary modification which are due to the inverse continuity property~\ref{asm:icpG} on $G$ being slightly different as a consequence of the necessity of distinguishing between $\thetaG$ and $\thetaGtilde$.
    
    Let $\widetilde r \geq r > 0$ and recall that $r\in(0,\min\{r_R,r_G,R^H_{L},(\delta_q/H_L)^{1/h_L}\}]$ by assumption.
	Using the definition of the consensus point $\mAlphaBeta{\varrho} = \int \theta \omegaa(\theta)/\Nnormal{\omegaa}_{L^1(\Ibeta{\varrho})}\,d\Ibeta{\varrho}(\theta)$ we can decompose
    \begin{equation} \label{proof:Laplace_principle:eq:decomposition}
	\begin{split}
	    \Nbig{\mAlphaBeta{\varrho} - \thetaG}_2
		&\leq \int_{B_{\widetilde{r}}(\thetaG)} \Nbig{\theta-\thetaG}_2 \frac{\omegaa(\theta)}{\N{\omegaa}_{L^1(\Ibeta{\varrho})}} d\Ibeta{\varrho}(\theta) \\
        &\qquad+ \int_{\left(B_{\widetilde{r}}(\thetaG)\right)^c} \Nbig{\theta-\thetaG}_2 \frac{\omegaa(\theta)}{\N{\omegaa}_{L^1(\Ibeta{\varrho})}} d\Ibeta{\varrho}(\theta).
	\end{split}
	\end{equation}
    \textit{First term in \eqref{proof:Laplace_principle:eq:decomposition}}:
    Since for all $\theta \in B_{\widetilde{r}}(\thetaG)$ it holds $\Nbig{\theta-\thetaG}_2\leq \widetilde{r}$, the first term in \eqref{proof:Laplace_principle:eq:decomposition} is bounded by $\widetilde{r}$.
    Recalling the definition of $G_r$ and introducing the notation $\widetilde{G}_r:=\sup_{\theta\in B_{r}(\thetaG)}G(\theta)-G(\thetaGtilde)$, 
    we choose $\widetilde r = (u+\widetilde{G}_r)^{\nu_G}/{\eta_G}$, which is a valid choice since
	\begin{equation}
	\begin{split}
	    \widetilde r
		=
        \frac{(u+\widetilde{G}_r)^{\nu_G}}{\eta_G}
		\geq \frac{\widetilde{G}_r^{\nu_G}}{\eta_G}
		&=
        \frac{\left(\sup_{\theta\in B_r(\thetaG)}G(\theta)-G(\thetaGtilde)\right)^{\nu_G}}{\eta_G} \\
        &\geq
        \sup_{\theta\in B_r(\thetaG)}\Nbig{\theta-\thetaGtilde}_2 \\
       &\geq r.
	\end{split}
	\end{equation}
    Here, 
    the inequality in the second line holds since the first part~\eqref{eq:icpG_1} of the inverse continuity property~\ref{asm:icpG} on $G$ holds on the set $B_r(\thetaG)\subset B_{r_G}(\thetaG)$ since $r\leq r_G$.
    To see the inequality in the last line recall that $r\leq r_G$ and $\thetaGtilde\in B_{r_G}(\thetaG)$ (note that the expression is minimal if $\thetaGtilde=\thetaG$, in which case the supremum is equal to $r$).

    \noindent
    \textit{Second term in \eqref{proof:Laplace_principle:eq:decomposition}}:
    Let us first note that because of Markov's inequality it holds for any $a>0$ that 
    \begin{equation} \label{eq:int_omegaa}
        \Nnormal{\omegaa}_{L^1(\Ibeta{\varrho})}
        \geq a \varrho\left(\left\{\theta\in B_R(0) : \exp(-\alpha G(\theta)) \geq a \text{ and } L(\theta)\leq \frac{2}{\beta}\int_{\beta/2}^{\beta} \qa{\varrho}\,da+\delta_q\right\}\right)
    \end{equation}
	Choosing $a = \exp\big(-\alpha (\widetilde{G}_r+G(\thetaGtilde))\big)$ and noticing that
	\begin{equation}
	\begin{split}
	    &\varrho\left(\left\{\theta\in B_R(0) : \exp(-\alpha G(\theta)) \geq \exp\big(-\alpha (\widetilde{G}_r+G(\thetaGtilde))\big) \text{ and } L(\theta)\leq \frac{2}{\beta}\int_{\beta/2}^{\beta} \qa{\varrho}\,da+\delta_q\right\}\right)\\
		&\qquad\qquad
        = \varrho\left(\left\{\theta\in B_R(0) : G(\theta) \leq \widetilde{G}_r+G(\thetaGtilde) \text{ and } L(\theta)\leq \frac{2}{\beta}\int_{\beta/2}^{\beta} \qa{\varrho}\,da+\delta_q\right\}\right)\\
        &\qquad\qquad
        = \varrho\left(\left\{\theta\in B_R(0) : G(\theta) \leq \sup_{\tilde\theta\in B_{r}(\thetaG)}G(\tilde\theta) \text{ and } L(\theta)\leq \frac{2}{\beta}\int_{\beta/2}^{\beta} \qa{\varrho}\,da+\delta_q\right\}\right)\\
        &\qquad\qquad
        \geq \varrho\left(B_R(0) \cap B_r(\thetaG) \cap \Qbeta{\varrho} \right)\\
        &\qquad\qquad
        = \varrho\left(B_r(\thetaG) \right),
	\end{split}
	\end{equation}
	where we reused the in the preliminaries shown set inclusions $B_r(\thetaG)\subset B_R(0)$ and $B_r(\thetaG)\subset\Qbeta{\varrho}$ in the last step,
    allows to conclude \eqref{eq:int_omegaa} to obtain  $\Nnormal{\omegaa}_{L^1(\Ibeta{\varrho})} \geq \exp\big(-\alpha (\widetilde{G}_r+G(\thetaGtilde))\big)\varrho\big(B_r(\thetaG)\big)$.
	With this we have
	\begin{equation}
	\begin{split}
		&\int_{\left(B_{\widetilde{r}}(\thetaG)\right)^c} \Nbig{\theta-\thetaG}_2 \frac{\omegaa(\theta)}{\N{\omegaa}_{L^1(\Ibeta{\varrho})}} d\Ibeta{\varrho}(\theta)\\
		&\qquad\qquad\leq\int_{\left(B_{\widetilde{r}}(\thetaG)\right)^c\cap\Qbeta{\varrho}} \Nbig{\theta-\thetaG}_2 \frac{\exp\big(-\alpha (G(\theta)-(\widetilde{G}_r+G(\thetaGtilde)))\big)}{\varrho\big(B_r(\thetaG)\big)} d\Ibeta{\varrho}(\theta)\\
		&\qquad\qquad\leq\! \frac{\exp\left(-\alpha \left(\inf_{\theta \in \left(B_{\widetilde{r}}(\thetaG)\right)^c\cap\Qbeta{\varrho}} G(\theta) - G(\thetaGtilde) - \widetilde{G}_r\right)\right)}{\varrho\big(B_r(\thetaG)\big)}\!\int\Nbig{\theta-\thetaG}_2 d\Ibeta{\varrho}(\theta)
	\end{split}
	\end{equation}
    for the second term in \eqref{proof:Laplace_principle:eq:decomposition}.
    Thus, for any $\widetilde r \geq r > 0$ we obtain for \eqref{proof:Laplace_principle:eq:decomposition} that
	\begin{align} \label{eq:aux_laplace_1}
	\begin{aligned}
	     \Nbig{\mAlphaBeta{\varrho} - \thetaG}_2
		&\leq \widetilde r + \frac{\exp\left(-\alpha \left(\inf_{\theta \in \left(B_{\widetilde{r}}(\thetaG)\right)^c\cap\Qbeta{\varrho}} G(\theta) - G(\thetaGtilde) - \widetilde{G}_r\right)\right)}{\varrho\big(B_r(\thetaG)\big)}\!\\
        &\qquad\qquad\qquad\qquad\cdot\int\N{\theta-\thetaG}_2 d\Ibeta{\varrho}(\theta).
	\end{aligned}
	\end{align}
    With the choice $\widetilde r = (u+\widetilde{G}_r)^{\nu_G}/{\eta_G}$
    it holds under the assumption $u + G_{r} + H_G{r_G}^{h_G} \leq G_{\infty}$ that
    \begin{align}
        \label{eq:proof:laplace:widetilde_r}
    \begin{split}
        \widetilde r
        = \frac{(u+\widetilde{G}_r)^{\nu_G}}{\eta_G}
        &= \frac{(u+G_r + (G(\thetaG)-G(\thetaGtilde)))^{\nu_G}}{\eta_G} \\
        &\leq \frac{(u+G_r + H_G{r_G}^{h_G})^{\nu_G}}{\eta_G}
        \leq \frac{G_{\infty}^{\nu_G}}{\eta_G}.
    \end{split}
    \end{align}
    In order to obtain the inequality in the next-to-last step be reminded that $\thetaGtilde\in B_{r_G}(\thetaG)$ and that $r_G\leq R^H_{G}$, allowing us to employ the H\"older continuity~\ref{asm:LipG} of $G$ to derive the upper bound $\absbig{G(\thetaG)-G(\thetaGtilde)}\leq H_G\Nbig{\thetaG-\thetaGtilde}_2^{h_G}\leq H_Gr_G^{h_G}$.
    Reordering \eqref{eq:proof:laplace:widetilde_r} shows $(\eta_G\widetilde r)^{1/\nu_G} \leq G_{\infty}$.
    Thus, recalling that $\Qbeta{\varrho}\subset\CN_{r_G}(\Theta)$ as established in \eqref{eq:proof:lem:laplace_LG:2} and using again both \eqref{eq:icpG_1} and \eqref{eq:icpG_2} of \ref{asm:icpG}, we have the first inequality in
    \begin{equation}
        \inf_{\theta \in \left(B_{\widetilde{r}}(\thetaG)\right)^c\cap\Qbeta{\varrho}} G(\theta) - G(\thetaGtilde) - \widetilde{G}_r 
        \geq \min\big\{G_{\infty}, (\eta_G \widetilde r)^{1/{\nu_G}}\big\} - \widetilde{G}_r = (\eta_G \widetilde r)^{1/{\nu_G}} - \widetilde{G}_r = u.
    \end{equation}
    The other two steps are a result of the definition of $\widetilde r$.
	Inserting this and the upper bound $\widetilde r \leq (u+G_r + H_G{r_G}^{h_G})^{\nu_G}/\eta_G$ from \eqref{eq:proof:laplace:widetilde_r} into \eqref{eq:aux_laplace_1},
    we obtain the result.
\end{proof}

\subsection{A Lower Bound for the Probability Mass around \texorpdfstring{$\thetaG$}{Good Global Minimizer} } \label{sec:lower_bound_prob_mass}

In this section, we establish a lower bound on the probability mass $\min_{t\in[0,T]}\rho_t(B_{r}(\thetaG))$ for any arbitrarily small radius $r > 0$ and a finite time horizon $T<\infty$,
which is crucial to ensure the feasibility of the choice \eqref{eq:beta_QQLP} and employ \eqref{eq:lem:laplace_LG} from Lemma~\ref{lem:laplace_LG}.

To do so, we define a smooth mollifier $\phi_r : \bbR^d \rightarrow [0,1]$ satisfying $\supp{\phi_r} = B_{r}(\thetaG)$.
One such possibility is
\begin{equation} \label{eq:mollifier}
	\phi_{r}(\theta) :=
	\begin{cases}
		\exp\left(1-\frac{r^2}{r^2-\N{\theta-\thetaG}_2^2}\right),& \textrm{ if } \Nbig{\theta-\thetaG}_2 < r,\\
		0,& \textrm{ else.}
	\end{cases}
\end{equation}
It then holds $\rho_t(B_{r}(\thetaG)) \geq \int \phi_r(\theta)\,d\rho_t(\theta)$,
where the right-hand side can be studied by using the weak solution property~\eqref{eq:weak_solution_identity} of $\rho$ as in Definition~\ref{def:weak_solution}.

\begin{proposition}[Lower bound for $\rho_t\big(B_{r}(\thetaG)\big)$, cf.\@{\cite[Proposition~4.6]{fornasier2021consensus}}]
    \label{lem:lower_bound_probability}
	Let $T > 0,\ r > 0$, and fix parameters $\alpha,\beta,R,\delta_q,\lambda,\sigma > 0$.
	Assume $\rho\in\CC([0,T],\CP(\bbR^d))$ weakly solves the Fokker-Planck Equation~\eqref{eq:fokker_planck} in the sense of Definition~\ref{def:weak_solution} with initial condition $\rho_0 \in \CP(\bbR^d)$ and for $t \in [0,T]$.
	Then, for all $t\in[0,T]$ we have
	\begin{align} 
		\rho_t\left(B_{r}(\thetaG)\right)
		&\geq \left(\int\phi_{r}(\theta)\,d\rho_0(\theta)\right)\exp\left(-pt\right), \text{ where} \label{eq:lower_bound_probability_rate} \\
        p &:= 2\max\left\{\frac{\lambda(cr+B\sqrt{c})}{(1-c)^2r}+\frac{\sigma^2(cr^2+B^2)(2c+d)}{(1-c)^4r^2},\frac{2\lambda^2}{(2c-1)\sigma^2}\right\}
        \label{eq:def_p}
	\end{align}
	for any $B<\infty$ with $\sup_{t \in [0,T]} \Nbig{\mAlphaBeta{\rho_t} - \thetaG}_2 \leq B$
	and for any $c \in (1/2,1)$ satisfying $d(1-c)^2\leq(2c-1)c$.
\end{proposition}

\begin{proof}
    The result follows immediately from \cite[Proposition~4.6]{fornasier2021consensus} after noticing that its proof does not rely on the specific form and relation of $\mAlphaBeta{\rho_t}$ and $\thetaG$ as long as $\sup_{t \in [0,T]} \Nbig{\mAlphaBeta{\rho_t} - \thetaG}_2 \leq B<\infty$, which is per assumption.
    This has been pointed out in Remark~4.7(i) in the extended arXiv version of \cite{fornasier2021consensus} (corresponding to \cite[Remark~4.7]{fornasier2021consensus}\footnote{This remark was removed in the published version~\cite{fornasier2021consensus} due to the journal's page limit policy.} in the journal version).
\end{proof}

\subsection{Proof of Theorem~\ref{thm:main}} \label{sec:proof_main}

We are now ready to combine the formerly established technical tools and present the proof of Theorem~\ref{thm:main}.

\begin{proof}[Proof of Theorem~\ref{thm:main}]
    Let us start by recalling that $G_r:=\sup_{\theta\in B_{r}(\thetaG)}G(\theta)-G(\thetaG)$ for $r>0$
    and by introducing the abbreviation
	\begin{align}\label{eq:c}
		c\left(\vartheta,\lambda,\sigma\right)
		:=
        \min\left\{
			\frac{\vartheta}{2}\frac{\left(2\lambda-d\sigma^2\right)}{\sqrt{2}\left(\lambda+d\sigma^2\right)}, 
			\sqrt{\vartheta\frac{\left(2\lambda-d\sigma^2\right)}{d\sigma^2}}
			\right\}.
	\end{align}
    Moreover, we define
    \begin{align} \label{eq:rGeps}
        r_{G,\varepsilon}
        :=
        \min\left\{\left(\frac{1}{2H_G}\left(\eta_G\frac{c\left(\vartheta,\lambda,\sigma\right)\sqrt{\varepsilon}}{2}\right)^{1/{\nu_G}}\right)^{1/{h_G}},R_G,R^H_{G},\left(\frac{G_{\infty}}{2H_G}\right)^{1/h_G}\right\},
    \end{align}
    which satisfies $r_{G,\varepsilon}\in(0,\min\{R_G,R^H_{G},(G_{\infty}/(2H_G))^{1/h_G}\}]$ by design.
    We further emphasize, that, by assumption, $\delta_q>0$ is sufficiently small (see Remark~\ref{remark:deltaq}) in the sense that it satisfies
    \begin{align}
    	\delta_{q}
    	\leq \frac{1}{2} \min\left\{L_\infty,(\eta_Lr_{G,\varepsilon})^{1/\nu_L}\right\}.
    \end{align}

    \textbf{Choice of $\beta$.}
    Let us further abbreviate by $\xi_{L,\varepsilon} := \min\left\{L_\infty,(\eta_Lr_{G,\varepsilon})^{1/\nu_L}\right\}-\delta_{q}$ and point out that $\xi_{L,\varepsilon}>0$ granted by the afore-discussed properties of $\delta_q$. 
    Further, define
    \begin{align} \label{eq:rHeps}
        r_{H,\varepsilon}
        := \min\left\{R^H_{L}, \left(\frac{\xi_{L,\varepsilon}}{H_L}\right)^{1/h_L}\right\},
    \end{align}
    which satisfies $r_{H,\varepsilon}\in(0,R^H_{L}]$ by design.

    We now choose $\beta\in(0,1)$ such that
    \begin{equation} \label{eq:beta}
        \beta <
        \beta_0
        := \frac{1}{2}\rho_0\big(B_{r_{H,\varepsilon}/2}(\thetaG)\big) \exp(-p_{H,\varepsilon}T^*),
    \end{equation}
    where $p_{H,\varepsilon}$ is as defined in \eqref{eq:def_p} in Proposition~\ref{lem:lower_bound_probability} with $B=c(\vartheta,\lambda,\sigma)\sqrt{\CV(\rho_0)}$ and with $r=r_{H,\varepsilon}$.
    Such choice of $\beta$ is possible since $\beta_0\in(0,1)$, which in turn is due to $p_{H,\varepsilon},T^*<\infty$, $\thetaG\in\supp{\rho_0}$ and $r_{H,\varepsilon}>0$ on the one hand, and $p_{H,\varepsilon}T^*>0$ and $\rho_0$ being a probability measure on the other.
    For such $\beta$ we have $\qbeta{\rho_t} \leq \underbar{L}+\xi_{L,\varepsilon}$ for all $t\in[0,T^*]$ for the following reason:
    Due to \ref{asm:LipL}, for all $\theta\in B_{r_{H,\varepsilon}}(\thetaG)$ it holds
    $L(\theta)-\underbar{L} \leq H_L\Nnormal{\theta-\thetaG}_2^{h_L}\leq H_Lr_{H,\varepsilon}^{h_L}\leq\xi_{L,\varepsilon}$.
    Thus, $B_{r_{H,\varepsilon}}(\thetaG)\subset \{\theta:L(\theta)-\underbar{L}\leq\xi_{L,\varepsilon}\}$.
    We further note that by Proposition~\ref{lem:lower_bound_probability} with $r_{H,\varepsilon}$, $p_{H,\varepsilon}$ and $B$ (for $r$, $p$ and $B$) as defined before, it holds for all $t\in[0,T^*]$ that 
    \begin{align}
	\begin{aligned}
	    \rho_{t}\big(B_{r_{H,\varepsilon}}(\thetaG)\big)
		&\geq
        \left(\int \phi_{r_{H,\varepsilon}}(\theta) \,d\rho_0(\theta)\right)\exp(-p_{H,\varepsilon}t) \\
		&\geq
        \frac{1}{2}\,\rho_0\big(B_{r_{H,\varepsilon}/2}(\thetaG)\big) \exp(-p_{H,\varepsilon}T^*) 
		> \beta,
	\end{aligned}
	\end{align}
	where we used in the second line the fact that $\phi_{r}$ (as defined in Equation~\eqref{eq:mollifier}) is bounded from below on $B_{r/2}(\thetaG)$ by $1/2$.
    The last inequality is by choice of $\beta$.
    With this and the aforementioned set inclusion, we have for all $t\in[0,T^*]$ that
    \begin{align}
        \beta
        <
        \rho_{t}\big(B_{r_{H,\varepsilon}}(\thetaG)\big)
        \leq
        \rho_{t}\big(\{\theta:L(\theta)-\underbar{L}\leq\xi_{L,\varepsilon}\}\big)
    \end{align}
    and thus by definition of $\qbeta{\dummy}$ as the infimum, that $\qbeta{\rho_t}\leq\underbar{L}+\xi_{L,\varepsilon}$ for all $t\in[0,T^*]$.

    \textbf{Choice of $\alpha$.}
    Let us further define 
    \begin{align}\label{eq:ueps}
		u_\varepsilon := \frac{1}{4}\min\bigg\{\!\left(\eta_G\frac{c\left(\vartheta,\lambda,\sigma\right)\sqrt{\varepsilon}}{2}\right)^{1/{\nu_G}}\!,G_{\infty}\bigg\}
        \quad \text{and} \quad
        \widetilde{r}_\varepsilon := \max_{s \in [0,r_{G,\varepsilon}]}\left\{\max_{\theta \in B_s(\thetaG)}G(\theta) -G(\thetaG) \leq u_\varepsilon\right\},
	\end{align}
    which satisfy $u_\varepsilon>0$ and $\widetilde{r}_\varepsilon\in(0,r_{G,\varepsilon}]$ by design.
    Notice that the continuity of $G$ ensures that there exists $s_{u_\varepsilon}>0$ such that $G(\theta)-G(\thetaG)\leq u_\varepsilon$ for all $\theta\in B_{s_{u_\varepsilon}}(\thetaG)$, thus yielding also $\widetilde{r}_\varepsilon>0$.
    Moreover, since $G_r$ is monotonously increasing in $r$, any $r_\varepsilon\in(0,\widetilde{r}_\varepsilon]\subset(0,r_{G,\varepsilon}]$,
    in particular the choice
    \begin{align}\label{eq:reps}
        r_\varepsilon
        :=
        \min\left\{r_R,\widetilde{r}_\varepsilon,R^H_{L},\left(\frac{\delta_q}{H_L}\right)^{1/h_L}\right\},
    \end{align}
    satisfies
    \begin{align}
        u_\varepsilon + G_{r_\varepsilon} + H_Gr_{G,\varepsilon}^{h_G} 
        \leq u_\varepsilon + G_{\widetilde{r}_\varepsilon} + H_Gr_{G,\varepsilon}^{h_G}
        \leq u_\varepsilon + G_{\widetilde{r}_\varepsilon} + \frac{G_{\infty}}{2}
        \leq 2u_\varepsilon + \frac{G_{\infty}}{2}
        \leq G_\infty,
    \end{align}
    where we utilized that $r_{G,\varepsilon}\leq(G_{\infty}/(2H_G))^{1/h_G}$ in the second step, the definitions of $\widetilde{r}_\varepsilon$ as well as $G_{r}$ in the third, and the definition of $u_\varepsilon$ in the last.

    All parameters $r_{G,\varepsilon}$, $\delta_q$, $\beta$, $r_\varepsilon$, and $u_\varepsilon$ are now in line with the requirements of Proposition~\ref{lem:laplace_LG} and it remains to choose $\alpha$
    such that 
	\begin{equation}
	   \label{eq:alpha}
	\begin{split}
		\alpha > 
		\alpha_0
		&:= \frac{1}{u_\varepsilon}\Bigg(\log\left(\frac{4\sqrt{2}}{c\left(\vartheta,\lambda,\sigma\right)}\right)-\log\left(\rho_{0}\big(B_{r_\varepsilon/2}(\thetaG)\big)\right)+\max\left\{\frac{1}{2}\log\left(\frac{\CV(\rho_0)}{\varepsilon}\right),p_\varepsilon T^*\right\}\!\Bigg)
	\end{split}
	\end{equation}
    where $p_{\varepsilon}$ is as $p$ defined in \eqref{eq:def_p} in Proposition~\ref{lem:lower_bound_probability} with $B=c(\vartheta,\lambda,\sigma)\sqrt{\CV(\rho_0)}$ and with $r=r_\varepsilon$.
    
    \textbf{Main proof.}
    Let us now define the time horizon $T_{\alpha,\beta} \geq 0$, which may depend on $\alpha$ and $\beta$,
    by
	\begin{align} \label{eq:endtime_T}
		T_{\alpha,\beta} := \sup\big\{t\geq0 : \CV(\rho_{t'}) > \varepsilon \text{ and } \N{\mAlphaBeta{\rho_t}-\thetaG}_2 < C(t') \text{ for all } t' \in [0,t]\big\}
	\end{align}
	with $C(t):=c(\vartheta,\lambda,\sigma)\sqrt{\CV(\rho_t)}$.
	Notice for later use that $C(0)=B$.

    Our aim is to show $\CV(\rho_{T_{\alpha,\beta}}) = \varepsilon$ with $T_{\alpha,\beta}\in\big[\frac{1-\vartheta}{(1+\vartheta/2)}\;\!T^*,T^*\big]$ and that we have at least exponential decay of $\CV(\rho_t)$ until time $T_{\alpha,\beta}$, i.e., until accuracy~$\varepsilon$ is reached.
	
	First, however, we ensure that $T_{\alpha,\beta}>0$.
    With the mapping~$t\mapsto\CV(\rho_{t})$ being continuous as a consequence of the regularity $\rho\in\CC([0,T^*], \CP_4(\bbR^d))$ (as established in Theorem~\ref{thm:well_posedness}) and $t\mapsto\!\Nbig{\mAlphaBeta{\rho_{t}} \!-\! \thetaG}_2$ being continuous as a consequence of the continuity of the mapping $t \mapsto \mAlphaBeta{\rho_t}$ (as established in Theorem~\ref{thm:well_posedness}), 
    $T_{\alpha,\beta}>0$ follows from the definition, since $\CV(\rho_{0}) > \varepsilon$ and $\Nbig{\mAlphaBeta{\rho_{0}} - \thetaG}_2 < C(0)$.
	While the former is immediate by assumption, for the latter an application of Proposition~\ref{lem:laplace_LG} with $r_{G,\varepsilon}$, $r_\varepsilon$, $u_\varepsilon$ and $\rho_0$ (for $r_G$, $r$, $u$ and $\varrho$) as defined at the beginning of the proof yields
    \begin{align} \label{eq:proof:lapl1}
	\begin{split}
        \N{\mAlphaBeta{\rho_{0}} - \thetaG}_2
        &\leq
        \frac{(u_\varepsilon+G_{r_\varepsilon} + H_Gr_{G,\varepsilon}^{h_G})^{\nu_G}}{\eta_G} + \frac{\exp\left(-\alpha u_\varepsilon\right)}{\rho_{0}\big(B_{r_\varepsilon}(\thetaG) \big)}\int\N{\theta-\thetaG}_2d\Ibeta{\rho_{0}}(\theta)\\
        &\leq
        \frac{(2u_\varepsilon+ H_Gr_{G,\varepsilon}^{h_G})^{\nu_G}}{\eta_G} + \frac{\exp\left(-\alpha u_\varepsilon\right)}{\rho_{0}\big(B_{r_\varepsilon}(\thetaG) \big)} \int\N{\theta-\thetaG}_2d\Ibeta{\rho_{0}}(\theta)\\
        &\leq
        \frac{c\left(\vartheta,\lambda,\sigma\right)\sqrt{\varepsilon}}{2} + \frac{\exp\left(-\alpha u_\varepsilon\right)}{\rho_{0}\big(B_{r_\varepsilon}(\thetaG) \big)}\sqrt{2\CV(\rho_0)}\\
		&\leq \frac{c\left(\vartheta,\lambda,\sigma\right)\sqrt{\varepsilon}}{2} + \frac{\exp\left(-\alpha u_\varepsilon\right)}{\rho_{0}\big(B_{r_\varepsilon}(\thetaG) \big)}\sqrt{2\CV(\rho_0)}\\
        &\leq c\left(\vartheta,\lambda,\sigma\right)\sqrt{\varepsilon} < c\left(\vartheta,\lambda,\sigma\right)\sqrt{\CV(\rho_0)} = C(0).
	\end{split}
	\end{align}
    Proposition~\ref{lem:laplace_LG} is applicable, since all assumptions hold as formerly verified.
    Thus, the first inequality in \eqref{eq:proof:lapl1} holds.
    The second step is due to the definition of $r_\varepsilon$,
    the third and fourth step follow after plugging in the definitions and properties of $r_{G,\varepsilon}$ and $u_\varepsilon$, and simplifying.
    Finally, the first inequality in the last line of \eqref{eq:proof:lapl1} holds by choice of $\alpha$ in \eqref{eq:alpha}.
    The remainder is by assumption and definition.    

	Next, we show that the functional $\CV(\rho_t)$ decays essentially exponentially fast in time.
    More precisely, we prove that, up to time $T_{\alpha,\beta}$, $\CV(\rho_t)$ decays
    \begin{enumerate}[label=(\roman*),labelsep=10pt,leftmargin=35pt]
        \item at least exponentially fast (with rate $(1-\vartheta)(2\lambda-d\sigma^2)$), and \label{enumerate:proof:atleastexpdecay}
        \item at most exponentially fast (with rate $(1+\vartheta/2)(2\lambda-d\sigma^2)$).\label{enumerate:proof:atmostexpdecay}
    \end{enumerate}
	To obtain \ref{enumerate:proof:atleastexpdecay}, recall that Equation~\eqref{eq:evolution_of_objective} in Lemma~\ref{lem:evolution_V} provides an upper bound on $\frac{d}{dt}\CV(\rho_t)$ given by
    \begin{equation}
	\begin{split}
	    \frac{d}{dt}\CV(\rho_t)
		\leq
		&-\left(2\lambda - d\sigma^2\right) \CV(\rho_t)
	    + \sqrt{2}\left(\lambda + d\sigma^2\right) \sqrt{\CV(\rho_t)} \N{\mAlphaBeta{\rho_t}-\thetaG}_2\phantom{.} \\
	    &+ \frac{d\sigma^2}{2} \N{\mAlphaBeta{\rho_t}-\thetaG}_2^2.
	\end{split}
	\end{equation}
	Combining this with the definition of $T_{\alpha,\beta}$ in \eqref{eq:endtime_T} we have by construction 
	\begin{align}
		\frac{d}{dt}\CV(\rho_t)
		\leq -(1-\vartheta)\left(2\lambda-d\sigma^2\right)\CV(\rho_t)
		\quad \text{ for all } t \in (0,T_{\alpha,\beta}).
	\end{align}
    Analogously, for \ref{enumerate:proof:atmostexpdecay}, by the second part of Lemma~\ref{lem:evolution_V}, Equation~\eqref{eq:evolution_of_objective_lower},
    we obtain a lower bound on $\frac{d}{dt}\CV(\rho_t)$ of the form
    \begin{equation}
    \begin{split}
        \frac{d}{dt}\CV(\rho_t)
        &\geq
        -\left(2\lambda - d\sigma^2\right) \CV(\rho_t)
        - \sqrt{2}\left(\lambda + d\sigma^2\right) \sqrt{\CV(\rho_t)} \N{\mAlphaBeta{\rho_t}-\thetaG}_2 \\
        &\geq -(1+\vartheta/2)\left(2\lambda - d\sigma^2\right) \CV(\rho_t)
        \quad \text{ for all } t \in (0,T_{\alpha,\beta}),
    \end{split}
    \end{equation}
    where the second inequality again exploits the definition of $T_{\alpha,\beta}$.
	Gr\"onwall's inequality now implies for all $t \in [0,T_{\alpha,\beta}]$ the upper and lower bound
	\begin{align}
		\CV(\rho_t)
		&\leq \CV(\rho_0) \exp\left(- (1-\vartheta)\left(2\lambda-d\sigma^2\right) t\right), \label{eq:evolution_J_no_H}\\
		\CV(\rho_t)
		&\geq \CV(\rho_0) \exp\left(- (1+\vartheta/2)\left(2\lambda-d\sigma^2\right) t\right), \label{eq:evolution_J_no_H_lower}
	\end{align}
    thereby proving \ref{enumerate:proof:atleastexpdecay} and \ref{enumerate:proof:atmostexpdecay}.
	We further note that the definition of $T_{\alpha,\beta}$ in \eqref{eq:endtime_T} together with the definition of $C(t)$ and \eqref{eq:evolution_J_no_H} permits to control
	\begin{align}
		&\max_{t \in [0,T_{\alpha,\beta}]} \N{\mAlphaBeta{\rho_{t}} - \thetaG}_2
		\leq \max_{t \in [0,T_{\alpha,\beta}]} C(t)\leq C(0).
		\label{eq:max_bound_distance_no_H}
	\end{align}
    
    To conclude, it remains to prove that $\CV(\rho_{T_{\alpha,\beta}}) = \varepsilon$ with $T_{\alpha,\beta}\in\big[\frac{1-\vartheta}{(1+\vartheta/2)}\;\!T^*,T^*\big]$.
    For this we distinguish the following three cases.

	\noindent
	\textbf{Case $T_{\alpha,\beta} \geq T^*$:}
	We can use the definition of $T^*$ in \eqref{eq:end_time_star_statement} and the time-evolution bound of $\CV(\rho_t)$ in \eqref{eq:evolution_J_no_H} to conclude that $\CV(\rho_{T^*}) \leq \varepsilon$.
	Hence, by definition of $T_{\alpha,\beta}$ in \eqref{eq:endtime_T} together with the continuity of $t\mapsto\CV(\rho_t)$, we find $\CV(\rho_{T_{\alpha,\beta}}) =\varepsilon$ with $T_{\alpha,\beta} = T^*$.
	
	\noindent
	\textbf{Case $T_{\alpha,\beta} < T^*$ and $\CV(\rho_{T_{\alpha,\beta}}) \leq \varepsilon$:}
	By continuity of $t\mapsto\CV(\rho_t)$, it holds for $T_{\alpha,\beta}$ as defined in \eqref{eq:endtime_T}, $\CV(\rho_{T_{\alpha,\beta}}) = \varepsilon$.
    Thus, $\varepsilon
        = \CV(\rho_{T_{\alpha,\beta}})
        \geq \CV(\rho_0) \exp\left(- (1+\vartheta/2)\left(2\lambda-d\sigma^2\right) T_{\alpha,\beta}\right)$ by \eqref{eq:evolution_J_no_H_lower}, which can be reordered as
    \begin{align}
        \frac{1-\vartheta}{(1+\vartheta/2)} \, T^*
        =\frac{1}{(1+\vartheta/2)\left(2\lambda-d\sigma^2\right)}\log\left(\frac{\CV(\rho_0)}{\varepsilon}\right)
        \leq T_{\alpha,\beta}
        < T^*.
    \end{align}
	
	\noindent
	\textbf{Case $T_{\alpha,\beta} < T^*$ and $\CV(\rho_{T_{\alpha,\beta}}) > \varepsilon$:}
	We shall show that this case can never occur by verifying that $\Nbig{\mAlphaBeta{\rho_{T_{\alpha,\beta}}} - \thetaG}_2 < C(T_{\alpha,\beta})$ due to the choices of $\alpha$ in~\eqref{eq:alpha} and $\beta$ in \eqref{eq:beta}.
	In fact, fulfilling simultaneously both $\CV(\rho_{T_{\alpha,\beta}})>\varepsilon$ and $\Nbig{\mAlphaBeta{\rho_{T_{\alpha,\beta}}} - \thetaG}_2 < C(T_{\alpha,\beta})$ would contradict the definition of $T_{\alpha,\beta}$ in \eqref{eq:endtime_T} itself.
	To this end, we apply again Proposition~\ref{lem:laplace_LG} with $r_{G,\varepsilon}$, $r_\varepsilon$ and $u_\varepsilon$ as defined at the beginning of the proof.
    Proposition~\ref{lem:laplace_LG} is again applicable, since all assumptions hold as formerly verified and since $T_{\alpha,\beta}\leq T^*$.
    Thus, the first inequality in \eqref{eq:proof:lapl2} below holds and the following steps are obtained when plugging in the definitions and properties of $r_\varepsilon$, $r_{G,\varepsilon}$ and $u_\varepsilon$.
    For the sharp inequality in the fourth step we recall that in this case we assumed $\varepsilon<\CV(\rho_{T_{\alpha,\beta}})$.
    The fifth step merely results from simplifying.
    We have
	\begin{align}  \label{eq:proof:lapl2}
	\begin{split}
        \N{\mAlphaBeta{\rho_{T_{\alpha,\beta}}} - \thetaG}_2
        &\leq
        \frac{(u_\varepsilon+G_{r_\varepsilon} + H_Gr_{G,\varepsilon}^{h_G})^{\nu_G}}{\eta_G} + \frac{\exp\left(-\alpha u_\varepsilon\right)}{\rho_{T_{\alpha,\beta}}\big(B_{r_\varepsilon}(\thetaG)\big)}\int\N{\theta-\thetaG}_2d\Ibeta{\rho_{T_{\alpha,\beta}}}(\theta)\\
        &\leq
        \frac{(2u_\varepsilon + H_Gr_{G,\varepsilon}^{h_G})^{\nu_G}}{\eta_G} + \frac{\exp\left(-\alpha u_\varepsilon\right)}{\rho_{T_{\alpha,\beta}}\big(B_{r_\varepsilon}(\thetaG)\big)}\int\N{\theta-\thetaG}_2d\Ibeta{\rho_{T_{\alpha,\beta}}}(\theta)\\
        &\leq\frac{c\left(\vartheta,\lambda,\sigma\right)\sqrt{\varepsilon}}{2} + \frac{\exp\left(-\alpha u_\varepsilon\right)}{\rho_{T_{\alpha,\beta}}\big(B_{r_\varepsilon}(\thetaG) \big)}\int\N{\theta-\thetaG}_2d\Ibeta{\rho_{T_{\alpha,\beta}}}(\theta)\\
        &<
        \frac{c\left(\vartheta,\lambda,\sigma\right)\sqrt{\CV(\rho_{T_{\alpha,\beta}})}}{2} + \frac{\exp\left(-\alpha u_\varepsilon\right)}{\rho_{T_{\alpha,\beta}}\big(B_{r_\varepsilon}(\thetaG) \big)}\sqrt{2\CV(\rho_{T_{\alpha,\beta}})}.
	\end{split}
	\end{align}
	Since, thanks to \eqref{eq:max_bound_distance_no_H}, it holds the bound $\max_{t \in [0,T_{\alpha,\beta}]}\Nbig{\mAlphaBeta{\rho_{t}} - \thetaG}_2 \leq B$ for $B=C(0)$, which is in particular independent of $\alpha$,  Proposition~\ref{lem:lower_bound_probability} guarantees that there exists a $p_\varepsilon>0$ not depending on $\alpha$ (but depending on $B$ and $r_\varepsilon$) with
	\begin{align}
	\begin{aligned}
	    \rho_{T_{\alpha,\beta}}(B_{r_\varepsilon}(\thetaG))
		&\geq \left(\int \phi_{r_\varepsilon}(\theta) \,d\rho_0(\theta)\right)\exp(-p_\varepsilon T_{\alpha,\beta}) \\
		&\geq \frac{1}{2}\,\rho_0\big(B_{r_\varepsilon/2}(\thetaG)\big) \exp(-p_\varepsilon T^* ) 
		> 0,
	\end{aligned}
	\end{align}
	where we used $\thetaG\in\supp{\rho_0}$ for bounding the initial mass $\rho_0$, the fact that $\phi_{r}$ (as defined in Equation~\eqref{eq:mollifier}) is bounded from below on $B_{r/2}(\thetaG)$ by $1/2$ and that $T_{\alpha,\beta}\leq T^*$.
	With this we can continue the chain of inequalities in~\eqref{eq:proof:lapl2} to obtain
	\begin{align} \label{eq:proof:lapl22}
	\begin{split}
		\N{\mAlphaBeta{\rho_{T_{\alpha,\beta}}} - \thetaG}_2
		&< \frac{c\left(\vartheta,\lambda,\sigma\right)\sqrt{\CV(\rho_{T_{\alpha,\beta}})}}{2}  + \frac{2\exp\left(-\alpha u_\varepsilon\right)}{\rho_{0}\big(B_{r_\varepsilon/2}(\thetaG)\big)\exp(-p_\varepsilon T^*)}\sqrt{2\CV(\rho_{T_{\alpha,\beta}})}\\
		&\leq c\left(\vartheta,\lambda,\sigma\right)\sqrt{\CV(\rho_{T_{\alpha,\beta}})}
		= C(T_{\alpha,\beta}) ,
	\end{split}
	\end{align}
    where the last inequality in \eqref{eq:proof:lapl22} holds by choice of $\alpha$ in \eqref{eq:alpha}.
	This establishes the desired contradiction, again as consequence of the continuity of the mappings~$t\mapsto\CV(\rho_{t})$ and~$t\mapsto\Nbig{\mAlphaBeta{\rho_{t}} - \thetaG}_2$.
\end{proof}

\section{Numerical Experiments}\label{sec:experiments}

In this section, we conduct a series of numerical experiments to demonstrate the practicability as well as efficiency of the CB\textsuperscript{2}O method proposed in \eqref{eq:dyn_micro}. 
For this purpose, we first present in Section \ref{sec:numerics_alg} a practical implementation of CB\textsuperscript{2}O.
Thereafter, we showcase in Sections \ref{sec:numerics_COPT} and \ref{sec:numerics_SRL}, respectively, two examples for solving nonconvex bi-level optimization problems of the form~\eqref{eq:bilevel_opt}.
In Section \ref{sec:numerics_COPT}, we consider the special instance of classical constrained global optimization in lower-dimensional settings.
We compare CB\textsuperscript{2}O with the methods~\cite{borghi2021constrained, carrillo2021consensus, carrillo2024interacting,fornasier2020consensus_sphere_convergence,fornasier2020consensus_hypersurface_wellposedness} designed specifically for this task.
In Section~\ref{sec:numerics_SRL}, we then tackle a sparse representation learning task~\cite{gong2021Biojective}, which aims at learning a sparse feature representation on a supervised dataset by using neural networks.
With this we demonstrate that the CB\textsuperscript{2}O method is also applicable in high-dimensional real-data applications.
A further machine learning example in the setting of (clustered) federated learning~\cite{carrillo2024fedcbo}
is tackled in the recent paper \cite{trillos2024attack}, where CB\textsuperscript{2}O is demonstrated to be robust against backdoor attacks by leveraging a bi-level optimization problem of the form~\eqref{eq:bilevel_opt}.

\subsection{\texorpdfstring{Implementation of the CB\textsuperscript{2}O Algorithm}{The CB2O Algorithm}}\label{sec:numerics_alg}

To implement the CB\textsuperscript{2}O algorithm,
we consider an Euler-Maruyama discretization of the interacting $N$ particle system introduced in \eqref{eq:dyn_micro} with time discretization step size $\Delta t > 0$.
At each time step $k = 0,1,2, \dots$, the particles $\{\theta^i\}_{i=1}^N$ are updated according to the iterative update rule
\begin{equation}
\label{eq:dyn_micro_discrete}
\theta_{k+1}^i = \theta_k^i -  \Delta t 
\lambda \left( \theta_k^i - \mAlphaBeta{\rho_k^N} \right) + \sqrt{\Delta t} \sigma D \left(\theta_k^i - \mAlphaBeta{\rho_k^N} \right)B_k^i, \qquad \text{for } i \in \{1,\dots,N\},
\end{equation}
where $B_k^i \sim \CN (0,I_{d\times d})$.
The consensus point $\mAlphaBeta{\rho_k^N}$ is defined analogously to \eqref{eq:ConsensusPoint_FiniteParticles} for the particle positions $\{\theta_k^i\}_{i=1}^N$.
For readers' convenience, let us recall
\begin{equation}
\label{eq:csp_finite_particles_discrete}
\begin{split}
    \mAlphaBeta{\rho_k^N} := \sum_{\theta_k^i: L(\theta_k^i)\leq \qbeta{\rho_k^N}} \theta_k^i \frac{\omegaa(\theta_k^i)}{\sum_{\theta_k^j: L(\theta_k^j)\leq \qbeta{\rho_k^N}} \omegaa(\theta_k^j)},
    \qquad \text{with} \quad
    \omegaa(\theta) := \exp \left( -\alpha G(\theta) \right).
\end{split}
\end{equation}
Here the sub-level set $\Qbeta{\rho_k^N}$ and the $\beta$-quantile function $\qbeta{\rho_k^N}$ are defined analogously to \eqref{eq:Qbeta_FiniteParticles} and \eqref{eq:qbeta_FiniteParticles}, respectively, as
\begin{equation}
\label{eq:Qbeta_qbeta_alg}
\Qbeta{\rho_k^N} = \left\{ \theta \in \R^d : L(\theta) \leq \qbeta{\rho_k^N} \right\} \qquad \text{and} \qquad \qbeta{\rho_k^N} = L \left(\theta_k^{\#\lceil \beta N \rceil} \right),
\end{equation}
with the empirical measure given by $\rho_k^N := \frac{1}{N} \sum_{i=1}^N \delta_{\theta_k^i}$.

Algorithm \ref{alg:cb2o_alg} below now summarizes the CB\textsuperscript{2}O algorithm.

\def\th{\theta}
\def\l{\left}
\def\r{\right}
\def\ll{\lVert}
\def\rl{\rVert}
\def\qd{\quad}
\def\atantwo{\text{atan2}}
\def\estop{\varepsilon_{\text{stop}}}
\def\cstop{c_{\text{stop}}}
\def\tcb{\textcolor{blue}}
\def\tcr{\textcolor{red}}

\begin{algorithm}[htb]
\setstretch{1.25}
\caption{CB\textsuperscript{2}O Algorithm}
\label{alg:cb2o_alg}
\begin{algorithmic}[1]
\REQUIRE
CB\textsuperscript{2}O hyperparamerters~$\lambda,\sigma,\alpha,\beta$; discretization time step size $\Delta t$;  stopping criteria $\estop$; maximal termination round $K$;  
\STATE Initialize particles $\{\theta_0^i\}_{i=1}^N \subset \R^d$;\\
\STATE Initialize the stopping criterion $\cstop > \estop$ and set iteration counter $k = 0$;\\
\WHILE{$\cstop > \estop \text{ and } k \in \{0,\dots, K\}]$}
\STATE Evaluate the $\beta$-quantile $\qbeta{\rho_k^N}$ and estimate the quantile set $\Qbeta{\rho_k^N}$ using \eqref{eq:Qbeta_qbeta_alg};
\STATE Compute the consensus point $\mAlphaBeta{\rho_k^N}$ according to \eqref{eq:csp_finite_particles_discrete} from the particle positions $\{\theta_k^i\}_{i=1}^N \cap \Qbeta{\rho_k^N}$;
\STATE Update the particle positions $\{\theta^i_{k+1}\}_{i=1}^N$ by evaluating  \eqref{eq:dyn_micro_discrete}; 
\STATE Set $k\gets k+1$ and evaluate the stopping criterion $\cstop \gets \frac{1}{dN}\sum_{i=1}^N \N{\theta_{k+1}^i - \mAlphaBeta{\rho_k^N}}_2^2$;
\ENDWHILE
\ENSURE Consensus point $\mAlphaBeta{\rho^N_{\text{stop}}}$.
\end{algorithmic}
\end{algorithm}

\begin{remark}
    Algorithm \ref{alg:cb2o_alg} describes the standard implementation of the CB\textsuperscript{2}O algorithm.
    Many useful numerical tricks, including random batch methods \cite{carrillo2019consensus}, particle re-initialization \cite{carrillo2019consensus, bailo2024cbx}, hyperparameter cooling strategies \cite{fornasier2021convergence}, amongst others, proposed in the past to improve the performance of CBO-type methods, can also be applied immediately to our CB\textsuperscript{2}O algorithm.
    We will discuss some of them in the later experimental sections, and refer to \cite{carrillo2019consensus,fornasier2021convergence,bailo2024cbx} for more detailed and elaborate descriptions.
\end{remark}

\subsection{Constrained Optimization}
\label{sec:numerics_COPT}
Standard constrained optimization problems can be reformulated as bi-level optimization problems of the form \eqref{eq:bilevel_opt}. 
In particular,
any optimization problem with equality constraints,
\begin{equation}
    \min_{\theta\in\bbR^d} G(\theta) \quad \text{s.t.\@}\quad g_\ell(\th) = 0, \quad \ell = 1, \dots, m,
\end{equation}
can be equivalently written as
\begin{equation}\label{eq:reformulate_copt}
    \min_{\theta^*\in \Theta} G(\theta^*) \quad \text{s.t.\@}\quad \Theta =  \argmin_{\theta \in \bbR^d} L(\th) := \sum_{l=1}^mg_l^2(\th).
\end{equation}
Thanks to this reformulation, we can employ our proposed CB\textsuperscript{2}O methods, as described in Algorithm \ref{alg:cb2o_alg}, to constrained optimization problems. 

\subsubsection{Baseline Comparisons}\label{subsec:COPT_baseline_comparison}

Let us first describe in detail the experimental setup for constrained optimization problem.\\

\noindent \textbf{Optimization Problem Setup.}
We choose for the upper-level objective function $G$ the two-dimensional Ackley function, i.e.,
\begin{equation}\label{eq:ackley}
	\begin{aligned}
		&G(\th) = -A\exp\l(-a\sqrt{\frac{b^2}d\Nbig{\th - \hat{\th}}_2^2}\r)- \exp\l(\frac1d\sum_{i=1}^d\cos\left(2\pi b\big(\th_i - \hat{\th}_i\big)\right)\r)+e^1+A,
	\end{aligned}
\end{equation}
where $A = 20$, $a = 0.2$, $b = 3$, $d=2$, and $\hat{\th} =(1/2,1/3)$. Furthermore, we consider two different constraints, a circular and a star constraint, i.e.,
\begin{subequations}
\begin{align}
\text{Circular constraint:} \qd &\left\{(\theta_1, \theta_2) \in \R^2 \, | \, \theta_1^2 + \theta_2^2 - 1 = 0 \right\}, \label{eq:circular_constraint}\\
\text{Star constraint:}\qd 
& \left\{(\theta_1, \theta_2) \in \R^2 \, | \, \theta_1^2 + \theta_2^2 - (1 + 0.5 \sin (5\atantwo(\theta_2, \theta_1)))^2 = 0 \right\}.\label{eq:star_constraint}
\end{align}
\end{subequations}
With respect to the constraints \eqref{eq:circular_constraint} and \eqref{eq:star_constraint}, respectively, the unique constrained global minimizers of $G$ are
\begin{subequations}
    \begin{align}
        \text{Circular:}\quad \thetaG &= (0.781475, 0.623937),\\ 
        \text{Star:}\quad \thetaG &= (0.482208, 0.468687).
    \end{align}
\end{subequations}
Please find visualizations of the two-dimensional Ackley function with circular and star constraints in Figures \ref{subfig:circle} and \ref{subfig:star}, respectively.
\begin{figure}[!htb]
	\centering
	\subcaptionbox{\label{subfig:circle}Circular constraint}{
		\includegraphics[width=0.4\textwidth]{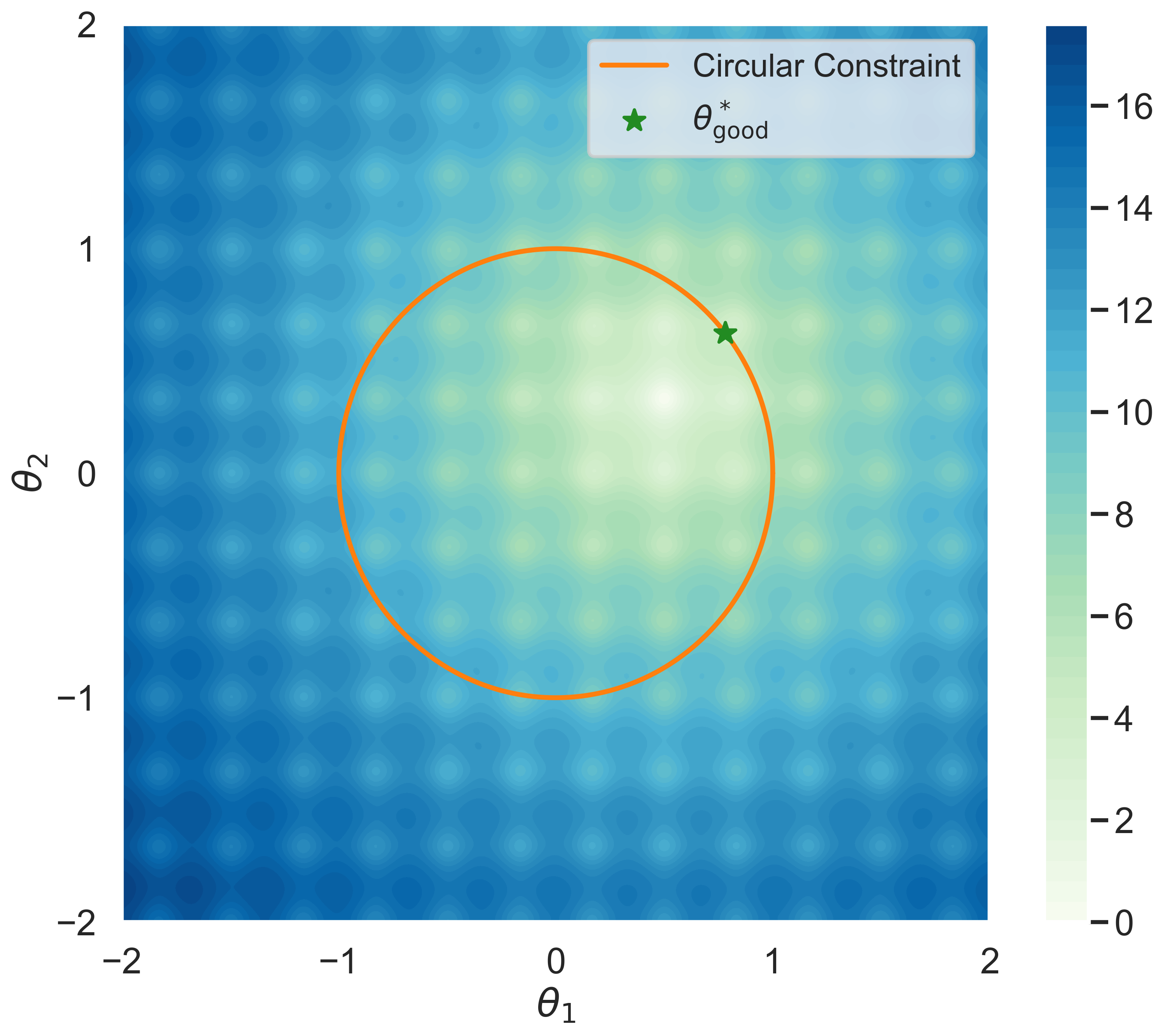}
	}
    \hspace{1em}
	\subcaptionbox{\label{subfig:star}Star constraint}{
		\includegraphics[width=0.4\textwidth]{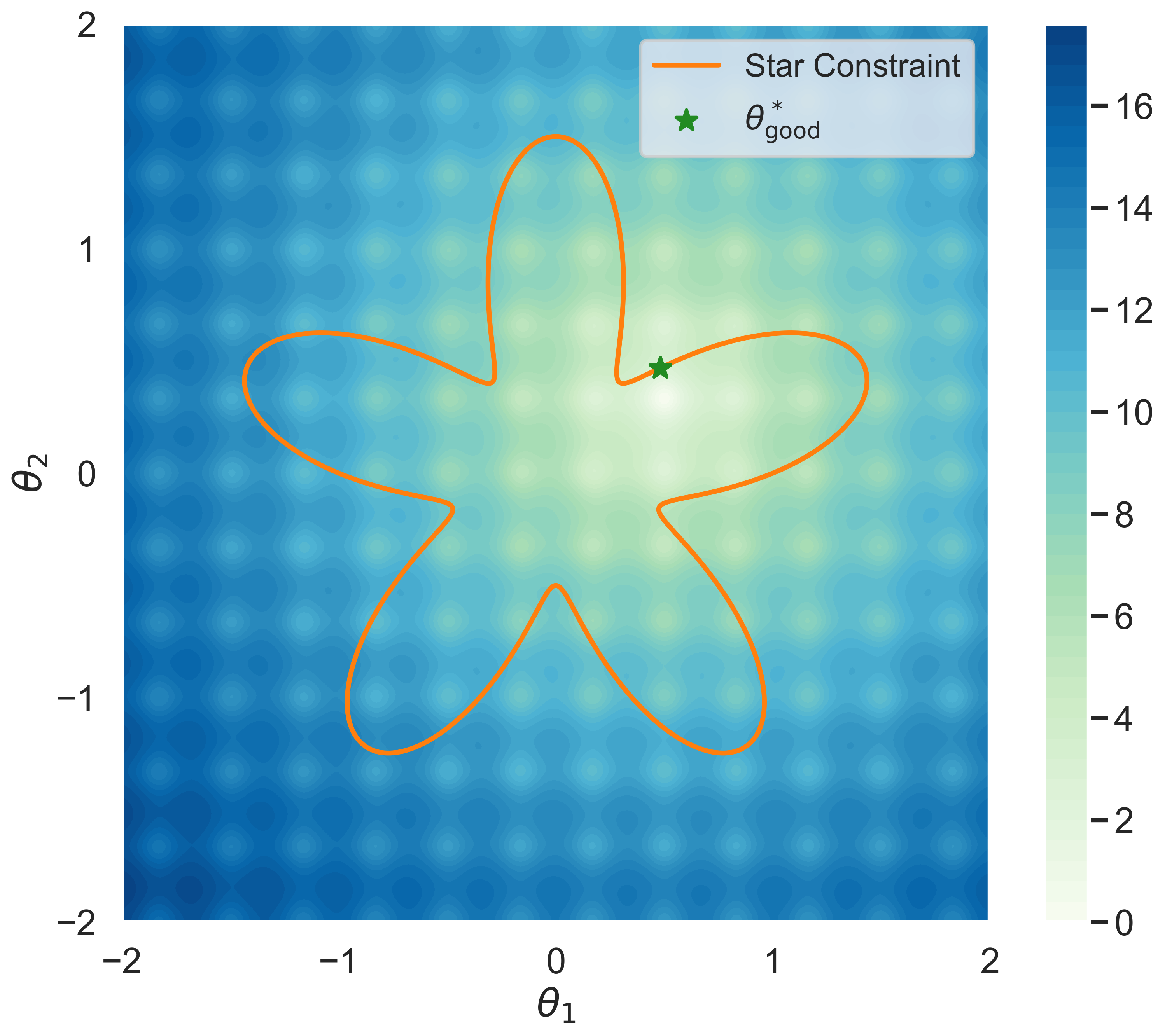}}
	\caption{2D Ackley function with two different constraint sets.}
\end{figure}
Thanks to the reformulation~\eqref{eq:reformulate_copt}, we obtain the equivalent lower-level objective functions $L$ for circular and star constraints given by
\begin{subequations}
\begin{align}
\text{Circular:} \quad L(\theta) &:= \left(\theta_1^2 + \theta_2^2 - 1 \right)^2, \\
\text{Star:} \quad L(\theta) &:= \left(\theta_1^2 + \theta_2^2 - (1 + 0.5 \sin (5\atantwo(\theta_2, \theta_1)))^2 \right)^2,
\end{align}
\end{subequations}
respectively.

\noindent \textbf{Baselines \& Implementations.}
We compare our CB\textsuperscript{2}O (Algorithm \ref{alg:cb2o_alg}) with four baseline methods: Penalized CBO \cite{carrillo2021consensus}, Adaptive Penalized CBO \cite{borghi2021constrained}, CBO with Gradient Force (GF) \cite{carrillo2024interacting} and Projected CBO \cite{fornasier2020consensus_hypersurface_wellposedness}.
The shared hyperparameters across all five algorithms are,
\[
\begin{aligned}
\lambda = 1, \qquad \sigma = 1, \qquad \alpha = 30,\qquad \Delta t = 0.01, \qquad  T = 300.
\end{aligned}
\]
The stopping criteria are set to $\estop = 0$ for the circular constraint and $\estop = 10^{-3}$ for the star constraint.
Let us now give the details about the baselines and their specific hyperparameters.
\begin{itemize}
    \item \textit{Penalized CBO.} The Penalized CBO method \cite{carrillo2021consensus} transforms the constrained optimization problem into an unconstrained problem by linearly combining the upper- and lower-level objective functions~$G$ and $L$ with a penalty parameter $\LM$, and then applies the standard CBO method \cite{pinnau2017consensus} on the resulting unconstrained problem.
    The reformulated objective is $\LM L + G$, where $\LM$ is a key hyperparameter.
    We set $\LM = 100$ for both the circular and the star constraint problems.
    \item \textit{Adaptive Penalized CBO.} The Adaptive Penalized CBO method \cite{borghi2021constrained} also applies the standard CBO to the unconstrained objective $\LM_k L + G$, yet dynamically adjusts $\LM_k$ at each iteration $k$ based on the constraint violations of the current particles.
    If the violation is below the tolerance threshold $1/\sqrt{\zeta_k}$, the tolerance decreases according to $\zeta_{k+1} = \eta_{\zeta} \zeta_k$ for some $\eta_{\zeta} > 1$. Otherwise, $\LM_k$ increases by a factor $\eta_{\LM} > 1$, i.e., $\LM_{k+1} = \eta_{\LM} \LM_k $.
    Please find the choices of hyperparameters in Appendix~\ref{sec:additional_copt}.
    
    \item \textit{CBO with GF.} CBO with gradient forcing solves the constrained optimization problem by adding a gradient force term $-\LM \nabla L$ to the standard CBO dynamics with $G$ as objective function.
    A large $\LM$ encourages particles to rapidly concentrate around the constraint set.
    We set $\LM = 100$ for both the circular and the star constraint problems.
    \item \textit{Projected CBO.} The Projected CBO methods modifies the standard CBO method \cite{pinnau2017consensus} by projecting the drift and diffusion terms onto the tangent space of the constraint set and adding an extra term to ensure the dynamics remains on the constraint set despite the Brownian motion.
    This method does not introduce additional hyperparameters.
    \item \textit{CB\textsuperscript{2}O.} In addition to shared hyperparameters, CB\textsuperscript{2}O introduces a key hyperparameter $\beta$.
    For this study, $\beta$ is set to be constant.
    Details of the chosen values are provided in Appendix~\ref{sec:additional_copt}.
\end{itemize}

\begin{remark}\label{rem:beta_too_small}
    In practical implementations of the CB\textsuperscript{2}O algorithm, the quantile parameter $\beta$ cannot be chosen too small,
    since the number of particles $N$ is finite.
    Specifically, if $\lceil \beta N \rceil \leq 1$, only one particle is selected at each time step to when computing the consensus point according to \eqref{eq:csp_finite_particles_discrete}. 
    In this case, the consensus point is simply at every time step the particle with the smallest value of the lower-level objective function $L$. 
    Consequently, CB\textsuperscript{2}O behaves in such case similarly to the standard CBO method \cite{pinnau2017consensus} with a large value of the hyperparameter $\alpha$, when $L$ is used as the objective function, neglecting completely the upper-level objective function $G$. 
    This hinders the algorithm's ability to converge to the target minimizer $\thetaG$.
    Therefore, in practical implementations of the CB\textsuperscript{2}O algorithm, we impose a lower bound on the hyperparameter $\beta$, denoted as $\betaMin$, which is the value such that $\lceil \beta N \rceil = 2$.
\end{remark}

\noindent \textbf{Performance Metrics.}
We evaluate the algorithm's performance using the averaged $\ell^2$-distance between the optimizer each algorithm finds and the target global minimizer $\thetaG$ over $100$ simulations, referred to as \textit{precision}.
Comparisons are conducted in two scenarios, where
(i) all algorithms use the same number of particles ($N = 100$), and where (ii) all algorithms run for approximately the same total time. 
For each algorithm, this requires to determine the maximum number of particles it can handle within a fixed runtime. 
Runtime is measured in seconds over $100$ simulations using MATLAB.

\subsubsection{Experimental Results}

We now report on the results of our numerical experiments.

\begin{itemize}
    \item \textbf{Circular constraint.}
    We first compare the algorithms using the same number of particles, i.e., setting (i).
    As shown in Table \ref{tab:circular_N100} and Figure \ref{subfig:circle_same_particles}, both CBO with Gradient Force and Projected CBO achieve the best precision but require relatively long running times.
    In contrast, the CB\textsuperscript{2}O algorithm demonstrates a good balance, maintaining strong performance while using significantly less computational time.
    The reduced computational cost of CB\textsuperscript{2}O is due to its use of only $\lceil \beta N \rceil$ particles to compute the consensus point at each iteration, which significantly reduces complexity, particularly for large particle counts.
    This efficiency enables CB\textsuperscript{2}O algorithm to utilize a greater number of particles under the same computational cost.
    
    To investigate further, we compare the algorithms under equivalent computational costs (running times), i.e., setting (ii).
    From Table \ref{tab:circular} and Figure \ref{subfig:circle_same_time},
    we observe that while CBO with Gradient Force and Projected CBO exhibit faster initial convergence, CB\textsuperscript{2}O ultimately achieves the best precision among all methods given the same computational budget.
    
    Notably, although Adaptive Penalized CBO shows a convergence pattern similar to CB\textsuperscript{2}O,
    the later consistently outperforms it in both precision and computational efficiency.
    Furthermore, Adaptive Penalized CBO is sensitive to its specific hyperparameters, such as $\eta_{\zeta}$ and $\zeta_{\chi}$, whereas CB\textsuperscript{2}O remains stable across a broader range of values of the hyperparameter $\beta$. 
    For a detailed discussion, please see Section \ref{subsec:COPT_ablation_study}.

\begin{table}[!htb]
\centering
\caption{Comparison of different algorithms for the two-dimensional Ackley function with a circular constraint, evaluated using the same number of particles.}
\label{tab:circular_N100}
\renewcommand\arraystretch{1.25}
\resizebox{0.95\textwidth}{!}{
    \begin{sc}
\begin{tabular}{cccc}
\toprule
    \bf Methods & \bf Number of particles & \bf Precision & \bf Running time (s)  \\
\midrule
Penalized CBO & $N = 100$ & $9.3\times10^{-3}$ & 3.49 \\
\midrule
    Adaptive Penalized CBO  & $N = 100$ & $5.1\times 10^{-3}$ & 2.23\\
\midrule
    CBO with GF & $N = 100$ & $\mathbf{1\times 10^{-3}}$ & 10.92\\
\midrule
    Projected CBO  & $N = 100$ & $1.4\times 10^{-3}$ & 6.12\\
\midrule
CB\textsuperscript{2}O & $N = 100$ & $\mathbf{4\times 10^{-3}}$ & $\mathbf{1.16}$ \\
\bottomrule
\end{tabular}
    \end{sc}}
\end{table}

\begin{table}[!htb]
\centering
\caption{Comparison of different algorithms for the two-dimensional Ackley function with a circular constraint, evaluated under approximately equal running times.}
\label{tab:circular}
\renewcommand\arraystretch{1.25}
\resizebox{0.95\textwidth}{!}{
    \begin{sc}
\begin{tabular}{cccc}
\toprule
    \bf Methods & \bf Number of particles & \bf Precision & \bf Running time (s) \\
    \midrule
    Penalized CBO & $N =500$  & $9.3\times 10^{-3}$ & $5.69$ \\
    \midrule
    Adaptive Penalized CBO & $N =500$0 & $2\times 10^{-3}$ & $6.43$\\
    \midrule
    CBO with GF & $N = 50$ & $1.4\times 10^{-3}$ & $6.18$\\
    \midrule
    Projected CBO  & $N = 80$ & $1.8\times 10^{-3}$ & $5.75$\\
    \midrule
    CB\textsuperscript{2}O  & $N = 1500$ & $\mathbf{1\times 10^{-3}}$ & $5.55$\\
\bottomrule
\end{tabular}
    \end{sc}
}
\end{table}

\begin{figure}[!htb]
	\centering
	\subcaptionbox{Setting (i) using the same number of particles $N=100$\label{subfig:circle_same_particles}}{
		\includegraphics[width=0.45\textwidth]{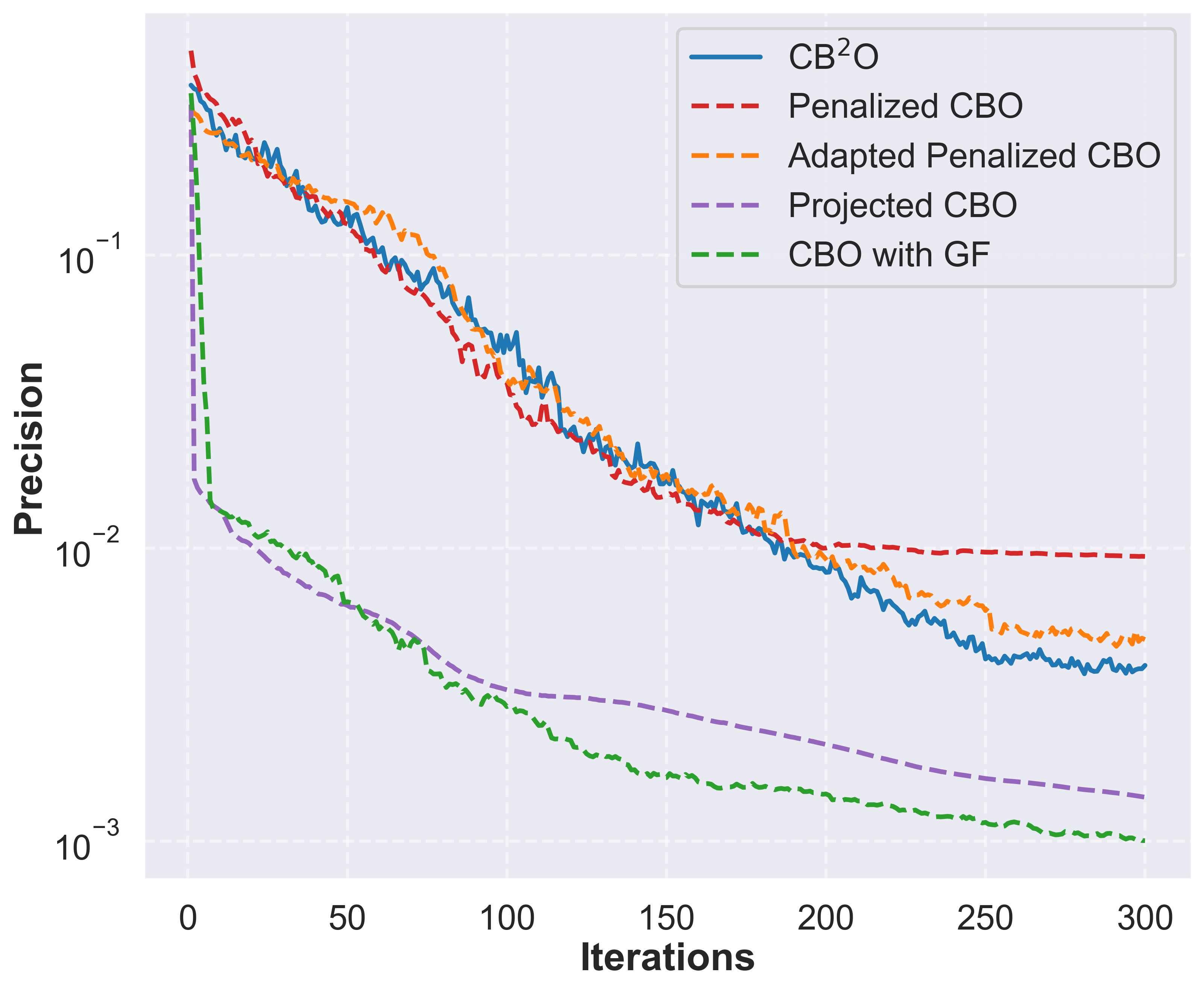}
	}
    \hspace{1em}
	\subcaptionbox{Setting (ii)  using approximately the same running time\label{subfig:circle_same_time}}{
		\includegraphics[width=0.45\textwidth]{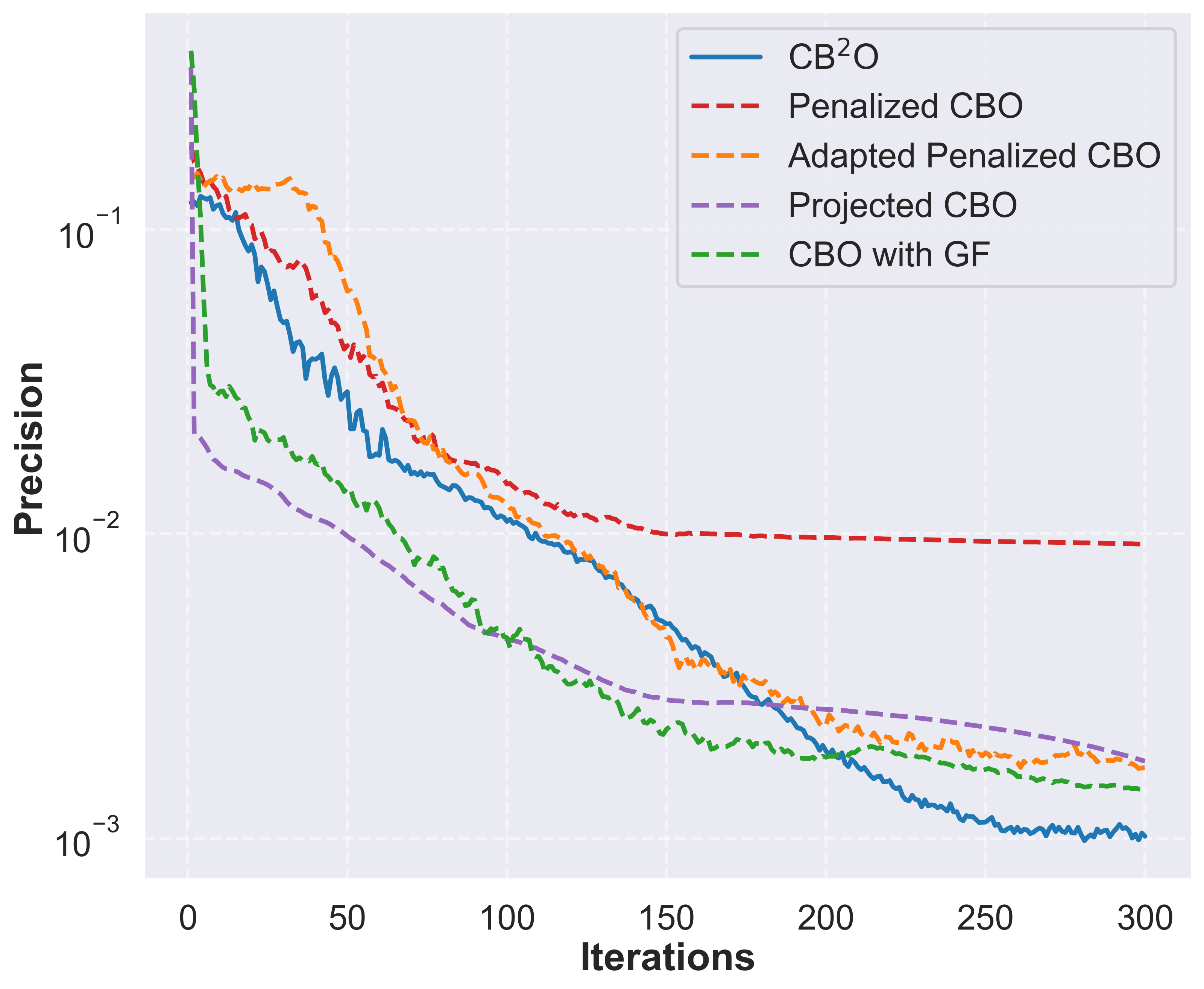}
	}	
	\caption{Comparison of different algorithms for the two-dimensional Ackley function with a circular constraint in settings (i) and (ii), respectively.}
	\label{fig:circular}
\end{figure}

\item \textbf{Star constraint.} Similar to the circular constraint case, we compare the performance of different algorithms using the same number of particles (setting (i)) and under approximately the same running time (setting (ii)).  
Due to the complexity of the star constraint, applying projection is not straightforward, so we exclude Projected CBO from this comparison.
From Table \ref{tab:star_N100} and Figure \ref{subfig:star_same_particles}, we observe that even using the same number of particles, CB\textsuperscript{2}O already outperforms most of the baselines, and has comparable performances as Penalized CBO meanwhile having less computational cost.
Furthermore, when  CB\textsuperscript{2}O is allowed to use more particles but maintaining the same computational cost as others, as shown in Table \ref{tab:star} and Figure \ref{subfig:star_same_time}, it consistently demonstrates superior performance in both precision and efficiency.
Notably, CBO with Gradient Force requires significantly longer running times to achieve reasonable precision and behaves highly unstable compared to CB\textsuperscript{2}O. 
This is primarily due to the computationally expensive matrix inversion step in the CBO with Gradient Force algorithm and the greatly nonconvex nature of the star constraint. 
These results underscore the advantages of CB\textsuperscript{2}O in handling nonconvex constraint sets effectively.

\begin{table}[!htb]
\centering
\caption{Comparison of different algorithms for the two-dimensional Ackley function with a star constraint, evaluated using the same number of particles.}
\label{tab:star_N100}
\renewcommand\arraystretch{1.25}
\resizebox{0.95\textwidth}{!}{
    \begin{sc}
\begin{tabular}{cccc}
\toprule
    \bf Methods & \bf Number of particles & \bf Precision & \bf Running time (s) \\
    \midrule
    Penalized CBO & $N =100$  & $8.7 \times 10^{-3}$ & $1.89$ \\
    \midrule
    Adaptive Penalized CBO & $N = 100$ & $11.2 \times 10^{-3}$ & $2.17$ \\
    \midrule
    CBO with GF & $N = 100$ & $10 \times 10^{-3}$ & $2.41 \times 10^4$ \\
    \midrule
    CB\textsuperscript{2}O  & $N = 100$ & $\mathbf{8 \times 10^{-3}}$ & $\mathbf{1.29}$ \\
\bottomrule
\end{tabular}
    \end{sc}
}
\end{table}

\begin{table}[!htb]
\centering
\caption{Comparison of different algorithms for the two-dimensional Ackley function with a star constraint, evaluated under approximately the same running time.}
\label{tab:star}
\renewcommand\arraystretch{1.25}
\resizebox{0.95\textwidth}{!}{
    \begin{sc}
\begin{tabular}{cccc}
\toprule
    \bf Methods & \bf Number of particles & \bf Precision & \bf Running time (s) \\
    \midrule
    Penalized CBO & $N =500$  & $6.1 \times 10^{-3}$ & $6.26$ \\
    \midrule
    Adaptive Penalized CBO & $N = 500$ & $4.4 \times 10^{-3}$ & $7.03$ \\
    \midrule
    CBO with GF & $N =500$ & $20.9 \times 10^{-3}$ & $1.17 \times 10^3$ \\
    \midrule
    CB\textsuperscript{2}O  & $N = 2000$ & $\mathbf{3.2 \times 10^{-3}}$ & $\mathbf{5.09}$ \\
\bottomrule
\end{tabular}
    \end{sc}
}
\end{table}

\begin{figure}[!htb]
	\centering
	\subcaptionbox{Setting (i) using the same number of particles $N=100$\label{subfig:star_same_particles}}{
		\includegraphics[width=0.45\textwidth]{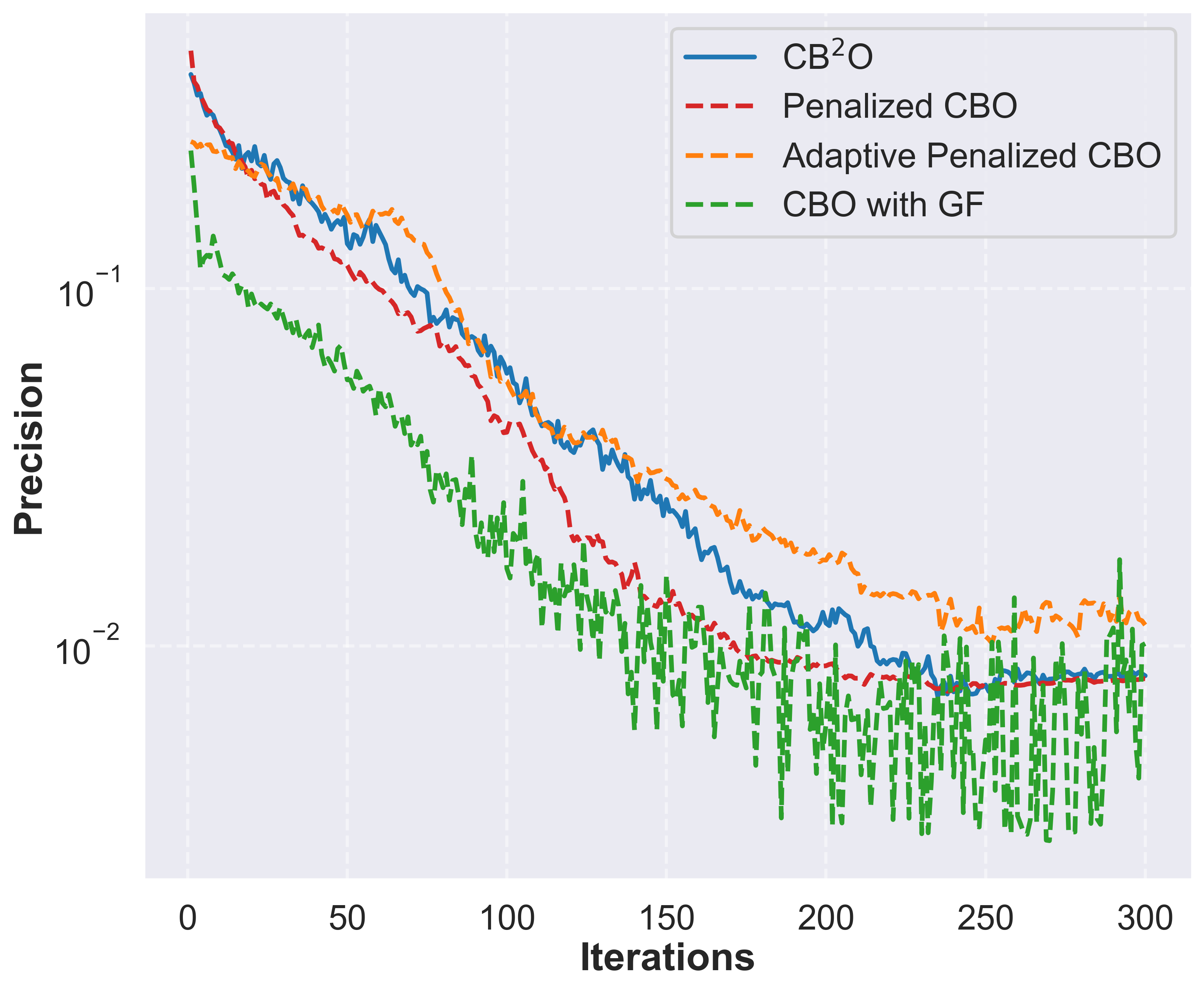}
	}
    \hspace{1em}
	\subcaptionbox{Setting (ii)  using approximately the same running time\label{subfig:star_same_time}}{
		\includegraphics[width=0.45\textwidth]{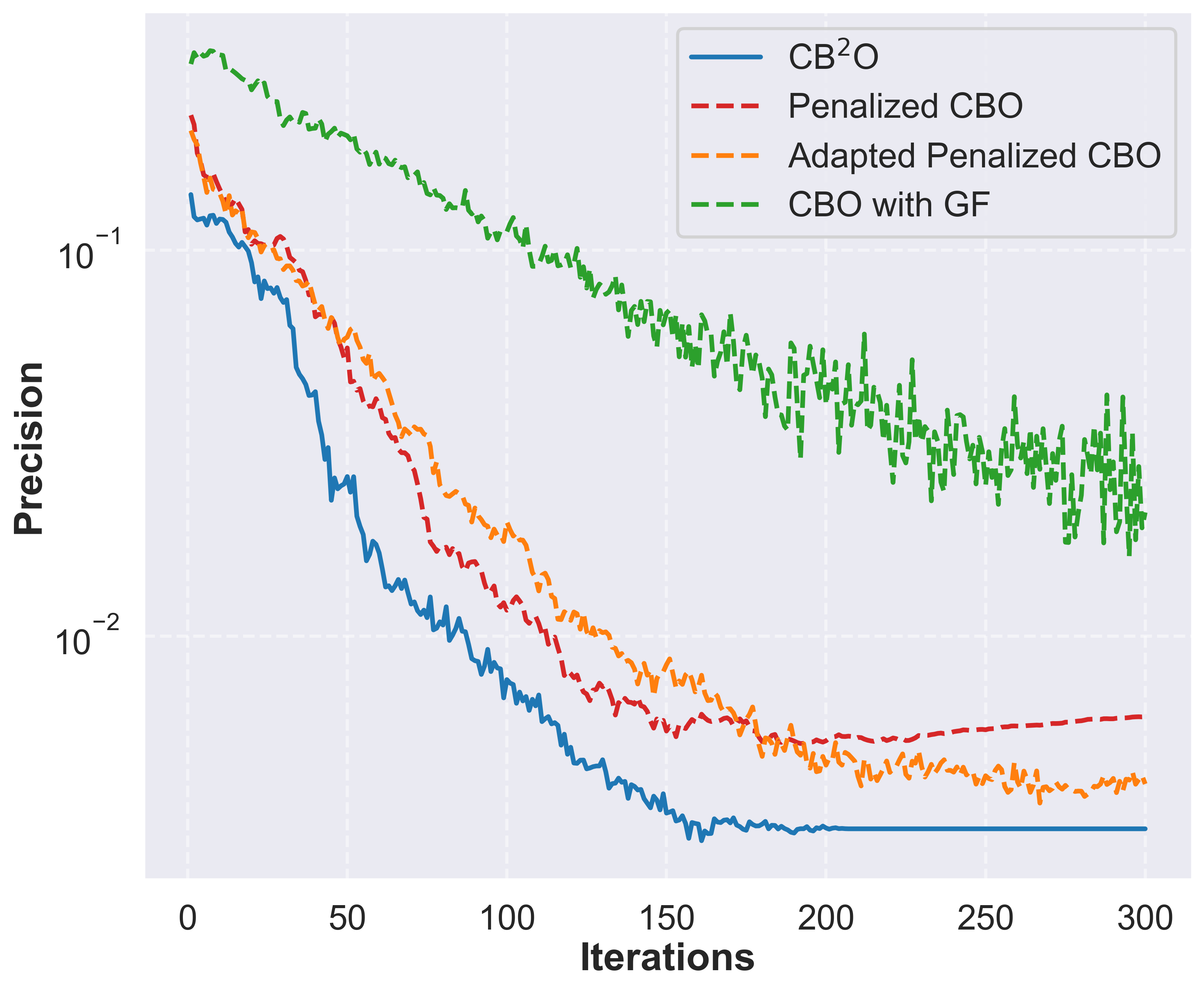}
	}	
	\caption{Comparison of different algorithms for the two-dimensional Ackley function with a star constraint in settings (i) and (ii), respectively.}
	\label{fig:star}
\end{figure}

\end{itemize}

\begin{remark}
From the two experiments above, we observe that CB\textsuperscript{2}O is computationally efficient, as it computes the consensus point at each iteration using only a subset of particles. 
However, this approach also means that the unselected particles is not fully utilized, which generally necessitates using a larger total number of particles to achieve good performance.
This observation motivates further exploration into strategies for accelerating  CB\textsuperscript{2}O by better leveraging the unselected particles, which we leave for future work.
\end{remark}

\subsubsection{The Influence of Hyperparameters}\label{subsec:COPT_ablation_study}
In this section, we investigate more deeply the influence of the hyperparameters $N$, $\beta$, $\alpha$, and $\estop$ on the performance of the CB\textsuperscript{2}O algorithm in terms of both precision and running time through an ablation study, see Figure~\ref{fig:eq2_distance}.
Specifically, in the experiments below, we vary one or two hyperparameters at a time while keeping all others fixed.
Detailed information on the hyperparameter settings are provided in Appendix~\ref{sec:additional_copt}.\\

\noindent \textbf{Ablation Study: Influence of the Hyperparameters on Precision.}
\begin{itemize}
\item \textit{Impact of the particle number $N$.}
In Figure \ref{subfig:different_N_distance}, we observe that increasing the number of particles $N$ from $10^2$ to $10^3$ improves the algorithm's performance. 
However, the precision plateaus as $N$ increases further from $10^3$ to $10^5$. 
This phenomenon can be understood through theoretical intuition by considering a hypothetical ``mean-field'' version of the CB\textsuperscript{2}O algorithm with the same hyperparameters as in Figure \ref{subfig:different_N_distance} but with infinitely many particles.
This mean-field algorithm achieves a fixed precision determined by the choice of hyperparameters.
For the finite particle CB\textsuperscript{2}O algorithm, its precision can be decomposed into two parts. First, the mean-field approximation error (how closely the finite particle algorithm approximates the mean-field algorithm) and second the intrinsic precision of the mean-field algorithm.
When $N$ is small, the finite particle algorithm's performance is limited by the mean-field approximation error. 
As $N$ increases, the algorithm's precision approaches the one of the mean-field algorithm. 
While our paper does not formally justify the mean-field approximation property, these observations align with theoretical results in the CBO literature \cite{huang2021MFLCBO,fornasier2021consensus,gerber2023mean}.

\item \textit{Impact of the quantile parameter $\beta$.}
From Figure \ref{subfig:different_beta_distance}, we make the following three observations.
First, for a fixed number of particles $N=1000$ and other hyperparameters held constant, the optimal precision occurs for $\beta \in [0.01, 0.05]$. 
Furthermore, selecting $\beta$ within the range $[0.002, 0.1]$ results in relatively good algorithm performance, indicating that the CB\textsuperscript{2}O algorithm is relatively insensitive to the choices of $\beta$. 
Second, choosing $\beta$ too large causes the CB\textsuperscript{2}O algorithm to fail in finding the target global minimizer~$\thetaG$.
This occurs because, with large $\beta$, the quantile set $\Qbeta{\dummy}$  may include the regions with high loss in terms of $L$ but low values of $G$, therefore causing the algorithm to move towards the global minimizer of $G$ but not the desired minimizer $\thetaG$.
This behavior is consistent with the theoretical results discussed in Remark~\ref{remark:beta0}.
Third, and conversely, choosing $\beta$ being too small degrades the algorithm performance, as discussed in Remark~\ref{rem:beta_too_small}.
However, for a fixed number of particles $N$, selecting $\beta$ such that $\lceil \beta N \rceil = 2$ already results in good performance.
These observations demonstrate that the choice of the quantile parameter 
$\beta$ in CB\textsuperscript{2}O is not restrictive, offering flexibility without compromising performance.

We repeat the experiments for varying values of $\beta$ for different numbers of particles, namely $N = 500$ and $N=2000$, respectively. 
The three patterns discussed earlier can be observed in analogous form in Figures \ref{subfig:beta_sameN_500_distance} and \ref{subfig:beta_sameN_2000_distance}.

\item \textit{Joint impact of $N$ and $\beta$.} Figure \ref{subfig:N_and_beta_distance} shows that as the number of particles $N$ increases and the quantile parameter $\beta$ decreases while keeping $\beta N = 50$ constant, the precision of CB\textsuperscript{2}O algorithm improves consistently.
Combined with the observation from Figure \ref{subfig:different_N_distance} that the precision plateaus when N exceeds $10^3$ for a fixed $\beta$, it suggests that smaller $\beta$ improves the algorithm performance in the ``mean-field'' regime.
These findings align with the theoretical perspective discussed in Remark~\ref{remark:beta0}.

\item \textit{Impact of the temperature parameter $\alpha$.} As shown in Figure \ref{subfig:different_alpha_distance}, the algorithm's  precision consistently improves with increasing $\alpha$.
This aligns with the theoretical perspective discussed in Remark~\ref{remark:alpha0}, as well as other CBO literature \cite{carrillo2018analytical,fornasier2021convergence}. 

\item \textit{Impact of the stopping criterion $\estop$.}
The choice of stopping criterion $\estop$ has a significant impact on the algorithm's precision, as illustrated in Figure \ref{subfig:different_estop_distance}. 
A smaller $\estop$ requires the algorithm to run for more iterations before stopping, resulting in improved performance.
While selecting a small $\estop$ is ideal for achieving high precision, this hyperparameter can be adjusted to balance precision and running time when the algorithm is operating under computational cost limitations. 
Further discussions on the influence of $\estop$ on the running time are provided below.

\begin{figure}[!htb]
	\centering
\subcaptionbox{\footnotesize Number of particles $N$\label{subfig:different_N_distance}}{\includegraphics[height=0.16\textheight, keepaspectratio=true]{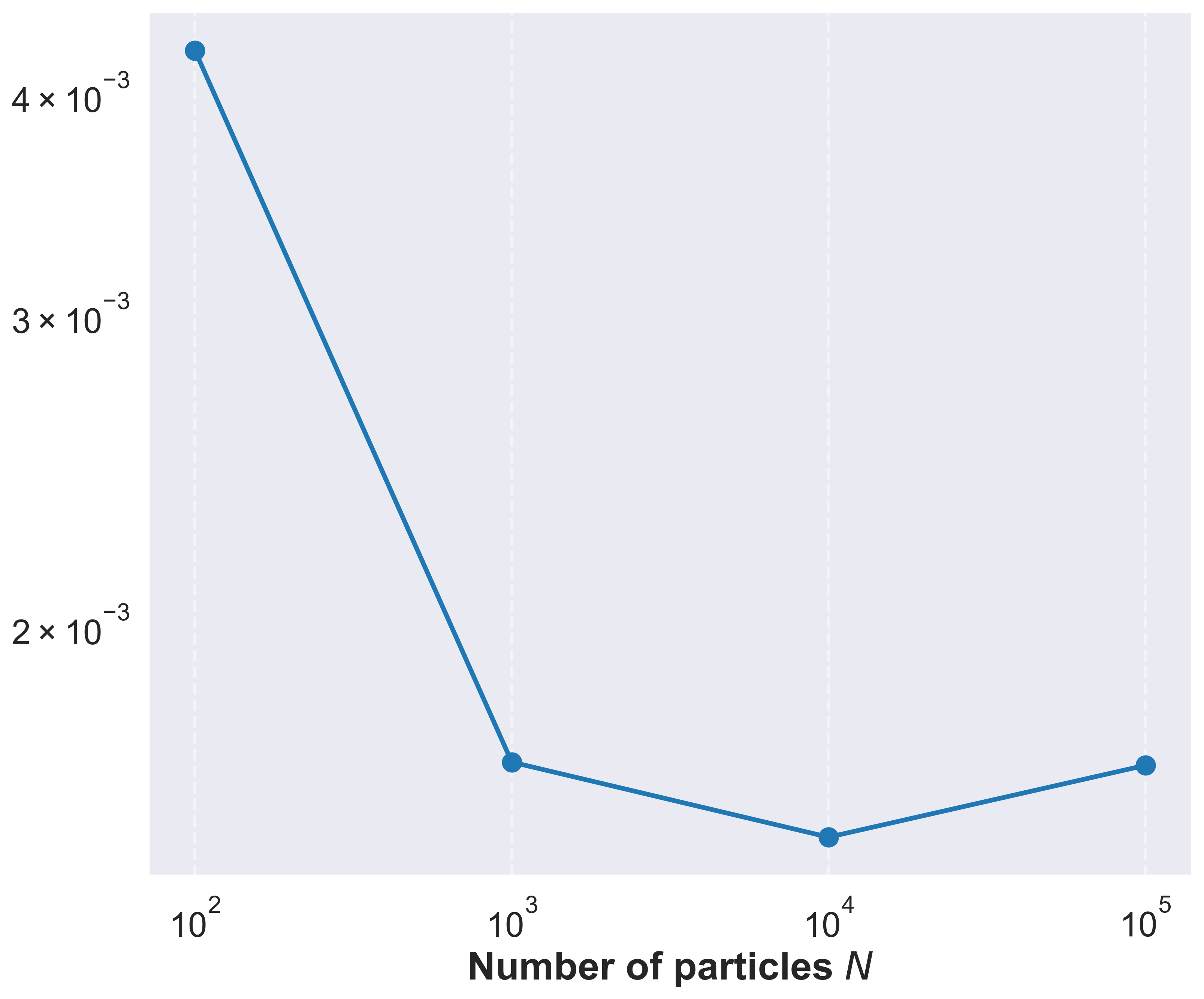}
}
\hspace{1em}
\subcaptionbox{\footnotesize Quantile parameter $\beta$ ($N=1000$)\label{subfig:different_beta_distance}}{
\includegraphics[height=0.16\textheight, keepaspectratio=true]{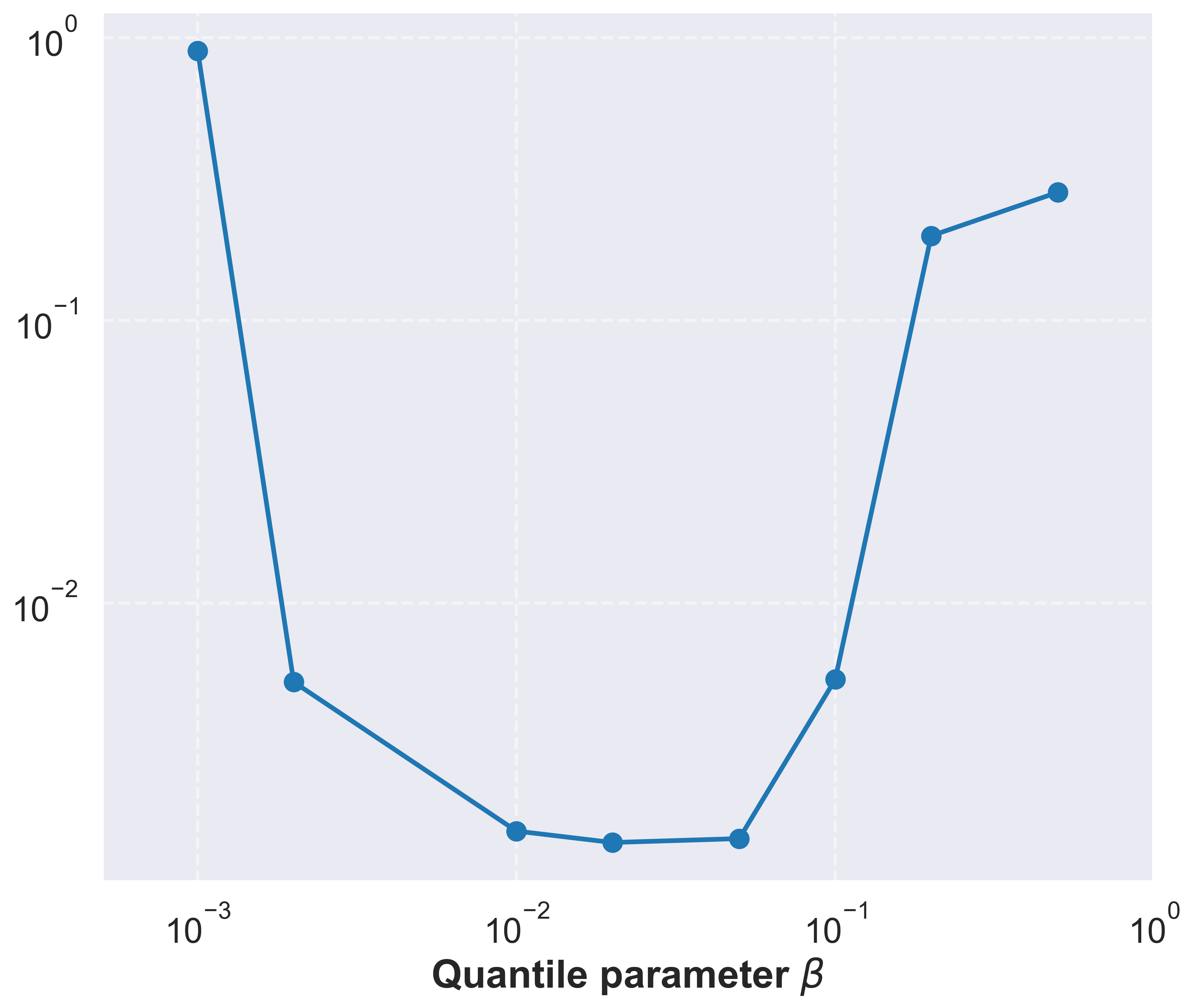}
}
\vspace{1em}

\subcaptionbox{\footnotesize $(N, \beta)$-pairs with fixed $\beta N$\label{subfig:N_and_beta_distance}}{
\includegraphics[height=0.16\textheight, keepaspectratio=true]{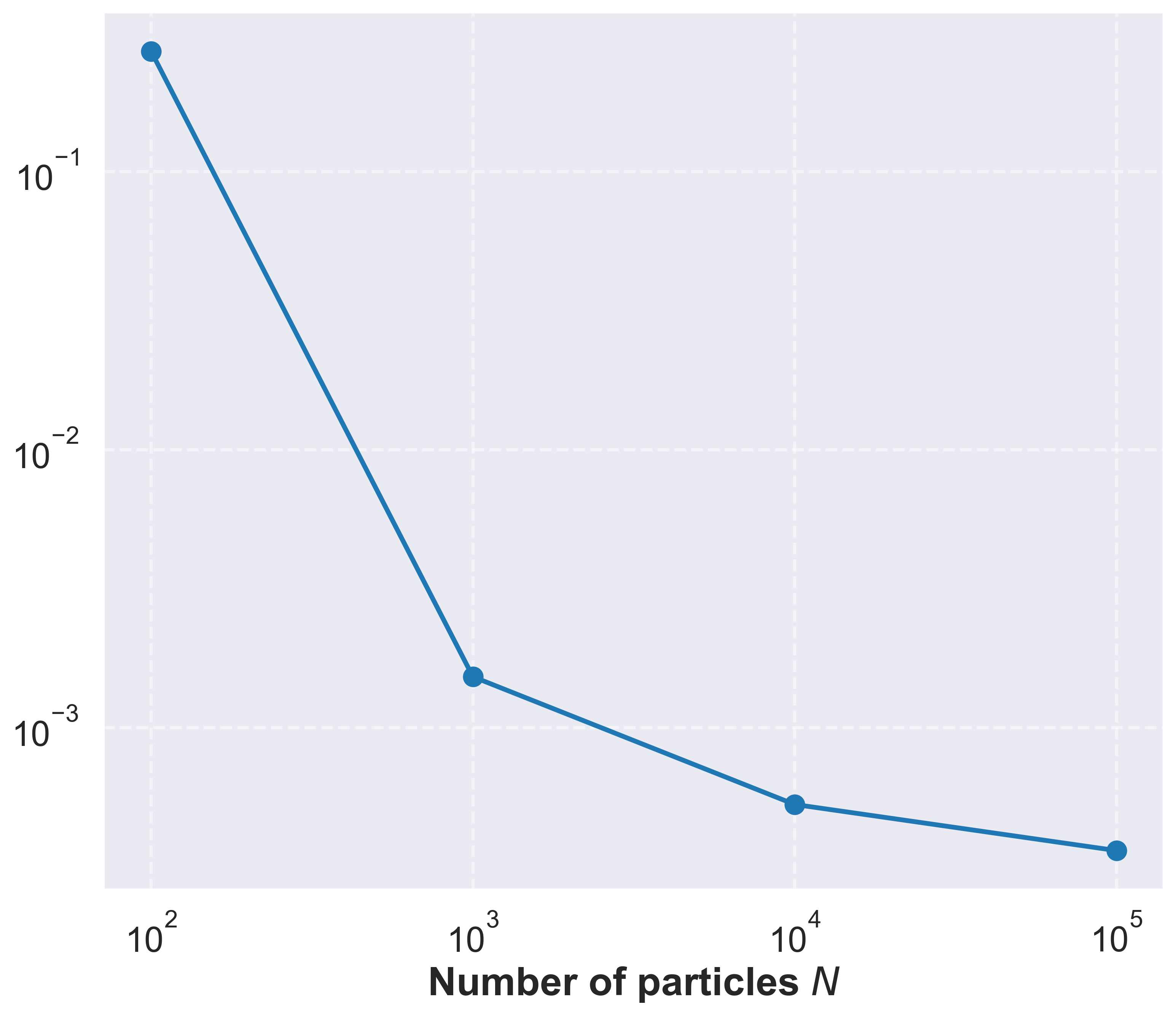}
} 
\hspace{1em}
\subcaptionbox{\footnotesize Temperature parameter $\alpha$\label{subfig:different_alpha_distance}}{
\includegraphics[height=0.16\textheight, keepaspectratio=true]{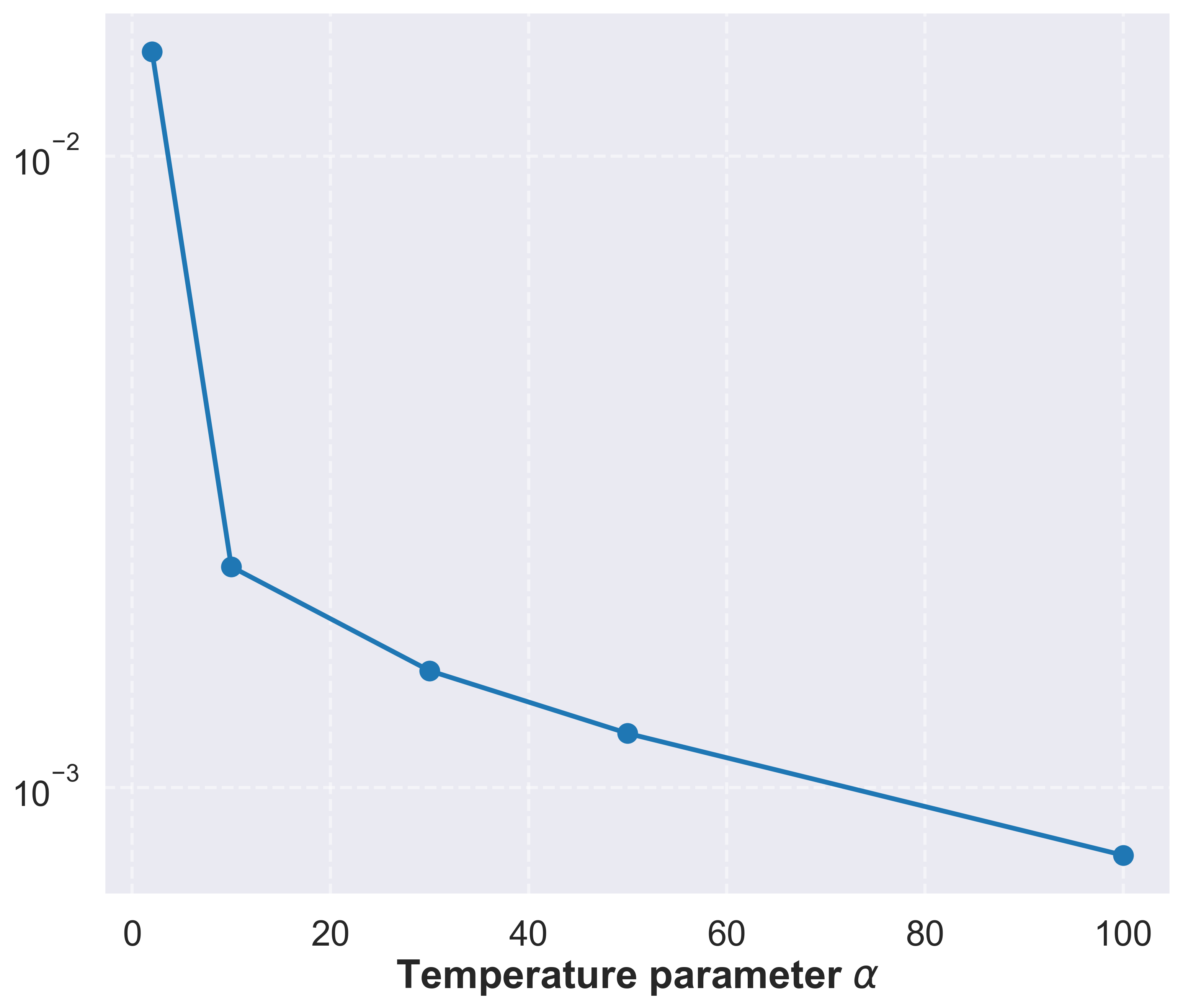}
}
\hspace{1em}
\subcaptionbox{\footnotesize Stopping criteria $\estop$\label{subfig:different_estop_distance}}{
\includegraphics[height=0.16\textheight, keepaspectratio=true]{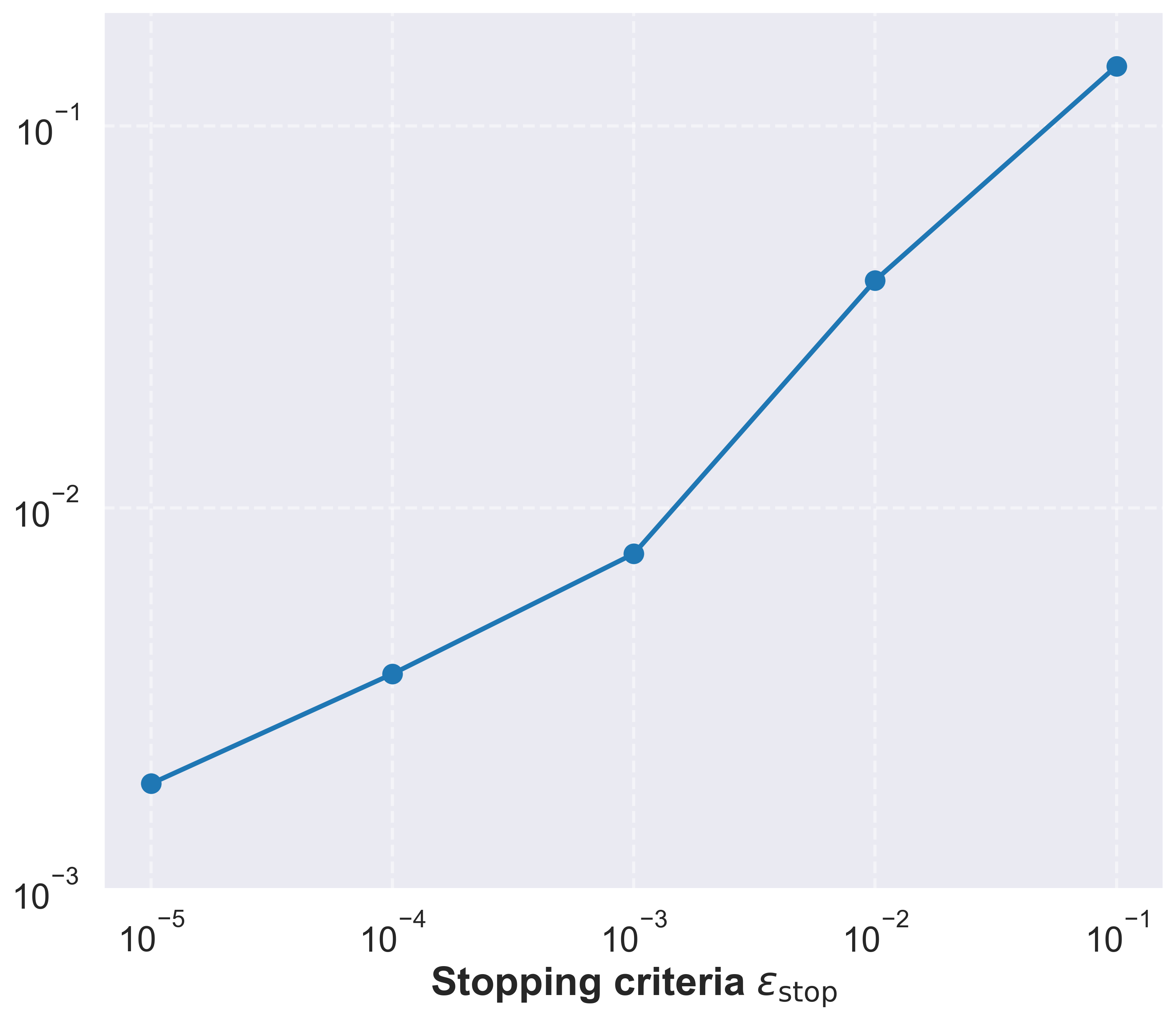}
}
	\caption{The influence of different hyperparameters on the performance of the CB\textsuperscript{2}O algorithm.
    The vertical axis in each plot corresponds to the {precision} of the algorithm, i.e., the averaged $\ell^2$-distance between the output of the algorithm and the target global minimizer $\thetaG$ averaged over $100$ simulations.}
	\label{fig:eq2_distance}
\end{figure}
\end{itemize}

\begin{figure}[!htb]
	\centering
	\subcaptionbox{\footnotesize Quantile parameter $\beta$ ($N=500$)\label{subfig:beta_sameN_500_distance}}{
		\includegraphics[height=0.16\textheight,keepaspectratio=true]{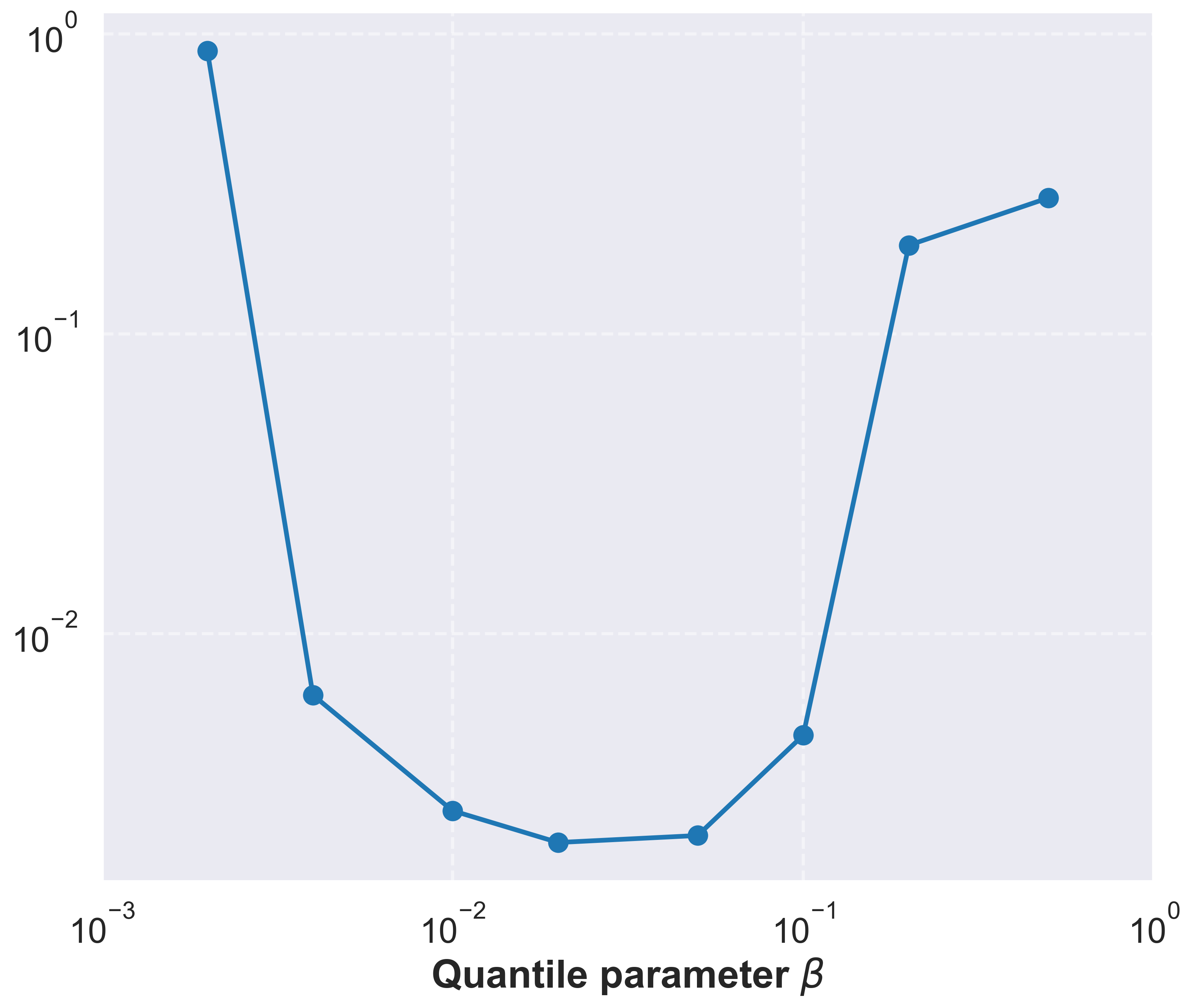}
	}
    \hspace{1em}
	\subcaptionbox{\footnotesize Quantile parameter $\beta$ ($N=2000$)\label{subfig:beta_sameN_2000_distance}}{
		\includegraphics[height=0.16\textheight,keepaspectratio=true]{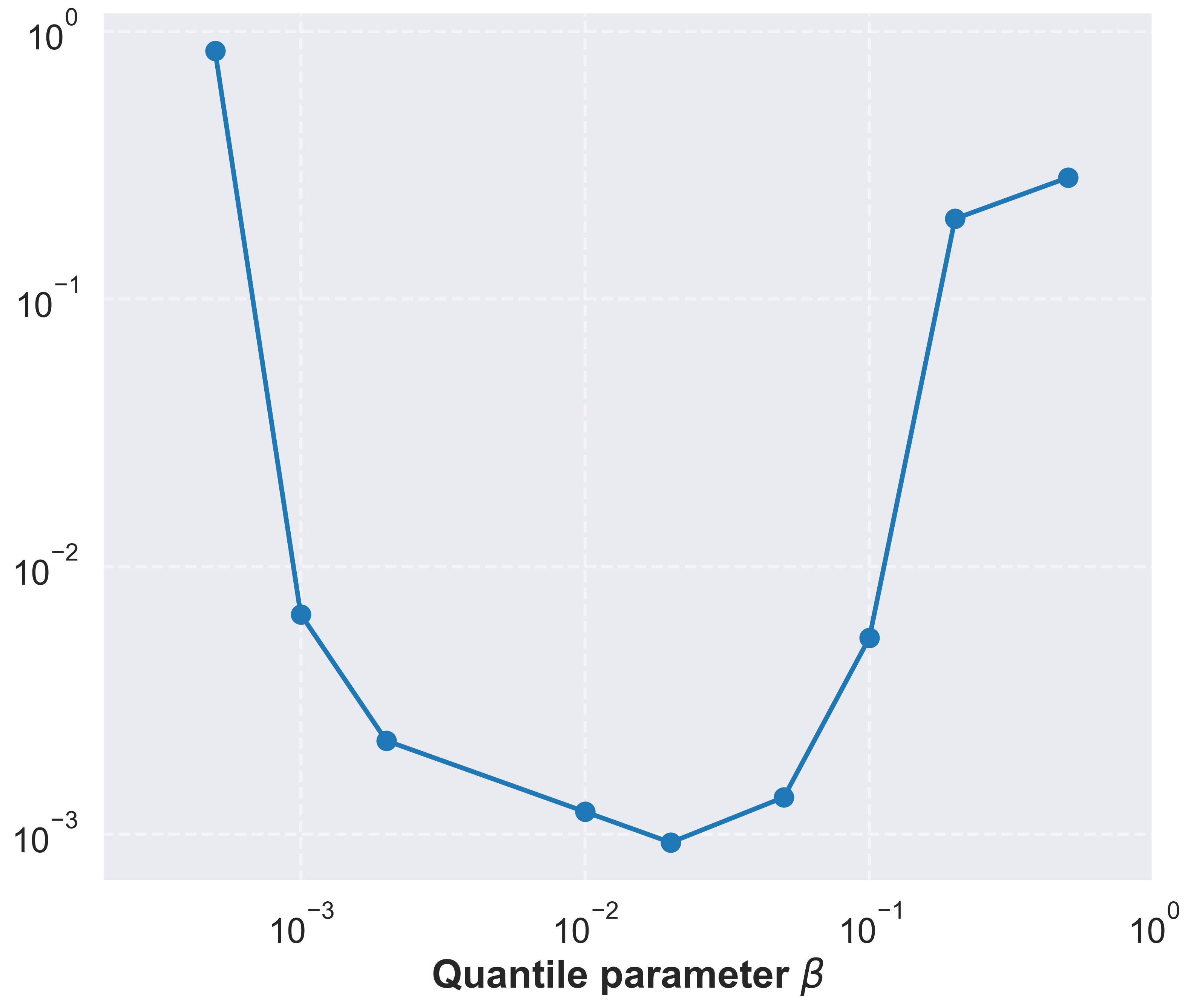}
	}	
	\caption{The influence of the hyperparameter $\beta$ on the precision of the CB\textsuperscript{2}O algorithm for different numbers of particles $N$ averaged over $100$ simulations.}
	\label{fig:beta_sameN_500_and_2000}
\end{figure}

\noindent \textbf{Influence of the Hyperparameters on the Running Time.} 
Figures \ref{subfig:different_N_runningtime}, \ref{subfig:different_beta_runningtime} and \ref{subfig:N_and_beta_runningtime} reveal that the running time of the CB\textsuperscript{2}O algorithm is predominately influenced by the total number of particles $N$, and the quantile $\beta$ affects the constant factor in the running time scaling. 
Specifically, $\beta$ determines the proportion of particles used in each iteration, impacting the computational overhead but not the overall scaling trend.
Figures \ref{subfig:different_alpha_runningtime} and \ref{subfig:different_estop_runningtime} demonstrate that the stopping criterion $\estop$ significantly affects the running time, as stricter criteria require more iterations for convergence. 
In contrast, the temperature parameter $\alpha$ has minimal impact on the overall running time. \\

\begin{figure}[!htb]
	\centering
	\subcaptionbox{\footnotesize Number of particles $N$\label{subfig:different_N_runningtime}}{
		\includegraphics[height=0.16\textheight, keepaspectratio=true]{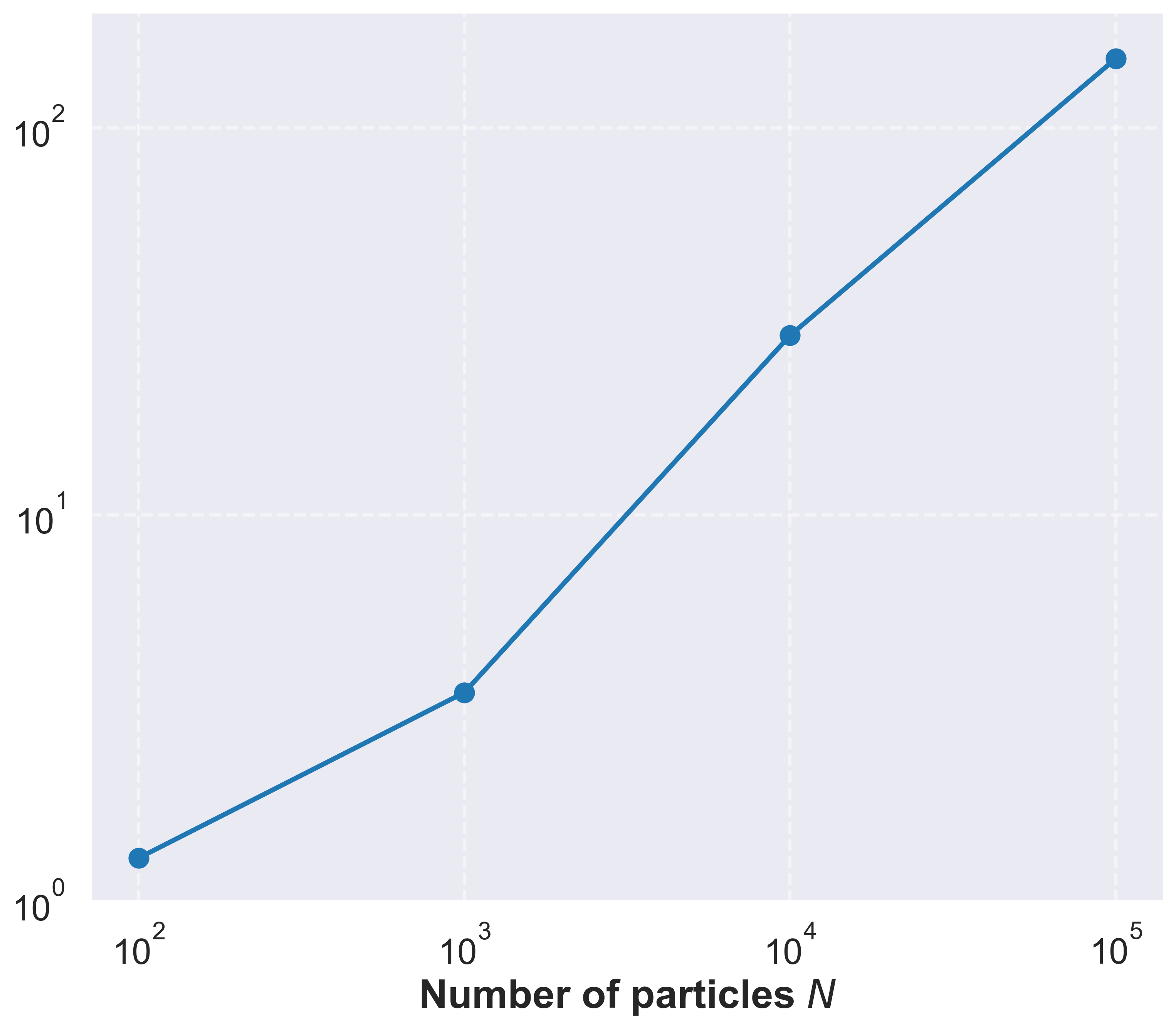}
	}
    \hspace{1em}
	\subcaptionbox{\footnotesize Quantile parameter $\beta$\label{subfig:different_beta_runningtime}}{
		\includegraphics[height=0.16\textheight, keepaspectratio=true]{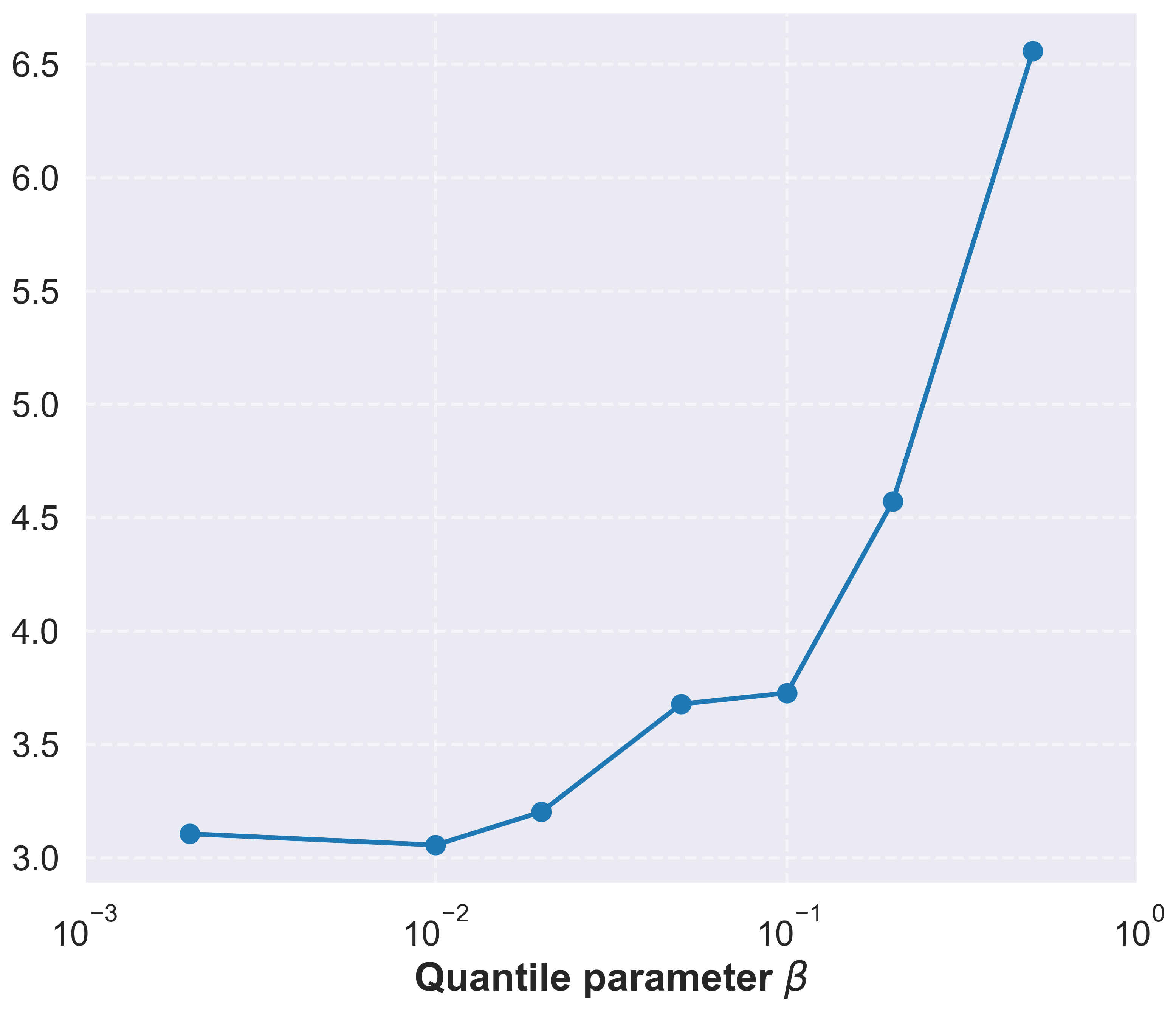}
	}
    \vspace{1em}
    
    \subcaptionbox{\footnotesize $(N, \beta)$-pairs with fixed $\beta N$\label{subfig:N_and_beta_runningtime}}{
		\includegraphics[height=0.16\textheight, keepaspectratio=true]{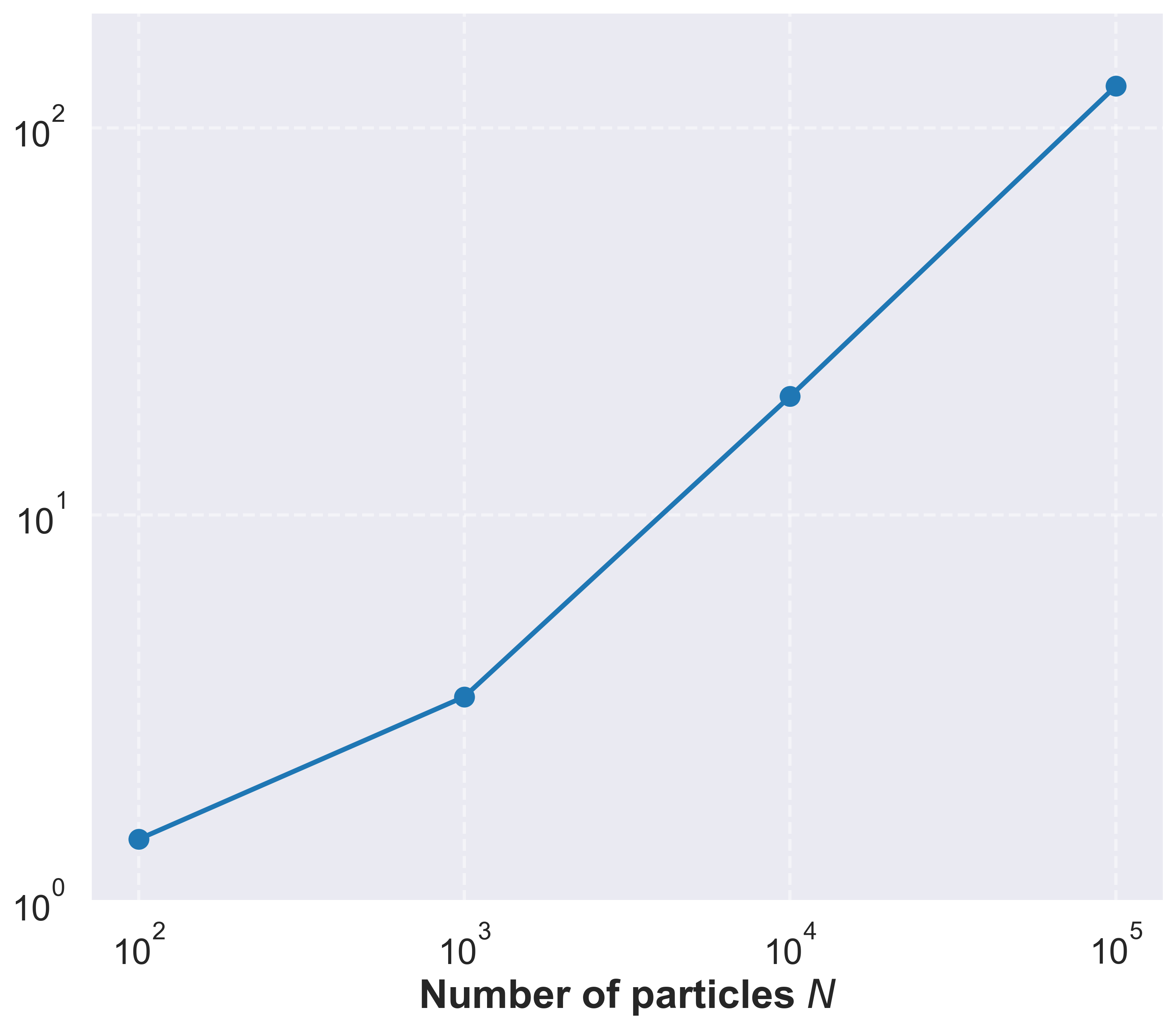}
	}
    \hspace{1em}
	\subcaptionbox{\footnotesize Temperature parameter $\alpha$\label{subfig:different_alpha_runningtime}}{
		\includegraphics[height=0.16\textheight, keepaspectratio=true]{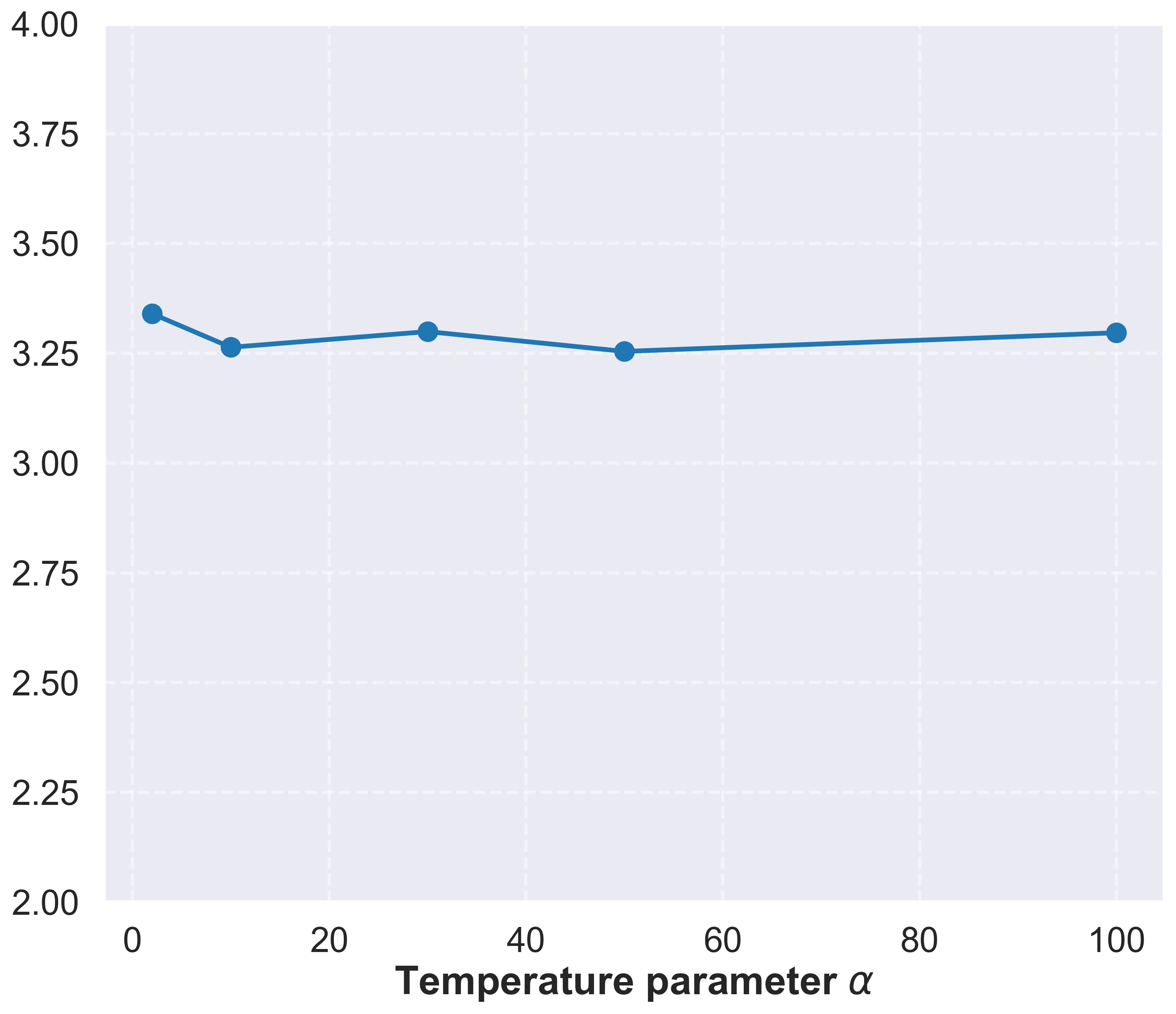}
	}
    \hspace{1em}
	\subcaptionbox{\footnotesize Stopping criteria $\estop$\label{subfig:different_estop_runningtime}}{
		\includegraphics[height=0.16\textheight, keepaspectratio=true]{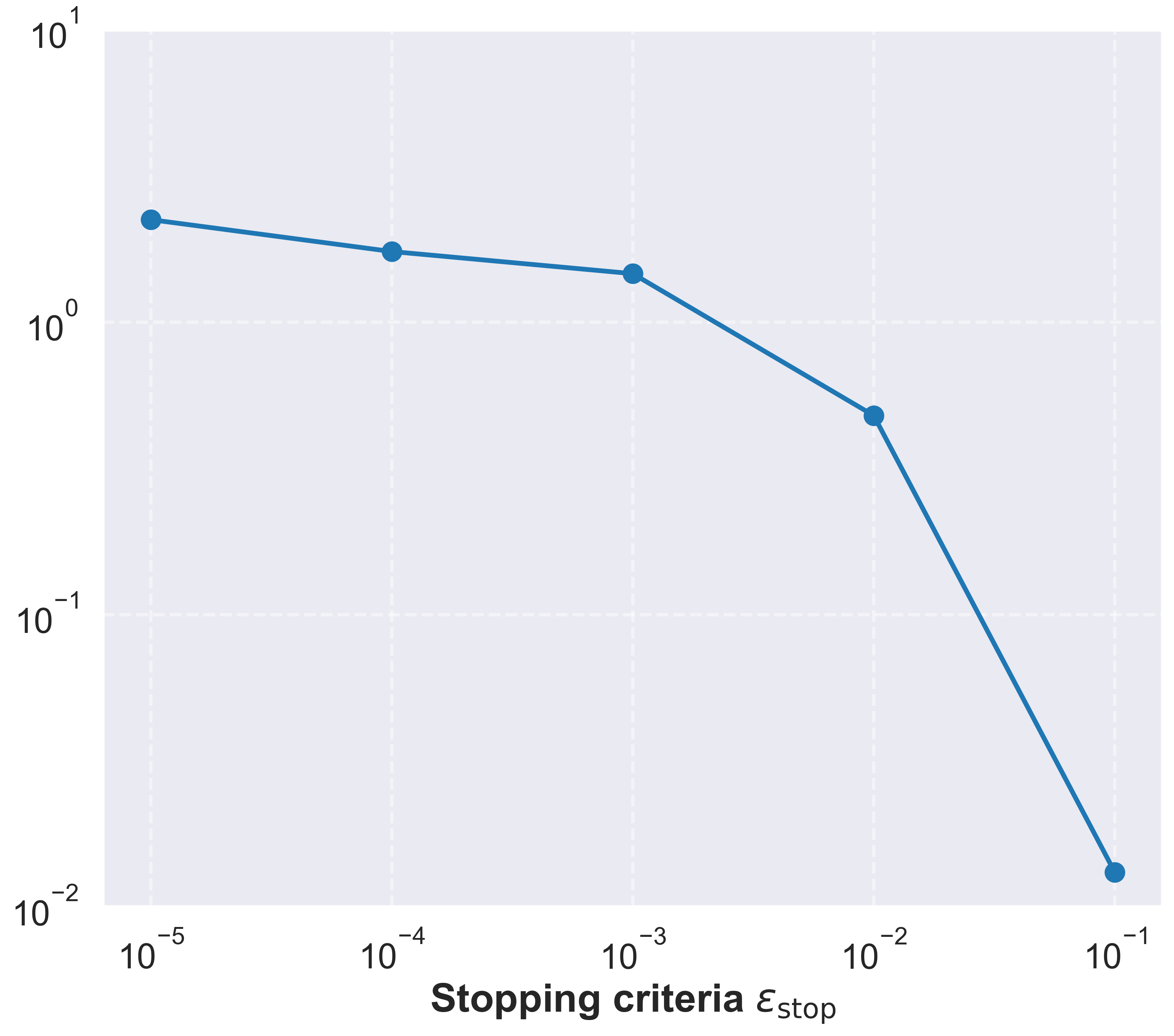}
	}
	\caption{The influence of different hyperparameters on the running time of the CB\textsuperscript{2}O algorithm.
    The vertical axis in each plot corresponds to the algorithm's total running time (seconds) averaged over $100$ simulations.}
	\label{fig:eq2_runningtime}
\end{figure}

\noindent In summary, we can draw the following three general conclusions.
First, the quantile parameter $\beta$ plays a pivotal role in determining the performance of the CB\textsuperscript{2}O method for a fixed number of particles $N$. 
The results demonstrate that the algorithm is relatively stable and robust across a reasonably large range of $\beta$ values for every fixed number of particles $N$.
Second, the number of particles $N$ and the stopping criterion $\estop$ can be adjusted to achieve a balance between the algorithm's performance and computational efficiency, offering flexibility to run the algorithm under different computational constraints.
And, third, in this two-dimensional example, the temperature parameter $\alpha$ can be set rather large provided that the number of particles $N$ is not too small. 
However, setting $\alpha = \infty$ is generally not ideal in practice, as elaborated on in Remark~\ref{rem:choicealphafinite}.

\subsection{Sparse Representation Learning}\label{sec:numerics_SRL}

In this section, we consider the task of sparse representation learning using neural networks \cite{gong2021Biojective}.
Specifically, we want to learn sparse feature representations on a supervised dataset $\CD$ that contains feature-label pairs $(x,y) \in \R^n \times \R$, while maintaining high prediction accuracy. 
Training a neural network that admits a sparse representation as well as high prediction accuracy for this task 
can be formulated as a bi-level optimization problem of form~\eqref{eq:bilevel_opt}.
Encoding the prediction accuracy in the lower-level objective function $L$
and promoting sparsity of the feature representations through the upper-level function $G$,
we can formalize the sparse representation learning task as 
\begin{equation}
    \label{eq:opt_srl}
    \min_{\theta^*\in \Theta} G(\theta^*) := \E_{(x,y) \in \CD} \left[ \N{h_{\theta^*} (x)}_p^p \right]
    \quad \text{s.t.\@}\quad
    \theta^* \in \Theta := \argmin_{\theta \in \R^d} L(\theta) := \E_{(x,y) \in \CD} \left[ \ell \left(\phi_{\theta} (h_{\theta} (x)), y\right) \right].
\end{equation}
Here, $h_{\theta}: \R^n \mapsto \R^m$ denotes the feature representation map, and $\phi_{\theta}: \R^m \mapsto \R$ the prediction head.
For instance, in a neural network, $h_{\theta}$ comprises all layers except the final one, while $\phi_{\theta}$ represents the last layer.
Within the lower-level objective function~$L$,
the function~$\ell: \R \times \R \mapsto \R$ represents the loss function, such as the cross-entropy loss used later on during the experiment, which penalizes deviations between the model's predictions~$\phi_{\theta}(h_{\theta}(x))$ and the true labels~$y$.
The parameter~$p\in[0,1]$ in the upper-level objective function~$G$ is chosen to promote a sparse representation.
In our experiment, we set $p = 0.5$ to enforce nonconvex regularization.

\begin{remark}
    Disclaimer: It is important to note that our primary goal is not to challenge state-of-the-art methods for this task, such as gradient-based approaches, which may be generically better suited for training neural networks.
    Rather, we consider the task of sparse representation learning as a well-understood high-dimensional benchmark bi-level optimization problem to firstly demonstrate the effectiveness of a mathematically rigorously analyzable method such as CB\textsuperscript{2}O (even without gradient term) and to secondly numerically analyze the influence of the hyperparameters on the behavior of CB\textsuperscript{2}O in practical machine learning applications.
    Following Remark~\ref{rem:gradient_drift}, it is straightforward to integrate and exploit gradient information in CB\textsuperscript{2}O, which will certainly improve the performance as observed for various other tasks~\cite{riedl2022leveraging}.
\end{remark}

\subsubsection{Experimental Setup}

Let us now describe in detail the experimental setup for the sparse representation learning task.

\textbf{Dataset \& Model Setup.}
We use a subset of the MNIST dataset \cite{lecun1998gradient}, consisting of $10,000$ training and $10,000$ test images.
Following the setup in \cite{fornasier2021convergence,riedl2022leveraging}, we use a neural network (NN) with two convolutional layers, two max-pooling layers, and one dense layer as our base model.
A batch normalization step is applied after each ReLU activation.
We denote by $h_{\theta}$ all layers up to the final dense layer,
and by $\phi_{\theta}$ the final dense layer.

\textbf{Baselines \& Implementations.} 
Since the simple bi-level optimization problem can be regarded as a constrained optimization problem, where the constraint set is defined as the global minimizers $\Theta$ of the lower-level objective function $L$,
some CBO-type algorithms designed for constrained problems \cite{carrillo2021consensus} are applicable.
However, several other constrained CBO algorithms \cite{albers2019ensemble, fornasier2020consensus_hypersurface_wellposedness, fornasier2020consensus_sphere_convergence, fornasier2021anisotropic, borghi2021constrained, carrillo2019consensus}, which depend heavily on the structure (e.g., convexity) of the constraint set, are not suitable for our problem.
This is because the constraint set $\Theta$ is not explicitly available (see discussions in Remark~\ref{remark:not_applicable_alg}). 

We compare our CB\textsuperscript{2}O algorithm as described in Algorithm~\ref{alg:cb2o_alg} with two baseline algorithms: Standard CBO \cite{carrillo2019consensus, fornasier2021convergence} and Penalized CBO \cite{carrillo2021consensus}.
The general setup for all three methods follows the setting used in \cite{fornasier2021convergence}. 
In particular, we use mini-batches when evaluating both the upper- and lower-level objective functions $G$ (if necessary) and $L$, with a mini-batch size of $\nmini = 60$.
At each time step one mini-batch of training data is processed.
One epoch comprises as many time step as are necessary to loop once through all training samples, i.e., we have $\lceil 10000/60 \rceil$ time steps per epoch.
Cooling strategies are applied to the parameters $\alpha$ and $\sigma$.
Specifically, $\alpha$ is multiplied by $2$ after each epoch, and the diffusion parameter $\sigma$ obeys $\sigmaEpoch = \sigma_0 / \log_2 (\text{epoch} + 2)$.
For the remaining parameters, we choose $\lambda = 1$, $\sigma_0 = \sqrt{0.4}$, $\alphaInit = 50$, and a step size of $\Delta t= 0.1$.
The default number of particles is set to $N = 100$ if there is no further specification.
We initialize all particles (recall that each particle contains the parameters~$\theta$ of one NN) from $\CN (0, I_{d\times d})$, 
all of which are involved in computing the consensus point and updated at every time step (referred to as the ``full update'' in \cite{carrillo2019consensus,fornasier2021convergence}).  
Additionally, we employ the re-initialization ad-hoc method introduced in \cite[Remark~2.4]{carrillo2019consensus}, where we perturb the positions of all particles by independent Gaussian noise with variance~$\sigma^2$ if the consensus point remains unchanged for $30$ consecutive time steps.

Below, we provide details on the baseline methods.
\begin{itemize}
    \item \textit{Standard CBO.} The standard CBO algorithm \cite{carrillo2019consensus, fornasier2021convergence} was implemented using exclusively the lower-level objective function $L$ as the objective.
    This baseline highlights that minimizing the classification loss $L$ alone does not automatically result in a neural network with sparse feature representations,
    motivating and necessitating the inclusion of the upper-level objective function $G$.
    It further provides a lower bound on the prediction accuracy for CBO-type methods.

    \item \textit{Penalized CBO.} The same as the one described in Section \ref{subsec:COPT_baseline_comparison}, i.e.
    applying the Standard CBO \cite{pinnau2017consensus,fornasier2021consensus} to the objective $\LM L + G$, where $\LM$ is a key hyperparameter to tune.
    
    \item \textit{CB\textsuperscript{2}O.} In addition to following the standard CBO parameter setup, we consider two types of schedulers for quantile parameter $\beta$. (i) constant $\beta$, and 
    (ii) decaying $\beta$.
    In particular, we apply a cooling strategy, where $\beta$ follows the schedule $\betaEpoch = \beta_0\DF^{\text{epoch}} $ with a decay factor $\kappa \in (0,1)$ until reaching prefixed minimum value~$\betaMin$.
\end{itemize}

\noindent \textbf{Performance Metrics.} We compare the performances of the different algorithms using three metrics:
\textit{training loss} (value of the lower-level objective function $L$ evaluated on the training set),
\textit{norm of feature representation} (value of upper-level objective function $G$ evaluated on the training set),
and \textit{test accuracy} (the model's prediction accuracy on the test set).
These three metrics are assessed based on the location of the consensus point (representing the parameters of an NN) generated by each algorithm.
Experiments are conducted with $5$ different random seeds for all algorithms, and we report the averages of the results.

\begin{remark}\label{remark:not_applicable_alg}
    Certain constrained CBO algorithms like Penalized CBO with adapted Lagrange multiplier \cite{borghi2021constrained},
    Projected CBO \cite{ fornasier2020consensus_hypersurface_wellposedness, fornasier2020consensus_sphere_convergence, fornasier2021anisotropic} and CBO with gradient force \cite{carrillo2019consensus} are not implemented here, as they may be unsuitable for the high-dimensional simple bi-level optimization problem at hand.
    For instance, in \cite{borghi2021constrained}, the adapted Lagrange multiplier is updated based on constraint violations of the particles, which
    is inapplicable since the constrained set $\Theta$ is unknown.
    Similarly, the projection step during the Projected CBO algorithms \cite{albers2019ensemble, fornasier2020consensus_hypersurface_wellposedness, fornasier2020consensus_sphere_convergence, fornasier2021anisotropic} is not feasible without explicit knowledge of the set $\Theta$. 
    Furthermore, CBO with gradient force \cite{carrillo2019consensus} requires computationally expensive matrix inversions, making it prohibitive in high-dimensional problems.
\end{remark}

\subsubsection{Experimental Results}  
Let us now report the results of our numerical experiments.

\begin{itemize}
\item \textbf{CB\textsuperscript{2}O with constant hyperparameter $\beta$.}
In Figure \ref{fig:CB2O_constant_beta}, we compare the performance of the standard CBO algorithm with CB\textsuperscript{2}O  using different constant values of the hyperparameter $\beta$.
Both algorithms employ $N=100$ particles. 
We can make the following observations:

(i) The standard CBO method minimizes exclusively the lower-level objective function $L$, leading to a strong performance in terms of training loss and test accuracy.
However, the norm of the feature representation generated by its learned neural network is large,
illustrating that minimizing $L$ alone does not result in a network with sparse feature representations, emphasizing the need for a sparsity-promoting regularization function $G$.

(ii) CB\textsuperscript{2}O with $\beta = 0.01$ behaves almost identical to the standard CBO algorithm in all three performance metrics, which aligns with our discussion in Remark~\ref{rem:beta_too_small}.

(iii) As the quantile parameter $\beta$ increases, the learned NNs exhibit smaller norms of the feature representations, but the model's performance w.r.t.\@ training loss and test accuracy gradually declines.
This indicates a trade-off between the model's prediction performance and the sparsity of feature representation.
CB\textsuperscript{2}O with $\beta=0.02$ achieves the best balance between these two factors, producing a neural network with a relatively small norm of the feature representation while maintaining nearly the same prediction performance as the model obtained by standard CBO.
This observation aligns with our theoretical justification of the choice of the quantile parameter $\beta$ in Remark~\ref{remark:beta0}, where we show that $\beta$ must be sufficiently small to ensure a good performance of the CB\textsuperscript{2}O algorithm.

Let us further compare in Figure \ref{fig:CB2O_constant_beta_N_diff} the performance of CB\textsuperscript{2}O across various combinations of the particle numbers $N$ and the quantile parameter $\beta$.
Notably, for a fixed number of particles $N$, choosing $\beta$ such that $\lceil \beta N \rceil = 2$ consistently yields the best trade-off between the model prediction accuracy and the sparsity of the feature representations.
This highlights the simplicity of tuning the quantile parameter $\beta$ for this task.
Moreover, when $\lceil \beta N \rceil = 2$, we observe, that using $N=100$ particles yields roughly the same performance as using $N=200$ particles, while $N=50$ particles result in a slightly lower prediction accuracy.

\begin{figure}[!htb]
\centering
\includegraphics[width=0.98\textwidth]{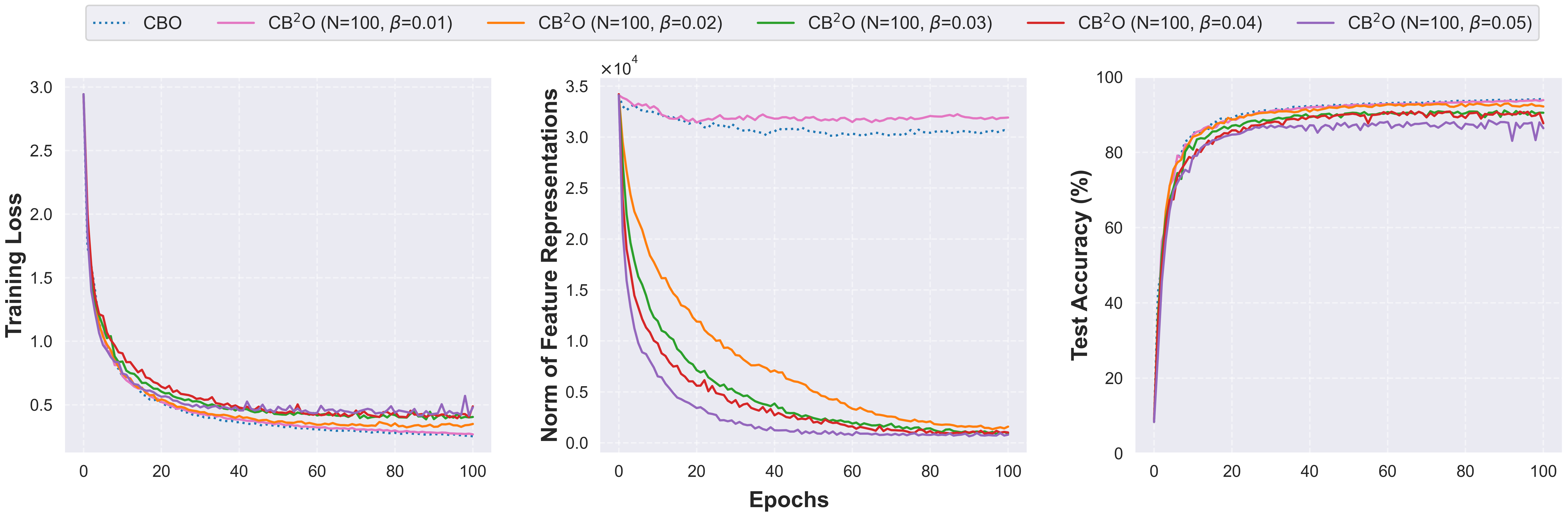}
\caption{Comparison between standard CBO and CB\textsuperscript{2}O for different constant values of the quantile hyperparameter $\beta$.
All algorithms employ $N=100$ particles.
The sub-figures from left to right are comparing the training loss $L$, the norm of the feature representation $G$, and the test accuracy against training epochs, respectively.} 
\label{fig:CB2O_constant_beta}
\end{figure}
\begin{figure}[!htb]
\centering
\begin{subfigure}[t]{0.98\textwidth}
\centering
\includegraphics[width=1.0\textwidth]{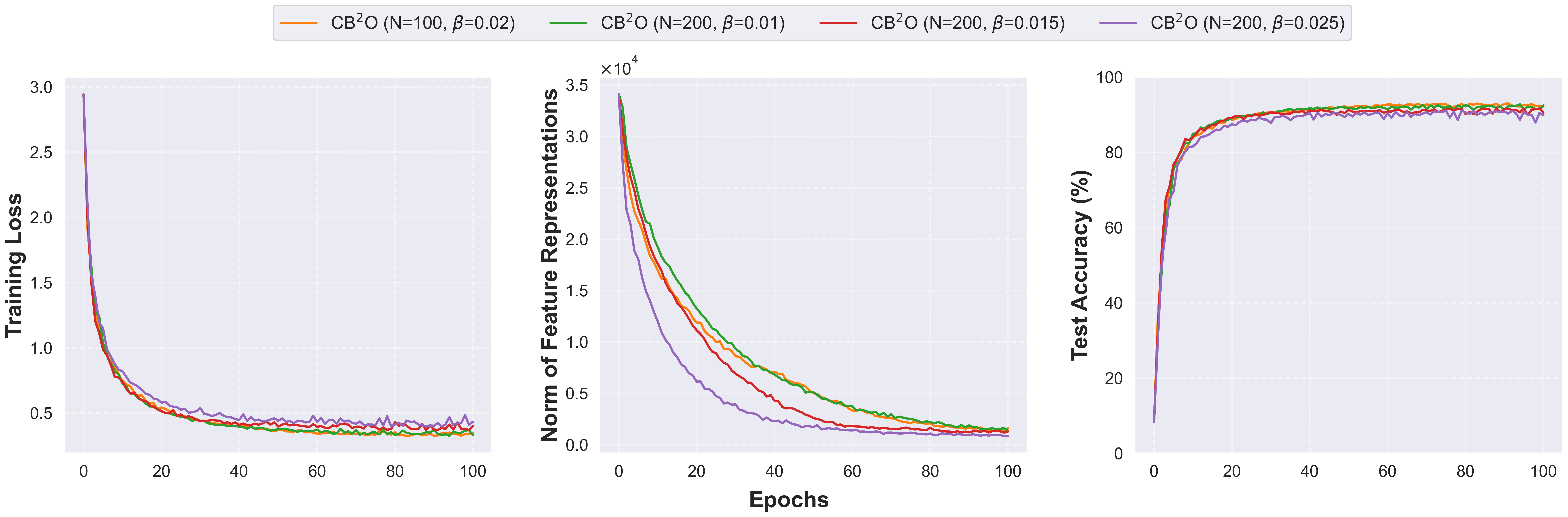}
\end{subfigure}
\begin{subfigure}[t]{0.98\textwidth}
\centering
\includegraphics[width=1.0\textwidth]{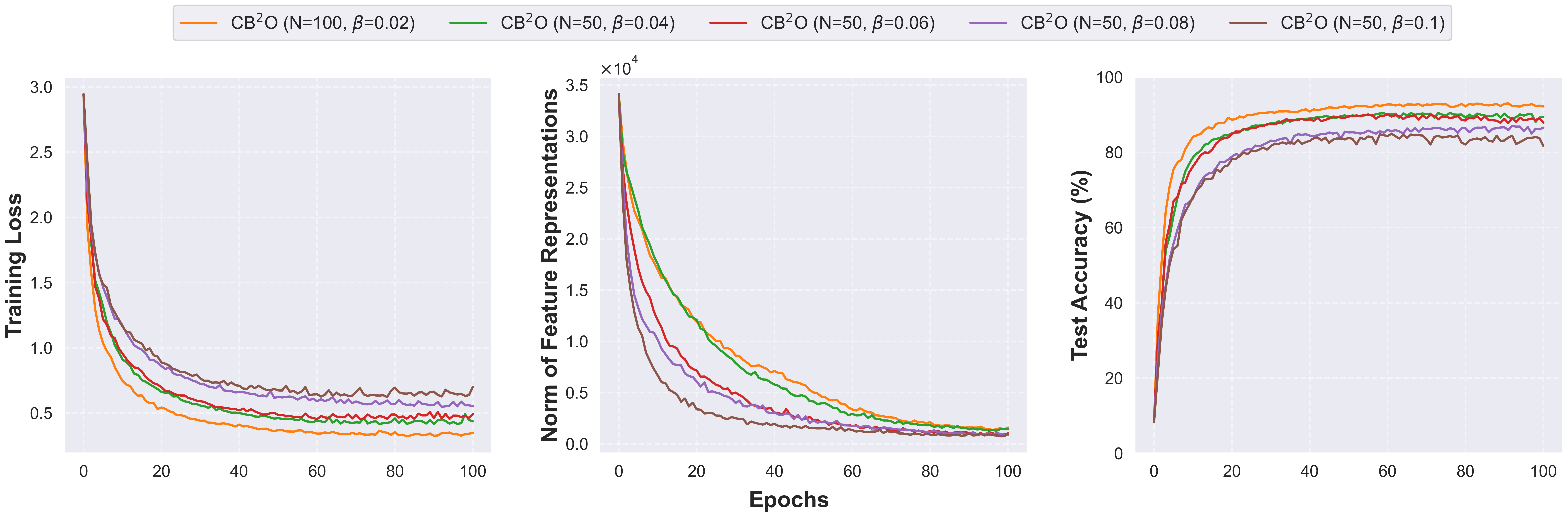}
\end{subfigure}
\caption{Comparison of CB\textsuperscript{2}O for different number~$N$ of particles and quantile hyperparameter~$\beta$.
In the top row, we compare CB\textsuperscript{2}O using $N=200$ particles and different constant values of the parameter $\beta$ with CB\textsuperscript{2}O using $N=100$ particles and $\beta=0.02$.
In the bottom row, we compare CB\textsuperscript{2}O using $N=50$ particles and different constant values of the parameter $\beta$ with CB\textsuperscript{2}O using $N=100$ particles and $\beta=0.02$.
The sub-figures from left to right are comparing the training loss $L$, the norm of the feature representation $G$, and the test accuracy against training epochs, respectively.} 
\label{fig:CB2O_constant_beta_N_diff}
\end{figure}

\item \textbf{Comparisons between CB\textsuperscript{2}O and Penalized CBO.}
We compare our CB\textsuperscript{2}O algorithm with Penalized CBO for different Lagrange multipliers $\LM$, and depict the results in Figure \ref{fig:penalized_CBO}.
We observe that Penalized CBO with $\LM = 5\times  10^4$ provides the best trade-off between model performance and the sparsity of the feature representations among all different choices of Lagrange multipliers.
Moreover, it has a similar convergence behavior as CB\textsuperscript{2}O with constant value $\beta = 0.02$ in terms of both upper- and lower-objectives $G$ and $L$.
However, it is generally harder to tune the hyperparameter $\LM$ since it requires taking into account the scaling between the objective functions $G$ and $L$.
In particular,
Penalized CBO with $\LM$ being too large will only minimize the lower-level objective $L$, and therefore behaves the same as standard CBO minimizing $L$. 
On the other hand, Penalized CBO with $\LM$ being too small will mainly minimize the objective function $G$, which will also fail the goal of finding $\thetaG$.
This phenomenon can be observed in Figure \ref{fig:penalized_CBO}.
Specifically, Penalized CBO with $\LM = 10^6, 10^7$ produces good performance on the prediction accuracy, while not enforcing the sparsity of the feature representations.
On the other extreme, smaller $\LM = 10^3, 10^4$ show good convergence in terms of the objective function $G$, but the model's prediction performance declines.
In comparison, tuning the hyperparameter $\beta$ in CB\textsuperscript{2}O is much simpler. 
For a fixed number $N$ of particles, choosing $\beta$ such that $\lceil \beta N\rceil = 2$ can already delivers good performance.
This observation aligns with similar findings in the constrained optimization task, as discussed in Section \ref{subsec:COPT_ablation_study}.

\begin{figure}[!htb]
\centering
\includegraphics[width=0.95\textwidth]{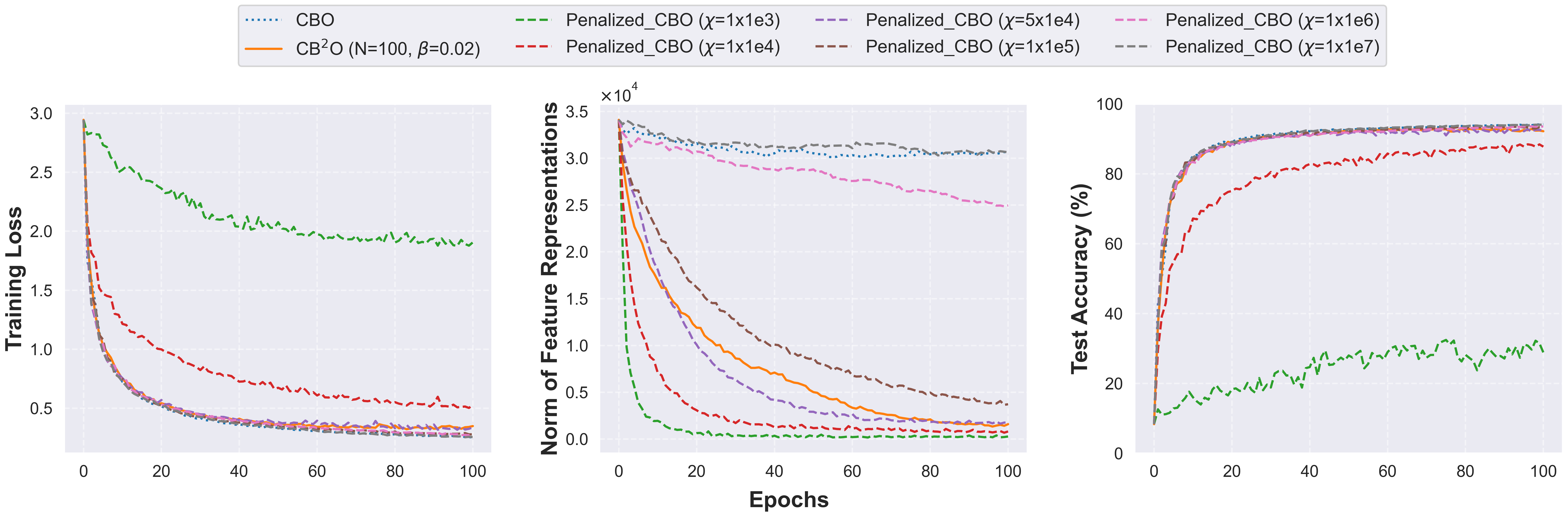}
\caption{Comparison between CB\textsuperscript{2}O with $\beta=0.02$ and Penalized CBO with different Lagrange multipliers $\LM$.
The sub-figures from left to right are comparing the training loss $L$, the norm of the feature representation $G$, and the test accuracy against training epochs, respectively.} 
\label{fig:penalized_CBO}
\end{figure}

\item \textbf{CB\textsuperscript{2}O with decaying hyperparameter $\beta$.}
Figure \ref{fig:CB2O_constant_beta} suggests that while choosing $\beta = 0.02$ gives the best overall performance by balancing the prediction accuracy and sparsity of the feature representations, a small value of the quantile hyperparameter $\beta$ also leads to a relatively slow decay of the norm of the feature representations compared to larger choices of $\beta$.
This motivates us to consider a decaying scheduler for $\beta$ such that it induces faster convergence on the function $G$, while maintaining good performance in terms of the objective function $L$.
In particular, we consider the scheduler $\betaEpoch = \min\, \{\beta_0 \DF^{\text{epoch}}, \betaMin \}$, i.e., $\beta$ with initialized value $\beta_0$ decaying every epoch with decay factor $\DF$ until it reaches a prefixed lower threshold $\betaMin$.
In Figure \ref{fig:CB2O_decay_beta}, we consider different pairs of initialized hyperparameter value $\beta_0$ and decay factor $\DF$ with minimum value $\betaMin = 0.02$.
We observe that larger $\beta_0$ enjoys faster convergence in terms of the norm of the feature representations (objective function $G$) and large decay factors $\DF$ ensure maintaining the overall model's performance (objective function $L$).

\begin{figure}[!htb]
\centering
\includegraphics[width=0.95\textwidth]{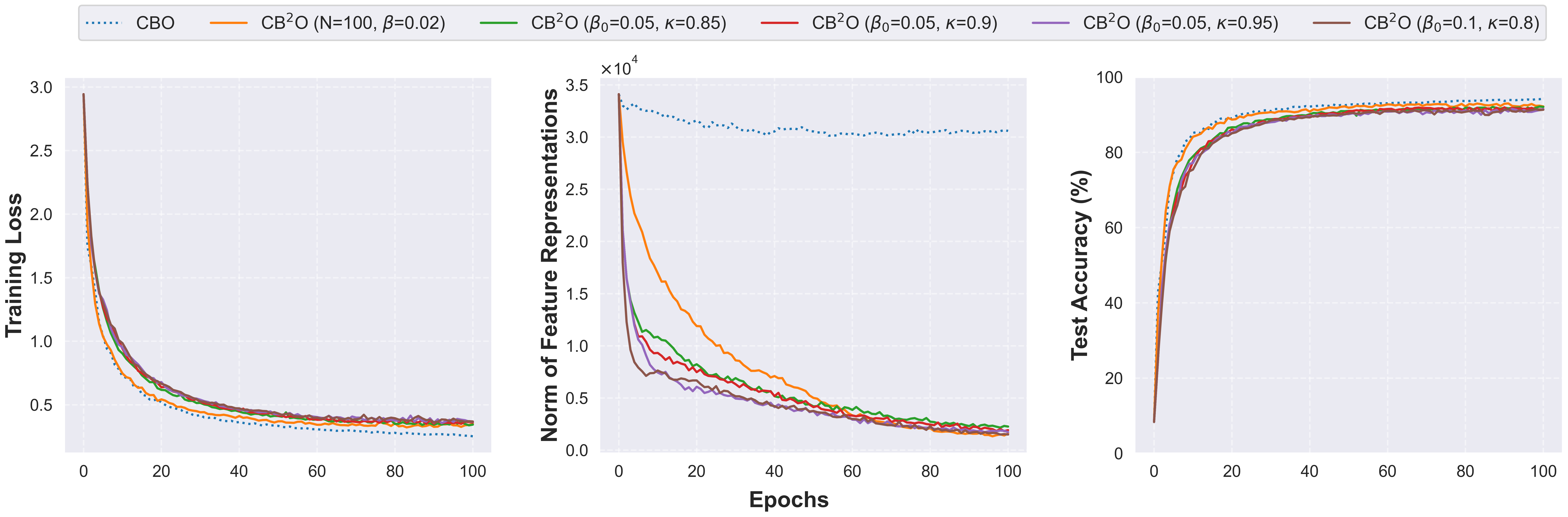}
\caption{Comparison between standard CBO and CB\textsuperscript{2}O for decaying $\beta$ with different decay factors.
The sub-figures from left to right are comparing the training loss $L$, the norm of the feature representation $G$, and the test accuracy against training epochs, respectively.} 
\label{fig:CB2O_decay_beta}
\end{figure}

\end{itemize}

\section{Conclusions}
\label{sec:conclusions}
In this paper,
we propose consensus-based bi-level optimization,
abbreviated as CB\textsuperscript{2}O,
a multi-particle metaheuristic derivative-free optimization method designed to solve nonconvex bi-level optimization problems of the form~\eqref{eq:bilevel_opt},
which are also known as simple bi-level optimization problems.
To ensure that the bi-level optimization problem is solved in an intrinsic manner, our method leverages within the computation of the consensus point a carefully designed particle selection principle implemented through a suitable choice of a quantile on the level of the lower-level objective, together with a Laplace principle-type approximation w.r.t.\@ the upper-level objective function.

On the theoretical side,
we give an existence proof of solutions to the mean-field dynamics, which is a nonlinear nonlocal Fokker-Planck equation, and establish its global convergence to the global target minimizer of the bi-level optimization problem in mean-field law.
To achieve the former, we first develop a new stability estimate for our consensus point w.r.t.\@ a combination of Wasserstein and $L^2$ perturbations,
and consecutively resort to PDE considerations extending the classical Picard iteration to construct a solution.
For the latter, we develop a novel quantitative large deviation result with an integrated selection mechanism in terms of the quantile function.
We believe that both those developments are of independent interest and may become fundamental analytical tools for evolutionary algorithms and metaheuristics in general.

On the experimental side,
we demonstrate the effectiveness of the CB\textsuperscript{2}O algorithm on both low-dimensional constrained global optimization problems and high-dimensional machine learning applications.

There are several theoretical considerations of the proposed CB\textsuperscript{2}O system that have not been explored in this paper and are left for future research. This includes the well-posedness of the finite particle CB\textsuperscript{2}O system \eqref{eq:dyn_micro}, the uniqueness of solutions to the mean-field system \eqref{eq:dyn_macro} and the mean-field approximation of the finite particle system to the mean-field system. 
Since the consensus point $\mAlphaBetanoarg$, as defined in \eqref{eq:consensus_point}, is not stable under the Wasserstein perturbations, the tools developed in \cite{carrillo2018analytical,fornasier2021consensus,huang2021MFLCBO, gerber2023mean,huang2024uniform} to prove the above-mentioned properties for the standard CBO system \cite{pinnau2017consensus} do not directly apply to the CB\textsuperscript{2}O method.
This necessitates the development of more general and less rigid proof strategies, which we leave for future work.
However, we believe that the tools developed in this paper are of foundational nature to those endeavors.

\section*{Acknowledgements}
All authors acknowledge the kind hospitality of the Institute for Computational and Experimental Research in Mathematics (ICERM) during the ICERM workshop ``Interacting Particle Systems: Analysis, Control, Learning and Computation.''

\section*{Funding}

NGT was supported by the NSF grant DMS-2236447, and, together with SL would like to thank the IFDS at UW-Madison and NSF through TRIPODS grant DMS-2023239 for their support.
KR acknowledges the financial support from the Technical University of Munich and the Munich Center for Machine 
Learning, where most of this work was done.
His work there has been funded by the German Federal Ministry of Education and Research, and the Bavarian State Ministry for Science and the Arts.
KR moreover acknowledges the financial support from the University of Oxford.
For the purpose of Open Access, the author has applied a CC BY public copyright license to any Author Accepted Manuscript (AAM) version arising from this submission.
YZ was supported by the NSF grant DMS-2411396.

\bibliographystyle{abbrv}
\bibliography{biblio}

\newpage
\appendix


\section{Technical Lemmas}
\label{sec:AppendixTechnical}

\subsection{Technical Lemmas for Section \ref{sec:conspoint_stability}}

\begin{restatable}{lemma}{lemWass}\label{lem:1dWass}
Let $L \in \CC(\bbR^d)$ satisfy Assumption \ref{asm:LipObjL}.
Moreover, let $\varrho,\widetilde\varrho \in \CP_2(\R^d)$. 
Then for any $\beta \in (0,1]$ it holds
\begin{equation}
\left| \int_{\beta/2}^{\beta} \qa{\varrho} da  - \int_{\beta/2}^{\beta} \qa{\widetilde{\varrho}} da \right| \leq C_{\beta, L} W_2 \left(\varrho, \widetilde{\varrho} \right),
\end{equation}
where $\qa{\dummy}$ is defined in \eqref{eq:q_beta} and $C_{\beta, L}$ is a constant depending on $\beta$ and Lipschitz coefficient $C_2$ defined in Assumption \ref{asm:LipObjL}.
\end{restatable}

\begin{proof}
One can compute that
\begin{equation}
\begin{aligned}
\left| \int_{\beta/2}^{\beta} \qa{\varrho} da - \int_{\beta/2}^{\beta} \qa{\widetilde{\varrho}} da \right| &\leq \sqrt{\frac{\beta}{2}} \left( \int_{0}^{1} \left| q_a^L[\varrho] - q_a^L[\widetilde{\varrho}] \right|^2 da \right)^{\frac{1}{2}}\\
&= \sqrt{\frac{\beta}{2}} W_2 \left( L_{\sharp} \varrho, L_{\sharp} \widetilde{\varrho} \right)\\
&\leq C_2 \sqrt{\frac{\beta}{2}}  W_2 \left( \varrho, \widetilde{\varrho} \right),
\end{aligned}
\end{equation}
where in the equality we have used the fact that in one dimension the Wasserstein distance between two measures can be computed in terms of the $L^2$-distance between their quantile functions, and in the last line we have used the fact that $L$ is Lipschitz.
\end{proof}

\subsection{Technical Lemmas for Section \ref{sec:well_posedness_proof}}

\begin{lemma}\label{lem:upperBound}
Let $L \in \CC(\bbR^d)$ and $G \in \CC(\bbR^d)$ satisfy \ref{asm:lowerBound_G} and \ref{asm:growthBound_G}.
Moreover, let $\varrho \in \CP_2(\R^d)$ satisfy $\int \N{\theta}_2^2 d\varrho \leq K$ and $\varrho (\Qbeta{\varrho} ) \geq c$ with $K < \infty$ and $c > 0$.
Then
\begin{equation}
    \frac{\exp(-\alpha \minobj)}{\N{\omegaa}_{L^1(\Ibeta{\varrho})}} 
    \leq \frac{1}{c} \exp\left(\frac{\alpha}{c} C_1(1 + K)\right)
    =: c_K.
\end{equation}
\end{lemma}

\begin{proof}
By the definitions of $\Ibeta{\dummy}$ and $\Qbeta{\dummy}$ in \eqref{eq:I_beta} and \eqref{eq:Q_beta}, one can compute that
\begin{equation}
\begin{aligned}
    \int_{\R^d} \exp \left(-\alpha \left(G(\theta) - \minobj\right) \right) d\Ibeta\varrho (\theta) & = \varrho \left(\Qbeta{\varrho} \right) \int_{\R^d} \frac{1}{\varrho \left(\Qbeta{\varrho} \right)} \exp \left(-\alpha \left(G(\theta) - \minobj\right) \right) d\Ibeta\varrho (\theta) \\
    &\geq \varrho \left(\Qbeta{\varrho} \right) \exp \left(-\frac{\alpha}{\varrho \left(\Qbeta{\varrho} \right)} \int_{\R^d} \left(G(\theta) - \minobj\right) d\Ibeta\varrho (\theta) \right)\\
    & \geq c \exp \left( -\frac{\alpha}{c} \int_{\R^d} (G(\theta) - \minobj) d\Ibeta{\varrho}(\theta) \right) \\
    &\geq c \exp \left(-\frac{\alpha}{c} \int_{\R^d} C_1 \left(1 + \|\theta\|_2^2\right) d\Ibeta\varrho (\theta) \right)\\
    &\geq c \exp \left(-\frac{\alpha}{c} \int_{\R^d} C_1 \left(1 + \|\theta\|_2^2\right) d\varrho (\theta) \right)\\
    & \geq c \exp\left(-\frac{\alpha}{c} C_1(1 + K)\right).
\end{aligned}
\end{equation}
Here the first inequality uses Jensen's inequality, and for the second and third inequalities we use assumption $\varrho (\Qbeta{\varrho} ) \geq c$ and the growth condition \ref{asm:growthBound_G} for $G$, respectively.
We conclude the proof by reordering the above computations.
\end{proof}

\begin{lemma}\label{lem:bound_Grad}
Let the dimension $d \geq 3$ and let $\widetilde\varrho\in\CP(\bbR^d)$.
Then, for any $r > 0$ and any $\varrho\in\CP(\bbR^d)$ it holds
\begin{equation}
    \int_{B_{r}(0)} \|\nabla \varrho\|_2 d\theta
    \leq C \left( \int_{B_r(0)} \|\nabla \varrho\|_2^2 \|\theta - \mAlphaBetaRED{\widetilde\varrho}\|_2^2 d\theta \right)^{\frac{1}{2}},
\end{equation}
for a constant $C=C(d,r)$.
\end{lemma}

\begin{proof}
\textbf{First case: $\|\mAlphaBetaRED{\widetilde\varrho}\|_2 \leq r+1$.}
By Cauchy-Schwarz inequality,
\begin{equation}
\begin{aligned}
    \int_{B_{r}(0)} \|\nabla \varrho\|_2 d\theta
    &= \int_{B_r(0)} \|\nabla \varrho\|_2 \frac{\|\theta - \mAlphaBetaRED{\widetilde\varrho}\|_2}{\|\theta - \mAlphaBetaRED{\widetilde\varrho}\|_2} d\theta\\
    &\leq \left( \int_{B_{r}(0)} \frac{1}{\|\theta - \mAlphaBetaRED{\widetilde\varrho}\|_2^2} d\theta \right)^{\frac{1}{2}} \left( \int_{B_r(0)} \|\nabla \varrho\|_2^2 \|\theta - \mAlphaBetaRED{\widetilde\varrho}\|_2^2 d\theta \right)^{\frac{1}{2}}\\
    &\leq \left( \int_{B_{2r+1}(0)} \frac{1}{\left\|\theta\right\|_2^2} d\theta \right)^{\frac{1}{2}} \left( \int_{B_r(0)} \|\nabla \varrho\|_2^2 \|\theta - \mAlphaBetaRED{\widetilde\varrho}\|_2^2 d\theta \right)^{\frac{1}{2}}\\
    &\leq \left(\int_0^{2r+1} \int_{\partial B_r(0)} \frac{1}{r^2} dS \right)^{\frac{1}{2}} \left( \int_{B_r(0)} \|\nabla \varrho\|_2^2 \|\theta - \mAlphaBetaRED{\widetilde\varrho}\|_2^2 d\theta \right)^{\frac{1}{2}} \\
    &= \text{Vol}\left(S^{d-1} \right)^{\frac{1}{2}} \left( \frac{(2r+1)^{d-2}}{d-2} \right)^{\frac{1}{2}} \left( \int_{B_r(0)} \|\nabla \varrho\|_2^2 \|\theta - \mAlphaBetaRED{\widetilde\varrho}\|_2^2 d\theta \right)^{\frac{1}{2}},
\end{aligned}
\end{equation}
where $\text{Vol}(S^{d-1})$ denotes the volume of unit sphere in $\R^d$.

\textbf{Second case: $\|\mAlphaBetaRED{\widetilde\varrho}\|_2 > r+1$.}
By Cauchy-Schwarz inequality,
\begin{equation}
\begin{aligned}
\int_{B_{r}(0)} \|\nabla \varrho\|_2 d\theta &\leq \text{Vol}(B_r(0))^{\frac{1}{2}} \left( \int_{B_r(0)} \|\nabla \varrho\|_2^2 d\theta \right)^{\frac{1}{2}}\\
&\leq \text{Vol}(B_r(0))^{\frac{1}{2}} \left( \int_{B_r(0)} \|\nabla \varrho\|_2 \|\theta - \mAlphaBetaRED{\widetilde\varrho}\|_2^2 d\theta \right)^{\frac{1}{2}}.
\end{aligned}
\end{equation}
\end{proof}

\begin{lemma}\label{lem:L_infinity_norm_bound}
Given fixed $T > 0$, $\lambda, \sigma > 0$, and $u \in \CC \left([0,T], \R^d \right)$, let $\rho $ be the solution to the following linear Fokker-Planck equation
\begin{equation}
    \partial_t \rho_t = \lambda \nabla \cdot \left( \left( \theta - u_t \right)\rho_t \right) + \frac{\sigma^2}{2} \Delta \left( \left\|\theta - u_t \right\|_2^2 \rho_t \right)
\end{equation}
in the weak sense with $\rho_0 \in L^{\infty} (\R^d)$.
Then there exists a constant $C=C(T, d, \lambda, \sigma, \lVert u \rVert_{C([0,T], \R^d)}) > 0$ such that $\|\rho_t\|_{L^{\infty} (\R^d)} \leq C$ for all $t \in [0,T]$.
\end{lemma}

\begin{proof}
For any $2 \leq p < \infty$, we can compute that
\begin{equation}
\begin{aligned}
\frac{1}{p} \frac{d}{dt} \|\rho\|_{L^p(\R^d)}^p &= \left\langle \partial_t \rho_t, \rho_t^{p-1} \right\rangle_{L^2(\R^d)}\\
&= \lambda \int \nabla \cdot \left( \left(\theta - u_t \right) \rho_t \right) \rho_t^{p-1} d\theta + \frac{\sigma^2}{2} \int \Delta \left( \left\|\theta - u_t \right\|_2^2 \rho_t\right) \rho_t^{p-1}d\theta\\
&= -\lambda (p-1) \int \left(\theta - u_t \right)^T \nabla \rho_t \rho_t^{p-1} d\theta - \frac{\sigma^2}{2} (p-1) \int \nabla \left( \left\|\theta - u_t \right\|_2^2 \rho_t \right)^T \nabla \rho_t \rho_t^{p-2} d\theta\\
&= - (p-1)\left(\lambda + \sigma^2 \right)  \int \left(\theta - u_t \right)^T \nabla \rho_t \rho_t^{p-1} d\theta - \frac{\sigma^2}{2} (p-1) \int \left\| \theta - u_t \right\|_2^2 \left\| \nabla \rho_t \right\|_2^2 \rho_t^{p-2} d\theta\\
&= - (p-1)\left(\lambda + \sigma^2 \right)  \int \left(\theta - u_t \right)^T \nabla \left( \frac{1}{p} \rho_t^p \right) d\theta - \frac{\sigma^2}{2} (p-1) \int \left\| \theta - u_t \right\|_2^2 \left\| \nabla \rho_t \right\|_2^2 \rho_t^{p-2} d\theta\\
&= \frac{p-1}{p} d \left(\lambda + \sigma^2 \right) \int |\rho_t|^p d\theta - \frac{\sigma^2}{2} (p-1) \int \left\| \theta - u_t \right\|_2^2 \left\| \nabla \rho_t \right\|_2^2 \rho_t^{p-2} d\theta\\
&\leq \frac{p-1}{p} d \left(\lambda + \sigma^2 \right) \left\| \rho_t \right\|^p_{L^p(\R^d)} .
\end{aligned}
\end{equation}
By Grönwall's inequality, we know that
\begin{equation}
    \left\| \rho_t \right\|_{L^p(\R^d)} \leq \exp \left( \frac{p-1}{p} d (\lambda + \sigma^2) t \right) \left\| \rho_0 \right\|_{L^p(\R^d)} .
\end{equation}
Therefore, for any $t \in [0,T]$ we have
\begin{equation}
\begin{aligned}
\left\| \rho_t \right\|_{L^{\infty}(\R^d)} = \lim_{p \rightarrow \infty} \left\| \rho_t \right\|_{L^p(\R^d)} &\leq \lim_{p \rightarrow \infty} \exp \left( \frac{p-1}{p} d (\lambda + \sigma^2) t \right) \left\| \rho_0 \right\|_{L^p(\R^d)}\\
&= \exp \left( d (\lambda + \sigma^2) t \right) \left\| \rho_0 \right\|_{L^{\infty}(\R^d)}\\
&\leq \exp \left( d (\lambda + \sigma^2) T \right) \left\| \rho_0 \right\|_{L^{\infty}(\R^d)} .
\end{aligned}
\end{equation}
\end{proof}

\begin{lemma}\label{lem:Cb_conv_to_L2_conv}
Let $\{\rho_{\tau}\}_{\tau \in [0,T]} \subset L^2 (\R^d)$ and $\|\rho_{\tau}\|_{L^2(\R^d)} \leq C$ with a uniform constant $C > 0$ for any $\tau \in [0,T]$. Suppose that for every $s$ and every $\phi \in \CC_b(\R^d)$ we have 
\[ \lim_{t \rightarrow s} \int \phi(\theta) \rho_t(\theta) d \theta = \int \phi(\theta) \rho_s(\theta) d \theta.\]
Then, for every $s\in [0,T]$ we have that $\rho_t$ converges weakly in $L^2(\R^d)$ towards $\rho_s$ as $t \rightarrow s$.
\end{lemma}

\begin{proof}
    Since $\CC_c(\R^d)$ is dense in $L^2(\R^d)$, for any $\phi \in L^2(\R^d)$ and $\varepsilon > 0$, there exists a function $\psi \in \CC_c(\R^d)$ such that $\|\phi - \psi\|_{L^2 (\R^d)} < \varepsilon/2$.
Applying the triangle inequality, Cauchy-Schwarz inequality, and the fact that $\|\rho_{\tau}\|_{L^2(\R^d)} \leq C$ uniformly for all $\tau \in [0,T]$, we deduce that
\begin{equation}
\begin{aligned}
&\left| \int \phi(\theta) \rho_t(\theta) d\theta - \int \phi(\theta) \rho_s(\theta) d\theta \right|\\
&\qquad\,\leq \left| \int \left( \phi(\theta) - \psi(\theta) \right) \rho_t(\theta) d\theta \right| + \left| \int \psi(\theta) \left(\rho_t(\theta) - \rho_s(\theta) \right) d\theta \right| + \left| \int \left( \phi(\theta) - \psi(\theta) \right) \rho_s(\theta) d\theta \right|\\
&\qquad\,\leq \left\| \phi - \psi \right\|_{L^2(\R^d)} \left( \|\rho_t\|_{L^2(\R^d)} + \|\rho_s\|_{L^2(\R^d)} \right) + \left| \int \psi(\theta) \left(\rho_t(\theta) - \rho_s(\theta) \right) d\theta \right|\\
&\qquad\, < C\varepsilon + \left| \int \psi(\theta) \left(\rho_t(\theta) - \rho_s(\theta) \right) d\theta \right| .
\end{aligned}
\end{equation}
The latter term in the last line above converges to $0$ as $t \rightarrow s$ by the assumption and the fact $\psi \in \CC_c(\R^d) \subset \CC_b(\R^d)$, implying that
\begin{equation}
    \limsup_{t \rightarrow s} \left| \int \phi(\theta) \rho_t(\theta) d\theta - \int \phi(\theta) \rho_s(\theta) d\theta \right| \leq C\varepsilon,
\end{equation}
which implies the desired result since $\varepsilon$ was arbitrary.
\end{proof}

The following is a standard fact from functional analysis.

\begin{lemma}\label{lem:limiting_func_Sob_reg}
For $l \geq 1$, let $\{u^n\}_n \subset W^{1,\infty}([0,T], H^l(\R^d))$ be a sequence of functions with a uniformly bounded norm in this space.
Suppose $u$ is such that $u^n \rightarrow u$ strongly in $\CC ([0,T], L_{\text{loc}}^2 (\R^d))$.
Then $u$ has the regularity $u \in W^{1, \infty} ([0,T], H^l(\R^d))$.
\end{lemma}

\begin{proof}

By the Banach-Alaoglu theorem \cite{rudin1991functional} (applied in the Banach space $W^{1,\infty} ([0,T], H^l(\R^d))$), there exists a subsequence of $u^n$ (not relabeled) and a function $\tilde u$ such that
\begin{equation*}
u^n \rightharpoonup^* \tilde u \qquad \text{in}\;\; W^{1,\infty} ([0,T], H^l (\R^d)) \, .
\end{equation*}
That is, for any test function $\phi \in L^1 ([0,T], H^l(\R^d))$, we have
\begin{subequations}
\begin{equation}\label{eq:weak_star_conv}
\lim_{n \rightarrow +\infty} \int_0^T \langle u^n, \phi(t, \dummy) \rangle_{H^l(\R^d)} dt = \int_0^T \langle \tilde u, \phi(t, \dummy) \rangle_{H^l(\R^d)} dt \, ,
\end{equation}
\begin{equation}\label{eq:weak_star_conv_deri}
\lim_{n \rightarrow +\infty} \int_0^T \langle \partial_t u^n, \phi(t, \dummy) \rangle_{H^l(\R^d)} dt = \int_0^T \langle \partial_t \tilde u, \phi(t, \dummy) \rangle_{H^l(\R^d)} dt \, .
\end{equation}
\end{subequations}
Our goal now is to show that $\tilde u =u$. In turn, it will suffice to show that $\tilde u \xi = u \xi$ for any arbitrary smooth cutoff function $\chi \in \CC_c^{\infty}(\R^d)$. 

Let $v^n(t, \theta) := \chi(\theta) u^n(t,\theta)$ and $v(t, \theta) := \chi(\theta) \tilde u(t,\theta)$. From the strong convergence of $u^n \rightarrow u$ in $\CC ([0,T], \Lloc (\R^d))$, we know that $v^n = \chi u^n \rightarrow \chi u$ strongly in $\CC ([0,T], L^2(\R^d))$.
Thus, together with \eqref{eq:weak_star_conv}, we obtain, for any $\phi \in L^1([0,T], H^l(\R^d))$, 
\begin{equation*}
\int_0^T \langle v, \phi(t, \dummy) \rangle_{L^2(\R^d)} dt = \lim_{n \rightarrow +\infty} \int_0^T \langle v^n, \phi(t, \dummy) \rangle_{L^2(\R^d)} dt = \int_0^T \langle \chi u, \phi(t, \dummy) \rangle_{L^2(\R^d)} dt.
\end{equation*}
Using a density argument, the above implies that $v= \chi u$, finishing with this the proof.

\end{proof}

\begin{lemma}\label{lem:limiting_func_L_infty_reg}
Let $\{u^n\}_n \subset L^{\infty} ([0,T] \times \R^d)$ be a sequence of functions with a uniformly bounded norm in this space.
Suppose $u$ is such that $u^n \rightarrow u$ strongly in $\CC ([0,T], \Lloc(\R^d))$.
Then $u$ has the regularity $u \in L^{\infty} ([0,T], L^{\infty} (\R^d))$.
\end{lemma}

\begin{proof}
Since $u^n$ strongly converges to $u$ in $\CC ([0,T], \Lloc(\R^d))$, it follows that for any fixed compact set $K \subset \R^d$, up to subsequence (not relabeled), we have 
\begin{equation*}
\lim_{n \rightarrow +\infty }u^n (t, x) = u(t,x) \qquad \text{for a.e. } (t,x) \in [0,T] \times K \, .
\end{equation*}
Thus, we conclude that
\begin{equation*}
    |u(t,x)| = \lim_{n \rightarrow +\infty}| u^n(t,x)| \leq \sup_n \|u^n\|_{L^{\infty}([0,T] \times \R^d)} < +\infty
\end{equation*}
for a.e. $(t,x) \in [0,T] \times K$.
By the arbitrariness of $K \subset \R^d$, we conclude that $u \in L^{\infty} ([0,T], L^{\infty} (\R^d))$.
\end{proof}

\begin{lemma}\label{lem:ContinuityInTime}
Let $\rho \in L^2 \left([0,T], L^2(\R^d) \right)$ satisfy 
\begin{enumerate}
    \item $\lim_{t \rightarrow s} \|\rho_t\|^2_{L^2(\R^d)} = \|\rho_s\|^2_{L^2 (\R^d)}$ for any $s \in [0,T]$.
    \item For any $s\in [0,T]$, $\rho_t$ 
 weakly converges to $\rho_s$ in $L^2 (\R^d)$ as $t \rightarrow s$, i.e.
 \begin{equation}
     \lim_{t \rightarrow s}\int_{\R^d} \varphi(\theta) \rho_t(\theta) d\theta = \int_{\R^d} \varphi(\theta) \rho_s(\theta) d\theta,
 \end{equation}
 for every $\varphi \in L^2(\R^d)$.
\end{enumerate}
Then $\rho_t$ strongly converges to $\rho_s$ as $t \rightarrow s$ in $L^2(\R^d)$.
 In other words, $\rho \in \CC \left([0,T], L^2(\R^d) \right)$.
\end{lemma}

\begin{proof}
For any $s \in [0,T]$, we have
\begin{equation}
\begin{aligned}
\lim_{t \rightarrow s} \left\| \rho_t - \rho_s \right\|^2_{L^2 (\R^d)} &= \lim_{t \rightarrow s}  \|\rho_t\|^2_{L^2(\R^d)} - 2 \lim_{t \rightarrow s} \langle \rho_t, \rho_s \rangle_{L^2(\R^d)} + \|\rho_s\|^2_{L^2(\R^d)}\\
&= \|\rho_s\|^2_{L^2(\R^d)} - 2\|\rho_s\|^2_{L^2(\R^d)} + \|\rho_s\|^2_{L^2(\R^d)} = 0,  
\end{aligned}
\end{equation}
where the second equality uses the fact that $\phi_s \in L^2(\R^d)$ and $\rho_t$ converges to $\rho_s$ weakly in $L^2(\R^d)$.
\end{proof}

\begin{lemma}\label{lem:WassConv}
Suppose that $\varrho$ is a density function on $\R^d$ and that $\{ \varrho^n \}_{n \in \bbN} \subset \CP_4(\R^d)$ is a sequence of densities satisfying
\begin{enumerate}
\item $\varrho^n$ strongly converges to $\varrho$ in $L^2_{\text{loc}}(\R^d)$,
\item  $\sup_{n \in \bbN} \int \|\theta\|_2^4 \varrho^n(\theta) d\theta \leq C $ for some constant $C > 0$.
\end{enumerate}
Then $\{\varrho^n\}_{n \in \bbN}$ converges to $\varrho$ in the Wasserstein-$2$ distance, i.e.,
\[ \lim_{n \rightarrow \infty} W_2(\varrho^n, \varrho)=0 . \]
\end{lemma}

\begin{proof}
Since weak convergence of probability measures and convergence in second moments implies convergence in the Wasserstein-$2$ distance (see, e.g., \cite[Chapter~6]{villani20090oldandnew}), we will show that $\varrho^n$ converges to $\varrho$ weakly and that $\int_{\R^d} \|\theta\|_2^2 \varrho^n d\theta \rightarrow \int_{\R^d} \|\theta\|_2^2 \varrho d\theta$ as $n \rightarrow \infty$.

For given $r > 0$, we have
\begin{equation}\label{eq:auxIneq}
\int_{\R^d \backslash B_r(0)} \|\theta\|_2^2 \varrho^n(\theta) d\theta \leq \frac{1}{r^2} \int_{\R^d \backslash B_r(0)} \|\theta\|_2^4 \varrho^n(\theta) d\theta \leq \frac{C}{r^2} .
\end{equation}
Then for any $\varepsilon > 0$, we pick $r_{\varepsilon} = \sqrt{\frac{C}{\varepsilon}}$, which ensures that $\int_{\R^d \backslash B_{r_{\varepsilon}}(0)} \|\theta\|_2^2 \varrho^n(\theta) d\theta \leq \varepsilon$ for $n \in \bbN$. 
From Lemma \ref{lem:semiContinuity}, we know that $\int \|\theta\|_2^4 \varrho d\theta \leq C$.
Then the same computation in \eqref{eq:auxIneq} holds for $\varrho$, and we also have $\int_{\R^d \backslash B_{r_{\varepsilon}}(0)} \|\theta\|_2^2 \varrho(\theta) d\theta \leq \varepsilon$.
Then one can compute
\begin{equation}
\begin{aligned}
&\left| \int_{\R^d} \|\theta\|_2^2 \varrho^n(\theta) d\theta - \int_{\R^d} \|\theta\|_2^2 \varrho(\theta) d\theta \right|^2 \\
&= \left| \int_{B_{r_{\varepsilon}}(0)} \|\theta\|_2^2 \left(\varrho^n(\theta) - \varrho(\theta) \right)d\theta + \left(\int_{\R^d \backslash B_{r_{\varepsilon}}(0)} \|\theta\|_2^2 \varrho^n(\theta) d\theta - \int_{\R^d \backslash B_{r_{\varepsilon}}(0)} \|\theta\|_2^2 \varrho(\theta) d\theta \right) \right|^2\\
&\leq 2 \left| \int_{B_{r_{\varepsilon}}(0)} \|\theta\|_2^2 \left(\varrho^n(\theta) - \varrho(\theta) \right)d\theta \right|^2 + 2\left|\int_{\R^d \backslash B_{r_{\varepsilon}}(0)} \|\theta\|_2^2 \varrho^n(\theta) d\theta - \int_{\R^d \backslash B_{r_{\varepsilon}}(0)} \|\theta\|_2^2 \varrho(\theta) d\theta  \right|^2\\
&\leq 2 \left( \int_{B_{r_{\varepsilon}}(0)} \|\theta\|_2^4 d\theta \right) \left\| \varrho^n - \varrho \right\|^2_{L^2(B_{r_{\varepsilon}(0)})} + 4 \left(\left| \int_{\R^d \backslash B_{r_{\varepsilon}}(0)} \|\theta\|_2^2 \varrho^n(\theta) d\theta \right|^2 + \left| \int_{\R^d \backslash B_{r_{\varepsilon}}(0)} \|\theta\|_2^2 \varrho(\theta) d\theta \right|^2 \right)\\
&\leq 2 C_{r_{\varepsilon}} \left\| \varrho^n - \varrho \right\|^2_{L^2(B_{r_{\varepsilon}(0)})} + 8 \varepsilon^2 .
\end{aligned}
\end{equation}
By the fact that $\lim_{n \rightarrow \infty} \|\varrho^n - \varrho\|^2_{L^2(B_{r_{\varepsilon}})} = 0$, we have
\begin{equation}
\limsup_{n \rightarrow \infty} \left| \int_{\R^d} \|\theta\|_2^2 \varrho^n(\theta) d\theta - \int_{\R^d} \|\theta\|_2^2 \varrho(\theta) d\theta \right|^2 \leq 8 \varepsilon^2 .
\end{equation}
By the arbitrariness of $\varepsilon$, we know that $\varrho^n$ converges to $\varrho$ in second moment.
Next we show that $\varrho^n$ weakly converges to $\varrho$ with similar idea.
For any given $\varphi \in C_b(\R^d)$, denote $M := \sup_{\theta \in \R^d} |\varphi(\theta)|$.
For any $\varepsilon > 0$, with the similar computations as in \eqref{eq:auxIneq}, we could pick $\widetilde{r}_{\varepsilon} = \sqrt{\frac{M}{\varepsilon}}$ such that
\begin{equation}
\int_{\R^d \backslash B_{\widetilde{r}_{\varepsilon}}(0)} |\varphi(\theta)| \varrho^n(\theta) d\theta \leq \varepsilon \quad \text{and} \quad \int_{\R^d \backslash B_{\widetilde{r}_{\varepsilon}}(0)} |\varphi(\theta)| \varrho(\theta) d\theta \leq \varepsilon .
\end{equation}
Then 
\begin{equation}
\begin{aligned}
&\left| \int_{\R^d} \varphi(\theta) \varrho^n(\theta) d\theta - \int_{\R^d} \varphi(\theta) \varrho(\theta) d\theta \right|^2 \\
& \leq 2\left| \int_{B_{\widetilde{r}_{\varepsilon}}(0)} \varphi(\theta) \left(\varrho^n(\theta) - \varrho(\theta) \right) d\theta \right|^2 + 2\left|\int_{\R^d \backslash B_{\widetilde{r}_{\varepsilon}}(0)} \varphi(\theta) \varrho^n(\theta) d\theta - \int_{\R^d \backslash B_{\widetilde{r}_{\varepsilon}}(0)} \varphi(\theta) \varrho(\theta) d\theta  \right|^2\\
&\leq 2 \left(\int_{B_{\widetilde{r}_{\varepsilon}}(0)} \varphi(\theta)^2 d\theta \right) \left\| \varrho^n - \varrho \right\|^2_{L^2(B_{\widetilde{r}_{\varepsilon}}(0))} + 4 \left(\left| \int_{\R^d \backslash B_{\widetilde{r}_{\varepsilon}}(0)} \varphi(\theta) \varrho^n(\theta) d\theta \right|^2 + \left| \int_{\R^d \backslash B_{\widetilde{r}_{\varepsilon}}(0)} \varphi(\theta) \varrho(\theta) d\theta \right|^2 \right)\\
&\leq C(r_{\varepsilon}, M) \left\| \varrho^n - \varrho \right\|^2_{L^2(B_{\widetilde{r}_{\varepsilon}}(0))} + 8\varepsilon^2,
\end{aligned}
\end{equation}
which implies that
\begin{equation}
\limsup_{n \rightarrow \infty} \left| \int_{\R^d} \varphi(\theta) \varrho^n(\theta) d\theta - \int_{\R^d} \varphi(\theta) \varrho(\theta) d\theta \right|^2 \leq 8 \varepsilon^2 .
\end{equation}
Again by the arbitrariness of $\varepsilon > 0$ and $\varphi \in C_b(\R^d)$, we obtain that $\varrho^n$ weakly converges to $\varrho$.
\end{proof}

\begin{lemma}\label{lem:semiContinuity}
Let $\varrho$ and $\{\varrho^n\}_{n \in \mathbb{N}}$ satisfy the same assumptions as in Lemma \ref{lem:WassConv}.
Then
\begin{equation}
    \int_{\R^d} \|\theta\|_2^4 \varrho d\theta \leq \liminf_{n \rightarrow \infty} \int_{\R^d} \|\theta\|_2^4 \varrho^n d\theta .
\end{equation}
\end{lemma}
\begin{proof}
For any fixed $r > 0$, we choose a $C^{\infty}_c(\R^d)$ bump function $\psi_r$ that has support in $B_{2r}(0)$ with $|\psi_r| \leq 1$ and is identically $1$ on $B_r(0)$.
In particular, we can construct
\begin{equation}\label{eq:bumpFuncton}
\psi_r(\theta) := 1 - g \left( \frac{\|\theta\|_2^2 - r^2}{3r^2} \right)
\end{equation}
where
\begin{equation}
g(x) := \frac{f(x)}{f(x) + f(1-x)} \qquad\text{with}\qquad f(x) := e^{-\frac{1}{x}} \mathds{1}_{\{x > 0\}}.
\end{equation}
Then one can compute
\begin{equation}
\int_{\R^d} \N{\theta}_2^4 \varrho d\theta = \lim_{r \rightarrow \infty}\int_{\R^d} \N{\theta}_2^4 \psi_r(\theta) \varrho d\theta = \lim_{r \rightarrow \infty} \lim_{n \rightarrow \infty} \int_{\R^d} \N{\theta}_2^4 \psi_r(\theta) \varrho^n d\theta \leq \liminf_{n \rightarrow \infty} \int_{\R^d} \N{\theta}_2^4 \varrho^n d\theta.
\end{equation}
Here, the first equality is due to the monotone convergence theorem (Theorem of Beppo-Levi) making use of the fact that the sequence $\N{\dummy}_2^4\psi_r(\dummy)\varrho$ is monotonously non-decreasing since $\psi_r (\dummy) \leq \psi_{\tilde{r}} (\dummy)$ if $ 0< r \leq \tilde{r}$.
The second equality uses that $\varrho^n$ strongly converges to $\varrho$ in $L^2_{\text{loc}}(\R^d)$ together with $\N{\dummy}_2^4\psi_r(\dummy)\in L^2_{\text{loc}}(\R^d)$ for fixed $r>0$.
Eventually, the last inequality just uses that $\psi_r\leq1$.
\end{proof}

\newpage

\newpage

\section{Supplementary Experimental Details for Constrained Optimization}\label{sec:additional_copt}


\subsection*{Baseline Comparisons}
We provide details on the selection of algorithm-specific hyperparameters for both CB\textsuperscript{2}O and Adaptive Penalized CBO algorithms studied in Section \ref{subsec:COPT_baseline_comparison}.
\begin{itemize}
    \item \textbf{CB\textsuperscript{2}O.}\\
    For Table \ref{tab:circular_N100} and Figure \ref{subfig:circle_same_particles}, $\beta = 1/20$;\\
    For Table \ref{tab:circular} and Figure \ref{subfig:circle_same_time}, $\beta = 1/30$;\\
    For Table \ref{tab:star_N100} and Figure \ref{subfig:star_same_particles}, $\beta = 1/20$;\\
    For Table \ref{tab:star} and Figure \ref{subfig:star_same_time}, $\beta = 1/40$.

    \item \textbf{Adaptive Penalized CBO.} The Adaptive Penalized CBO adapts the penalization parameter $\LM_k$ by introducing two other hyperparameters $\eta_{\LM}$ and $\eta_{\zeta}$. 
    We test various hyperparameter configurations to identify the best performance in each scenario, and the final hyperparameter settings are as follows:\\
    For Table \ref{tab:circular_N100} and Figure \ref{subfig:circle_same_particles}, $\LM_0 = 1$, $\eta_{\LM}=1.1$, $\zeta_0 = 0.1$ and $\eta_{\zeta} = 1.4$;\\
    For Table \ref{tab:circular} and Figure \ref{subfig:circle_same_time}, $\LM_0 = 10$, $\eta_{\LM}=1.05$, $\zeta_0 = 0.1$ and $\eta_{\zeta} = 1.4$;\\
    For Table \ref{tab:star_N100} and Figure \ref{subfig:star_same_particles}, $\LM_0 = 1$, $\eta_{\LM}=1.1$, $\zeta_0 = 0.1$ and $\eta_{\zeta} = 1.4$;\\
    For Table \ref{tab:star} and Figure \ref{subfig:star_same_time}, $\LM_0 = 50$, $\eta_{\LM}=1.05$, $\zeta_0 = 0.1$ and $\eta_{\zeta} = 1.4$. 
\end{itemize}

\subsection*{Ablation Study}
We provide details on the hyperparameter settings used in the ablation study conducted in Section \ref{subsec:COPT_ablation_study}.
Specifically, we set the reference hyperparameter values as follows:
\begin{equation}\label{eq:reference_hyper}
    N = 10^3, \qd \beta = 1/20,\qd \alpha = 30, \qd \estop = 0 \, .
\end{equation} 
In each ablation experiment, one or two hyperparameters are varied at a time, while the others are kept fixed as their reference values.

For Figures \ref{subfig:different_N_distance} and \ref{subfig:different_N_runningtime}, we vary the number of particles as $N = 10^2$, $10^3$, $10^4$, and $10^5$.

For Figures \ref{subfig:different_beta_distance} and \ref{subfig:different_beta_runningtime}, we test $\beta = 1/500$, $1/100$, $1/50$, $1/20$, $1/10$, $1/5$, and $1/2$. 

For Figures \ref{subfig:N_and_beta_distance} and \ref{subfig:N_and_beta_runningtime}, we vary $(N,\beta)$ as follows: $(10^2, 0.5)$, $(10^3, 0.05)$, $(10^4, 0.005)$, and $(10^5, 0.0005)$.

For Figures \ref{subfig:different_alpha_distance} and \ref{subfig:different_alpha_runningtime}, we vary $\alpha = 2$, $10$, $30$, $50$, and $100$. 

For Figures \ref{subfig:different_estop_distance} and \ref{subfig:different_estop_runningtime}, we test $\estop = 0$, $10^{-5}$, $10^{-4}$, $10^{-3}$, $10^{-2}$, and $10^{-1}$.


\end{document}